\documentclass[a4paper,10pt]{article}
\usepackage{amssymb,amsmath,amsthm}
\usepackage{graphicx}

\usepackage{amsfonts}
\usepackage{latexsym}
\usepackage[english]{babel}
\usepackage{hyperref}

\numberwithin{equation}{section}

\newtheorem{theorem}{Theorem}[section]
\newtheorem{proposition}{Proposition}[section]
\newtheorem{lemma}{Lemma}[section]
\newtheorem{corollary}{Corollary}[section]

\newtheorem{remark}{Remark}[section]
\newtheorem{remarks}{Remark}[section]
\newtheorem{definition}{Definition}[section]
\newcommand{\be}{\begin{equation}}
\newcommand{\ee}{\end{equation}}

\newcommand{\e}{\varepsilon}

\newcommand{\R}{\mathbb R}
\newcommand{\C}{\mathbb C}

\newcommand{\Z}{\mathbb Z}

\newcommand{\N}{\mathbb N}
\newcommand{\T}{\mathbb T}

\renewcommand{\b }{\beta }

\newcommand{\ii }{{\rm i} }

\renewcommand{\l }{\lambda }

\newcommand{\p}{\pi}

\begin{document}

\title{{\bf Normal form coordinates for the KdV equation 
having expansions in terms of pseudodifferential operators
}}



\author{
Thomas Kappeler\footnote{Supported in part by the Swiss National Science Foundation.}  ,
Riccardo Montalto\footnote{Supported in part by the Swiss National Science Foundation.} 
}

\maketitle

\noindent
{\bf Abstract.} 
Near an arbitrary finite gap potential we construct real analytic, canonical coordinates
for the KdV equation on the torus having the following two main properties:  (1) up to a remainder term, which is smoothing to any given order, 
the coordinate transformation is a pseudodifferential operator of order 0 with principal part given by the Fourier transform and  
(2) the pullback of the KdV Hamiltonian is in normal form up to order three and
the corresponding Hamiltonian vector field admits an expansion in terms of a paradifferential operator.
Such coordinates are a key ingredient for studying the stability of finite gap solutions
of the KdV equation under small, quasi-linear perturbations.

\medskip


\medskip



\tableofcontents

\section{Introduction}\label{introduzione paper}
\label{1. Introduction}
The goal of this paper is to construct canonical coordinates for the Korteweg-de Vries (KdV) equation on the circle
   \begin{equation}
   \label{1.1} \partial_t u = - \partial_x^3 u + 6 u \partial_x u , \quad x \in \T := \R/ \Z\,,
   \end{equation}
   taylored for studying the stability of finite gap solutions of \eqref{1.1}, also referred to as periodic multisolitons, under quasi-linear perturbations. 
   To state our main results, we first need to make some preliminary considerations and introduce some notations. It is well known 
   that \eqref{1.1} is well-posed on the Sobolev spaces $H^s = \{ q \in H^s_\C: \, q \mbox{ real valued}\}$ with $s \geq - 1$ (cf \cite{KT} and references therein) where 
   for any $s \in \R,$
   \begin{equation}\label{def Hs intro}
   H^s_\C \equiv H^s(\T, \C) : = \big\{ q = \sum_{n \in \Z} q_n e^{2 \pi \ii n x} :  \,
   \| q \|_s < \infty\big\} \,,
  \quad \| q \|_s = \big( \sum_{n \in \Z} \langle n \rangle^{2 s} |q_n|^2 \big)^{\frac12}\,,
  \,\,\, \langle n \rangle := {\rm max}\{1, |n| \}\,.
   \end{equation}
  Note that $\int_0^1 u(t, x)\, d x$ is a prime integral for equation \eqref{1.1}. Without loss of generality, we restrict
   our attention to the case where $u$ has zero mean value (cf \cite[Section 13]{KP}), i.e., we consider solutions $u(t, x)$ of \eqref{1.1} in 
   $H_0^s$ with $s \ge -1$ where for any $s \in \R,$
   \begin{equation}
   H_0^s = \big\{ q \in H^s : \,\,  \int_0^1 q(x)\, d x = 0 \big\}\,.
   \end{equation} 
   It is well known that equation \eqref{1.1} can be written as a Hamiltonian PDE, $\partial_t u = \partial_x \nabla { H}^{kdv}(u)$ 
   where $\partial_x $ is the Gardner Poisson structure (with $\partial_x^{- 1}$ being the corresponding symplectic structure) 
   and $\nabla { H}^{kdv}$ denotes the $L^2$-gradient of the KdV Hamiltonian 
   $$
   H^{kdv}(q) := \frac12\int_0^1 (\partial_x q)^2\,dx + \int_0^1 q^3\, d x\,. 
   $$
   According to \cite{KP} (cf also \cite{KMT}), there are canonical coordinates $x_n = x_n(q)$, $y_n = y_n(q)$, $n \geq 1$, defined on $L^2_0 \equiv H_0^0$
   so that when expressed in these coordinates, the KdV equation takes the form 
   $$
   \dot x_n = \omega_n^{kdv} y_n, \qquad \dot y_n = - \omega_n^{kdv} x_n\,, \qquad \forall n \geq 1\,,
   $$
   where $\omega_n^{kdv}$ denote the KdV frequencies.
   To be more precise, introduce for any $s \in \R$ the sequence space
   $$
   h^s_{0,\C} \equiv h^s(\Z \setminus \{ 0 \}, \C) = \big\{ w = (w_n)_{n \neq 0} \subset \C : \| w \|_s < \infty \big\}\,, \qquad
     \| w \|_s := \big( \sum_{n \neq 0} |n|^{2 s} |w_n|^2 \big)^{\frac12}
   $$
   and its real subspace  $ h_0^s := \big\{ (w_n)_{n \neq 0} \in h^s_{0,\C} : w_{- n} = \overline{w}_n \,\, \forall n \geq 1 \big\}$
   and define the weighted complex coordinates $z_{\pm n} \equiv z_{\pm n}(q)$, 
   \begin{equation}\label{definition z_n}
   z_n := \sqrt{n \pi}(x_n - \ii y_n), \quad z_{- n} := \sqrt{n \pi} (x_n + \ii y_n), \quad \forall n \geq 1\,, 
   \end{equation}
   where $\sqrt{\cdot} \equiv \sqrt[+]{\cdot}\,$ denotes the principal branch of the square root.
   The results in \cite{KP} imply that the transformation, referred to as Birkhoff map,
   $$
   \Phi^{kdv} : L^2_0 \equiv H^0_0 \to h^0_0, \quad q \mapsto (z_n(q))_{n \neq 0}\,,
   $$
   is canonical in the sense that 
   $$
   \{ z_n, z_{- n} \} = \int_0^1 \nabla z_n \partial_x \nabla z_{- n}\, d x = 2 \pi \ii n, \quad \forall n \geq 1\,,
   $$
   whereas the brackets between all other coordinate functions vanish, and has the property that for any $s \in \Z_{\geq 0}$, 
   its restriction to $H^s_0$ is a real analytic diffeomorphism with range $h^s_0$, $\Phi^{kdv} : H^s_0 \to h^s_0$. 
   In terms of the coordinates $z_n(q), n \ne 0,$ referred to as complex Birkhoff coordinates, the action variables
   $ I_n(q)$ are defined by
   \begin{equation}\label{definition actions}
  I_n(q)  = \frac{1}{2 \pi n} z_n(q) z_{- n}(q) \ge 0\,, \quad \forall n \ge 1\,.
   \end{equation}
   The sequences $I(q)= (I_n(q))_{n \ge1}$ fill out the whole positive quadrant  $\ell^{1,1}_+$ of  $\ell^{1,1}$
   where for any $r \ge 0$, the weighted $\ell^1-$space $\ell^{1,r}$ is defined by
   \begin{equation}\label{weighted ell_1}
   \ell^{1,r} \equiv \ell^{1,r}(\N, \R) := \{ I=(I_n)_{n \geq 1} \subset \R : \, \sum_{n =1}^\infty n^r |I_n| < \infty \}\,, \qquad \N \equiv \Z_{\geq 1}.
   \end{equation}
    A key feature of the Birkhoff map is that the KdV Hamiltonian, expressed in the coordinates $z_n,$ $n \ne 0,$
   $$
    {H}^{kdv} \circ \Psi^{kdv} : h^1_0 \to \R\,, \quad \Psi^{kdv} := (\Phi^{kdv})^{- 1}\,,
   $$
   is in fact a function ${\cal H}^{kdv }$ of the actions $I$ alone. More precisely,
   ${\cal H}^{kdv } :\ell^{1, 3}_+ \to \R$  is a real analytic map.
   The KdV frequencies are then defined by $\omega_n^{kdv} := \partial_{I_n} {\cal H}^{kdv}$.
   Finally,  the differential 
   $d_0 \Phi^{kdv} : L^2_0 \to h^0_0$ of $\Phi^{kdv}$ at $q = 0$  is the Fourier transform $\mathcal F$ (cf \cite[Theorem 9.8]{KP})
   and hence $d_0 \Psi^{kdv}$ the inverse Fourier transform ${\cal F}^{- 1}$ where for any $s \in \Z,$
   \begin{equation}\label{Fourier transform}
   {\cal F} : H^s_{0} \to h^s_0,  \quad q \mapsto (q_n)_{n \neq 0}, \quad q_n := \int_0^1 q(x) e^{- 2 \pi \ii n x}\, d x\,.
   \end{equation} 
   Let
   \begin{equation}\label{definition S}
  S_+ \subseteq \N \,\mbox{ finite}, \quad   S := S_+ \cup (- S_+) \qquad {\mbox{and}} \qquad S_+^\bot := \N \setminus S_+\,, \quad S^\bot := S_+^\bot \cup (-S^\bot_+)\,.
   \end{equation}
   We denote by $M_S \subset L^2_0$ the manifold of $S$-gap potentials,
   $$
   M_S := \big\{ q \in L^2_0  : z_n(q) = 0 \,\, \forall \, n \in S^\bot  \big\}, \quad
   $$
   and by $M_S^o$ the open subset of $M_S$, consisting of proper $S$-gap potentials,
   $$
   M_S^o := \{ q \in M_S : \, z_n(q) \ne 0 \, \, \forall \, n \in S \}\,.
   $$
   Note that $M_S$ is contained in $\cap_{s \geq 0} H^s_0$ and hence consists of $C^\infty$-smooth potentials  
   and that $M_S^o$ can be parametrized by the action-angle coordinates 
   $ \theta_S = (\theta_k)_{k \in S_+},$  $I_S = (I_k)_{k \in S_+},$
   $$
    {\cal M}_S^o  \to M_S^o, \,\, (\theta_S, I_S) \mapsto \Psi^{kdv}(z(\theta_S, I_S))\,,
   \qquad {\cal M}_S^o := \T^{S_+} \times \R^{S_+}_{> 0}
   $$
   where for any $n \ne 0$, $z_n = z_n(\theta_S, I_S)$ is given by 
   $$
    z_{\pm n} := \sqrt{2 \pi n I_n} e^{\mp \ii \theta_n}, \quad \forall n \in S_+, \qquad 
   z_n = 0, \quad \forall n \in S^\bot .
   $$
 Introduce for any $s \in \R$
   $$
  h^s_{\bot, \C} := h^s(S^\bot, \C)\,, \qquad  h^s_\bot := \big\{ z_\bot \in h^s_{\bot, \C} : z_{- n} = \overline{z}_n, \quad \forall n \in S^\bot \big\}, 
   $$
   as well as the maps, related to the Fourier transform,
   $$
   \mathcal F_\bot: H^s_0 \to h^s_\bot, \, q \mapsto \, (q_n)_{n \in S^\bot} \,, \quad 
   \mathcal F_\bot^{-1}: h^s_\bot \to H^s_0, \, (z_n)_{n \in S^\bot} \mapsto \sum_{n \in S^\bot} z_n e^{2\pi \ii nx}\,.
   $$
   By a slight abuse of notation, we view ${\cal M}_S^o \times h^s_\bot$, $s \in \R$, as a subset of $h^s_0$ 
   and denote its elements by
   $$
   {\frak x} := (\theta_S, I_S, z_\bot) , \quad \theta_S := (\theta_n)_{n \in S_+}, \,\,\, I_S := (I_n)_{n \in S_+}, \,\,\, z_\bot := (z_n)_{n \in S^\bot}\,. 
   $$
   It is endowed by the standard Poisson bracket, given by 
   $$
   \{ I_n, \theta_n \} = 1, \,\,\, \forall n \in S_+, \qquad \{ z_n, z_{- n} \} = 2 \pi \ii n, \,\,\, \forall n \in S^\bot_+,
   $$
   whereas the brackets between all other coordinate functions vanish.
   For any $s \in \R$, we denote by $\widehat{\frak x} = (\widehat \theta_S, \widehat I_S, \widehat z_\bot)$ elements in the tangent space $E_s$ 
   of ${\cal M}_S^o \times h^s_\bot$ at any given point $\frak x \in {\cal M}_S^o \times h^s_\bot$ where $E_s$ is given by
   $$
   E_s = \R^{S_+} \times \R^{S_+} \times h^s_\bot.
   $$ 
    Furthermore, for any $k \ge 1, $ $\partial_x^{-k} : H^s_{ \C} \to H^{s+k}_{0,\C}$ is  the bounded linear operator, defined by
$$
\partial_x^{-k}[e^{2\pi \ii nx}] = \frac{1}{(2\pi \ii n)^k} e^{2\pi \ii nx}\, , \quad \forall n \ne 0\,,  
\qquad \mbox{and} \qquad   \partial_x^{-k}[1] = 0\,.
$$   
Finally, the standard inner products on $L^2_0$ and on $h^0_0$
are defined by
$$
\langle f , g \rangle \equiv \langle f , g \rangle_{L^2_0} = \int_0^1 f(x) g(x) d x\,, \,\,\, \forall f, g \in L^2_0\,, \qquad 
\langle z , w \rangle \equiv \langle z , w \rangle_{h^0_0}  = \sum_{n \ne 0} z_n w_{-n} \,, \,\, \, \forall z, w \in h^0_0\,.
$$
Note that $\langle \cdot , \cdot \rangle_{L^2_0}$ and  $\langle \cdot , \cdot \rangle_{h^0_0}$ extend as complex valued bilinear forms 
to $L^2_{0, \C}$ and respectively, $h^0_{0, \C}$. In the sequel, restrictions of these inner products to various subspaces  and extensions
as dual pairings will be denoted in the same way and the gradient of a functional $F$ corresponding to these inner products by $\nabla F$.
In more detail, for a $C^1-$functional $F: h^0_0 \to \C$, one has
$dF[\widehat z] = \sum_{n \ne 0} \widehat z_n \partial_{z_n} F = \langle  \nabla F, \widehat z \rangle$
with the nth component of $\nabla F$ given by $(\nabla F)_n = \partial_{z_{-n}}  F$.
Furthermore, for any given Banach spaces $Y_1,$ $Y_2$, we denote by ${\cal B}( Y_1 , Y_2)$ the space of bounded linear operators from $Y_1$ to $Y_2$, endowed with the operator norm. 
\begin{theorem}\label{modified Birkhoff map}
Let $S_+ \subseteq \N$ be finite
and ${\cal K}$ be a subset of ${\cal M}_S^o$ of the form $\T^{S_+} \times {\cal K}_1$ where  ${\cal K}_1$ is a compact subset of $\R^{S_+}_{> 0}$. 
Then there exists an open bounded neighbourhood ${\cal V}$ of ${\cal K} \times \{ 0 \}$ in ${\cal M}_S^o \times h^0_\bot$ 
and a canonical real analytic diffeomorphism 
$\Psi : {\cal V} \to \Psi({\cal V}) \subseteq L^2_0\,, \, {\frak x} = (\theta_S, I_S, z_\bot)\mapsto q$ 
with the property that $\Psi$ satisfies 
$$
\Psi(\theta_S, I_S, 0) = \Psi^{kdv}(\theta_S, I_S, 0), \quad \forall (\theta_S, I_S, 0) \in {\cal V}\,,
$$
and is compatible in the sense explained below with the scale of Sobolev spaces $H^s_0, s \in \Z_{\geq 0}$, so that the following holds:
\begin{description}
\item[({\bf AE1})] For any integer $N \ge 1$, $\Psi$ admits an asymptotic expansion on $\mathcal V$ of the form  
$$
\Psi(\theta_S, I_S, z_\bot) = \Psi^{kdv}(\theta_S, I_S, 0) + {\cal F}_\bot^{- 1}[z_\bot] + 
\sum_{k = 1}^N a_k(\theta_S, I_S, z_\bot; \Psi) \, \partial_x^{- k} {\cal F}_\bot^{- 1}[z_\bot] 
+ {\cal R}_{N}(\theta_S, I_S, z_\bot; \Psi) 
$$ 
where  ${\cal R}_{N}(\theta_S, I_S, 0; \Psi) = 0$ and where
for any $s \in \Z_{\geq 0}$  and $1 \le k \le N$ 
$$
a_k^\Psi : {\cal V} \to H^s,\,  {\frak x} \mapsto  a_k^\Psi({\frak x}) \equiv a_k ({\frak x}; \Psi),
\qquad {\cal R}_N^\Psi : {\cal V} \cap \big( {\cal M}_S^o \times h^s_\bot \big) \to H^{s + N +1},\, {\frak x} \mapsto  {\cal R}_N^\Psi ({\frak x}) \equiv \mathcal R_N({\frak x}; \Psi),
$$
are real analytic maps satisfying the tame estimates of Theorem \ref{modified Birkhoff map 2} below. 
\item[({\bf AE2})]
For any $\frak x  \in {\cal V}$,
the transpose $d \Psi(\frak x )^t$ 
(with respect to the standard inner products) of the differential
$d \Psi(\frak x) : E_1 \to H^1_0$ is a bounded operator $d \Psi(\frak x )^t : H^1_0 \to E_1$. For any $\widehat q \in H^1_0$
and any integer $N \ge 1$, $d \Psi(\frak x )^t[\widehat q]$ admits an expansion of the form 
$$
d \Psi(\frak x)^t [\widehat q]= 
\big( \, 0, \, 0, \, {\cal F}_\bot [\widehat q]  + {\cal F}_\bot \circ \sum_{k = 1}^N  a_k(\frak x; d \Psi^t)  \partial_x^{- k}\widehat q \, 
+ {\cal F}_\bot \circ \sum_{k = 1}^N  \mathcal A_k(\frak x; d \Psi^t)[\widehat q]  \, \partial_x^{- k} \mathcal F_\bot^{-1}[ z_\bot] \big)
+ {\cal R}_{N}(\frak x; d \Psi^t)[\widehat q]
$$
where for any $s \in \N$ and $1 \le  k \le N$, 
$$
a_k^{d \Psi^t} : {\cal V} 
\to H^s\,, \,\, \frak x \mapsto a_k^{d \Psi^t}(\frak x) \equiv a_k (\frak x; d\Psi^t )\,, 
$$
$$
\mathcal A_k^{d \Psi^t} : {\cal V} 
\to  {\cal B}(H^1_0, H^s)\,, \,\,  \frak x \mapsto \mathcal A_k^{d \Psi^t}(\frak x) \equiv \mathcal A_k (\frak x; d\Psi^t )\,, 
$$
$$
 {\cal R}_N^{d \Psi^t} : {\cal V} \cap ({\cal M}_S^o \times h^s_\bot) \to {\cal B}(H^s_0, E_{s + N +1}), \,\,
\frak x \mapsto {\cal R}_N^{d \Psi^t}(\frak x) \equiv  {\cal R}_N(\frak x; d \Psi^t)
$$
are real analytic maps, satisfying the tame estimates of Theorem \ref{modified Birkhoff map 2} below. 
Furthermore, the coefficient $a_1( \frak x ; d \Psi^t)$ satisfies
$a_1( \frak x; d \Psi^t) = - a_1( \frak x; \Psi)$.
\item[({\bf AE3})] The Hamiltonian ${\cal H} := H^{kdv} \circ \Psi : {\cal V} \cap (\mathcal M_S^o \times h^1_\bot) \to \R$ is in normal form up to order three. More precisely,  
$$
{\cal H}(\theta_S, I_S, z_\bot) = {\cal H}^{kdv}(I_S, 0) + \sum_{n \in S_+^\bot} \Omega_n^{kdv}(I_S, 0) z_n z_{- n} + {\cal P}(\theta_S, I_S, z_\bot)
$$
where $ {\cal P} : {\cal V} \cap (\mathcal M_S^o \times h^1_\bot) \to \R$ is real analytic and satisfies ${\cal P}(\theta_S, I_S, z_\bot) = O(\| z_\bot \|_1 \| z_\bot \|_0^2)$, 
$\Omega_n^{kdv} : = \frac{1}{2 \pi n} \omega_n^{kdv}$, 
$n \in S_+^\bot$, and $ \omega_n^{kdv}$ denote the KdV frequencies introduced above. Furthermore for any integer $N \ge 1$, 
there exists an integer $\sigma_N \ge N$ (loss of regularity) so that the gradient $\nabla {\cal P}(\theta_S, I_S, z_\bot)$ of ${\cal P}$ with components 
$\nabla_{\theta_S} {\cal P}$, $\nabla_{I_S} {\cal P}$, and $\nabla_{z_\bot} {\cal P}$ admits an expansion of the form 
$$
\nabla {\cal P}(\theta_S, I_S, z_\bot) =
\big(\, 0, \, 0, \, {\cal F}_\bot \circ \sum_{k = 0}^N  T_{a_k^{\nabla {\cal P}}} \, \partial_x^{- k} {\cal F}_\bot^{- 1}[z_\bot] \, \big)   + {\cal R}_N^{{\nabla {\cal P}}}(\theta_S, I_S, z_\bot)
$$
where for any integers $s \ge 0$ and $0 \le  k \le N$, 
$$
\begin{aligned}
& a_k^{\nabla {\cal P}} : {\cal V} \cap ({\cal M}_S^o \times h^{s + \sigma_N}_\bot) \to H^s \,,  \qquad
 {\cal R}_N^{{\nabla {\cal P}}} : {\cal V} \cap \big( {\cal M}_S^o \times h^{s \lor \sigma_N}_\bot \big) \to E_{s + N + 1}
\end{aligned}
$$
are real analytic and satisfy the tame estimates of Theorem \ref{modified Birkhoff map 2} below. 
\end{description}
\noindent
Here $T_{a_k^{\nabla {\cal P}}}$ denotes the operator of paramultiplication with $a_k^{\nabla {\cal P}}$ (cf Appendix \ref{appendice B}) and 
the diffeomorphism $\Psi : {\cal V} \to \Psi({\cal V}) \subset L^2_0$ 
being compatible with the scale of Sobolev spaces $H^s_0$, $s \in \Z_{\geq 0}$, means that for any $s \in \Z_{\geq 0}$,
$\Psi \big( {\cal V} \cap ({\cal M}_S^o \times h^s_\bot) \big) \subseteq H^s_0$ and $\Psi : {\cal V} \cap ({\cal M}_S^o \times h^s_\bot) \to H^s_0$ 
is a real analytic diffeomorphism onto its image. 
\end{theorem}

\noindent
In applications, it is of interest to know whether the coordinate transformation $\Psi$ preserves the reversible structure, defined by the maps
$S_{rev} : L^2_0 \to L^2_0$,  $(S_{rev} q)(x) := q(- x)$, and
${\cal S}_{rev} : {\cal M}_S^o \times h^0_\bot \to {\cal M}_S^o \times h^0_\bot$ where 
\begin{equation}\label{definition reversible structure for actions angles}
 {\cal S}_{rev}(\theta_S, I_S, z_\bot) := (\theta^{rev}_S, I_S^{rev}, z_\bot^{rev})\,, \quad
\theta_n^{rev} = - \theta_n , \,\,\, I_n^{rev} = I_n, \,\,\, \forall n \in S_+\,, \quad z_n^{rev} = z_{- n}, \,\,\, \forall n \in S^\bot\,. 
\end{equation}
Note that for any $s \in \Z_{\ge 0}$, $S_{rev} : H^s_0 \to H^s_0$ and ${\cal S}_{rev} : {\cal M}_S^o \times h^s_\bot \to {\cal M}_S^o \times h^s_\bot$ 
are linear involutions and that without loss of generality, the neighbourhood ${\cal V}$ of Theorem \ref{modified Birkhoff map} can 
be chosen to be invariant under the map ${\cal S}_{rev}$, i.e., ${\cal S}_{rev} ({\cal V}) = {\cal V}$. 

\medskip

\noindent 
 {\bf Addendum to Theorem 1.1} \label{teo reversibilita}
 {\em  The maps $\Psi : {\cal V} \to L^2_0$, \, $\Psi^{kdv} : h^0_0 \to L^2_0$,  and \,${\cal F}^{- 1} : h^0_0 \to L^2_0$ preserve
 the  reversible structure, i.e.,
  $$
  \Psi \circ {\cal S}_{rev} = S_{rev} \circ \Psi, \qquad \Psi^{kdv} \circ {\cal S}_{rev} = S_{rev} \circ \Psi^{kdv}\,, 
  \qquad {\cal F}^{- 1} \circ {\cal S}_{rev} = S_{rev} \circ {\cal F}^{- 1}\,.
  $$
  and so do the maps in the asymptotic expansions {\bf (AE1)} ($\frak x \in \mathcal V$),
  $$
  a_k^\Psi ( {\cal S}_{rev} {\frak x}) =  (-1)^k S_{rev}  (a_k^\Psi ({\frak x}) ) \,, \quad 
  {\cal R}_N^\Psi ( {\cal S}_{rev} {\frak x}) =  S_{rev} ({\cal R}_N^\Psi ({\frak x}))\,,
  $$
  and the ones in the asymptotic expansions  {\bf (AE2)} ($\frak x \in \mathcal V \cap (\mathcal M_S^o \times h^1_\bot)$, $\widehat q \in H^1_0$), 
  $$
  a_k({\cal S}_{rev} \frak x; \, d \Psi^t) = (-1)^k  { S}_{rev}  (a_k(\frak x; \, d \Psi^t) )\,, \quad
  \mathcal  A_k({\cal S}_{rev} \frak x; \, d \Psi^t) [ S_{rev} \, \widehat q ] = (-1)^k  { S}_{rev}  (\mathcal A_k (\frak x; \, d \Psi^t) [ \widehat q ] )
  $$
  $$
  {\cal R}_N({\cal S}_{rev} \frak x ; \, d \Psi^t) [ S_{rev} \, \widehat q ] =  {\cal S}_{rev} ( {\cal R}_N( \frak x ; \, d \Psi^t) [ \, \widehat q ] ) \,.
  $$
  Furthermore, the Hamiltonians $H^{kdv}$, ${\cal H} = H^{kdv} \circ \Psi$, and ${\cal P}$ are reversible, meaning that 
  $$
  H^{kdv} \circ S_{rev} = H^{kdv}\,, \quad {\cal H} \circ {\cal S}_{rev} = {\cal H}\,, \quad {\cal P} \circ {\cal S}_{rev} = {\cal P}
  $$
  and the maps in the asymptotic expansion in {\bf (AE3)} preserve the reversible structure,
   $$
 a_k^ {\nabla \cal P} ( {\cal S}_{rev} {\frak x}) = (-1)^k S_{rev} (a_k^{\nabla \cal P} (\frak x) )\,, \quad \forall \, \frak x \in \mathcal V \cap ({\cal M}_S^o \times h^{1 + \sigma_N}_\bot)\,,
 $$
 $$
  {\cal R}_N^{\nabla \cal P} ( {\cal S}_{rev} \frak x) =  {\cal S}_{rev} ( {\cal R}_{N}^{\nabla \cal P} (\frak x) )\, \quad \forall \, \frak x \in \mathcal V \cap ({\cal M}_S^o \times h^{1 \lor \sigma_N}_\bot)\, . 
  $$}
  
  \smallskip
  
 \noindent
 Theorem \ref{modified Birkhoff map 2} below states tame estimates for the map $\Psi $ and the gradient $\nabla {\cal P}$ of the remainder term ${\cal P}$ 
 in the expansion of ${\cal H}$. In the sequel, we denote
 elements in the tangent space $E_s := \R^{S_+} \times \R^{S_+} \times h^s_\bot$  of ${\cal V} \cap ({\cal M}_S^o \times h^s_\bot)$
 at any given point $\frak x = (\theta_S, I_S, z_\bot)$ by $\widehat{\frak x} = (\widehat \theta_S, \widehat I_S, \widehat z_\bot)$.
 Throughout the paper, all the stated estimates for maps hold locally uniformly with respect to their arguments.
  \begin{theorem}\label{modified Birkhoff map 2}
  Let $N, l \in \N$. Then
  under the same assumptions as in Theorem \ref{modified Birkhoff map}, the following estimates hold:
  \begin{description}
  \item[(Est1)] For any $\frak x = (\theta_S, I_S, z_\bot) \in {\cal V}$, $1 \le k \le N$, \, $\widehat{\frak x}_1, \ldots, \widehat{\frak x}_l \in E_0$, $s \in \Z_{\geq 0}$,
  $$
  \begin{aligned}
  & \| a_k^\Psi(\frak x) \|_s \lesssim_{s, k} 1 
  \,,  \qquad
    \| d^l a_k^\Psi(\frak x)[\widehat{\frak x}_1, \ldots, \widehat{\frak x}_l]\|_s \lesssim_{s, k, l} \prod_{j = 1}^l \| \widehat{\frak x}_j\|_0\,. 
  \end{aligned}
  $$
  Simlarly, for any $\frak x  \in {\cal V} \cap \big( {\cal M}_S^o \times h^s_\bot \big)$, \, $\widehat{\frak x}_1, \ldots, \widehat{\frak x}_l \in E_s$, $s \in \Z_{\geq 0}$,
  $$
  \begin{aligned}
  & \| {\cal R}_N^\Psi(\frak x)\|_{s + N + 1} \lesssim_{s, N}  \| z_\bot \|_s\,,  \quad
   \| d^l {\cal R}_N^\Psi(\frak x)[\widehat{\frak x}_1, \ldots, \widehat{\frak x}_l]\|_{s + N + 1} 
   \lesssim_{s, N, l} \sum_{j = 1}^l \| \widehat{\frak x}_j\|_s \prod_{i \neq j} \| \widehat{\frak x}_i\|_0 + \| z_\bot \|_s \prod_{j = 1}^l \| \widehat{\frak x}_j\|_0\,.
  \end{aligned}
  $$
  \item[(Est2)] For any $\frak x = (\theta_S, I_S, z_\bot) \in {\cal V} \cap (\mathcal M_S^o \times h^1_\bot)$, $1 \le k \le N$, \, $\widehat{\frak x}_1, \ldots, \widehat{\frak x}_l \in E_1$, $s \in \N$,
  $$
  \begin{aligned}
  &  \| a_k^{{d \Psi}^t}(\frak x)\|_s \lesssim_{s, k} 1 +  \| z_\bot \|_1 \,, \qquad
   \| d^l a_k^{{d \Psi}^t}(\frak x)[\widehat{\frak x}_1, \ldots, \widehat{\frak x}_l]\|_s 
  \lesssim_{s, k, l}  \prod_{j = 1}^l \| \widehat{\frak x}_j\|_1\,. 
  \end{aligned}
  $$
  and
  $$
  \begin{aligned}
  &  \| \mathcal A_k^{{d \Psi}^t}(\frak x)\|_s \lesssim_{s, k} 1 +  \| z_\bot \|_1 \,, \qquad
   \| d^l \mathcal A_k^{{d \Psi}^t}(\frak x)[\widehat{\frak x}_1, \ldots, \widehat{\frak x}_l]\|_s 
  \lesssim_{s, k, l}  \prod_{j = 1}^l \| \widehat{\frak x}_j\|_1\,. 
  \end{aligned}
  $$
  Similarly, for any $\frak x \in {\cal V} \cap \big( {\cal M}_S^o \times h^s_\bot \big)$, 
  \, $\widehat{\frak x}_1, \ldots, \widehat{\frak x}_l \in E_s$, \, $\widehat q \in H^s_0$, $s \in \N$,
  $$
  \begin{aligned}
  & \| {\cal R}_N^{d\Psi^t} (\frak x) [\widehat q]\|_{s + N + 1} \lesssim_{s, N} \| \widehat q\|_s + \| z_\bot \|_s \| \widehat q\|_1\,,  \\
  & \| d^l \big( {\cal R}_N^{d \Psi^t} (\frak x) [\widehat q] \big)[\widehat{\frak x}_1, \ldots, \widehat{\frak x}_l]\|_{s + N + 1} 
  \lesssim_{s, N, l} \| \widehat q\|_s \prod_{j = 1}^l \| \widehat{\frak x}_j\|_1 + \| \widehat q\|_1 \sum_{j = 1}^l \| \widehat{\frak x}_j\|_s \prod_{i \neq j} \| \widehat{\frak x}_i\|_1 + 
  \| \widehat q\|_1 \| z_\bot\|_s \prod_{j = 1}^l \| \widehat{\frak x}_j \|_1\,.
    \end{aligned}
  $$
\item[(Est3)] For any $s \in \Z_{\ge 0}$, $\frak x = (\theta_S, I_S, z_\bot) \in {\cal V} \cap \big( {\cal M}_S^o \times h^{s + \sigma_N}_\bot\big)$, $\| z_\bot \|_{\sigma_N} \leq 1$, $1 \le k \le N$,
\, $\widehat{\frak x}_1, \ldots, \widehat{\frak x}_l \in E_{s + \sigma_N}$,
$$
  \| a_k^{\nabla \cal P}(\frak x)\|_{s} \lesssim_{s, k} \| z_\bot \|_{s + \sigma_N}\,, \quad
  \| d^l a_k^{\nabla \cal P}(\frak x)[\widehat{\frak x}_1, \ldots, \widehat{\frak x}_l]\|_s 
\lesssim_{s, k, l} \sum_{j = 1}^l \| \widehat{\frak x}_j\|_{s + \sigma_N} \prod_{n \neq j} \| \widehat{\frak x}_n\|_{\sigma_N} +
 \| z_\bot \|_{s + \sigma_N} \prod_{j = 1}^l \| \widehat{\frak x}_j\|_{\sigma_N}\,. 
$$
For any $s \in \Z_{\ge 0}$, $\frak x \in {\cal V} \cap \big( {\cal M}_S^o \times h_\bot^{s \lor\sigma_N}\big)$ with $\| z_\bot \|_{\sigma_N} \leq 1$,
  \, $ \widehat{\frak x} \in E_{s \lor\sigma_N}$, 
$$
 \| {\cal R}_N^{\nabla \cal P}(\frak x)\|_{s + N + 1} \lesssim_{s, N} \| z_\bot \|_{\sigma_N} \| z_\bot \|_{s \lor\sigma_N}\,, \quad
 \| d {\cal R}_N^{\nabla \cal P}(\frak x)[\widehat{\frak x}] \|_{s + N + 1} 
\lesssim_{s, N} \| z_\bot \|_{\sigma_N} \| \widehat{\frak x}\|_{s \lor \sigma_N} + \| z_\bot \|_{s \lor \sigma_N} \| \widehat{\frak x}\|_{\sigma_N}\,, 
$$
and if in addition $ \widehat{\frak x}_1, \ldots, \widehat{\frak x}_l \in E_{s \lor \sigma_N}$, $l \ge 2,$
$$
 \| d^l  {\cal R}_N^{\nabla \cal P}(\frak x)[\widehat{\frak x}_1, \ldots, \widehat{\frak x}_l]\|_{s + N + 1} 
\lesssim_{s, N, l} \sum_{j = 1}^l \| \widehat{\frak x}_j\|_{s \lor \sigma_N} \prod_{n \neq j} \| \widehat{\frak x}_n\|_{\sigma_N} + 
\| z_\bot \|_{s \lor \sigma_N} \prod_{j = 1}^l \|\widehat{\frak x}_j \|_{\sigma_N}\,. 
$$
  \end{description}
  \end{theorem}
 \noindent 
{\em Applications:} The Birkhoff coordinates are well suited to study the initial value problem of \eqref{1.1} (cf e.g. \cite{KT}, \cite{KM2} and references therein) 
and semilinear perturbations of \eqref{1.1} (cf e.g. \cite{KP}, \cite{K}). However, when equation \eqref{1.1} is expressed in Birkhoff coordinates, 
various features of the KdV equation and its perturbations such as being partial differential equations, get lost. 
On the other hand, due to the expansions $({\bf AE1}) - ({\bf AE3})$, the coordinates of Theorem \ref{modified Birkhoff map} allow 
to preserve the essence of such features and in the form stated turn out to be well suited to study quasi-linear perturbations of the KdV equation. 
   
   \medskip
   \noindent
 {\em Outline of the construction:} In his pioneering work \cite{K}, Kuksin presents a general scheme for proving KAM-type theorems 
 for integrable PDEs in one space dimension such as the KdV or the sine-Gordon (sG) equations, which possess a Lax pair formulation 
 and admit finite dimensional integrable subsystems foliated by invariant tori. Expanding on work of Krichever \cite{Krichever}, Kuksin considers bounded integrable subsystems of such a PDE 
 which admit action-angle coordinates. They are complemented by infinitely many coordinates 
 whose construction is based on a set of time periodic solutions, referred to as Floquet solutions of the PDE, obtained by linearizing 
 the PDE under consideration along a solution evolving in the integrable subsystem. It turns out that the resulting coordinate transformation 
 is typically not symplectic. Extending arguments of Moser and Weinstein to the given infinite dimensional setup (see \cite{K}, Lemma 1.4 and Section 1.7), 
 he constructs a second coordinate transformation so that the composition of the two transformations become symplectic. We follow Kuksin's scheme of the proof 
 by constructing $\Psi$ as the composition of $\Psi_L \circ \Psi_C$ of two transformations. 
 The $\Psi_L$ is given by the Taylor expansion of $\Psi^{kdv}$ of order one in the normal direction $z_\bot$ around $(\theta_S, I_S, 0)$,
   $$
   \Psi^{kdv}(\theta_S, I_S, 0) + d \Psi^{kdv}(\theta_S, I_S, 0)[(0,0, z_\bot)]\,.
   $$ 
   The neighbourhood ${\cal V}$ of ${\cal K} \times \{ 0 \}$ is chosen sufficiently small so that by the inverse function theorem,
    $\Psi_L$ is a real analytic diffeomorphism onto its image. 
    Using that $\Psi_L$ is given in terms of the Birkhoff map $\Psi^{kdv}$, we prove in a first step that $\Psi_L$ admits an asymptotic expansion and tame estimates corresponding to the ones of Theorems \ref{modified Birkhoff map}, \ref{modified Birkhoff map 2}. In a second step we establish the corresponding results for the symplectic corrector $\Psi_C$. 
The methods developed in this paper also apply to the defocusing NLS equation and can be used to provide corresponding asymptotic expansions
and estimates, thus complementing our previous work \cite{Kappeler-Montalto} on this equation.

   \medskip
   
   \noindent
   {\em Comments:} In view of the definition of $\Psi_L$, the map $\Psi = \Psi_L \circ \Psi_C$ can be considered as a symplectic version of the Taylor expansion of $\Psi^{kdv}$ of order $1$ 
   in normal directions at points of ${\cal M}_S^o \times \{ 0 \}$ and hence as a locally defined symplectic approximation of $\Psi^{kdv}$. Theorem \ref{modified Birkhoff map} in particular says 
   that $\Psi(\theta_S, I_S, z_\bot) - {\cal F}_\bot^{- 1}[z_\bot]$ maps ${\cal V} \cap ({\cal M}_S^o \times h^s_\bot)$ into $H^{s + 1}_0$ for any $s \geq 0$, i.e., that it is one-smoothing. 
   Such a property has previously been established for the Birkhoff map $\Psi^{kdv}$ near $0$ by Kuksin-Perelman \cite{Kuksin-Perelman}, Theorem 0.2, 
   and proved to hold on the entire phase space by Kappeler-Schaad-Topalov \cite{KST2}. Theorem \ref{modified Birkhoff map} says that for the map $\Psi$, a much stronger property holds: 
   up to a remainder term which is $(N+1)$-smoothing, $\Psi$ is a (nonlinear) pseudodifferential operator acting on ${\cal F}_\bot^{- 1}(h^0_\bot)$. 
   
   \medskip
   
   \noindent
   {\em Organization:} The maps $\Psi_L$ and $\Psi_C$ are studied in Section \ref{sezione mappa Psi L} and respectively, Section \ref{sezione Psi C}.
   The expansion of the KdV Hamiltonian in the new coordinates is treated in Section \ref{Hamiltoniana trasformata} 
   and a summary of the proofs of Theorem \ref{modified Birkhoff map} and Theorem \ref{modified Birkhoff map 2} is given in Section \ref{synopsis of proof}. 
   In Appendix \ref{Birkhoff map} -- 
   Appendix \ref{AppendixReversability}, 
   we present results needed for the analysis of the map  $\Psi_L$ in Section \ref{sezione mappa Psi L} and 
   in Appendix \ref{appendice B} we review material from the pseudodifferential and paradifferential calculus.


   \section{The map $\Psi_L$}\label{sezione mappa Psi L}
   In this section we define and study the map $\Psi_L$ described in Section \ref{introduzione paper}. First let us introduce some more notation. For $S \subset \Z$ finite as in \eqref{definition S},
   denote by $h_S^0 \subset \C^S$ the subspace given by
   \begin{equation}\label{definition h^0_S}
   h_S^0 := \big\{ z_S = (z_n)_{n \in S} \in \C^S :  \,z_{- n} = \overline{z}_n \,\, \forall n \in S_+ \big\}\,.
   \end{equation}
   By a slight abuse of terminology, for any $s \in \Z$, we identify $h^s_0$ with $h^0_S \times h^s_\bot$
   and write $(z_S, z_\bot)$ for $z \in h^s_0.$
   According to the Addendum to Theorem \ref{Theorem Birkhoff coordinates} at the end of Appendix \ref{Birkhoff map}, 
   the space $M_S$ of $S-$gap potentials is viewed as a (real analytic) submanifold of the weighted Sobolev space $H^{w_*}_0$ and
   the restriction of $\Phi^{kdv}$ to $M_S$ yields a real analytic diffeomorphism,
   ${\Phi^{kdv}}_{| M_S} : M_S \to h_S^0$. 
   We endow 
   $h_S^0$ with the pull back of the standard Poisson 
   structure on $h_0^0$ by the natural embedding $h_S^0 \hookrightarrow h_0^0$, where the standard Poisson structure is the one for which $\{ z_n, z_{- n} \} = 2 \pi \ii n$ 
   for any $n \geq 1$ whereas the Poisson brackets among all the other coordinates vanish. 

   Consider the partially linearized inverse Birkhoff map 
   \begin{equation}\label{definition Psi_L}
   \Psi_L : h_S^0 \times h_\bot^0 \to L^2_0\,, \,\,\, (z_S, z_\bot) \mapsto \Psi^{kdv}(z_S, 0) + d_\bot \Psi^{kdv}(z_S, 0)[z_\bot]
   \end{equation}
   where 
   $d_\bot \Psi^{kdv}(z_S, 0)$ denotes the Fr\'echet derivative of the map $z_\bot \mapsto \Psi^{kdv}(z_S, z_\bot)$, evaluated at the point $(z_S, 0)$. 
   By Theorem \ref{Theorem Birkhoff coordinates}, $\Psi_L$ is a real analytic map. 
   \begin{proposition}\label{prop Psi L Psi kdv} The map $\Psi_L$ has the following properties: (i) For any $z_S \in h^0_S,$
   $$
   \Psi_L (z_S, 0) = \Psi^{kdv}(z_S, 0) \qquad \mbox{and} \qquad d \Psi_L(z_S, 0)= d \Psi^{kdv}(z_S, 0).
   $$
   (ii) For any compact subset ${\mathcal K} \subseteq h_S^0$ there exists an open neighbourhood ${\mathcal V}$ of ${\mathcal K} \times \{ 0 \}$ in $h^0_0 \equiv h_S^0 \times h_\bot^0$ 
   so that for any integer $s \ge 0$, the restriction $\Psi_L|_{\mathcal V \cap h_0^s}$ is a map $\mathcal V \cap h_0^s  \to H^s_0$
   which is a real analytic diffeomorphism onto its image. The neighborhood $\mathcal V$ is chosen of the form $\mathcal V_S \times \mathcal V_\bot$
    where $\mathcal V_S$ is an open, bounded neighborhood of $\mathcal K$ in $h^0_S$ and $\mathcal V_\bot$ is an open ball
    in $h^0_\bot$ of sufficiently small radius, centered at zero.
(iii) For any $z = (z_S, z_\bot) \in {\cal V}$ and  $\widehat z = (\hat z_S,  \hat z_\bot ) \in h^0_S \times h_\bot^0$, 
\begin{equation}\label{formula d Psi L}
d \Psi_L(z) [\widehat z]  = d \Psi_L(z_S, 0)[\widehat z] + d_S \big( d_\bot \Psi^{kdv}(z_S, 0)[ z_\bot]  \big)[\widehat z_S]
\end{equation}
where the linear map $d \Psi_L(z_S, 0) =  d \Psi^{kdv}( z_S, 0)$
is canonical and $d_S \big( d_\bot \Psi^{kdv}(z_S, 0)[z_\bot]  \big)$ denotes the Fr\'echet derivative
of the map 
$$
\mathcal V_S \to L^2_0 ,\, z_S \mapsto d_\bot \Psi^{kdv}(z_S, 0)[z_\bot].
$$
\end{proposition}
 \begin{proof} (i) The stated formulas follow from the definition of $\Psi_L$ in a straightforward way. (ii) In view of Theorem \ref{Theorem Birkhoff coordinates},
 the claimed statements can be proved by using the same arguments as in the proof of the corresponding results for the defocusing NLS equation in \cite[Proposition 3.1]{Kappeler-Montalto}. 
 Item (iii) is proved in a straightforward way.
 \end{proof}

In a next step we want to analyze $d_\bot \Psi^{kdv}(z_S, 0)$ further. Consider the Hamiltonian vector fields $\partial_x \nabla_q z_{\pm n}$, $n \geq 1$, 
corresponding to the Hamiltonians given by the complex Birkhoff coordinates $z_{\pm n}$. 
Since $\Phi^{kdv}$ is canonical in the sense that $\{ z_n, z_{- n} \} = 2 \pi \ii n$ for any $n \ne 0$ 
whereas the brackets among all the other coordinates vanish, it follows that for any $q \in L^2_0$ and $n \geq 1$, 
   $$
   d_q \Phi^{kdv}[\partial_x \nabla_q z_{\pm n}] = X_{z_{\pm n}} 
   $$
   where $X_{z_{\pm n}}$ are the constant vector fields on $h^0_{0, \C}$ given by 
   $$
   X_{z_n} = - 2 \pi \ii n e^{(- n)}\,, \quad X_{z_{- n}} =  2 \pi \ii n e^{(n)}\,
   $$
   and $e^{(n)}, e^{(- n)}$ are the standard basis elements in the sequence space $h^0_{0, \C}$, 
   $$
   e^{(n)} = (\delta_{n, k})_{k \neq 0}, \qquad e^{(- n)} = (\delta_{- n, k})_{k \neq 0}\,.
   $$
   (Here we extended $d_q \Phi^{kdv}: L^2_{0} \to h^0_{0}$ as a $\C$-linear map $L^2_{0, \C} \to h^0_{0, \C}$). Hence for any $n \geq 1$, 
   \begin{equation}\label{formula vector field X z+- n}
   (d_q \Phi^{kdv})^{- 1}[e^{(n)}] = \frac{1}{ 2 \pi \ii n} \partial_x \nabla_q z_{- n}, \qquad  
   (d_q \Phi^{kdv})^{- 1}[e^{(- n)}] =  - \frac{1}{2 \pi \ii n} \partial_x \nabla_q z_{n}\,. 
   \end{equation}
   It then follows from \cite[Theorem 9.5, 9.7]{KP} and the formulas \eqref{armadillo 2}, \eqref{aramdillo 2} 
   for the functions $H_n$, $G_n$ together with their properties stated in Proposition \ref{Re/ImFloquet solutions} 
   that for any $q \in M_S$ and $n \in S_+^\bot$
   \begin{equation}\label{formula gradient zn}
   \partial_x \nabla_q z_n = \sqrt{n \pi} \partial_x \nabla_q (x_n - \ii y_n) = \sqrt{n \pi} \frac{\xi_n}{2} e^{- \ii \beta_n} \partial_x (H_n - \ii G_n)^2
   \end{equation}
   and similarly
   \begin{equation}\label{formula gradient z-n}
   \partial_x \nabla_q z_{- n} = \sqrt{n \pi} \partial_x \nabla_q (x_n + \ii y_n) = \sqrt{n \pi} \frac{\xi_n}{2} e^{\ii \beta_n} \partial_x (H_n + \ii G_n)^2
   \end{equation}
   where  $\beta_n = \sum_{\ell \in S_+} \beta_\ell^n$ (cf \cite[Theorem 8.5]{KP}) and $\xi_n = \sqrt{8 I_n / \gamma_n}$ (cf \cite[Theorem 7.3]{KP}). 
   Since $q \in M_S$ and $n \in S^\bot_+$ one has $\gamma_n(q) = 0$ and the factor $\xi_n(q)$  is obtained by a limiting argument. By a slight abuse of terminology,
we denote this limit also by $\sqrt{8 I_n(q) / \gamma_n^2(q)}$.
   The formulas \eqref{formula gradient zn} - \eqref{formula gradient z-n} allow to express $ \partial_x \nabla_q z_{\pm n}$ in terms of the Floquet solutions 
   $f_n(x) \equiv f_n(x, q)$ as follows (cf Appendix \ref{sezione floquet solution} for notations). 
   \begin{proposition}\label{lemma zn nabla q}
   For any $q \in M_S$ and $n \in S_+^\bot$, 
   $$
   \partial_x \nabla_q z_n = \sqrt{n \pi} \frac{\xi_n}{2} e^{- \ii \beta_n} \big(- \frac{2 \dot m_2(\tau_n)}{\ddot{\Delta}(\tau_n)} \big) \partial_x f_{- n}^2\,,
   \qquad
   \partial_x \nabla_q z_{- n} = \sqrt{n \pi} \frac{\xi_n}{2} e^{\ii \beta_n} \big( -\frac{2 \dot m_2(\tau_n)}{\ddot{\Delta}(\tau_n)} \big) \partial_x f_n^2\,.
   $$
   Hence by \eqref{formula vector field X z+- n}
   $$
   (d_q \Phi^{kdv})^{- 1}[e^{(n)}] = \sqrt{n \pi} \, \xi_n e^{\ii \beta_n} \frac{\dot m_2(\tau_n)}{\ddot{\Delta}(\tau_n)}  \frac{- 1}{2 \pi \ii n} \partial_x f_n^2\,,
   \qquad
   (d_q \Phi^{kdv})^{- 1}[e^{(- n)}] = \sqrt{n \pi} \, \xi_n e^{- \ii \beta_n} \frac{\dot m_2(\tau_n)}{\ddot{\Delta}(\tau_n)}  \frac{1}{2 \pi \ii n} \partial_x f_{- n}^2\,.
   $$
   \end{proposition}
   \begin{remark}
   At $q = 0$, $\sqrt{n \pi} \, \xi_n = 1$, $\beta_n = 0$, $\frac{\dot m_2(\tau_n)}{\ddot{\Delta}(\tau_n)} = - 1$ and 
   $f_{\pm n}(x) = e^{\pm  \pi \ii n x}$ for any $n \geq 1$, confirming that $(d_0 \Phi^{kdv})^{- 1}[e^{(\pm n)}] = e^{\pm 2 \pi \ii n x}$ (cf \cite{KP}). 
   \end{remark}
   \noindent
   For any given $q \in M_S$ and $n \in S_+^\bot$, introduce the functions $W_{\pm n}(x) \equiv W_{\pm n}(x, q)$ given by 
   \begin{equation}\label{definition Wn}
   W_n(x) := \sqrt{n \pi } \,  \xi_n \frac{\dot m_2(\tau_n)}{\ddot{\Delta}(\tau_n)} e^{\ii \beta_n} \frac{- 1}{2 \pi \ii n} \partial_x f_n^2(x)\,, \qquad 
    W_{-n}(x) := \sqrt{n \pi } \,  \xi_n \frac{\dot m_2(\tau_n)}{\ddot{\Delta}(\tau_n)} e^{- \ii \beta_n}\frac{1}{2 \pi \ii n} \partial_x f_{-n}^2(x)\,.
   \end{equation}
   We record that  for any $n \in S_+^\bot$, $W_{- n} = \overline{W_n}$ since $f_{-n} = \overline{f_n}$. 
   Combining Proposition \ref{prop Psi L Psi kdv} and Proposition \ref{lemma zn nabla q} 
   one obtains the following formula for the map $\Psi_L$:
   \begin{corollary}\label{definition Psi_1 z_S}
   For any $z = (z_S, z_\bot) \in {\cal V}$, one has
   $ d_\bot \Psi^{kdv}(z_S, 0)[z_\bot]  = \Psi_1(z_S)[z_\bot]$ where
   \begin{equation}\label{forma finale Psi L def}
     \Psi_1(z_S)[z_\bot](x):= \sum_{n \in S^\bot} z_n W_n(x, q)\,, \quad q = \Psi^{kdv}(z_S, 0)\,.
   \end{equation}
   \end{corollary}
   \noindent
   Note that $\Psi_1(z_S)[z_\bot] = \Psi_L(z) - q$ is linear in $z_\bot$. 
   Since $q \in M_S$ is a finite gap potential, it is $C^\infty$-smooth and so is $W_n(x, q)$. Next we want to show that $\Psi_L(z)$ admits  an expansion of the type stated in Theorem \ref{modified Birkhoff map}. 
   Recall from the Addendum to Theorem \ref{Theorem Birkhoff coordinates} at the end of Appendix \ref{Birkhoff map} that for any $q \in M_S$, $V^*_{q, S}$ denotes a neighborhood of $q$,
   consisting of complex valued $S-$gap potentials in the weighted Sobolev space $H^{w_*}_{0, \C}$ so that the restriction of $\Phi^{kdv}$ to $V^*_{q, S}$
   is a real analytic diffeomorphism onto its image $\Phi^{kdv}(V^*_{q, S}) \subset h^0_{S, \C}$.
 Combining Theorem \ref{teorema espansione fn appendice} and Lemma \ref{lemma appendice espansione xi n} - Lemma \ref{asymptotics beta_n}
 of Appendix \ref{appendix asymptotics} and using that
 $$
 \sum_{n \in S^\bot} z_n \frac{1}{(2 \pi \ii n)^k} e^{2\pi \ii nx} = \partial_x^{-k} ( \sum_{n \in S^\bot} z_n  e^{2\pi \ii nx} ) = \partial_x^{-k} \mathcal F_\bot^{-1}[z_\bot]
 $$
 one obtains the following
\begin{theorem}\label{lemma asintotica Floquet solutions}
(i) Let  $q \in M_S$ and $N \in \N$. Then for any $p \in V^*_{q, S}$,
$W_n(x) \equiv W_n(x, p)$, $n \in S^\bot$, has an expansion as $|n| \to \infty$ of the form 
\begin{equation}\label{final asymptotic Wn}
W_n(x, p) = e^{2 \p \ii n x} \Big( 1  + \sum_{k = 1}^N \frac{W_{k}^{ae}(x, p)}{(2 \pi \ii n)^k} + \frac{{\cal R}_{N}^{W_n}(x, p)}{(2 \pi \ii n)^{N + 1}} \Big)
\end{equation}
where for any $s \in \Z_{\geq 0}$, 
$W_k^{ae} : V^*_{q, S} \to H^s_{ \C}$, $p \mapsto W_{k}^{ae}(\cdot, p)$, $k \ge 1$, are real analytic and 
$ {\cal R}_{N}^{W_n} : V^*_{q, S}  \to H^s_\C, \,p  \mapsto {\cal R}_{N}^{W_n}(\cdot, p)$, $n \in S^\bot$,
are analytic and  satisfy for any $j \geq 0$, 
$$
{\rm sup}_{\begin{subarray}{c}
0 \leq x \leq 1 \\
n \in S^\bot
\end{subarray}} |\partial_x^j {\cal R}_N^{W_n}(x, p)| \leq C_{N, j}\,.
$$
The constants $C_{N, j}$ can be chosen locally uniformly for $p \in V^*_{q, S}$. 
By a slight abuse of terminology, in the sequel, we will view $W_k^{ae}(\cdot, q)$ and ${\cal R}_{N}^{W_n}(\cdot, q)$  as functions of $z_S$, 
$$
W_k^{ae}(\cdot, z_S) \equiv W_k^{ae}(\cdot, \, \Psi^{kdv}(z_S, 0))\,, \qquad  {\cal R}_{N}^{W_n}(\cdot, z_S) \equiv {\cal R}_{N}^{W_n}(\cdot, \, \Psi^{kdv}(z_S, 0))\,.
$$
\noindent
(ii) For any $z_S \in h^0_S$, 
the linear operator $\Psi_1(z_S)$, given by
$$
\Psi_1(z_S) : h^0_\bot \to L^2_0, \, \widehat z_\bot \mapsto \Psi_1(z_S)[\widehat z_\bot]= \sum_{n \in S^\bot} \widehat z_n W_n( \cdot, q)\,, \qquad q = \Psi^{kdv}(z_S, 0)\,,
$$
has the property that for any $s \in \Z_{\ge 0}$, its restriction to $h^s_\bot$ is a bounded linear operator $h^s_\bot \to H^s_0$.
Furthermore, up to a remainder,  the operator $\Psi_1(z_S): h^0_\bot \to L^2_0$ is a pseudodifferential operator of order $0$. 
More precisely, $\Psi_1(z_S)$ has an expansion to any order $N \ge 1$ of the form
$$
\Psi_1(z_S)=  \big( \text{Id} +  \sum_{k = 1}^N a_{k}(z_S; \Psi_1) \partial_x^{- k} \big) \circ {\cal F}_\bot^{-1} + {\cal R}_{N}(z_S; \Psi_1)\,, \qquad
$$
where
$$
a_{k}(z_S; \Psi_1) := W_{k}^{ae}(\cdot, z_S), \,\, k \ge 1, \qquad
{\cal R}_{N}(z_S; \Psi_1)[\widehat z_\bot] (x) := \sum_{n \in S^{\bot}} \widehat z_n \frac{{\cal R}_{N}^{W_n}(x, z_S)}{(2 \pi \ii n)^{N + 1}} e^{2 \pi \ii nx}\,.
$$
For any $s \ge 0$, the restriction of ${\cal R}_{N}(z_S; \Psi_1)$ to $h^s_\bot$ defines a bounded linear operator $h^s_\bot \to H^{s+ N+1}$ and the map
$$
h^0_S \to {\cal B}(h^s_\bot, H^{s + N+1}), \, z_S \mapsto  {\cal R}_{N}(z_S; \Psi_1)\,,
$$
is real analytic. Corresponding properties hold for the map $\Psi_L: {\cal V} \to L^2_0,$ defined
in Proposition~\ref{prop Psi L Psi kdv}, 
\begin{equation}\label{definizione Psi L Phi 1}
\Psi_L(z_S, z_\bot) =  q + \Psi_1(z_S)[z_\bot] = q + {\cal F}_\bot^{-1}[z_\bot] + \sum_{k = 1}^N a_{k}(z_S; \Psi_L) \partial_x^{- k} {\cal F}_\bot^{-1}[z_\bot]+ 
{\cal R}_{N}(z_S; \Psi_L)[z_\bot]
\end{equation}
where
$$
a_{k}(z_S; \Psi_L):= a_{k}(z_S; \Psi_1), \,\,\, k \ge 1 \qquad  {\cal R}_{N}(z_S; \Psi_L):= {\cal R}_{N}(z_S; \Psi_1)
$$
\end{theorem}

\smallskip

\begin{remark}\label{notation coefficients expansion}
(i) Note that the pseudodifferential operator $ \big( \text{Id} +  \sum_{k = 1}^N a_{k}(z_S; \Psi_1) \partial_x^{- k} \big) \circ {\cal F}_\bot^{-1} $
defines a bounded linear operator $h^s_\bot \to H^s$ for any $s \in \R$ whereas the  remainder ${\cal R}_{N}(z_S; \Psi_1)$
defines a bounded linear operator  $h^s_\bot \to H^{s +N + 1}$ for any $s \ge -N - 1$. \\
(ii) Whenever possible, we will use similar notation for the coefficients of the expansion of the various quantities such as $\Psi_1(z_S)$.
If the coefficients are operators, we use the upper case letter $A$
and write $\mathcal A_k$ for the kth coefficient, whereas when they are functions (or operators, defined as the multiplication by a function),
we use the lower case letter $a$ and write $a_k$ for the k'th coefficient. The quantity, which is expanded, is indicated
as an argument of $\mathcal A_k$ and $a_k$.\\
(iii) The fact that up to a remainder term, $ \Psi_L(z_S, \cdot)$ is given by the pseudodifferential operator of order 0,  
$(Id + \sum_{k = 1}^N a_{k}(z_S; \Psi_1) \partial_x^{- k}) \circ \mathcal F^{-1}_\bot$, 
 acting on the scale of Hilbert spaces $h^s_\bot$, $s \in \Z_{\ge 0}$, is at the heart of this paper.
The result shows that the differential of the Birkhoff map $z \mapsto \Psi^{kdv}(z)$ at a finite gap potential, has distinctive features.
 \end{remark}
 
\medskip

A straightforward application of Theorem~\ref{lemma asintotica Floquet solutions}(ii)
 yields an expansion of the transpose operator $\Psi_1(z_S)^t$ of $\Psi_1(z_S)$.  
Since  ${\cal F}_\bot ^{- 1}$ is the restriction of the inverse of the Fourier transform to $h^0_\bot$, 
 the transpose  ${\cal F}_\bot^{- t} := ({{\cal F}_\bot^{- 1}})^t$ of ${\cal F}_\bot^{- 1}$ with respect to the standard inner products in $L^2_0$ and $h^0_\bot$ 
is given by the Fourier transform, i.e., for any $\widehat q \in L^2_0,$
$$
\langle {\cal F}^{- 1}_\bot[z_\bot], \widehat q \rangle = 
\int_0^1 \sum_{n \in S^\bot} z_n e^{2 \pi \ii n x} \widehat q (x) d x = \sum_{n \in S^\bot} z_n \int_0^1 \widehat q(x) e^{2 \pi \ii n x} d x = \sum_{n \in S^\bot} z_n \widehat q_{-n}
= \langle z_\bot, {\cal F}_\bot \widehat q \rangle\,.
$$
\begin{corollary}\label{lemma asintotiche Phi 1 q}
For any $z_S \in h^0_S$, $q = \Psi^{kdv}(z_S, 0)$, and $N \in \N$,  
$\Psi_1(z_S)^t : L^2 \to h^0_\bot, \widehat q \mapsto (\langle W_{-n}(\cdot, q), \, \widehat q \rangle )_{n \in S^\bot}$ 
 has an expansion of the form
\begin{align}
\Psi_1(z_S)^t & = {\cal F}_\bot  \circ ( \text{Id} + \sum_{k = 1}^N  a_k( z_S; \Psi_1^t)   \partial_x^{- k} ) + 
{\cal R}_N( z_S;  \Psi_1^t) \label{espansione finale Phi 1 t}
\end{align}
where for any $s \geq 0$, the coefficients $h^0_S \to H^s$, $z_S \mapsto  a_k( z_S; \Psi_1^t)$, $k \ge 1$, and the remainder
$h^0_S \to {\cal B}(H^s, h^{s + N + 1}_\bot)$, $z_S \mapsto {\cal R}_N( z_S;  \Psi_1^t),$ are real analytic. 
Furthermore, $a_k( z_S; \Psi_1^t) = - a_k( z_S; \Psi_1)$.
Corresponding properties hold for the map 
$$
\mathcal V \to \mathcal B(L^2, h^0_0), z \mapsto d\Psi_L(z)^t. 
$$
For any $z \in \mathcal V,$ $N \in \N$, $d\Psi_L(z)^t$ has an expansion of the form
\begin{align}
d\Psi_L(z)^t & = \big( 0, {\cal F}_\bot  \circ ( \text{Id} + \sum_{k = 1}^N  a_k( z; d\Psi_L^t)   \partial_x^{- k} ) \big) + 
{\cal R}_N( z;  d\Psi_L^t) \label{expansion Phi L t}
\end{align}
where $a_k( z; d\Psi_L^t) = a_k( z_S; \Psi_1^t)$ and where for any integer $s \ge 0,$  $\mathcal V \cap h^s_0 \to B(L^2, h^{s+N+1}_0), z \mapsto {\cal R}_N( z;  d\Psi_L^t)$
is real analytic.
\end{corollary}
\begin{remark}\label{extension of remainder} 
 Again we record that the pseudodifferential operator $ {\cal F}_\bot  \circ \big( \text{Id} + \sum_{k = 1}^N  a_k( z_S; \Psi_1^t)   \partial_x^{- k} \big)$
defines a bounded linear operator $H^s \to h^s_\bot$ for any $s \in \R$ whereas the  remainder ${\cal R}_N( z_S;  \Psi_1^t)$
defines a bounded linear operator  $H^s \to h_\bot^{s +N + 1}$ for any $s \ge -N - 1$. 
\end{remark}
\begin{proof}
By Theorem~\ref{lemma asintotica Floquet solutions}(ii), 
$\Psi_1(z_S) =  {\cal F}_\bot^{-1}+ \sum_{k = 1}^N a_{k}(z_S; \Psi_1) \partial_x^{- k}  {\cal F}_\bot^{-1}+ {\cal R}_{N}(z_S; \Psi_1)$ where for any $z_\bot \in h^0_\bot,$
$$
{\cal R}_{N}(z_S; \Psi_1)[z_\bot] (x) := \sum_{n \in S^{\bot}} z_n \frac{{\cal R}_{N}^{W_n}(x, z_S)}{(2 \pi \ii n)^{N + 1}} e^{2 \pi \ii nx}\,.
$$
Note that the functions $a_{k}(z_S; \Psi_1)(x)$, $k \ge 1$, are real valued.
Taking into account that ${\cal F}_\bot^{- t} = {\cal F}_\bot$ and $(\partial_x^{-k})^t = (-1)^k \partial_x^{-k}$, the expansion of the transpose $\Psi_1(z_S)^t$ of $\Psi_1(z_S)$ then reads
$$
\Psi_1(z_S)^t = {\cal F}_\bot + \mathcal F_\bot \circ \sum_{k = 1}^N  (- 1)^k \partial_x^{- k} \circ a_{k}(z_S; \Psi_1)  + ({\cal R}_{N}(z_S; \Psi_1))^t\,.
$$
By Theorem~\ref{lemma asintotica Floquet solutions}(i), for any $ \widehat q \in H^s,$ 
$$
({\cal R}_{N}(z_S; \Psi_1))^t [\widehat q] = \big( \frac{1}{(2\pi \ii n)^{N+1}} \int_0^1 \widehat q(x) {\cal R}_{N}^{W_n}(x, z_S) e^{2 \pi \ii nx} \, dx \big)_{n \in S^\bot} \in h^{s + N +1}_\bot\,,
$$
$({\cal R}_{N}(z_S; \Psi_1))^t: H^s \to h^{s + N + 1}_\bot$ is bounded, and 
the map $h^0_S \to {\cal B}(H^s, h^{s + N + 1}_\bot)$, $z_S \mapsto {\cal R}_N( z_S;  \Psi_1^t),$ is real analytic.
Since by Lemma \ref{lemma composizione pseudo} and the notation introduced there,
$$
 \partial_x^{- k} \circ a_{k}(z_S; \Psi_1)   = 
  a_{k}(z_S; \Psi_1) \,  \partial_x^{ -k} + \sum_{j = 1}^{N - k} C_j(k) \, (\partial_x^j a_{k}(z_S; \Psi_1))  \,  \partial_x^{ -k - j} 
+  {\cal R}_{N, k, 0}^{\psi do}( a_{k}(z_S; \Psi_1) )
$$
one sees that $\Psi_1(z_S)^t$ admits an expansion of the form,
\begin{align}
\Psi_1(z_S)^t & = {\cal F}_\bot +  {\cal F}_\bot \circ  \sum_{k = 1}^N  a_k( z_S; \Psi_1^t)   \partial_x^{- k} + {\cal R}_{N}(z_S; \Psi_1^t)\,,  \nonumber
\end{align}
where  $a_k( z_S; \Psi_1^t)$, $ k \ge 1$, and ${\cal R}_{N}(z_S; \Psi_1^t)$ satisfy the claimed properties.  Since
$$
d \Psi_L(z) [\widehat z] = \big( d_S\Psi^{kdv}(z_S, 0) +  d_S ( \Psi_1(z_S)[z_\bot]) \big) [\widehat z_S] + \Psi_1(z_S)[\widehat z_\bot],
$$
the claimed properties of $d \Psi_L(z)^t$ follow from the ones of $\Psi_1(z_S)^t$.
\end{proof}

\medskip

Using results of Appendix \ref{AppendixReversability} and Appendix \ref{appendix asymptotics}, one obtains the following properties of the functions $W_n$, $n \in S^\bot$, and the map $\Psi_L$
with regard to the reversible structure, introduced in Section \ref{introduzione paper}. 

\smallskip

\noindent
{\bf Addendum to Theorem \ref{lemma asintotica Floquet solutions}} {\em (i) For any $z_S \in h_S^0$, $q = \Psi^{kdv} (z_S, 0)$ satisfies $S_{rev}q = \Psi^{kdv} (\mathcal S_{rev}(z_S, 0))$
and for any  $n \in S^\bot$, $x \in \R$,
$W_n(x, S_{rev}q) = W_{-n}( - x, q)$ as well as ( $k \ge 1$, $N \ge 1$)
\begin{equation}\label{symmetries of expansion}
W^{ae}_k(x, \mathcal S_{rev}z_S) = (-1)^k W^{ae}_{k}( - x, z_S)\,,  \qquad {\cal R}_{N}^{W_n}(x, \mathcal S_{rev}z_S) = (-1)^{N+1} {\cal R}_{N}^{W_{-n}}(- x, z_S)\,.
\end{equation}
(ii)
For any $z = (z_S, z_\bot) \in h^0_S \times h^0_\bot$ and $x \in \R$, 
$$
\big(\Psi_1(\mathcal S_{rev}z_S)[\, \mathcal S_{rev}z_\bot] \big)(x) =  \big( \Psi_1(z_S)[z_\bot] \big) (-x).
$$
As a consequence,  
for any $z \in \mathcal V$, $x \in \R$
\begin{equation}\label{symmetries of PsiL}
(\Psi_L( \mathcal S_{rev}z)) (x) = ( \Psi_L(z) ) (-x)\,, \qquad 
({\cal R}_{N}( \mathcal S_{rev}z_S; \Psi_1)[\mathcal S_{rev}z_\bot]) (x) = \big( {\cal R}_{N}(z_S; \Psi_1)[z_\bot] \big)(-x)\,.
\end{equation}
(iii) For any $z_S \in h^0_S$ and $\widehat q \in L^2_0$, one has
$ \Psi_1(\mathcal S_{rev} z_S)^t [S_{rev} \widehat q]  = \mathcal S_{rev}( \Psi_1(z_S)^t [ \widehat q])$. As a consequence,
 for any $k \ge 1$ and $N  \ge  1$,
$$
 a_k( \mathcal S_{rev}z_S; \Psi_1^t) (x)  = (-1)^k  a_k (z_S; \Psi_1^t) (-x) \,,  \qquad 
 {\cal R}_N(\mathcal S_{rev}z_S;  \Psi_1^t)[S_{rev} \widehat q]  = \mathcal S_{rev} \big( {\cal R}_N(z_S; \Psi_1^t) [\widehat q] \big)\,.
$$
}
\noindent
{\bf Proof of Addendum to Theorem \ref{lemma asintotica Floquet solutions}}
(i) By the Addendum to Theorem \ref{teorema espansione fn appendice}, we know that for any $q \in M_S$ and $n \in S^\bot_+$, $f_{\pm n}(x, S_{rev}q) = f_{\mp n}( - x, q)$. 
Furthermore, one has $\xi_n(S_{rev}q) = \xi_n(q)$, $\Delta(\lambda, S_{rev}q) = \Delta(\lambda, q ),$
$m_2(\lambda, S_{rev}q ) = m_2(\lambda, q )$ (Lemma \ref{lemma 2 reversibilita}) and $\b_n(S_{rev}q) = - \b_n(q)$ (Corollary \ref{corollario 3 reversibilita}).
In view of the definition \eqref{definition Wn} of $W_{\pm n}$ it then follows that $W_{\pm n}(x, S_{rev}q) = W_{\mp n}( - x, q)$
and in turn, comparing the expansion \eqref{final asymptotic Wn} of $W_{\pm n}(x, S_{rev}q)$ with the one of  $W_{\mp n}( - x, q)$,
one obtains the identities \eqref{symmetries of expansion}. (ii) By (i) one has for any $z_S \in h^0_S$ and $q = \Psi^{kdv} (z_S, 0)$,
$$
\big(\Psi_1(\mathcal S_{rev}z_S)[\, \mathcal S_{rev}z_\bot] \big)(x) = \sum_{n \in S^\bot} z_{-n} W_n( x, S_{rev} q) = \sum_{n \in S^\bot} z_{-n} W_{-n}( - x, q) = \big( \Psi_1(z_S)[z_\bot] \big) (-x)
$$
as well as $W^{ae}_k(x, \mathcal S_{rev}z_S) = (-1)^k W^{ae}_{k}( - x, z_S)$ and
$({\cal R}_{N}(\mathcal S_{rev}z_S; \, \Psi_1)[\mathcal S_{rev}z_\bot] )(x) = \big( {\cal R}_{N}(z_S; \Psi_1)[z_\bot] \big)(-x)$.
By \eqref{definizione Psi L Phi 1}, the claimed identities \eqref{symmetries of PsiL} then follow.
(iii) Recall that for any $z_S \in h^0_S$, $\widehat q \in L^2_0$, one has $\Psi_1(z_S)^t [\widehat q] = (\langle W_{-n}(\cdot, q), \, \widehat q \rangle )_{n \in S^\bot}$.
It then follows from item (i) that
$$
\Psi_1(\mathcal S_{rev}z_S)^t  [ S_{rev} \widehat q]  = \mathcal S_{rev} \big( \Psi_1(z_S)^t [  \widehat q] \big).
$$
Comparing the expansion \eqref{espansione finale Phi 1 t} for $\mathcal S_{rev} \big(\Psi_1(\mathcal S_{rev}z_S)^t  [\widehat q] \big) $  with the one for  $\Psi_1(z_S)^t [ S_{rev} \widehat q]$
and taking into account that $\mathcal S_{rev} \circ \mathcal F_\bot = \mathcal F_\bot \circ S_{rev}$ and $\partial_x \circ S_{rev} = - S_{rev} \circ \partial_x$ one sees that for any $k \ge 1,$
$$
 a_k (\mathcal S_{rev} z_S; \Psi_1^t) (x) = (-1)^k  a_k (z_S; \Psi_1^t) (-x)\,, \qquad   
 {\cal R}_N(\mathcal S_{rev}z_S;  \Psi_1^t)[ S_{rev}\widehat q] = \mathcal S_{rev} ( {\cal R}_N(z_S; \Psi_1^t) [ \widehat q] )\,.  \qquad \qquad \square
$$

\bigskip


In the remaining part of this section we describe the pull back $\Psi_L^* \Lambda_G$ of the symplectic form $\Lambda_G$ by the map $\Psi_L$, defined in Proposition~\ref{prop Psi L Psi kdv},
where $\Lambda_G$, defined by the Gardner Poisson structure, is given by
$$
\Lambda_G [\widehat u , \widehat v] = \langle  \partial_x^{-1} \widehat u ,  \widehat v \rangle = \int_0^1 ( \partial_x^{-1} \widehat u)(x) \widehat v(x) d x   \,, \qquad \forall  \widehat u, \widehat v  \in L^2_0\,.
$$
Note that $\Lambda_G = d \lambda_G$ where the one form $\lambda_G$, defined on $L^2_0$, is given by 
\begin{equation}\label{definition lambda G}
\lambda_G(u)[ \widehat v] = \langle \partial_x^{-1} u , \widehat v \rangle = \int_0^1 (\partial_x^{-1} u)(x) \widehat v(x) d x\,, \quad \forall u, \widehat v \in L^2_0\,.
\end{equation}
To compute the pull back of $\Lambda_G$ by $\Psi_L$, note that for any $z = (z_S, z_\bot) \in \mathcal V = \mathcal V_S \times \mathcal V_\bot$,
the derivative $d \Psi_L(z) $, when written in $1 \times 2$ matrix form, is given by (cf \eqref{forma finale Psi L def}) 
\begin{align}\label{formula differential of Psi L}
d \Psi_L(z) & = d \Psi_L(z_S, 0) + \begin{pmatrix}
d_S ( \Psi_1(z_S)[z_\bot] ) & 0
\end{pmatrix}  
 = \begin{pmatrix}
d_S\Psi^{kdv}(z_S, 0) &  \Psi_1(z_S)
\end{pmatrix} + \begin{pmatrix}
d_S ( \Psi_1(z_S)[z_\bot])  & 0
\end{pmatrix} \,. 
\end{align}

For any $\widehat z = (\widehat z_{S}, \widehat z_{\bot}) ,$  $\widehat w = (\widehat w_{S}, \widehat w_{\bot})  \in h^0_0$ one has
\begin{align}
& (\Psi_L^* \Lambda_G)(z) [\widehat z, \widehat w]  = \Lambda_G \big[ d \Psi_L(z) [\widehat z], \, d \Psi_L(z)[\widehat w ]\big] 
 =  \big\langle \partial_x^{- 1} d \Psi_L(z)[\widehat z], \, d \Psi_L(z)[\widehat w] \big\rangle \nonumber\\
&  =
\Big\langle \partial_x^{- 1} \big( \, d \Psi_L(z_S, 0)[\widehat z] \big)+  \partial_x^{- 1} \big( d_S( \Psi_1(z_S)[z_\bot]\big)[\widehat z_S]) , \,
d \Psi_L(z_S, 0)[\widehat w]  + d_S  \big(\Psi_1(z_S)[z_\bot] \big) [\widehat w_S]  \Big\rangle \,. \nonumber
\end{align}
Since by construction, $d \Psi_L(z_S , 0) : h^0_0  \to  L^2_0$  is symplectic, one has 
$$
(\Psi_L^* \Lambda_G)(z_S, 0) = \Lambda
$$
where $\Lambda$ is the symplectic form on $h^0_0$,
\begin{equation}\label{1 forma Lambda M}
\Lambda [\widehat z, \widehat w] :=  \langle J^{-1}\widehat z, \widehat w \rangle = \sum_{n \ne 0} \frac{1}{ 2\pi \ii n} \widehat z_{ n} \widehat w_{-n }\,,
\quad \forall \widehat z, \widehat w \in h^0_0
\end{equation}
and $J^{-1}$ denotes the inverse of the diagonal operator, acting on the scale of Hilbert spaces $h^s_0,$ $s \in \R,$
\begin{equation}\label{definition J}
J: h^{s+1}_0 \to h^{s}_0 :\, (z_n)_{n \ne 0} \mapsto  (2\pi \ii n z_n)_{n \ne 0}\,.
\end{equation}
Note that $\Lambda = d \lambda$ where $\lambda$ is the one form on $h^0_0$, 
\begin{equation}\label{definition lambda}
\lambda(z)[\widehat w] : =  \langle J^{-1} z, \widehat w \rangle =  \sum_{n \ne 0} \frac{1}{2 \pi \ii n} z_n \widehat w_{-n}\,, \quad \forall z, \widehat w \in h^0_0\,.
\end{equation}
Altogether we have that 
\begin{equation}\label{Pull back of LambdaG}
 (\Psi_L^* \Lambda_G)(z) [\widehat z, \widehat w] 
 = \Lambda[\widehat z, \widehat w] +   \Lambda_L(z)[ \widehat z, \widehat w] \,, \qquad 
 \Lambda_L(z)[\widehat z, \widehat w]  = \big\langle \mathcal L(z)[\widehat z], \,  \widehat w \big\rangle
\end{equation}
where 
the operator $ \mathcal  L(z) : h^0_S \times h^0_\bot \to  h^0_S \times h^0_\bot$ has the form 
\begin{equation}\label{definition L z}
 \mathcal L(z) = \begin{pmatrix}
 \mathcal L_S^S(z) &  \mathcal L_S^{\bot}( z) \\
 \mathcal L_{\bot}^S( z) & 0
\end{pmatrix} 
\end{equation}
with $ \mathcal L_S^S(z) : h^0_S \to h^0_S$, $ \mathcal L_S^{\bot}(z) : h^0_\bot  \to h^0_S$, and $ \mathcal L_\bot^S : h^0_S \to h^0_\bot$ given by 
\begin{equation}\label{L SS botS Sbot}
\begin{aligned}
 \mathcal L_S^S(z) & := \big(d_S \Psi_1(z_S)[z_\bot]\big)^t \, \partial_x^{- 1} \big( d_S \Psi_1(z_S)[ z_\bot] \big)
 + (d_S \Psi^{kdv}(z_S, 0))^t  \partial_x^{- 1} \big( d_S \Psi_1(z_S)[ z_\bot] \big)\\
 & + \big( d_S \Psi_1(z_S)[ z_\bot] \big)^t \partial_x^{-1} d_S \Psi^{kdv}(z_S, 0) \,, \\
 \mathcal L_S^{\bot}(z) & :=  \big(d_S \Psi_1(z_S)[z_\bot] \big)^t \,  \partial_x^{- 1} \Psi_1(z_S) \,, \qquad
 \mathcal L_\bot^S (z)  :=\Psi_1(z_S)^t \, \partial_x^{- 1} \big( d_S \Psi_1(z_S)[z_\bot] \big) \,. 
\end{aligned}
\end{equation}
For any $z = (z_S, z_\bot) \in \mathcal V$, the operators $ \mathcal L(z),$ $ \mathcal L_S^S(z),$  $ \mathcal L_S^{\bot}(z),$ and $ \mathcal L_\bot^S(z)$ are bounded.
In the sequel, we will often write the operators \eqref{L SS botS Sbot} in the following way
\begin{equation}\label{L SS botS Sbot (1)}
\begin{aligned}
 \mathcal L_S^S(z) [\widehat z_S] & = \Big( \big\langle \partial_x^{- 1} d_S \big(\Psi_1(z_S)[z_\bot] \big)[\widehat z_S]\,,\, \, \partial_{z_n}\Psi_1(z_S)[z_\bot] \big\rangle \Big)_{n \in S}
  + \Big( \big\langle \partial_x^{- 1} d_S \big(\Psi_1(z_S)[z_\bot] \big)[\widehat z_S]\,,\, \, \partial_{z_n}\Psi^{kdv}(z_S, 0) \big\rangle \Big)_{n \in S}\\
  & +  \Big( \big\langle \partial_x^{- 1} d_S (\Psi^{kdv}(z_S, 0))[\widehat z_S] \,,\, \, \partial_{z_n}\Psi_1(z_S)[z_\bot] \big\rangle \Big)_{n \in S} \,, \\
 \mathcal L_S^{\bot}(z)[\widehat z_\bot] & =  \Big( \big\langle \partial_x^{- 1} \Psi_1(z_S)[\widehat z_\bot]\,,\,\, \partial_{z_n} \Psi_1(z_S)[z_\bot] \big\rangle \Big)_{n \in S} \,, \\
 \mathcal L_\bot^S(z) [\widehat z_S] & = \Psi_1(z_S)^t \partial_x^{- 1} d_S (\Psi_1(z_S)[z_\bot] ) [\widehat z_S] =
 \Big( \big\langle W_{-n}(\cdot, q), \,\,  \partial_x^{- 1} \, d_S (\Psi_1(z_S)[z_\bot] ) [\widehat z_S] \big\rangle \Big)_{n \in S^\bot} 
\end{aligned}
\end{equation}
where $q = \Psi^{kdv}(z_S, 0)$.
The operators $\mathcal L_S^S(z)$, $\mathcal L_S^{\bot}(z)$, and $\mathcal L_\bot^S(z)$ satisfy the following properties.
\begin{lemma}\label{espansione L S bot q z}
$(i)$ The maps 
$$
\mathcal V \to {\cal B}(h^0_S, h^0_S), \, z \mapsto  \mathcal L_S^S(z) \,, \qquad  
\mathcal V \to {\cal B}\big(h^0_\bot, h^0_S \big), \,  z \mapsto  \mathcal L_S^{\bot}(z)
$$
 are real analytic. Furthermore, the following estimates hold: for any $z= (z_S, z_\bot) \in \mathcal V$, $\widehat z_S \in h^0_S$, 
   and  $\widehat z_1, \ldots, \widehat z_l \in  h^0_0$, $l \geq 1$ ,
$$
\begin{aligned}
& \|  \mathcal L_S^S(z)[\widehat z_S]\| \lesssim \| \widehat z_S\| \| z_\bot \|_0 \,, \quad 
& \| d^l \big(  \mathcal L_S^S(z)[\widehat z_S] \big)[\widehat z_1, \ldots, \widehat z_l] \| 
\lesssim_{l} \| \widehat z_S\| \prod_{j = 1}^l \| \widehat z_j\|_0 \,, 
\end{aligned}
$$
and if in addition, $\widehat z_\bot \in h^0_\bot$,
$$
\begin{aligned}
& \|  \mathcal L_S^{\bot}(z) [\widehat z_\bot]\| \lesssim \| z_\bot \|_0 \| \widehat z_\bot \|_0\,, \qquad
 \| d^l \big(  \mathcal L_S^{\bot}(z) [\widehat z_\bot] \big)[\widehat z_1, \ldots, \widehat z_l] \| \lesssim_{l} \| \widehat z_\bot \|_0 \prod_{j = 1}^l \| \widehat z_j\|_0\,. 
\end{aligned}
$$
\noindent
$(ii)$  
For any $z = (z_S, z_\bot) \in \mathcal V$, 
$ \mathcal L_\bot^S (z)$ has an expansion of arbitrary order $N \ge 2$,
 \begin{align}\label{expansion L bot S}
 \mathcal L_{\bot}^S(z)  & = {\cal F}_\bot \circ \sum_{k = 2}^N  \mathcal A_k(z_S; \mathcal L_\bot^S)   \,\, 
 \partial_x^{- k} {\cal F}_\bot^{- 1}[z_\bot] + {\cal R}_{N}(z; \mathcal L_{\bot}^S)
\end{align}
where for any $s \geq 0$, $k \ge 1,$ the maps 
$$ \mathcal V_S \mapsto {\cal B}(h^0_S, H^s), \, z_S \mapsto \mathcal A_k(z_S; \mathcal L_\bot^S) \,, \qquad 
\mathcal V \cap h^s_0 \to {\cal B}\big(h^0_S, h^{s + N +1}_\bot \big), \, z \mapsto {\cal R}_{N}(z; \mathcal L_{\bot}^S)
$$ 
are real analytic. In particular, the operator $\mathcal L_\bot^S(z)$ is two smooothing. More precisely, for any $s \ge 0$,
$$
 \mathcal V \cap h_0^s \to {\cal B}(h^0_S, h^{s+2}_\bot), \, z \mapsto  \mathcal L_\bot^S(z)
$$
is real analytic. 
The coefficients  $\mathcal A_k(z_S; \mathcal L_\bot^S)$ are independent of $z_\bot$ and satisfy for any $s \ge 0,$ $z_S \in \mathcal V_S,$
$\widehat z_S \in h^0_S,$ $\widehat z_1, \ldots, \widehat z_l \in h^0_0$, $l \geq 1$, the following estimates
$$
\| \mathcal A_k ( z_S; \mathcal L_{\bot}^S) [\widehat z_S]  \|_s \lesssim_{s, k}  \| \widehat z_S \|\,, \qquad
 \| d^l \big( \mathcal A_k ( z_S; \mathcal L_{\bot}^S) [\widehat z_S]  \big) [\widehat z_1, \ldots, \widehat z_l] \|_{s}  \lesssim_{s, k, l}  \| \widehat z_S \|  \prod_{j = 1}^l \| \widehat z_j\|_0\,.
$$
Furthermore, for any $s \in \Z_{\ge 0}$, $z = (z_S, z_\bot) \in \mathcal V \cap h^s_0$, $\widehat z_S \in h^0_S$, and
$\widehat z_1, \ldots, \widehat z_l \in h^s_0$, $l \geq 1$, ${\cal R}_{N}(z; \mathcal L_{\bot}^S) [\widehat z_S]$ satisfies
$ \| {\cal R}_{N}(z; \mathcal L_{\bot}^S) [\widehat z_S]\|_{s + N+1} \lesssim_{s, N} \| \widehat z_S\|  \| z_\bot \|_s$ and 
$$
 \| d^l \big( {\cal R}_{N}(z; \mathcal L_{\bot}^S) [\widehat z_S]  \big) [\widehat z_1, \ldots, \widehat z_l] \|_{s + N+1} 
\lesssim_{s, N, l}  \| \widehat z_S \| \big( \sum_{j = 1}^l \| \widehat z_j\|_s \prod_{i \neq j} \| \widehat z_i\|_0 \, + \, \| z_\bot \|_s \prod_{j = 1}^l \| \widehat z_j\|_0 \big) \,. 
$$
(iii) As a consequence, for any integer $s \ge 0$, the map $\mathcal V \cap h^s_0 \to \mathcal B(h^0_0, h^{s+2}_0), z \mapsto  \mathcal L(z)$ is real analytic. Furthermore,
for any $z = (z_S, z_\bot) \in \mathcal V \cap h^s_0$ and $\widehat z \in h^0_0$, it satisfies the estimates
\begin{equation}\label{estimate mathcal L}
\|  \mathcal L(z) [\widehat z]\|_{s+2} \le C(s; \mathcal L)  \| \widehat z \|_0 \| z_\bot \|_s
\end{equation}
and if in addition $\widehat z_1, \ldots, \widehat z_l \in h^s_0$, $l \geq 1$, one has
\begin{equation}\label{estimate derivative mathcal L}
\| d^l ( \mathcal L(z) [\widehat z] )[\widehat z_1, \ldots , \widehat z_l] \|_{s+2} 
\le C(s, l; \mathcal L)  \| \widehat z\|_0 \big( \sum_{j = 1}^l \| \widehat z_j\|_s \prod_{i \neq j} \| \widehat z_i\|_0 \, + \, \| z_\bot \|_s \prod_{j = 1}^l \| \widehat z_j\|_0 \big) 
\end{equation}
for some constants $C(s; \mathcal L) \ge 1$, $C(s, l; \mathcal L) \ge 1$. 
\end{lemma}
\begin{remark}\label{extension of mathcal L(z)}
Recall that by Remark \ref{notation coefficients expansion}(i),  $ \partial_x^{- 1} \Psi_1(z_S): h^{-1}_\bot \to H^0_0$ 
is a bounded linear operator for any $z \in \mathcal V$. Since $\partial_{z_n} \Psi_1(z_S)[z_\bot] \in H^0_0$ for any $n \in S$, it then follows that
$ \mathcal L_S^{\bot}(z): h^{-1}_\bot \to h^0_S$
and in turn $\mathcal L(z): h^{-1}_0 \to h^0_0$ are bounded linear operators. Estimates, corresponding to the ones for $\mathcal L_S^{\bot}(z)$ and $\mathcal L(z)$ 
of Lemma \ref{espansione L S bot q z}, continue to hold, when these operators are extended to $h^{-1}_\bot$ and, respectively, $h^{-1}_0$.
\end{remark}
\begin{proof}
The lemma follows in a straightforward way by using the properties of the maps $\Psi_1(z_S)$ and $\Psi_1(z_S)^t$
(cf Lemma \ref{lemma asintotiche Phi 1 q}) and  the expansion of the composition $\partial_x^{-n} \circ a \, \partial_x^{-k}$ (cf Lemma \ref{lemma composizione pseudo} in Appendix~\ref{appendice B}). 
\end{proof}

\smallskip

Finally, we discuss the properties of the symplectic forms $\Lambda_G$, $\Lambda$, and $\Psi_L^*\Lambda_G$ with respect to
the reversible structures introduced in Section \ref{introduzione paper}. 
First note that for any $\widehat u, \widehat v \in L^2_0,$
$$
(S_{rev}^*\Lambda_G)[ \widehat u, \widehat v] = \Lambda_G[ S_{rev}\widehat u, S_{rev}\widehat v] =
\int_0^1 \partial_x^{-1}( \widehat u(-x) ) \widehat v(-x) d x =  - \Lambda_G[ \widehat u, \widehat v]
$$
and similarly, for any $\widehat z, \widehat w \in h^0_0,$
$$
(\mathcal S_{rev}^*\Lambda)[ \widehat z, \widehat w] = \Lambda[ \mathcal S_{rev}\widehat z, \mathcal S_{rev}\widehat w] =
\sum_{n \ne 0} \frac{1}{2\pi \ii n} \widehat z_{-n} \widehat w_{n}  = - \Lambda[ \widehat z, \widehat w]\,.
$$
By  the Addendum to Theorem \ref{lemma asintotica Floquet solutions}, the pullback $\mathcal S_{rev}^* \Psi_L^*\Lambda_G$ can then be computed as
$$
(\mathcal S_{rev}^*\Psi_L^*) \Lambda_G  = \Psi_L^* (S_{rev}^* \Lambda_G) = - \Psi_L^*\Lambda_G
$$
implying together with \eqref{Pull back of LambdaG} that
$$
\mathcal S_{rev}^*\Lambda_L = - \Lambda_L \,.
$$
It then follows that the operators $ \mathcal L_S^S(z)$, $ \mathcal L_S^{\bot}(z)$, and $ \mathcal L_\bot^S(z)$ have the following symmetry properties.

\medskip

\noindent
{\bf Addendum to Lemma \ref{espansione L S bot q z}} {\em For any $z = (z_S, z_\bot) \in h^0_S \times h^0_\bot$ and any $\widehat z_S \in h^0_S$, $\widehat z_\bot \in h^0_\bot$,
$$
 \mathcal L_S^S(\mathcal S_{rev} z) [\mathcal S_{rev} \widehat z_S] = -  \mathcal S_{rev} \big( \mathcal L_S^S(z) [\widehat z_S] \big)\,, \quad 
$$
$$
 \mathcal L_S^{\bot}(\mathcal S_{rev}z)[\mathcal S_{rev} \widehat z_\bot]   = - \mathcal S_{rev} \big( \mathcal L_S^{\bot}(z)[\widehat z_\bot] \big)\,, \quad
 \mathcal L_\bot^S(\mathcal S_{rev}z) [\mathcal S_{rev} \widehat z_S]  \big)= - \mathcal S_{rev} \big( \mathcal L_\bot^S(z) [\widehat z_S] \big)\,.
$$
By \eqref{expansion L bot S} it then follows that
$$
\mathcal A_k(\mathcal S_{rev} z_S; \mathcal L_\bot^S) [\mathcal S_{rev} \widehat z_S] (x)  = - (-1)^k \mathcal A_k(z_S; \mathcal L_\bot^S) [\widehat z_S] (-x)\,,
$$
$$
{\cal R}_{N}(\mathcal S_{rev}z;  \mathcal L_{\bot}^S)[\mathcal S_{rev} \widehat z_S]  = -  \mathcal S_{rev} \big({\cal R}_{N}(z; \mathcal L_{\bot}^S)[\widehat z_S] \big) \,.
$$
}

\section{The map $\Psi_C$}\label{sezione Psi C}

\noindent
In this section we construct the symplectic corrector $\Psi_C$. Our approach is based on a well known method of Moser and Weinstein,
implemented for an infinite dimensional setup in \cite{K} (cf also \cite{Kappeler-Montalto}).
We begin by briefly outlining the construction. At the end of Section \ref{sezione mappa Psi L}, we introduce the symplectic forms $\Lambda$
and $\Psi_L^*\Lambda_G.$ They are defined on $\mathcal V = \mathcal V_S \times \mathcal V_\bot$ and are related as follows ($z \in \mathcal V,$  $\widehat z, \widehat w \in h^0_0$)
$$
\Psi^*_L \Lambda_G (z) [\widehat z, \widehat w] = \Lambda[\widehat z, \widehat w] + \Lambda_L(z) [\widehat z, \widehat w]\,, 
\qquad \Lambda_L(z) [\widehat z, \widehat w] = \langle \mathcal L(z) [\widehat z], \widehat w \rangle\,, 
$$
where $\mathcal L(z)$ is the operator defined by \eqref{definition L z}. Our candidate for $\Psi_C$ is $\Psi_X^{0,1}$ where $X \equiv X(\tau, z)$
is a non-autonomous vector field, defined for $z \in \mathcal V$ and $0 \le \tau \le 1$,
so that $(\Psi_X^{0, 1})^* (\Psi^*_L \Lambda_G) = \Lambda$.
The  flow $\Psi^{\tau_0, \tau}_X$, corresponding to the vector field $X$, is required to be well defined on a neighborhood $\mathcal V'$ (cf Lemma \ref{estimates for flow}) for $0 \le \tau_0, \tau \le 1$ 
and to satisfy the standard normalization conditions $\Psi^{\tau_0, \tau_0}_X(z) = z$ for any $z \in \mathcal V'$ and $0 \le \tau_0 \le 1.$
To find $X$ with the desired properties, introduce the one parameter family of two forms,
 $$
 \Lambda_\tau (z) = \Lambda + \tau \Lambda_L(z)\,, \quad 0 \le \tau \le 1\,.
 $$
Note that $\Lambda_0 = \Lambda$, $\Lambda_1 = \Psi_L^* \Lambda_G$, and $(\Psi_X^{0,0})^* \Lambda_0 = \Lambda_0.$
The desired identity $(\Psi_X^{0,1})^* \Lambda_1 = \Lambda_0$ then follows if one can show that $(\Psi_X^{0,\tau})^* \Lambda_\tau$
is independent of $\tau$, that is, $\partial_\tau \big( (\Psi_X^{0,\tau})^* \Lambda_\tau \big) = 0$. By Cartan's identity,
$$
\partial_\tau \big( (\Psi_X^{0,\tau})^* \Lambda_\tau \big) =  (\Psi_X^{0,\tau})^* \big( \partial_\tau \Lambda_\tau + d (\Lambda_\tau [ X(\tau, \cdot), \cdot] ) \big)
= (\Psi_X^{0,\tau})^* \big(  \Lambda_L + d (\Lambda_\tau [ X(\tau, \cdot), \cdot] ) \big)\,.
$$
Hence we need to choose the vector field $X( \tau, z)$ in such a way that
\begin{equation}\label{equation for X}
\Lambda_L(z) +  d \big(\Lambda_\tau(z) [X(\tau, z),  \cdot ] \big) = 0 \,, \qquad  \Lambda_\tau(z) [X( \tau, z),  \cdot ] = \langle J^{-1} \mathcal L_\tau(z) [X(\tau, z)] , \, \cdot \rangle
\end{equation}
where for any $0 \le \tau \le 1$ and $z \in \mathcal V,$ the operator $\mathcal L_\tau(z) : h^0_0  \to  h^0_0$ is defined by
\begin{equation}\label{definition L tau}
\mathcal L_\tau(z) :=  {\rm Id} + \tau J \mathcal L(z) 
\end{equation}
and where $J^{-1}$ is the inverse of the diagonal operator $J$, defined by \eqref{definition J}.
In a next step we want to rewrite $\Lambda_L(z)$ as the differential of a properly chosen one form. First note that
since $\Lambda_G = d \lambda_G$ (cf \eqref{definition lambda G}) and $\Lambda = d \lambda$ (cf \eqref{definition lambda}),
the two form $\Lambda_L$ is closed, $\Lambda_L = d (\lambda_1 - \lambda_0)$ where $\lambda_1 := \Psi^*_L \lambda_G$ and $\lambda_0 := \lambda$.
 Furthermore, by Lemma \ref{espansione L S bot q z}, $\mathcal L(z_S, 0) = 0$ and hence $\Lambda_L(z_S, 0) = 0$ for any $z_S \in \mathcal V_S$. 
It then follows by the Poincar\'e Lemma (cf e.g. \cite[Appendix 1]{Kappeler-Montalto}) that  $d \lambda_L = \Lambda_L$ where
$$
\lambda_L(z) [\widehat z] :=  \int_0^1 \langle \mathcal L(z_S, tz_\bot)[0, z_\bot],\, (\widehat z_S, t \widehat z_\bot) \rangle d t  
\stackrel{\eqref{definition L z}}{=} \int_0^1 \langle \mathcal L^\bot_S(z_S, t z_\bot)[z_\bot], \, \widehat z_S \rangle dt\,.
$$
Since $\mathcal L^\bot_S(z_S, t z_\bot) = t \mathcal L^\bot_S(z_S, z_\bot)$ (cf \eqref{L SS botS Sbot (1)}) one is then led to 
\begin{equation}
   \lambda_{L} ( z) [ \widehat z] = \langle{ \mathcal E}(z) \,,\, \widehat z \rangle\,, \qquad 
   \mathcal E(z) := (\mathcal E_S(z), 0)\in h^0_S \times h^0_{\bot}\,,
   \qquad z \in \mathcal V, \,\, \, \widehat z \in  h^0_0\label{forma finale 1 forma Kuksin}
 \end{equation}
 where
 \begin{equation}\label{forma finale 1 forma Kuksin 2}
 \mathcal E_S: \mathcal V \to h^0_S, \, z \mapsto \mathcal E_S(z) := \frac12 \mathcal L_S^\bot(z)[z_\bot]\,.
  \end{equation}
 In view of \eqref{equation for X}, we will choose $X$ so that
 \begin{equation}\label{second equation for X}
 \mathcal E(z) + J^{-1}\mathcal L_\tau(z)[X(\tau, z)] = 0\,, \qquad \forall z \in \mathcal V\,, \,\, 0 \le \tau \le 1\,.
 \end{equation}
 Arguing as in the proof of \cite[Lemma 4.1]{Kappeler-Montalto} one can show that after shrinking the ball $\mathcal V_\bot$, if needed, 
 $\mathcal L_\tau(z)$ is invertible for any $0 \le \tau \le 1$ and $z \in {\cal V}$. In view of Lemma \ref{espansione L S bot q z}, 
 the following version of \cite[Lemma 4.1]{Kappeler-Montalto} holds:
\begin{lemma}\label{stime cal Lt inverso}
After shrinking the ball ${\cal V}_\bot \subset h^0_{\bot}$ in ${\cal V} = {\cal V}_S \times {\cal V}_\bot$, if needed, 
for any $s \ge 0,$ $z \in \mathcal V \cap h^s_0$, and $\tau \in [0, 1],$ the operator
${\cal L}_\tau(z) : h^s_0 \to h^{s}_0$ is invertible with inverse ${\cal L}_\tau(z)^{-1} : h^{s}_0 \to h^{s}_0$ given by the Neumann series,
\begin{equation}\label{definizione Lt (w)}
 {\cal L}_\tau(z)^{- 1} =  \text{Id} + \sum_{n \geq 1}(- 1)^n  ( \tau J   \mathcal L(z) )^n.
 \end{equation}
Furthermore, for any $s \ge 0$, the map
$$
 [0, 1] \times  ({\cal V} \cap h^s_0) \to {\cal B}(h^{s}_0, h^{s + 1}_0), \quad (\tau, z) \mapsto {\cal L}_\tau(z)^{- 1} - \text{Id} = - \tau J \mathcal L(z) {\cal L}_\tau(z)^{- 1}
$$
is real analytic and the following estimates hold: for any $z \in {\cal V} \cap h^s_0$, $0 \le \tau \le 1$,
$\widehat z, \widehat z_1, \ldots, \widehat z_l \in h^s_0$, $l \ge 1$,
$$
 \| ({\cal L}_\tau (z)^{- 1} -  \text{Id}) [\widehat z] \|_{s+1 } \lesssim_s \| z_\bot \|_s \| \widehat z\|_0\,, 
 $$
 $$ 
 \| d^l \big( ({\cal L}_\tau (z)^{- 1} - \text{Id}) [\widehat z] \big)[\widehat z_1, \ldots, \widehat z_l] \|_{s + 1} \lesssim_{s, l} 
 \| \widehat z\|_0 \sum_{j = 1}^l \| \widehat z_j\|_s \prod_{i \neq j} \| \widehat z_i\|_0  +  \| \widehat z\|_0 \| z_\bot\|_s \prod_{j = 1}^l \| \widehat z_j\|_0\,. 
 $$
 \end{lemma}
Note that by \eqref{forma finale 1 forma Kuksin 2} and \eqref{L SS botS Sbot (1)}, $\mathcal E(z)$ and hence $\lambda_L(z)$ are quadratic expressions in $z_\bot$. 
Applying Lemma \ref{espansione L S bot q z} to $\mathcal E(z)$, one obtains the following estimates:
   \begin{lemma}\label{lemma E(z)}
   The map ${\cal V} \to h^0_S \times h^0_\bot$, $z \mapsto \mathcal E(z) = (\mathcal E_S(z), 0)$ is real analytic. Furthermore, for any 
   $z \in {\cal V}$, $ \widehat z_1, \ldots, \widehat z_l \in h^0_0$, $l \geq 1$,  one has 
   $$
   \begin{aligned}
  &  \|  \mathcal E_S(z)\| \lesssim \| z_\bot \|_0^2 \,, \qquad \| d \mathcal E_S(z)[\widehat z_1]\| \lesssim \| z_\bot \|_0 \| \widehat z_1 \|_0 \,, \qquad
  & \| d^l \mathcal E_S(z)[\widehat z_1, \ldots, \widehat z_l]\| \lesssim_{ l} \prod_{j = 1}^l \| \widehat z_j\|_0\,, \quad l \geq 2\,.  
   \end{aligned}
   $$
   \end{lemma}
\noindent 
Since ${\cal L}_\tau(z)$ is invertible (cf  Lemma \ref{stime cal Lt inverso}), equation \eqref{second equation for X} can be solved for $X(\tau, z)$,
  \begin{equation}\label{definizione campo vettoriale ausiliario}
 X(\tau, z) := - {\cal L}_\tau(z)^{- 1} [ J \mathcal E(z)] \,, \quad \forall z \in \mathcal V, \,\, \tau \in [0, 1].
 \end{equation}  
 Note that by Lemma \ref{lemma E(z)}, $J \mathcal E(z)$ is $C^\infty-$smooth.   
Hence it follows from  Lemma \ref{stime cal Lt inverso} that 
 for any integer $s \ge 0,$ $z \in \mathcal V \cap h^s_0$
 $$
 X(\tau, z)  =   - J \mathcal E(z)   -  \big({\cal L}_\tau(z)^{- 1} - \text{Id} \big) [J\mathcal E(z)] = - J \mathcal E(z) + \tau J \mathcal L(z) [X(\tau, z)]\,.
 $$
Lemma \ref{stime cal Lt inverso} and Lemma \ref{lemma E(z)}  then lead to the following results (cf Lemma \cite[4.3]{Kappeler-Montalto}).
 \begin{lemma}\label{estimates for X} For any $s \ge 0,$ the non-autonomous vector field
 $$
 X: [0, 1] \times (\mathcal V \cap h^s_0) \to  h^{s+1}_0
 $$
 is real analytic and the following estimates hold: for any $z \in \mathcal V \cap h^s_0,$ $0 \le \tau \le 1$, $\widehat z \in h^s_0$,
 $$
 \| X(\tau, z) \|_{s+1} \lesssim_s \|z_\bot\|_s \|z\|_0\,, \quad \|dX(\tau, z))[\widehat z_1]\|_{s+1} \lesssim_s \|z_\bot\|_s \|\widehat z\|_0 + \|z_\bot\|_0 \|\widehat z\|_s\,,
 $$
 and for any $\widehat z_1, \,  \ldots , \widehat z_l \in h^s_0$, $l \ge 2$,
 $$
 \| d^l X(\tau, z) [\widehat z_1, \, \ldots , \widehat z_l] \|_{s+1} \lesssim_{s, l}  \sum_{j=1}^l \|\widehat z_j\|_s \prod_{i \ne j}^l \|\widehat z_i\|_0  +  \|z_\bot\|_s \prod_{j = 1}^l \|\widehat z_j\|_0\,.
 $$
 \end{lemma}
 By a standard contraction argument, there exists an open neighborhood $\mathcal V_S' \subset \mathcal V_S$ of $\mathcal K \subset h^0_S$ and a ball $\mathcal V_\bot' \subset \mathcal V_\bot,$ centered at $0$,
 so that for any $\tau, \tau_0 \in [0,1]$, the flow map $\Psi_X^{\tau_0, \tau}$  of  the non-autonomous differential equation $ \partial_\tau z = X(\tau, z)$ 
 is well defined on $\mathcal V' := \mathcal V_S' \times \mathcal V_\bot'$ and
 \begin{equation}\label{definition Psi_X {tau_0, tau}}
 \Psi_X^{\tau_0, \tau} : \mathcal V' \to \mathcal V
 \end{equation}
 is real analytic. Arguing as in the proof of \cite[Lemma 4.4]{Kappeler-Montalto} one shows that $\Psi_X^{\tau_0, \tau} - id$ 
 is one smoothing. More precisely, the following holds.
 \begin{lemma}\label{estimates for flow}
 Shrinking the ball ${\cal V}_\bot' \subset h^0_{\bot}$ in ${\cal V}' = {\cal V}_S' \times {\cal V}_\bot'$, if needed, it follows that 
for any $s \ge 0$, $\tau_0, \tau \in [0, 1],$ the map
$\Psi_X^{\tau_0, \tau} - id: \mathcal V' \cap h^s_0 \to h^{s+1}_0$ is real analytic and for any $z \in \mathcal V' \cap h^s_0$, $0 \le \tau_0, \tau \le 1$,
$\widehat z \in h^s_0$,
$$
\| \Psi_X^{\tau_0, \tau}(z)  - z\|_{s+1} \lesssim_s \|z_\bot\|_s \|z_\bot\|_0\,, \qquad 
\| (d \Psi_X^{\tau_0, \tau}(z)  - Id)[\widehat z]\|_{s+1}  \lesssim_s \|z_\bot\|_s \|\widehat z\|_0 + \|z_\bot\|_0 \|\widehat z\|_s
$$
and for any $\widehat z_1, \ldots , \widehat z_l \in h^s_0$, $l \ge 2$,
$$
\| d^l \Psi_X^{\tau_0, \tau}(z)[\widehat z_1, \ldots, \widehat z_l] \|_{s+1}  \lesssim_s  \sum_{j = 1}^l \|\widehat z_j\|_s \prod_{i \ne j}\|\widehat z_i\|_0 + \|z_\bot\|_s \prod_{j=1}^l \|\widehat z_j\|_0
$$
 \end{lemma} 

Our aim is  to derive expansions for the flow maps  $\Psi_X^{\tau_0, \tau}(z)$. 
To this end we derive such expansions for ${\cal L}_\tau(z)^{- 1}$ and in turn for the vector field $X(\tau, z)$.
Recall that ${\cal L}_\tau(z)^{- 1}$ is given by the Neumann series \eqref{definizione Lt (w)} and hence we
first derive an expansion for the operators $(J \mathcal L(z))^n$. 
It is convenient to introduce the projections 
\begin{equation}\label{Pi S Pi bot}
\Pi_S : h^0_S \times h^0_{ \bot } \to h^0_S \times h^0_{ \bot }\,,  (\widehat z_S, \widehat z_\bot) \mapsto (\widehat z_S, 0)\,, \quad 
\Pi_\bot : h^0_S \times h^0_{ \bot } \to h^0_S \times h^0_{ \bot }\,, (\widehat z_S, \widehat z_\bot ) \mapsto (0, \widehat z_\bot)
\end{equation}
and the maps 
\begin{equation}\label{pi S}
  \pi_S : h^0_S \times h^0_{ \bot } \to h^0_S\,, \quad z = (z_S, z_\bot) \mapsto z_S\,, \qquad
  \pi_\bot : h^0_S \times h^0_{ \bot } \to h^0_{ \bot }\,, \quad z= (z_S , z_\bot) \mapsto z_\bot\,.
  \end{equation}   
\noindent
Furthermore, let $J_S: = \pi_S J \pi_S$, $J_\bot: = \pi_\bot J \pi_\bot$. Then
$J_S^{-1} = \pi_S J^{-1} \pi_S$, $J_\bot^{-1}= \pi_\bot J^{-1} \pi_\bot$, or, more explicitly,
$$
\big\langle J_S^{- 1} \widehat z_S, \widehat w_S \big\rangle = \sum_{n \in S} \frac{1}{\ii 2\pi n} \widehat z_{n} \widehat w_{- n}  \,, \,\,\, \forall \widehat z_S, \widehat w_S \in h^0_S\,,
\qquad \big\langle J_\bot^{- 1} \widehat z_\bot, \widehat w_\bot \big\rangle = \sum_{n \in S^\bot} \frac{1}{\ii 2\pi n} \widehat z_{n} \widehat w_{- n}\,,
\,\,\,  \forall \widehat z_\bot, \widehat w_\bot \in h^0_\bot \,.
$$
\begin{lemma}\label{lemma potenze JM L(q,z)}
 For any $n \ge 1$, $z = (z_S, z_\bot) \in \mathcal V$, 
$(J \mathcal L(z))^n$ has an expansion of arbitrary order $N \ge 1$,
$$
   \begin{pmatrix}
0 & 0 \\
{\cal F}_\bot \circ \sum_{k = 1}^N  {\cal A}_k^S(z; (J\mathcal L)^n) \,   \partial_x^{- k} {\cal F}_\bot^{- 1}[z_\bot] \,\,\,
& {\cal F}_\bot \circ \sum_{k = 1}^N  {\cal A}_k^\bot(z; (J\mathcal L)^n ) \, \partial_x^{- k} {\cal F}_\bot^{- 1}[z_\bot]
\end{pmatrix}  + {\cal R}_{N}(z; (J\mathcal L)^n)
$$
where for any integers $s \geq 0$, $1 \le k \le N,$ the maps 
$$ \mathcal V \to {\cal B}(h^0_S, H^s), \, z \mapsto \mathcal A_k^S(z; (J\mathcal L)^n) \,, \qquad 
\mathcal V \to {\cal B}(h^0_\bot, H^s), \, z \mapsto \mathcal A_k^\bot(z; (J\mathcal L)^n) \,, 
$$
$$
\mathcal V \cap h^s_0 \to {\cal B}\big(h^0_0, h^{s + N +1}_0 \big), \, z \mapsto {\cal R}_{N}(z;  (J\mathcal L)^n)
$$ 
are real analytic. For any $z = (z_S, z_\bot) \in \mathcal V$, $\widehat z = (\widehat z_S, \widehat z_\bot) \in h^0_0$,
${\cal A}_k^S(z; (J\mathcal L)^n )[\widehat z_S]$ and  ${\cal A}_k^\bot(z; (J\mathcal L)^n )[\widehat z_\bot]$ satisfy the estimates
\begin{equation}\label{stime Neumann}
\begin{aligned}
& \| {\cal A}_k^S(z; (J\mathcal L)^n ) [\widehat z_S] \|_{s} \lesssim_{s, k}\, \| \widehat z_S\| \,  (C(k) \| z_\bot \|_0)^{n - 1} \, , \\  
& \| {\cal A}_k^\bot(z; (J\mathcal L)^n )[\widehat z_\bot] \|_{s} \lesssim_{s, k} \, \| \widehat z_\bot \|_0 (C(k) \| z_\bot \|_0)^{n - 1} 
\end{aligned}
\end{equation}
whereas for any integer $s \ge 0$, $z= (z_S, z_\bot) \in \mathcal V \cap h^s_0$, $\widehat z \in h^0_0$
\begin{equation}
 \| {\cal R}_N(z; (J\mathcal L)^n)[\widehat z]\|_{ s + N + 1} \lesssim_{s, N} \, \| z_\bot \|_s  \| \widehat z\|_0 (C_0(N)  \| z_\bot \|_0 )^{n - 1} \,,
\end{equation}
for some constants $C(k), C_0(N) \ge 1$. 
Furthermore, the following estimates hold for the derivatives of these maps: for $k, l \ge 1$, there exists a constant $C(k, l) \ge 1$, 
so that for any $z_S \in \mathcal V_S$, $\widehat z_S \in h^0_S$, $\widehat z_\bot \in h^0_\bot$, $\widehat z_1, \ldots , \widehat z_l \in h^0_0$,
$$
\| d^l \big( {\cal A}_k^S(z; (J \mathcal L)^{n}) [\widehat z_S] \big) [ \widehat z_1, \ldots , \widehat z_l] \|_s
\lesssim_{s, l, k} \,  \| \widehat z_S\|  \,  (\prod_{j =1}^l \|\widehat z_j\|_0) \, (C(k, l)  \| z_\bot \|_0 )^{0 \lor {(n -1 - l)}}  \,,
$$
$$
\| d^l \big( {\cal A}_k^\bot(z; (J \mathcal L)^{n}) [\widehat z_\bot] \big) [ \widehat z_1, \ldots , \widehat z_l] \|_s
\lesssim_ {s, l, k} \,  \| \widehat z_\bot\|_0  \,  (\prod_{j =1}^l \|\widehat z_j\|_0) \, (C(k, l)  \| z_\bot \|_0 )^{0 \lor {(n - 1 - l)}}  \,.
$$
Finally, there exist constants $C_0(N, l) \ge 1$, $l \ge 1$,  so that for any $z \in \mathcal V \cap h^s_0$, $\widehat z \in h^0_0$, $\widehat z_1, \ldots , \widehat z_l \in h^s_0$, 
\begin{align}
\| d^l ({\cal R}_N(z; (J \mathcal L)^n) & [ \widehat z])  [\widehat z_1, \ldots , \widehat z_l] \|_{ s + N + 1}  \nonumber\\
&  \lesssim_{s, l, N} \,  \| \widehat z\|_0  \, 
 \big( \, \sum_{j = 1}^l \| \widehat z_j \|_s \prod_{i \ne j} \|\widehat z_i \|_0  \, + \,  \| z_\bot \|_s   \prod_{j =1}^l \|\widehat z_j\|_0 \,  \big) \, (C_0(N, l)  \| z_\bot \|_0 )^{0 \lor {(n - l)}}  \,.
 \nonumber
\end{align}
We refer to Remark \ref{notation coefficients expansion} where we comment on the notation introduced for the coefficients in such expansions. 
Furthermore, we recall that  $\mathcal V_\bot$ denotes the open ball in $h^0_\bot,$ centered at $0,$ whose radius is smaller than one and downscaled according to our needs.
\end{lemma}
\begin{proof}
We begin by proving the claimed statements for $n = 1$. By \eqref{definition L z}, the operator $\mathcal L(z)$, $z \in \mathcal V$, can be written as  
$$
\mathcal L(z) = 
\begin{pmatrix}
0 & 0 \\
\mathcal L_\bot^S(z) & 0 
\end{pmatrix} + \begin{pmatrix}
\mathcal L_S^S(z) & \mathcal L_S^\bot(z) \\
0 & 0
\end{pmatrix}\,. 
$$
 Using that $J_\bot \circ {\cal F}_\bot = {\cal F}_\bot \circ \partial_x$, it then follows from Lemma \ref{espansione L S bot q z} that 
\begin{equation} \label{formula JM L q z}
J \mathcal L(z)   = \begin{pmatrix}
0 & 0 \\
{\cal F}_\bot \circ \sum_{k = 1}^N  {\cal A}_k^S(z_S; J \mathcal L) \, \partial_x^{- k} {\cal F}_\bot^{- 1}[z_\bot]  & 0
\end{pmatrix} + {\cal R}_{N}(z; J\mathcal L ) 
\end{equation}
where ${\cal A}_k^S(z_S; J \mathcal L)$ and ${\cal R}_{N}(z; J\mathcal L )$ are obtained from 
$\mathcal A_k(z_S; \mathcal L_\bot^S)$ and ${\cal R}_{N}(z; \mathcal L_{\bot}^S)$, given by Lemma \ref{espansione L S bot q z},
in a straightforward way. It then follows that for any $s \geq 0$ and $1 \le k \le N$, the maps 
$$
{\cal V}_S \to {\cal B}( h^0_S, H^s), z_S \mapsto {\cal A}^S_k(z_S; J \mathcal L)\, ,  \qquad
\mathcal V \cap h^s_0 \to {\cal B}\big(h^0_0, h^{s + N +1}_0 \big), z \mapsto {\cal R}_N(z; J \mathcal L)
$$ 
are real analytic and that ${\cal R}_N(z; J \mathcal L)$ is of the form
$$
{\cal R}_{N}(z; J \mathcal L) = \begin{pmatrix}
{\cal R}_{N}(z; J \mathcal L)_S^S & {\cal R}_{N}(z; J \mathcal L)_S^\bot  \\
{\cal R}_{N}(z; J \mathcal L)_\bot^S  & 0
\end{pmatrix}\,.
$$
For any integer $s \ge 0$, $z_S \in \mathcal V_S$, $\widehat z_S \in h^0_S$, 
 $\widehat z_1, \ldots , \widehat z_l \in h^0_0$, $l \ge 1$, one has
 \begin{equation}\label{estimate of mathcal A k ^ S ( z_ S J L)}
  \| {\cal A}_k^S(z_S; J \mathcal L) [\widehat z_S] \|_{s} \lesssim_{s, k} \| \widehat z_S\| \,, \quad
   \| d^l \big({\cal A}_k^S(z_S; J \mathcal L) [\widehat z_S] \big)[\widehat z_1, \dots ,  \widehat z_l]\|_{s} \lesssim_{s, k, l} \| \widehat z_S\| \prod_{j =1}^l \|\widehat z_j\|_0\,.
 \end{equation}
Furthermore,  for any integer $s \ge 0$, $z = (z_S, z_\bot) \in \mathcal V \cap h^s_0,$ $\widehat z \in h^0_0$, 
 $\widehat z_1, \ldots , \widehat z_l \in h^s_0$, $l \ge 1$, the remainder term satisfies
$ \| {\cal R}_{N}(z; J \mathcal L) [\widehat z ] \|_{ s + N + 1} \lesssim_{s, N} \| z_\bot \|_s \| \widehat z\|_0$ and 
\begin{equation}\label{stima cal QN cal A n 1}
 \| d^l \big({\cal R}_N(z; J \mathcal L)[\widehat z]\big) [ \widehat z_1, \ldots , \widehat z_l ] \|_{ s + N + 1} \lesssim_{s, N, l}
\| \widehat z \|_0 \,  \big( \sum_{j = 1}^l \| \widehat z_j\|_s \prod_{i \ne j} \| \widehat z_i \|_0 \, +  \, \|z_\bot\|_s \prod_{j =1}^l \|\widehat z_j\|_0  \big) \,.\\
\end{equation}
To prove the claimed statements for $n+1$, $n \ge 1$, write $( J \mathcal L(z))^{n + 1} = J \mathcal L(z) (J \mathcal L(z) )^n$. 
By the expansion \eqref{formula JM L q z} it follows that
$ ( J \mathcal L(z) )^{n + 1}$ is of the form
\begin{equation}
 \begin{pmatrix}
0 & 0 \\
{\cal F}_\bot \circ \sum_{k = 1}^N   {\cal A}_{k}^S(z ) \,  \partial_x^{- k} {\cal F}_\bot^{- 1}[z_\bot]  
& {\cal F}_\bot \circ  \sum_{k = 1}^N  {\cal A}_k^\bot(z )  \, \partial_x^{- k} {\cal F}_\bot^{- 1}[z_\bot]  
\end{pmatrix} + {\cal R}_N(z)
\end{equation}
where ${\cal A}_{k}^S(z ) \equiv {\cal A}_{k}^S(z; (J \mathcal L)^{n + 1} )$, ${\cal A}_k^\bot(z ) \equiv {\cal A}_k^\bot(z; (J \mathcal L)^{n + 1} )$,
and ${\cal R}_N(z) \equiv {\cal R}_N(z; (J \mathcal L)^{n + 1})$ are given by
\begin{equation}
\begin{aligned}
 {\cal A}_k^S(z; (J \mathcal L)^{n + 1}) := {\cal A}_k^S(z_S; J \mathcal L)  & \circ  (J \mathcal L (z) )^n)_S^S \,, \qquad  
{\cal A}_k^\bot(z;(J \mathcal L)^{n + 1} ) := {\cal A}_k^S (z_S; J \mathcal L) \circ \big( (J \mathcal L (z) )^n \big)_S^\bot \,, \\
& {\cal R}_N(z; (J \mathcal L)^{n+1}) :=  {\cal R}_{N}(z; J \mathcal L) \circ (J \mathcal L(z))^{n} \,. \nonumber
\end{aligned}
\end{equation}
It then follows that these maps are real analytic as claimed in the statement of the lemma. Furthermore, for any $s \ge 0$, $z = (z_S, z_\bot) \in \mathcal V,$ $\widehat z_S \in h^0_S$
one has by \eqref{estimate of mathcal A k ^ S ( z_ S J L)}
$$
\| {\cal A}_k^S(z; (J \mathcal L)^{n + 1}) [\widehat z_S]\|_s \lesssim_{s, k}  \|( (J \mathcal L (z) )^n )_S^S [\widehat z_S]\| \lesssim_{s, k} ( C_0 \|z_\bot \|_0 )^n \| \widehat z_S\|
$$
where 
$ C_0 :=  C(0; \mathcal L)$ is given by \eqref{estimate mathcal L}.
Similarly, again by \eqref{estimate of mathcal A k ^ S ( z_ S J L)}, for any $z = (z_S, z_\bot) \in \mathcal V$, $\widehat z_\bot \in h^0_\bot$ one has
$$
\| {\cal A}_k^\bot(z;(J \mathcal L)^{n + 1} )[\widehat z_\bot] \|_s \lesssim_{s, k}  \|( (J \mathcal L (z) )^n )_S^\bot [\widehat z_\bot]\| 
\lesssim_{s, k} ( C_0 \|z_\bot \|_0 )^n \| \widehat z_\bot\|_0\,,
$$
whereas for any $s \ge 0,$ $z = (z_S, z_\bot) \in \mathcal V \cap h^s_0,$ $\widehat z = (\widehat z_S, \widehat z_\bot) \in h^0_0$
$$
\|  {\cal R}_N(z; (J \mathcal L)^{n+1}) [\widehat z] \|_{s+N+1} \lesssim_{s, N} \| z_\bot \|_s \| (J \mathcal L (z) )^n [\widehat z]\|_0 
\lesssim_{s, N} \| z_\bot \|_s  ( C_0 \|z_\bot \|_0 )^n \| \widehat z\|_0\,.
$$
Next we estimate the derivatives of ${\cal A}_k^S(z; (J \mathcal L)^{n + 1}) [\widehat z_S]$. 
For any $s \ge 0$, $z_S \in \mathcal V_S$, $\widehat z_S \in h^0_S$, $\widehat z_1, \ldots , \widehat z_l \in h^0_0$, $l \ge 1$,
the estimate of $\| d^l \big( {\cal A}_k^S(z; (J \mathcal L)^{n + 1}) [\widehat z_S] \big) [ \widehat z_1, \ldots , \widehat z_l] \|_s$ is obtained from the estimates of
$$
\| d^m \big( {\cal A}_k^S(z; J \mathcal L) [\widehat w_{l-m}] \big) [ \widehat z_1, \ldots , \widehat z_m] \|_s\,, \qquad 
\widehat w_{l-m}:= d^{l-m} \big( ( (J\mathcal L (z))^n )^S_S [\widehat z_S] \big) [ \widehat z_{m +1}, \ldots , \widehat z_l] \in h^0_S\,, 
$$ 
and $\| \widehat w_{l-m} \|$ where $0 \le m \le l$. By \eqref{estimate of mathcal A k ^ S ( z_ S J L)}, one has
$$
\| d^m \big( {\cal A}_k^S(z; J \mathcal L) [\widehat w_{l-m}] \big) [ \widehat z_1, \ldots , \widehat z_m] \|_s 
\lesssim_{s, m} \| \widehat w_{l-m}\| \prod_{j =1}^m \|\widehat z_j\|_0\,.
$$
Note that we introduced the element $\widehat w_{l-m}$ to indicate that in  $\| d^m \big( {\cal A}_k^S(z; J \mathcal L) [\widehat w_{l-m}] \big) [ \widehat z_1, \ldots , \widehat z_m] \|_s$
the derivative $d^m$ does not act on $\widehat w_{l-m}$.
Increasing the constant $C_0$ and/or decreasing the radius of the ball $\mathcal V_\bot$ depending on the size of $l$ 
it follows from \eqref{estimate mathcal L} - \eqref{estimate derivative mathcal L} that
$$
\| \widehat w_{l - m} \| 
\lesssim_{s, l - m}  \| \widehat z_S\| \, (\prod_{j = m+ 1}^l \| \widehat z_j\|_0) \, ( C_0 \|z_\bot \|_0 )^{0 \lor {(n-l +m)}}\,.
$$
Combining these estimates implies that
$$
\| d^l \big( {\cal A}_k^S(z; (J \mathcal L)^{n + 1}) [\widehat z_S] \big) [ \widehat z_1, \ldots , \widehat z_l] \|_s
\lesssim_{s, l, k} \,  \| \widehat z_S\|  \,  (\prod_{j =1}^l \|\widehat z_j\|_0) \, (C_0  \| z_\bot \|_0 )^{0 \lor {(n - l)}}  \,.
$$
In the same way one shows that for any $s \ge 0$, $z_S \in \mathcal V_S$, $\widehat z_\bot \in h^0_\bot$, $\widehat z_1, \ldots , \widehat z_l \in h^0_0$, $l \ge 1$,
$$
\| d^l \big( {\cal A}_k^\bot(z; (J \mathcal L)^{n + 1}) [\widehat z_\bot] \big) [ \widehat z_1, \ldots , \widehat z_l] \|_s
\lesssim_{s, l, k} \,  \| \widehat z_\bot\|_0  \,  (\prod_{j =1}^l \|\widehat z_j\|_0) \, (C_0  \| z_\bot \|_0 )^{0 \lor {(n - l)}}  \,.
$$
Finally, the claimed estimate for $\| d^l ({\cal R}_N(z; (J \mathcal L)^n)[ \widehat z]) [\widehat z_1, \ldots , \widehat z_l] \|_{ s + N + 1} $ 
where $s \ge 0$, $z \in \mathcal V \cap h^s_0$, $\widehat z \in h^0_0$, $\widehat z_1, \ldots , \widehat z_l \in h^s_0$, $l \ge 1$,
follows from the estimates of 
$$
\| d^m ({\cal R}_N(z; J \mathcal L)[ \widehat v_{l-m}]) [\widehat z_1, \ldots , \widehat z_m] \|_{ s + N + 1} \,, \quad 
\widehat v_{l-m} := d^{l -m} \big( (J \mathcal L(z))^n[ \widehat z] \big) [\widehat z_{m +1}, \ldots , \widehat z_l]
$$ 
and $\| \widehat v_{l-m} \|_0$ where $0 \le m \le l.$ Indeed, by \eqref{stima cal QN cal A n 1}, 
$$ 
\| d^m \big({\cal R}_N(z; J \mathcal L)[\widehat v_{l - m}]\big) [ \widehat z_1, \ldots , \widehat z_m ] \|_{ s + N + 1} \lesssim_{s, N, m}
\| \widehat v_{l-m} \|_0 \,  \big( \sum_{j = 1}^m \| \widehat z_j\|_s \prod_{i \ne j, 1 \le i \le m} \| \widehat z_i \|_0 \, +  \, \|z_\bot\|_s \prod_{j =1}^m \|\widehat z_j\|_0  \big) \,.
$$
Increasing the constant $C_0$ and/or decreasing the radius of the ball $\mathcal V_\bot$ depending on the size of $l$, 
it follows from \eqref{estimate mathcal L} - \eqref{estimate derivative mathcal L} that
$$
\| \widehat v_{l - m} \|_0 \le 
\| d^{l -m} \big( (J \mathcal L(z))^n [\widehat z] \big) [\widehat z_{m +1}, \ldots , \widehat z_l] \|_0
\lesssim_{s, l - m}  \| \widehat z\|_0 \, ( \prod_{j = m+ 1}^l \| \widehat z_j\|_0 ) \, ( C_0 \|z_\bot \|_0 )^{0 \lor {(n-l +m)}}\,.
$$
Combining these estimates yields the claimed estimate for $\| d^l ({\cal R}_N(z; (J \mathcal L)^n)[ \widehat z]) [\widehat z_1, \ldots , \widehat z_l] \|_{ s + N + 1},$
with constants chosen appropriately.
\end{proof}
Recall that $\mathcal V_\bot$ denotes the open ball in $h^0_\bot,$ centered at $0,$ whose radius is smaller than one and downscaled at several instances in the course of our analysis.
\begin{lemma}\label{lemma campo vettoriale}
For any $N \ge 1,$  $X(\tau, z) = - \mathcal L_\tau(z)^{-1} [J \mathcal E(z)]$ ($0 \le \tau \le 1$, $z \in \mathcal V$)  has an expansion of the form
\begin{equation}\label{espansione campo vettoriale correttore}
 X(\tau, z) =    \Big( \, 0,  \,\, {\cal F}_\bot \circ \sum_{k = 1}^N    a_k (\tau, z; X ) \,  \partial_x^{- k} {\cal F}_\bot^{- 1}[z_\bot]  \, \Big)  + {\cal R}_N(\tau, z; X)
\end{equation}
where for any $s \ge 0$ and $k \ge 1,$ the maps
$$
 [0, 1] \times \mathcal V \to H^s\,, \,\, (\tau, z) \mapsto a_k(\tau, z; X)\,, \quad
  [0, 1] \times (\mathcal V \cap h^s_0 ) \to  h^{s + N +1}_0\,, \,\, (\tau, z) \mapsto {\cal R}_N (\tau, z; X)
$$
are real analytic. Furthermore, for any $0 \le \tau \le 1$, $z \in \mathcal V$, $ \widehat z \in h^0_0$, 
$$
 \| a_k(\tau, z; X)\|_s \lesssim_{s, k} \| z_\bot \|_0^2\,,  \quad \| d a_k(\tau, z; X)[\widehat z]  \|_s \lesssim_{s, k} \| z_\bot \|_0 \| \widehat z\|_0
 $$
 and for any $ \widehat z_1, \ldots, \widehat z_l \in h^0_0$, $l \ge 2,$
$$
 \| d^l a_k(\tau, z; X)[\widehat z_1, \ldots, \widehat z_l] \|_s \lesssim_{s, k, l} \prod_{j = 1}^l \| \widehat z_j\|_0\,.
$$
For any $z \in  \mathcal V \cap h^s_0$, $0 \le \tau \le 1$,  $\widehat z \in  h^s_0$, the remainder term ${\cal R}_N(\tau, z; X)$ satisfies
$$
 \| {\cal R}_N(\tau, z; X) \|_{s + N +1} \lesssim_{s, N}  \| z_\bot \|^2_0 +  \| z_\bot \|^2_0 \| z_\bot \|_s \lesssim_{s, N}    \| z_\bot \|_0 \| z_\bot \|_s\,, 
 $$
 $$
 \| d {\cal R}_N(\tau, z; X)[\widehat z] \|_{s + N +1} \lesssim_{s, N} \| z_\bot \|^2_0 \| \widehat z\|_s + \| z_\bot \|_0 \| \widehat z\|_0(1 + \| z_\bot \|_s)  
 \lesssim_{s, N} \| z_\bot \|_0 ( \| \widehat z\|_s + \| z_\bot \|_s \| \widehat z\|_0) \,, 
 $$
 whereas for any $\widehat z_1, \ldots, \widehat z_l \in h^s_0$, $l \geq 2$,
 $$
 \| d^l {\cal R}_N(\tau, z; X)[\widehat z_1, \ldots, \widehat z_l] \|_{s + N + 1} \lesssim_{s, N, l} \sum_{j = 1}^l \| \widehat z_j\|_s \prod_{i \neq j} \| \widehat z_i\|_0 + \| z_\bot \|_s \prod_{j = 1}^l \| \widehat z_j\|_0 \,. 
$$
\end{lemma}
\begin{proof}
By the Neumann series expansion, one has for any $z \in \mathcal V$, $0 \le \tau \le 1$,
\begin{align}
- {\cal L}_\tau(z)^{- 1} [J \mathcal E(z)] & = - J \mathcal E(z) + \sum_{n \geq 1}(- 1)^{n + 1} \tau^n (  J  \mathcal L(z) )^n [J \mathcal E(z)] \,. \label{against the current 0}
\end{align}
Since $J \mathcal E(z) = (J_S \mathcal E_S(z), 0)$, Lemma \ref{lemma potenze JM L(q,z)} yields that for any $n \geq 1$,
$$
 (- 1)^{n + 1} \tau^n (J  \mathcal L(z))^n [J \mathcal E(z)]   = \big( 0, \, 
{\cal F}_\bot  \circ \sum_{k = 1}^N    a_{k}(\tau, z; (J  \mathcal L)^n [J \mathcal E]) \, \partial_x^{- k} {\cal F}_\bot^{- 1}[z_\bot] \, \big)  
+  {\cal R}_N(\tau, z; (J  \mathcal L)^n [J \mathcal E]) 
$$
where 
$$
a_k(\tau, z; (J  \mathcal L)^n [J \mathcal E]) := (- 1)^{n + 1} \tau^n {\cal A}_k^S(z; (J  \mathcal L)^n )[J_S \mathcal E_S(z)] \,, 
$$
$$
{\cal R}_N(\tau, z; (J  \mathcal L)^n [J \mathcal E] )  := (- 1)^{n + 1} \tau^n {\cal R}_N(z; (J  \mathcal L)^n) [J \mathcal E(z) ] \,.
$$
By applying the estimates of Lemma \ref{lemma E(z)} and Lemma \ref{lemma potenze JM L(q,z)}, one gets for any $s \ge 0$, $z \in \mathcal V$, $0 \le \tau \le 1$,
\begin{equation}\label{stime a n k X N k nel lemma}
\| a_k(\tau, z; (J  \mathcal L)^n [J \mathcal E]) \|_s \lesssim_{s, k} \| z_\bot \|_0^2 \big( C(k) \| z_\bot \|_0 \big)^{n - 1}
\end{equation}
and for any $s \ge 0$, $z \in \mathcal V \cap h^s_0$, $0 \le \tau \le 1$,
\begin{equation}\label{estimate R N J L ^ n}
\| {\cal R}_N(\tau, z; (J  \mathcal L)^n [J \mathcal E]) \|_{s + N +1} \lesssim_{s, N} \| z_\bot \|_0^2 \| z_\bot \|_s  \big( C_0 (N)  \| z_\bot \|_0 )^{n-1}\,. 
\end{equation}
In view of \eqref{against the current 0} we define for any $1 \le k \le N,$
\begin{equation}\label{z n X N def}
 a_k (\tau, z; X) := \sum_{n \geq 1} a_k(\tau, z; (J  \mathcal L)^n [J \mathcal E]) \,, \quad
 {\cal R}_N (\tau,  z; X) := - J \mathcal E(z) +  \sum_{n \geq 1} {\cal R}_N(\tau, z; (J  \mathcal L)^n [J \mathcal E])\,.
\end{equation}
Shrinking the radius of the ball $\mathcal V_\bot$, if needed, one can assume that  $C (k)  \| z_\bot \|_0 < 1$, $C_0 (N)  \| z_\bot \|_0 < 1$ for any 
$1 \le k \le N$, $z_\bot \in \mathcal V_\bot$.
The expansion  \eqref{espansione campo vettoriale correttore}, the analyticity statement, 
and the claimed estimates for $ \| a_k(\tau, z; X)\|_s$ and $ \| {\cal R}_N(\tau, z; X) \|_{s + N +1}$ 
then follow from the estimates \eqref{stime a n k X N k nel lemma}, \eqref{estimate R N J L ^ n}, and Lemma \ref{lemma E(z)}, Lemma \ref{lemma potenze JM L(q,z)}. 
The estimates for the derivatives of $a_k(\tau, z; X)$ and ${\cal R}_N(\tau, z; X)$ follow by similar arguments. 
Indeed, it follows from Lemma \ref{lemma potenze JM L(q,z)}
that for any for $1 \le k \le N$, $z \in \mathcal V$, $0 \le \tau \le 1$, $\widehat z \in h^0_0$,
$$
\| d a_k(\tau, z; (J  \mathcal L)^n [J \mathcal E]))[\widehat z]  \|_s \lesssim_{s, k} 
\| d \big( {\cal A}_k^S(z; (J \mathcal L)^{n}) [\widehat w] \big) [ \widehat z] \|_s |_{\widehat w = J_S \mathcal E_S(z)}
+ \| {\cal A}_k^S(z; (J \mathcal L)^{n}) [ \, d (J_S \mathcal E_S(z)) [ \widehat z] \, ]  \|_s
$$
Using that by Lemma \ref{lemma E(z)}, $\|J_S \mathcal E_S(z)\| \lesssim \|z_\bot\|_0^2$ and $\| d (J_S \mathcal E_S(z)) [\widehat z]\|_0 \lesssim \|z_\bot\|_0 \| \widehat z \|_0$
one then concludes from Lemma \ref{lemma potenze JM L(q,z)}
$$
\| d \big( {\cal A}_k^S(z; (J \mathcal L)^{n}) [\widehat w] \big) [ \widehat z] \|_s |_{\widehat w = J_S \mathcal E_S(z)} 
 \lesssim_{s, k} \|z_\bot\|_0^2  \|\widehat z\|_0 (C(k, 1)  \| z_\bot \|_0 )^{0 \lor {(n -2)}}
$$
and
$$
\| {\cal A}_k^S(z; (J \mathcal L)^{n}) [ \, d (J_S \mathcal E_S(z)) [ \widehat z] \, ]  \|_s \lesssim_{s, k}
\|z_\bot\|_0 \| \widehat z \|_0  (C(k, 1)  \| z_\bot \|_0 )^{n -1}\,.
$$
In view of the definition \eqref{z n X N def} of $a_k(\tau, z; X)$ one then obtains the claimed estimate 
$\| d a_k(\tau, z; X)[\widehat z]  \|_s \lesssim_{s, k} \| z_\bot \|_0 \| \widehat z\|_0$.
The estimates for  $\| d^l a_k(\tau, z; X)[\widehat z_1, \ldots, \widehat z_l]\|_s$
with $l \ge 2$ are derived in a similar fashion.
Finally let us consider the estimates of the derivatives of the remainder term.
For any $n \ge 1$, $z \in \mathcal V \cap h_0^s$, $0 \le \tau \le 1$, $\widehat z \in h_0^s,$
it follows from Lemma \ref{lemma potenze JM L(q,z)} and the product rule that
\begin{align}
\| d ({\cal R}_N(z; (J  \mathcal L)^n) [J \mathcal E(z) ])  [\widehat z] \|_{s + N +1} \lesssim_{s, N} 
&  \|J \mathcal E(z)\|_0 \big(\| \widehat z\|_s + \| z_\bot \|_s \|\widehat z \|_0  \big) \big( C_0 (N)  \| z_\bot \|_0 )^{n-1} \nonumber \\
& +  \| d J \mathcal E(z) [\widehat z]\|_0 |\|z \|_s \big( C_0 (N)  \| z_\bot \|_0 )^{n-1} \nonumber
\end{align}
Using again that by Lemma \ref{lemma E(z)}, $\|J \mathcal E(z)\|_0 \lesssim \|z_\bot\|_0^2$ and $\| d J \mathcal E(z) [\widehat z]\|_0 \lesssim \|z_\bot\|_0 \| \widehat z \|_0$
and taking into account the definition \eqref{z n X N def} of $ {\cal R}_N (\tau,  z; X)$ one sees that 
$$
\| d {\cal R}_N(\tau, z; X) [\widehat z] \|_{s+ N + 1} \lesssim_{s, N} 
\|z_\bot\|_0 \| \widehat z \|_0 + \|z_\bot\|_0^2 \big(\| \widehat z\|_s + \| z_\bot \|_s \|\widehat z \|_0  \big) +   \|z_\bot\|_0 \| z_\bot \|_s \|\widehat z \|_0\,,
$$
yielding the claimed estimate for $\| d {\cal R}_N(\tau, z; X) [\widehat z] \|_{s+ N + 1}$. The ones for $\| d^l {\cal R}_N(\tau, z; X)[\widehat z_1, \ldots, \widehat z_l] \|_{s + N + 1}$
with $l \ge 2$ are derived in a similar fashion.
\end{proof}

After these preliminary considerations we can now state the main result of this section, saying  that for any $\tau_0, \tau \in [0, 1]$, 
the flow map $\Psi_X^{\tau_0, \tau}$, defined on $\mathcal V'$ and with values in $\mathcal V$, admits an expansion,
referred to as parametrix for the solution of the initial value problem of $\partial_\tau z = X(\tau, z)$. 
\begin{theorem}\label{espansione flusso per correttore}
(i) For any $\tau_0, \tau \in [0, 1]$, $N \in \N$, and $z = (z_S, z_\bot) \in \mathcal V'$,
\begin{equation}\label{espansione asintotica Psi tau 0 tau}
\Psi_X^{\tau_0, \tau}(z) = (z_S, \, z_\bot) +  \big( 0, \, \, {\cal F}_\bot \circ\sum_{k = 1}^N   a_k(z; \Psi_X^{\tau_0, \tau}) \,\, \partial_x^{- k}{\cal F}_\bot^{- 1}[z_\bot] \big)  + {\cal R}_N( z; \Psi_X^{\tau_0, \tau})
\end{equation}
where for any $\tau_0, \tau \in [0, 1]$, $1 \le k \le N$, and $s \geq 0$, the maps 
$$
\mathcal V' \to H^s, z \mapsto a_k(z; \Psi_X^{\tau_0, \tau}), \qquad
\mathcal V' \cap h^s_0 \to  h^{s + N +1}_0, z \mapsto {\cal R}_N(z; \Psi_X^{\tau_0, \tau})
$$ 
are real analytic. Furthermore, for any $z \in \mathcal V'$,  $ \widehat z \in h^0_0$,
$$
 \| a_k(z; \Psi_X^{\tau_0, \tau} ) \|_s \lesssim_{s, k} \| z_\bot \|_0^2\,, \qquad  \| d a_k(z; \Psi_X^{\tau_0, \tau} )[\widehat z]\|_s \lesssim_{s, k} \| z_\bot \|_0 \| \widehat z\|_0 
 $$
 and for any $ \widehat z_1, \ldots, \widehat z_l \in h^0_0$, $l \geq 2$,
$$
 \| d^l a_k(z; \Psi_X^{\tau_0, \tau} ) [\widehat z_1, \ldots, \widehat z_l]\|_s \lesssim_{s, k, l} \prod_{j = 1}^l \| \widehat z_j\|_0\,. 
$$
The remainder term satisfies the following estimates: for any $z \in \mathcal V'  \cap h^s_0$,  $\widehat z \in h^s_0$
$$
 \|{\cal R}_N(z; \Psi_X^{\tau_0, \tau}) \|_{s + N +1} \lesssim_{s, N} \| z_\bot \|_s \| z_\bot \|_0\,, \qquad
 \| d {\cal R}_N(z; \Psi_X^{\tau_0, \tau}) [\widehat z]\|_{s + N +1} \lesssim_{s, N} \| z_\bot \|_s \| \widehat z\|_0 + \| z_\bot \|_0 \| \widehat z\|_s \,, 
 $$
and for any  $\widehat z_1, \ldots, \widehat z_l \in h^s_0$, $l \geq 2$,
 $$
 \| d^l {\cal R}_N(z; \Psi_X^{\tau_0, \tau})[\widehat z_1, \ldots, \widehat z_l] \|_{s + N +1} \lesssim_{s, N, l} \, 
 \sum_{j = 1}^l \| \widehat z_j \|_s \prod_{i \neq j} \| \widehat z_i\|_0  \, + \, \| z_\bot \|_s \prod_{j = 1}^l \| \widehat z_j\|_0\,.
$$
(ii) In particular, the statements of item (i)  hold for $\Psi_C := \Psi_X^{0, 1}: \mathcal V' \to \Psi_X^{0, 1}(\mathcal V')$, referred to as symplectic corrector,
and $ \Psi_X^{1, 0} : \mathcal V' \to  \Psi_X^{1, 0}(\mathcal V')$,
which by a slight abuse of terminology with respect to its domain of definition we refer to as the inverse of  $\Psi_C$ and denote by $\Psi_C^{- 1}$.
The expansion of the map $\Psi_C$ is then written as ($z \in \mathcal V'$)
$$
\Psi_C(z) = z +  \big( 0 , \,  {\cal F}_\bot  \circ \sum_{k = 1}^N  a_k(z;  \Psi_C) \partial_x^{- k} {\cal F}_\bot^{- 1}[z_\bot ] \, \big) + {\cal R}_N(z; \Psi_C)
$$
where
$$
a_k(z; \Psi_C) := a_k(z; \Psi_X^{0, 1})\,, \quad {\cal R}_N( z; \Psi_C) := {\cal R}_N(z; \Psi_X^{0, 1})\,.
$$
Similarly, the expansion for the inverse $\Psi_C^{- 1}(z)$, $z \in \mathcal V'$, is written as
$$
\Psi_C(z)^{- 1} = z +  \big( 0 , \,  {\cal F}_\bot  \circ \sum_{k = 1}^N  a_k(z;  \Psi_C^{- 1}) \partial_x^{- k} {\cal F}_\bot^{- 1}[z_\bot ] \, \big) + {\cal R}_N(z; \Psi_C^{- 1})
$$ 
where
$$
a_k(z;  \Psi^{- 1}_C) := a_k(z;  \Psi_X^{1, 0})\,, \quad {\cal R}_N(z; \Psi_C^{- 1}) := {\cal R}_N(z; \Psi_X^{1, 0}).
$$ 
\end{theorem}
\begin{proof}
Clearly, item (ii) is a direct consequence of (i). Since the proof of item (i) is quite lengthy, we divide it up into several steps.
First note that the flow map $\Psi^{\tau_0, \tau} \equiv \Psi_X^{\tau_0, \tau}$ is a bounded nonlinear operator acting on $\mathcal V' \cap h^s_0$, $s \ge 0$, satisfying the integral equation
\begin{equation}\label{Duhamel ODE X}
\Psi^{\tau_0, \tau}(z) = z + \int_{\tau_0}^\tau X(t, \Psi^{\tau_0, t}(z))\, d t\,.  
\end{equation}
Using the latter equation, the coefficients $a_k(z; \Psi^{\tau_0, \tau})$, $k \ge 1,$ and the remainder term ${\cal R}_N(z; \Psi_X^{\tau_0, \tau})$ of the parametrix  \eqref{espansione asintotica Psi tau 0 tau}
are determined inductively. 
By \eqref{espansione campo vettoriale correttore}, one obtains for any $0 \le  \tau_0, t \le 1,$ $z \in \mathcal V'$, 
\begin{equation}\label{hoch schule - 1}
X(t, \Psi^{\tau_0, t}(z))  = 
\Big( \, 0, \,\, {\cal F}_\bot \circ \sum_{k = 1}^N  a_k(t, \Psi^{\tau_0, t}(z); X)  \,\, \partial_x^{- k} {\cal F}_\bot^{- 1}[\pi_\bot \Psi^{\tau_0, t}(z)]   \,\Big)  + {\cal R}_N(t , \Psi^{\tau_0, t}(z); X)\,.
\end{equation}
{\it Expansion of $\partial_x^{- k} {\cal F}_\bot^{- 1}[\pi_\bot \Psi^{\tau_0, t}(z)]$, $1 \le k \le N$:} To find {\em candidates} for the coeffcients $a_k(z; \Psi_X^{\tau_0, \tau})$
we argue formally and substitute the expansion \eqref{espansione asintotica Psi tau 0 tau}
into the expression $\partial_x^{- k} {\cal F}_\bot^{- 1}[\pi_\bot \Psi^{\tau_0, t}(z)]$ yielding
\begin{align}
\partial_x^{- k} {\cal F}_\bot^{- 1}[\pi_\bot \Psi^{\tau_0, t}(z)] & = \partial_x^{- k} \Big( {\cal F}_\bot^{- 1}[z_\bot] + 
\sum_{j = 1}^N  a_j(z; \Psi^{\tau_0, t}) \,\, \partial_x^{- j}{\cal F}_\bot^{- 1}[z_\bot] 
 + {\cal F}_\bot^{- 1} \pi_\bot {\cal R}_N(z; \Psi^{\tau_0, t})  \Big) 
\nonumber\\
& = \partial_x^{- k} \Big( {\cal F}_\bot^{- 1}[z_\bot] + \sum_{j = 1}^{N - k} a_j(z; \Psi^{\tau_0, t}) \,\, \partial_x^{- j} {\cal F}_\bot^{- 1}[z_\bot] \, \Big) + {\cal R}_{N, k}^{(1)}(t, z; \tau_0) \label{hoch schule 0}
\end{align}
where 
\begin{equation}\label{cal R n N (1)}
{\cal R}_{N, k}^{(1)}(t, z; \tau_0) := \partial_x^{- k} \Big(  \sum_{j = N - k + 1}^{N }  a_j(z; \Psi^{\tau_0, t}) \,\, \partial_x^{- j} {\cal F}_\bot^{- 1}[z_\bot]  + {\cal F}_\bot^{- 1} \pi_\bot {\cal R}_N(z; \Psi^{\tau_0, t}) \Big) \,.
\end{equation}
Using Lemma \ref{lemma composizione pseudo}(i) and the notation established there, one has 
$$
\partial_x^{- k} \big( a_j(z; \Psi^{\tau_0, t}) \,  \partial_x^{- j}{\cal F}_\bot^{- 1}[z_\bot] \big)= 
\sum_{i = 0}^{N - k - j} C_i( k, j) (\partial_x^i a_j(z; \Psi^{\tau_0, t}) ) \, \partial_x^{- k - j - i} {\cal F}_\bot^{- 1}[z_\bot] + {\cal R}_{N, k, j}^{(2)}(t, z; \tau_0)
$$ 
where 
\begin{equation}\label{definition R (2)}
{\cal R}_{N,  k, j}^{(2)}(t, z; \tau_0) :=  {\cal R}_{N,  k, j}^{\psi do}(a_j(z; \Psi^{\tau_0, t})) \, {\cal F}_\bot^{- 1}[z_\bot]\,.
\end{equation}
By Lemma \ref{lemma composizione pseudo}, for any $z \in \mathcal V' \cap h^s_0$, $s \ge 0$, $0 \le t, \tau_0 \le 1$, $1 \le j \le N$
\begin{equation}\label{stima cal R n k N (2)}
\| {\cal R}_{N, k, j}^{(2)}(t, z; \tau_0)\|_{s + N + 1} \lesssim_{s, N} {\rm max}_{1 \le i  \le N}\| a_i (z ;  \Psi^{\tau_0, t})\|_{s + 2 N}  \| z_\bot \|_s\,. 
\end{equation}
Hence, \eqref{hoch schule 0} reads
$$
\partial_x^{- k} {\cal F}_\bot^{- 1}[\pi_\bot \Psi^{\tau_0, t}(z)] 
  = \partial_x^{- k} {\cal F}_\bot^{- 1}[z_\bot] + \sum_{j = 1}^{N - k} \sum_{i  =0}^{N - k - j} C_i(k, j)(\partial_x^i  a_j(z; \Psi^{\tau_0, t})) \,\,  \partial_x^{- k - j - i} {\cal F}_\bot^{- 1}[z_\bot]   +
  {\cal R}_{N, k}^{(3)}( t, z; \tau_0)
$$
where 
\begin{equation}\label{definizione cal R n N (3)}
{\cal R}_{N, k}^{(3)}(t, z; \tau_0) := {\cal R}_{N, k}^{(1)}(t, z;  \tau_0) +  \sum_{j = 1}^{N - k} {\cal R}_{N, k, j}^{(2)}(t, z;  \tau_0) \,. 
\end{equation}
Changing in the double sum $\sum_{j = 1}^{N - k} \sum_{i  =0}^{N - k - j}$  the index $i$ of summation to $n:= i+j$ and then interchanging the order of summation, one obtains
$$
\sum_{j = 1}^{N - k} \sum_{i  =0}^{N - k - j} C_i(k, j)(\partial_x^i  a_j) \,\,  \partial_x^{- k - j - i} 
=\sum_{n = 1}^{N - k} \sum_{j=1}^{n} C_{n-j} (k, j)(\partial_x^{n-j}  a_j) \,\,  \partial_x^{- k - n} 
$$
implying that  $\partial_x^{- k} {\cal F}_\bot^{- 1}[\pi_\bot \Psi^{\tau_0, t}(z)] $ equals
\begin{equation}\label{hoch schule 2}
 \partial_x^{- k} {\cal F}_\bot^{- 1}[z_\bot] + \sum_{n = 1}^{N - k} \Big( \sum_{ j=1}^{n} C_{n -j}(k , j) (\partial_x^{n-j} a_j(z; \Psi^{\tau_0, t})) \Big) \partial_x^{- k - n} {\cal F}_\bot^{- 1}[z_\bot] + 
{\cal R}_{N, k}^{(3)}(t, z;  \tau_0) \,.
\end{equation}

\medskip

\noindent
{\it Expansion of $ \sum_{k = 1}^N  a_k(t, \Psi^{\tau_0, t}(z); X)  \,\, \partial_x^{- k} {\cal F}_\bot^{- 1}[\pi_\bot \Psi^{\tau_0, t}(z)]$:} 
To simplify notation, introduce
\begin{equation}\label{definition a _ k (t,z; tau)}
a_k(t, z; \tau_0) := a_k(t, \Psi^{\tau_0, t}(z); X)
\end{equation}
and then substitute \eqref{hoch schule 2} into $ \sum_{k = 1}^N  a_k(t, \Psi^{\tau_0, t}(z); X)  \,\, \partial_x^{- k} {\cal F}_\bot^{- 1}[\pi_\bot \Psi^{\tau_0, t}(z)]$ to get
\begin{align}
   \sum_{k = 1}^N & a_k(t, z; \tau_0) \,\,  \partial_x^{- k} {\cal F}_\bot^{- 1}[\pi_\bot \Psi^{\tau_0, t}(z)]    \, 
  = \,  \sum_{k = 1}^N   a_k(t, z; \tau_0) \,\,    \partial_x^{- k} {\cal F}_\bot^{- 1}[z_\bot] \,\,   \label{sumatra 0}\\
& + \,  \sum_{k = 1}^N \sum_{n = 1}^{N - k} \sum_{ j = 1}^n C_{n-j}(k, j) a_k(t, z; \tau_0) (\partial_x^{n - j} a_j(z; \Psi^{\tau_0, t}))  \, \partial_x^{-k - n } {\cal F}_\bot^{- 1}[z_\bot] + 
\sum_{k = 1}^N   a_k(t, z; \tau_0) {\cal R}_{N, k}^{(3)}(t, z;  \tau_0)  \nonumber
\end{align} 
Changing the index of summation $n$ to $l:= k+n$ and then interchanging the sum with respect to $k$ and $l$ and in turn with respect to $k$ and $j$, 
the triple sum in \eqref{sumatra 0} becomes
\begin{align}
& \sum_{k = 1}^N \sum_{l = k+ 1}^{N} \sum_{ j = 1}^{l - k} C_{l - k -j}(k, j) a_k(t, z; \tau_0) (\partial_x^{l - k - j} a_j(z; \Psi^{\tau_0, t}))  \, 
\partial_x^{-l } {\cal F}_\bot^{- 1}[z_\bot] \nonumber \\
& =  \sum_{l = 2}^N \Big( \sum_{k= 1}^{l - 1} \sum_{ j = 1}^{l - k} C_{l - k -j}(k, j) a_k(t, z; \tau_0) (\partial_x^{l - k - j} a_j(z; \Psi^{\tau_0, t})) \Big) \, 
\partial_x^{-l } {\cal F}_\bot^{- 1}[z_\bot] \nonumber \\
& =  \sum_{l = 2}^N \Big( \sum_{j= 1}^{l - 1} \sum_{ k = 1}^{l - j} C_{l - k -j}(k, j) a_k(t, z; \tau_0) (\partial_x^{l - k - j} a_j(z; \Psi^{\tau_0, t})) \Big) \, 
\partial_x^{-l } {\cal F}_\bot^{- 1}[z_\bot]\,.
\end{align}
{\it Expansion of $X(t, \Psi^{\tau_0, t}(z))$:}
Writing $k$ for $l$ and $n$ for $k$,  the expansion \eqref{hoch schule - 1} of $X(t, \Psi^{\tau_0, t}(z))$ takes the form
\begin{equation}\label{expansion for X(t, Psi ^ {tau_0, t}(z))}
\Big( \, 0, \,\, \mathcal F_\bot \big(  a_1(t, z; \tau_0) \, \partial_x^{- 1} {\cal F}_\bot^{- 1}[z_\bot] \big) \,
+ \, \mathcal F_\bot \circ \sum_{k = 2}^N \big( a_k(t, z; \tau_0) +  b_k(t, z; \tau_0) \big) \, \partial_x^{-k } {\cal F}_\bot^{- 1}[z_\bot] \,  \Big) 
+\, {\cal R}_N^{(4)}(t, z; \tau_0)
\end{equation}
where
\begin{equation}\label{definition b k}
b_k(t, z; \tau_0)  = 
 \sum_{j= 1}^{k - 1} \sum_{ n = 1}^{k - j} C_{k - n -j}(n, j) a_n(t, z; \tau_0) \, \partial_x^{k - n - j} a_j(z; \Psi^{\tau_0, t})
\end{equation}
and
\begin{equation}\label{cal R N (3)}
{\cal R}_N^{(4)}( t, z; \tau_0) := \Big( 0, \,\, {\cal F}_\bot \circ   \sum_{k = 1}^N  a_k(t, z; \tau_0) {\cal R}_{N, k}^{(3)}(t, z; \tau_0) \Big) 
+  \,{\cal R}_N(t , \Psi^{\tau_0, t}(z); X)\,.
\end{equation}

\noindent
{\it Definition and estimates of $a_k(z; \Psi^{\tau_0, t})$, $1 \le k \le N$:}
The fact that for any given $1 \le k \le N$, the coefficient $b_k(t, z; \tau_0)$ only depends on the unknown coefficients $a_j(z; \Psi^{\tau_0, t})$ with $1 \le j \le k-1,$ but not on $a_k(z; \Psi^{\tau_0, t})$
allows to determine $a_k(z; \Psi^{\tau_0, t})$ recursively by using equation \eqref{Duhamel ODE X}. Indeed,
substituting \eqref{expansion for X(t, Psi ^ {tau_0, t}(z))} for $X(t, \Psi^{\tau_0, t}(z))$ into equation \eqref{Duhamel ODE X} and combining it with \eqref{espansione asintotica Psi tau 0 tau},
leads, {\em up to remainder terms}, to the equations
$$
\mathcal F_\bot   \big( a_1(z; \Psi^{\tau_0, \tau}) \,\, \partial_x^{- 1}{\cal F}_\bot^{- 1}[z_\bot] \big)  
=  \mathcal F_\bot \, \big( \int_{\tau_0}^\tau   a_1(t, z; \tau_0)d\, t \, \, \partial_x^{- 1} {\cal F}_\bot^{- 1}[z_\bot] \big)
$$
and for any $2 \le k \le N$, 
$$
\mathcal F_\bot   \big( a_k(z; \Psi^{\tau_0, \tau}) \,\, \partial_x^{- k}{\cal F}_\bot^{- 1}[z_\bot] )
= \,  \mathcal F_\bot  \big( \int_{\tau_0}^\tau (a_k(t, z; \tau_0) +  b_k(t, z; \tau_0)) d\,t \,\, \partial_x^{-k } {\cal F}_\bot^{- 1}[z_\bot] \big)\,.
$$
We then define for any $z \in \mathcal V',$ $0 \le \tau_0, \tau \le 1$, 
$$
a_1(z; \Psi^{\tau_0, \tau})  :=  \int_{\tau_0}^\tau a_1(t, z; \tau_0)\, d t
$$ 
and for any $k \ge 2,$ $a_k(z; \Psi^{\tau_0, \tau}) := \int_{\tau_0}^\tau ( a_k(t, z; \tau_0) + b_k(t, z; \tau_0) )\, d t $, or more explicitly,
\begin{equation}\label{formule b rho}
a_k(z; \Psi^{\tau_0, \tau}) 
= \int_{\tau_0}^\tau \big( a_k(t, z; \tau_0) +  \sum_{j= 1}^{k - 1} \sum_{ n = 1}^{k - j} C_{k - n -j}(n, j) a_n(t, z; \tau_0) \, \partial_x^{k - n - j} a_j(z; \Psi^{\tau_0, t}) \big)\, d t \,.
\end{equation}
To prove the claimed estimates for $a_k(z; \Psi^{\tau_0, \tau})$, we first estimate $a_k(t, z; \tau_0)$. Recall that by \eqref{definition a _ k (t,z; tau)},
$a_k(t, z; \tau_0) = a_k(t, \Psi^{\tau_0, t}(z); X)$. 
By Lemma \ref{lemma campo vettoriale} and Lemma \ref{estimates for flow} one has
for any $0 \le \tau_0, t  \le 1$, $z \in \mathcal V'$, and $s \ge 0,$
\begin{equation}\label{estimate a k (t, z, tau 0)}
\| a_k(t, z; \tau_0) \|_s \lesssim_{s, k} \| \pi_\bot \Psi^{\tau_0, t}(z) \|_0^2 \lesssim_{s, k}  \| z_\bot \|^2_0 \,.
\end{equation}
It then follows from the definition of $ a_1(z; \Psi^{\tau_0, \tau})$ that for any $s \ge 0$,
$$
\| a_1(z; \Psi^{\tau_0, \tau})  \|_s  \lesssim_s  \, \, \| z_\bot \|_0^2 \,
\quad \forall  \, 0 \le \tau_0, \tau \le 1, \, \forall \, z \in \mathcal V'\,.
$$
To prove corresponding estimates for $a_k(z; \Psi^{\tau_0, \tau})$ with $2 \le k \le N,$ we argue by induction.
Assume that 
for any $1 \le j \le k-1$ and $ s \geq 0$,
\begin{equation}\label{stima a Psi k completa nella dim}
\| a_j(z; \Psi^{\tau_0, \tau}) \|_s \lesssim_{s, j} \| z_\bot \|_0^2, \quad \forall  \, 0 \le \tau_0, \tau \le 1, \,  \forall  \, z \in \mathcal V'\,.
\end{equation}
By the estimate \eqref{estimate a k (t, z, tau 0)}, the definition \eqref{formule b rho} of $a_k(z; \Psi^{\tau_0, \tau})$,
 and the interpolation Lemma \ref{lemma interpolation}, 
one then concludes that the estimate \eqref{stima a Psi k completa nella dim} is also satisfied for $j = k$. 
Using the analyticity properties established for $a_k(\tau, z; \tau_0)$ and $\Psi^{\tau_0, \tau}(z)$, one verifies the ones stated for the coefficients $a_k(z; \Psi^{\tau_0, \tau})$.

\smallskip

\noindent
{\em Estimates of the derivatives of $a_k(z; \Psi^{\tau_0, \tau})$:} 
By Lemma \ref{lemma campo vettoriale}, Lemma \ref{estimates for flow} and the chain rule one has
for any $0 \le \tau_0, t  \le 1$, $z \in \mathcal V'$, $\widehat z \in h_0^0$, $s \ge 0,$
\begin{equation}\label{estimate d a k (t, z, tau 0)}
\| d a_k(t, z; \tau_0) [\widehat z] \|_s  \lesssim_{s, k}  
\| \pi_\bot \Psi^{\tau_0, \tau}(z) \|_0 \|d \Psi^{\tau_0, \tau}(z) [\widehat z] \|_0 \lesssim_s  \| z_\bot \|_0 \| \widehat z\|_0
\end{equation}
and if in addition, $\widehat z_1, \ldots , \widehat z_l \in h^0_0$, $l \ge 2,$
\begin{equation}\label{estimate d ^ l a k (t, z, tau 0)}
\| d^l a_k(t, z; \tau_0)[\widehat z_1, \ldots , \widehat z_l]  \|_s \lesssim_s  \prod_{j=1}^l \| \widehat z_j\|_0\,.
\end{equation}
By the definition of $a_1(z; \Psi^{\tau_0, \tau})$,  \eqref{estimate d a k (t, z, tau 0)} and \eqref{estimate d ^ l a k (t, z, tau 0)}
yield the claimed estimate for $\| d^l a_1(z; \Psi^{\tau_0, \tau})[\widehat z_1, \ldots , \widehat z_l]  \|_s$ for any $l \ge 1$.
To prove corresponding estimates for the derivatives of $a_k(z; \Psi^{\tau_0, \tau})$ with $2 \le k \le N,$ we again argue by induction.
Assume that for any $1 \le j \le k-1$ and $ s \geq 0$,
\begin{equation}\label{estimates for d a_k z; Psi tau_0, tau}
\| d a_j(z; \Psi^{\tau_0, \tau}) [\widehat z] \|_s \lesssim_{s, j} \| z_\bot \|_0 \| \widehat z\|_0, \quad \forall  \, 0 \le \tau_0, \tau \le 1, \,  \forall  \, z \in \mathcal V'\,, \, \widehat z \in h^0_0\,.
\end{equation}
By the definition \eqref{formule b rho} of $a_k(z; \Psi^{\tau_0, \tau})$, the estimate \eqref{estimate d a k (t, z, tau 0)} and the product rule it then follows that
\eqref{estimates for d a_k z; Psi tau_0, tau} also holds for $j = k.$
The estimates for  $\| d^l a_k(z; \Psi^{\tau_0, \tau}) [\widehat z_1, \ldots, \widehat z_l]\|_s$
with $l \ge 2$ are derived in a similar fashion.

\smallskip

\noindent
{\it Definition and estimate of ${\cal R}_N( z; \Psi^{\tau_0, \tau})$:}   The remainder term ${\cal R}_N( z; \Psi^{\tau_0, \tau})$
 is defined so that  the identity \eqref{espansione asintotica Psi tau 0 tau} holds,
$$
{\cal R}_N( z; \Psi_X^{\tau_0, \tau}) := 
\Psi_X^{\tau_0, \tau}(z) - z  - \big( 0, \, \, {\cal F}_\bot \circ\sum_{k = 1}^N   a_k(z; \Psi_X^{\tau_0, \tau}) \,\, \partial_x^{- k}{\cal F}_\bot^{- 1}[z_\bot] \big)
$$
where $a_k(z; \Psi_X^{\tau_0, \tau})$ are given by \eqref{formule b rho}. 
By \eqref{Duhamel ODE X} and the expansion \eqref{expansion for X(t, Psi ^ {tau_0, t}(z))} of $X$, 
${\cal R}_N( z; \Psi^{\tau_0, \tau})$ satisfies
$$
{\cal R}_N( z; \Psi^{\tau_0, \tau}) = \int_{\tau_0}^\tau {\cal R}_N^{(4)}(t, z; \tau_0) \,d t
$$
where by \eqref{cal R N (3)}
\begin{equation}\label{remainder R N Psi^ tau0, tau}
\int_{\tau_0}^\tau {\cal R}_N^{(4)}(t, z; \tau_0) \,d t =  \Big( 0, \,\, {\cal F}_\bot \circ  \sum_{k = 1}^N\int_{\tau_0}^\tau  a_k(t, z; \tau_0) {\cal R}_{N, k}^{(3)}(t, z; \tau_0)   \, d t \Big)+  
\int_{\tau_0}^\tau\,{\cal R}_N(t , \Psi^{\tau_0, t}(z); X) \,d t\,.
\end{equation}
We estimate the two components  $\pi_S \int_{\tau_0}^\tau {\cal R}_N^{(4)}(t, z; \tau_0) \,d t$ and $\pi_\bot \int_{\tau_0}^\tau {\cal R}_N^{(4)}(t, z; \tau_0) \,d t$ 
of $\int_{\tau_0}^\tau {\cal R}_N^{(4)}(t, z; \tau_0) \,d t $ separately. 
By \eqref{remainder R N Psi^ tau0, tau}
$$
\pi_S \int_{\tau_0}^\tau {\cal R}_N^{(4)}(t, z; \tau_0) \,d t  = \int_{\tau_0}^\tau \pi_S {\cal R}_{N}(t, \Psi^{\tau_0, t}(z) ; X)\, d t
$$
and 
$$
\pi_\bot \int_{\tau_0}^\tau {\cal R}_N^{(4)}(t, z; \tau_0) \,d t = \int_{\tau_0}^\tau \pi_\bot {\cal R}_{N}(t, \Psi^{\tau_0, t}(z); X )\, d t + 
{\cal F}_\bot \circ  \sum_{k = 1}^N\int_{\tau_0}^\tau  a_k(t, z; \tau_0) {\cal R}_{N, k}^{(3)}(t, z; \tau_0)   \, d t \,. 
$$
By Lemma \ref{lemma campo vettoriale} and Lemma \ref{estimates for flow}, for any $z \in \mathcal V'$, $0 \le \tau_0, t \le 1,$
\begin{equation}\label{stima X Psi N (q)}
\|\pi_S {\cal R}_{N}(t, \Psi^{\tau_0, t}(z) ; X) \| \lesssim_{N} \| {\cal R}_{N}(t, \Psi^{\tau_0, t}(z) ; X) \|_0  \lesssim_{N} \| z_\bot \|_0^2\,,
\end{equation}
implying that 
$$
\| \int_{\tau_0}^\tau \pi_S {\cal R}_{N}(t, \Psi^{\tau_0, t}(z) ; X)\, d t \| \lesssim_{s, N} \| z_\bot \|_0^2\,.
$$
By the definition \eqref{definizione cal R n N (3)} of ${\cal R}_{N, k}^{(3)}(t, z; \tau_0)$, one has
$$
\int_{\tau_0}^\tau  a_k(t, z; \tau_0) {\cal R}_{N, k}^{(3)}(t,  z; \tau_0) \,d t=  \int_{\tau_0}^\tau  a_k(t, z; \tau_0){\cal R}_{N, k}^{(1)}(t, z;  \tau_0) \,d t +  
\sum_{j = 1}^{N - k} \int_{\tau_0}^\tau  a_k(t, z; \tau_0){\cal R}_{N, k, j}^{(2)}(t, z;  \tau_0) \,d t\,.
$$
Furthermore, by  the definition \eqref{cal R n N (1)} of ${\cal R}_{N, k}^{(1)}(t, z;  \tau_0)$, the term $ \int_{\tau_0}^\tau  a_k(t, z; \tau_0){\cal R}_{N, k}^{(1)}(t, z;  \tau_0) \,d t $ equals 
\begin{equation}
 \int_{\tau_0}^\tau  a_k(t, z; \tau_0) \, \partial_x^{- k}  \big(  \sum_{j = N - k + 1}^{N }  a_j(z; \Psi^{\tau_0, t})\, \partial_x^{- j} {\cal F}_\bot^{- 1}[z_\bot] \big)  \,d t \,
 + \,  \int_{\tau_0}^\tau  a_k(t, z; \tau_0) \partial_x^{-k} {\cal F}_\bot^{- 1} \pi_\bot {\cal R}_N(z; \Psi^{\tau_0, t}) \,d t \,.
\end{equation}
Altogether, one concludes that $\pi_\bot {\cal R}_N(z; \Psi^{\tau_0, \tau})$ satisfies the integral equation
\begin{equation}\label{equazione integrale resto}
\pi_\bot {\cal R}_N(z; \Psi^{\tau_0, \tau}) =  B_N( \tau, z; \tau_0) 
+ \mathcal F_\bot \circ \int_{\tau_0}^\tau  \big( \sum_{k = 1}^N a_k(t, z; \tau_0) \partial_x^{-k} \big) {\cal F}_\bot^{- 1} \pi_\bot {\cal R}_N(z; \Psi^{\tau_0, t})\, d t
\end{equation}
where 
\begin{align}
B_N(\tau, z; \tau_0) :=  & \int_{\tau_0}^\tau \pi_\bot {\cal R}_{N}(t, \Psi^{\tau_0, t}(z); X )\, d t 
+ {\cal F}_\bot \circ  \sum_{k = 1}^N \sum_{j = 1}^{N - k} \int_{\tau_0}^\tau  a_k(t, z; \tau_0){\cal R}_{N, k, j}^{(2)}(t, z;  \tau_0) \,d t \nonumber \\
&  + {\cal F}_\bot \circ  \sum_{k = 1}^N \sum_{j = N - k + 1}^{N } \int_{\tau_0}^\tau  a_k(t, z; \tau_0) \partial_x^{- k} \big( \, a_j(z; \Psi^{\tau_0, t})  \,\, \partial_x^{- j} {\cal F}_\bot^{- 1}[z_\bot] \,  \big) \, \,d t \,.
\end{align}
By the estimates \eqref{stima cal R n k N (2)} of ${\cal R}_{N, k, j}^{(2)}(t, z;  \tau_0)$, 
the ones of $a_k(\tau, z; X)$ and ${\cal R}_N(\tau, z; X)$ given by Lemma \ref{lemma campo vettoriale}, 
the estimates of $\Psi^{\tau_0, t}(z)$, given by Lemma \ref{estimates for flow}, and the ones of $a_k(z; \Psi^{\tau_0, \tau})$, given by \eqref{stima a Psi k completa nella dim},
and using the interpolation Lemma \ref{lemma interpolation}, one obtains for any $s \ge 0,$
$$
\| B_N( \tau, z; \tau_0) \|_{s + N +1 } \lesssim_{s, N} \| z_\bot \|_s \| z_\bot \|_0\,, \quad \forall z \in \mathcal V' \cap h^s_0\,, \,\,\, \forall \,\,0 \le \tau_0, \tau \le 1\,. 
$$
Note that $ \sum_{k = 1}^N a_k(t, z; \tau_0) \partial_x^{-k}$ is a pseudodifferential operator of order $-1$ where by \eqref{estimate a k (t, z, tau 0)}
the coefficients $a_k(t, z; \tau_0)$ satisfy $ \| a_k(t, z; \tau_0)\|_s  \lesssim_{s, k} \| z_\bot \|^2_0 $.
Hence for any $z \in \mathcal V' \cap h^s_0$, $0 \le \tau_0, \tau \le 1$,
$$
 \|  \sum_{k = 1}^N a_k(t, z; \tau_0) \partial_x^{-k} {\cal F}_\bot^{- 1} \pi_\bot {\cal R}_N(z; \Psi^{\tau_0, t})  \|_{s+N+1} 
 \lesssim_{s, N} \| z_\bot \|_0^2 \,  \| {\cal F}_\bot^{- 1} \pi_\bot {\cal R}_N(z; \Psi^{\tau_0, t}) \|_{s + N + 1}\,.
 $$
By Gronwall's inequality and since $\mathcal V'_\bot$ is a ball of sufficiently small radius, the integral equation \eqref{equazione integrale resto} yields that for any $s \ge 0,$
\begin{equation}\label{stima X Psi N (z)}
\| \pi_\bot {\cal R}_{N}(z; \Psi^{\tau_0, \tau}) \|_{s + N + 1} \lesssim_{s, N} \| z_\bot \|_s \| z_\bot \|_0\,, \quad \forall z \in \mathcal V' \cap h^s_0 \,, \,\,\, \forall \,\,0 \le \tau_0, t \le 1\,. 
\end{equation}
The estimates \eqref{stima X Psi N (q)}, \eqref{stima X Psi N (z)} imply the claimed estimate of ${\cal R}_{N}(z; \Psi^{\tau_0, \tau})$. 
The stated analyticity property of ${\cal R}_N( z; \Psi_X^{\tau_0, \tau})$ then follows from the already established analyticity properties of  
$\Psi_X^{\tau_0, \tau}(z)$, $a_k(\tau, z; \tau_0)$, and $a_k(z; \Psi^{\tau_0, \tau})$ (cf e.g. \cite[Theorem A.5]{GK}).

\smallskip

\noindent
{\em Estimates of the derivatives of ${\cal R}_{N}(z ; \Psi^{\tau_0, \tau})$:} The estimates of the derivatives of ${\cal R}_{N}(z ; \Psi^{\tau_0, \tau})$  can be obtained in a similar way 
as the ones for ${\cal R}_{N}(z ; \Psi^{\tau_0, \tau})$. Indeed, for any $s \ge 0$, $z \in \mathcal V' \cap h_0^s$, $0 \le \tau_0, \tau \le 1$, $\widehat z \in h_0^s,$
one has
$$
 d{\cal R}_{N}(z ; \Psi^{\tau_0, \tau})[\widehat z]  =  
 \Big( 0, \,\, {\cal F}_\bot \circ  \sum_{k = 1}^N\int_{\tau_0}^\tau d \big( a_k(t, z; \tau_0) {\cal R}_{N, k}^{(3)}(t, z; \tau_0) \big) [\widehat z]   \, d t \, \Big)
 \, +  \, \int_{\tau_0}^\tau\, d \big( {\cal R}_N(t , \Psi^{\tau_0, t}(z); X) \big) [\widehat z] \,d t\,. 
$$
Again, we estimate 
$\pi_S \big( d{\cal R}_{N}(z ; \Psi^{\tau_0, \tau})[\widehat z] \big) = \int_{\tau_0}^\tau\,  \pi_S \big( d \big( {\cal R}_N(t , \Psi^{\tau_0, t}(z); X) \big) [\widehat z] \big) \,d t $
and $\pi_\bot \big( d{\cal R}_{N}(z ; \Psi^{\tau_0, \tau})[\widehat z] \big)$  
separately. By Lemma \ref{lemma campo vettoriale}, Lemma \ref{estimates for flow}, and the chain rule, one has
\begin{equation}\label{estimate of S projection of derivative}
\| \int_{\tau_0}^\tau\,  \pi_S \big( d \big( {\cal R}_N(t , \Psi^{\tau_0, t}(z); X) \big) [\widehat z] \big) \,d t \| \lesssim_{N} \| z_\bot \|_0 \|\widehat z\|_0
\end{equation}
whereas by \eqref{equazione integrale resto}, 
$\pi_\bot \big( d{\cal R}_{N}(z ; \Psi^{\tau_0, \tau})[\widehat z] \big)$
satisfies
\begin{equation}\label{equation for the derivative of the remainder}
\pi_\bot \big( d {\cal R}_N(z; \Psi^{\tau_0, \tau}) \big[\widehat z] \big)=  B^{(1)}_N( \tau, z; \tau_0) [ \widehat z]
+ \mathcal F_\bot \circ \int_{\tau_0}^\tau  
\big( \sum_{k = 1}^N a_k(t, z; \tau_0) \partial_x^{-k} \big) {\cal F}_\bot^{- 1} \pi_\bot \big( d {\cal R}_N(z; \Psi^{\tau_0, t}) [\widehat z] \big) \, d t
\end{equation}
with $B^{(1)}_N( \tau, z; \tau_0)[ \widehat z]$ given by
$$
B^{(1)}_N( \tau, z; \tau_0) [ \widehat z] = d \big( B_N( \tau, z; \tau_0) \big) [\widehat z] + 
\mathcal F_\bot \circ \int_{\tau_0}^\tau  
\big( \sum_{k = 1}^N  (d a_k(t, z; \tau_0)[\widehat z]) \partial_x^{-k} \big) {\cal F}_\bot^{- 1} \pi_\bot {\cal R}_N(z; \Psi^{\tau_0, t}) \, d t\,.
$$
Since
$$
\| B_N^{(1)}( \tau, z; \tau_0) [\widehat z] \|_{s + N +1 } \lesssim_{s, N} 
\| z_\bot \|_s  \| \widehat z \|_0 +  \| z_\bot \|_0 \| \widehat z \|_s \,, \quad \forall z \in \mathcal V' \cap h^s_0\,, \,\, \widehat z \in h^s_0\,,  \,\,0 \le \tau_0, \tau \le 1
$$
we conclude from \eqref{equation for the derivative of the remainder} by Gronwall's inequality that
\begin{equation}\label{estimate of bot projection of derivative}
\| \pi_\bot \big( d{\cal R}_{N}(z ; \Psi^{\tau_0, \tau})[\widehat z] \big) \|_{s+N+1} \lesssim_{s, N} 
\| z_\bot \|_s  \| \widehat z \|_0 +  \| z_\bot \|_0 \| \widehat z \|_s \,, \quad \forall z \in \mathcal V' \cap h^s_0\,,  \,\, \widehat z \in h^s_0\,,  \,\, 0 \le \tau_0, \tau \le 1\,. 
\end{equation}
The estimates  \eqref{estimate of S projection of derivative} and \eqref{estimate of bot projection of derivative}
imply the claimed estimate for $d {\cal R}_N(z ; \Psi^{\tau_0, \tau}) [\widehat z]$.
In a similar fashion, one derives the estimates for $\| d^l {\cal R}_N(z ; \Psi^{\tau_0, \tau}) [\widehat z_1, \ldots, \widehat z_l] \|_{s + N + 1}$
with $\widehat z_1, \ldots, \widehat z_l \in h^s_0$, $l \ge 2$.
\end{proof}

\bigskip

It turns out that the flow maps $\Psi_X^{\tau_0, \tau}$ and hence the symplectic corrector $\Psi_C$ and its inverse $\Psi_C^{-1}$
preserve the reversible structures, introduced in Section \ref{introduzione paper}, acts on. To state the result in more detail, note that without loss of generality, we may assume that the neighborhood 
$\mathcal V' = \mathcal V'_S \times \mathcal V'_\bot$ (cf  Lemma \ref{estimates for flow}) is invariant  under the map $\mathcal S_{rev}$.

\noindent
{\bf Addendum to Theorem \ref{espansione flusso per correttore}}
{\em (i) For any $0 \le \tau_0, \tau \le 1$,
$\Psi_X^{\tau_0, \tau} \circ \mathcal S_{rev} = \mathcal S_{rev} \circ \Psi_X^{\tau_0, \tau}$ on $\mathcal V'$
and for any $z \in \mathcal V'$, $x \in \mathbb R$, $N \in \N,$ and  $1 \le k \le N$, 
$$
   a_k( \mathcal S_{rev} z; \Psi_X^{\tau_0, \tau})(x) = (-1)^k a_k(z; \Psi_X^{\tau_0, \tau})( - x) \,, \quad  
    {\cal R}_N( \mathcal S_{rev} z; \Psi_X^{\tau_0, \tau}) = \mathcal S_{rev} ( {\cal R}_N( z; \Psi_X^{\tau_0, \tau}))\,.
$$
(ii) As a consequence, $\Psi_C$ and $\Psi_C^{-1}$ are invariant under $\mathcal S_{rev}$ on $\mathcal V'$,
\begin{equation}\label{reversability and flow}
\Psi_C \circ \mathcal S_{rev} = \mathcal S_{rev} \circ \Psi_C\,, \quad  \Psi_C^{-1} \circ \mathcal S_{rev} = \mathcal S_{rev} \circ \Psi_C^{-1}
\end{equation}
and for any $z \in \mathcal V'$, $x \in \mathbb R$, $N \in \N,$ and  $1 \le k \le N$, 
$$
   a_k( \mathcal S_{rev} z; \Psi_C)(x) = (-1)^k a_k(z; \Psi_C)( - x) \,, \quad  
    {\cal R}_N( \mathcal S_{rev} z; \Psi_C) = \mathcal S_{rev} ( {\cal R}_N( z; \Psi_C))\,.
$$
}
\noindent
{\bf Proof of Addendum to Theorem \ref{espansione flusso per correttore}} Clearly, item (ii) is a direct consequence of item (i).
By the  Addendum to Lemma \ref{espansione L S bot q z}, the operator $\mathcal L(z)$, introduced in \eqref{Pull back of LambdaG}, 
satisfies $\mathcal L( \mathcal S_{rev} z) \circ \mathcal S_{rev} = - \mathcal S_{rev} \circ \mathcal L(z)$ on $\mathcal V$.
It implies that for any $z \in \mathcal V$, 
$\mathcal E(\mathcal S_{rev} z) = - \mathcal S_{rev} \mathcal E(z)$ where $\mathcal E(z)$ has been introduced in \eqref{forma finale 1 forma Kuksin 2}.
Altogether we then conclude that the vector field $X(\tau, z),$ introduced in \eqref{definizione campo vettoriale ausiliario}, satisfies 
$$
X(\tau, \mathcal S_{rev}z) = \mathcal S_{rev}X(\tau, z)\,, \qquad \forall \, z \in \mathcal V, \,\, 0 \le \tau \le 1
$$
and hence by the uniqueness of the initial value problem of $\partial_\tau z = X(\tau, z)$,  the solution map satisfies 
$$
\Psi_X^{\tau_0, \tau}( \mathcal S_{rev}z) = \mathcal S_{rev} \Psi_X^{\tau_0, \tau}(z) \,, \quad \forall \, z \in \mathcal V' \,, \, \, 0 \le \tau_0,  \tau \le 1\,.
$$
The claimed identities for  $a_k(z; \Psi_X^{\tau_0, \tau})$ and ${\cal R}_N( z; \Psi_X^{\tau_0, \tau})$
then follow from the expansion \eqref{espansione asintotica Psi tau 0 tau}.
\hfill   $\square$

\bigskip
We now discuss two applications of Theorem \ref{espansione flusso per correttore}. The first one concerns the expansion 
of the transpose $d\Psi^{0, \tau}_X(z)^t$ of the differential $ d\Psi^{0, \tau}_X(z)$ which will be used in Section \ref{Hamiltoniana trasformata} 
in the proof of Lemma \ref{proprieta hamiltoniana cal P 3 (2b)}. Recall that for any $z \in \mathcal V',$ and $\widehat z$, $\widehat w \in h^0_0$
$$
\Lambda_\tau(z)[\widehat z, \widehat w] = \langle J^{-1}\mathcal L_\tau(z)[\widehat z], \widehat w  \rangle, \qquad
\mathcal L_\tau(z) = \text{Id} + \tau J \mathcal L(z), \quad 0 \le \tau \le 1
$$
and that the flow $\Psi^{0, \tau}_X$ satisfies $\partial_\tau \big( (\Psi^{0, \tau}_X)^* \Lambda_\tau \big) = 0$ and hence $ (\Psi^{0, \tau}_X)^* \Lambda_\tau = \Lambda_0$.
By the definition of the pullback this means that for any $z \in V',$ $0 \le \tau \le 1,$ $\widehat z$, $\widehat w \in h^0_0$,
$$
\langle J^{-1} \mathcal L_\tau (\Psi^{0, \tau}_X(z))  [d\Psi^{0, \tau}_X(z) [\widehat z]],  d\Psi^{0, \tau}_X(z) [\widehat w] \rangle
= \langle J^{-1} \widehat z, \widehat w \rangle \quad \text{or} \quad
d\Psi^{0, \tau}_X(z)^t  J^{-1} \mathcal L_\tau (\Psi^{0, \tau}_X(z))  d\Psi^{0, \tau}_X(z) = J^{-1}.
$$
Using that $d\Psi^{0, \tau}_X(z)^{-1} = d\Psi^{ \tau, 0}_X(\Psi^{0, \tau}_X(z))$ one obtains the following formula for $d\Psi^{0, \tau}_X(z)^t$,
\begin{equation}\label{formula for d Psi ^{0, tau}_X (z) ^ t }
d\Psi^{0, \tau}_X(z)^t = J^{-1} \,  d\Psi^{ \tau, 0}_X(\Psi^{0, \tau}_X(z)) \, \mathcal L_\tau (\Psi^{0, \tau}_X(z))^{-1} J.
\end{equation}
Note that $ d\Psi^{ \tau, 0}_X(\Psi^{0, \tau}_X(z))$ and $ \mathcal L_\tau (\Psi^{0, \tau}_X(z))^{-1}$ are bounded linear operators on $h^0_0$,
implying that $d\Psi^{0, \tau}_X(z)^t$ is one on $h^1_0$, and that these operators and their derivatives depend continuously on $0 \le \tau \le 1.$
\begin{corollary}\label{expansion of  of differential of Psi}
 For any $0 \le \tau \le 1$, $z  \in {\cal V'}$, the transpose $d\Psi^{0, \tau}_X(z)^t$
(with respect to the standard inner product) of the differential
$d\Psi^{0, \tau}_X(z)$ is a bounded linear operator $ d\Psi^{0, \tau}_X(z)^t : h^1_0 \to  h^1_0$ and for any $N \in \N$ and $\widehat z \in h^1_0$, 
$d\Psi^{0, \tau}_X(z)^t [\widehat z]$ admits an expansion of the form
$$
\big( \, 0, \,  \widehat z_\bot + {\cal F}_\bot \circ \sum_{k = 1}^N  a_k(z; (d\Psi^{0, \tau}_X)^t)  \partial_x^{- k} \mathcal F^{-1}_\bot [\widehat z_\bot] +  
{\cal F}_\bot \circ \sum_{k = 1}^N  \mathcal A_k(z; (d\Psi^{0, \tau}_X)^t)[\widehat z]  \partial_x^{- k} \mathcal F^{-1}_\bot z_\bot  \big) 
 + {\cal R}_{N}(z; (d\Psi^{0, \tau}_X)^t)[\widehat z]
$$
where for any integer $s \ge 0$ and $1 \le  k \le N$, 
$$
a_k( \, \cdot \, ; (d\Psi^{0, \tau}_X)^t) : {\cal V'} \to H^s\,, \,\, z \mapsto  a_k (z; (d\Psi^{0, \tau}_X)^t )\,, \quad
\mathcal A_k( \, \cdot \, ; (d\Psi^{0, \tau}_X)^t) : {\cal V'} \to \mathcal B( h_0^1, H^s), \,\, z \mapsto  \mathcal A_k (z; (d\Psi^{0, \tau}_X)^t )\,,
$$
$$
 {\cal R}_N( \, \cdot \, ;  (d\Psi^{0, \tau}_X)^t ) : {\cal V'} \cap h^s_0 \to {\cal B}(h^{s+1}_0, h_0^{s + 1 + N +1}), \,\,
z  \mapsto {\cal R}_N( z; (d\Psi^{0, \tau}_X)^t)
$$
are real analytic maps. Furthermore, for any $z \in {\cal V'}$, $1 \le k \le N$, $\widehat z_1, \ldots, \widehat z_l \in h^0_0$, $l \ge 2$,
    $$
   \| a_k( z; (d\Psi^{0, \tau}_X)^t)\|_s \lesssim_{s, k} \| z_\bot\|^2_0 \,, \qquad  \| d a_k( z; (d\Psi^{0, \tau}_X)^t)[\widehat z_1]\|_s  \lesssim_{s, k}  \| z_\bot\|_0 \| \widehat z_1\|_0 \,,
   $$
   $$
    \| d^l a_k( z; (d\Psi^{0, \tau}_X)^t)[\widehat z_1, \ldots, \widehat z_l]\|_s  \lesssim_{s, k, l}  \prod_{j = 1}^l \| \widehat z_j\|_0\,,
   $$
and for any  $z \in {\cal V'}$, $\widehat z \in h^1_0$, $\widehat z_1, \ldots, \widehat z_l \in h^0_0$, $l \ge 1$,
 $$ 
   \| \mathcal A_k( z; (d\Psi^{0, \tau}_X)^t) [\widehat z]\|_s \lesssim_{s, k}  \| z_\bot\|_0 \|\widehat z \|_1\,, \qquad
   \| d^l (\mathcal A_k( z; (d\Psi^{0, \tau}_X)^t)  [\widehat z] )[\widehat z_1, \ldots, \widehat z_l]\|_s  \lesssim_{s, k, l} \|\widehat z \|_1 \prod_{j = 1}^l \| \widehat z_j\|_0\,.
  $$
The remainder ${\cal R}_N( z; (d\Psi^{0, \tau}_X)^t)$ satisfies for any $z \in {\cal V'} \cap h^s_0$, $\widehat z \in h^{s+1}_0,$ and $\widehat z_1, \ldots, \widehat z_l \in h^s_0$, $l \in \N$,
  $$
   \| {\cal R}_N( z; (d\Psi^{0, \tau}_X)^t) [\widehat z]\|_{s + 1 + N + 1}  \lesssim_{s, N} \| z_\bot \|_0 \| \widehat z\|_{s+1} + \| z_\bot \|_s \| \widehat z\|_1\,,  \qquad \qquad \qquad \qquad \qquad
   $$
   $$
   \begin{aligned}
   \| d^l \big(  {\cal R}_N & ( z; (d\Psi^{0, \tau}_X)^t)  [\widehat z] \big)[\widehat z_1,  \ldots,  \widehat z_l]  \|_{s + 1 + N + 1} \\
  & \lesssim_{s, N, l} \| \widehat z\|_{s+1} \prod_{j = 1}^l \| \widehat z_j\|_0 + \| \widehat z\|_1 \sum_{j = 1}^l \| \widehat z_j\|_s \prod_{i \neq j} \| \widehat z_i\|_0 + 
   \| z_\bot\|_s  \| \widehat z\|_1 \prod_{j = 1}^l \| \widehat z_j \|_0\,.
    \end{aligned}
  $$
\end{corollary}
\begin{remark}
Corollary \ref{expansion of  of differential of Psi} holds in particular for $ d\Psi_C(z)^t = d\Psi^{0, 1}_X(z)^t$.
\end{remark}
\begin{proof}
The starting point is the formula \eqref{formula for d Psi ^{0, tau}_X (z) ^ t } for $d\Psi^{0, \tau}_X(z)^t$.
Note that $\| J \widehat z \|_{s} \lesssim_s \| \widehat z \|_{s + 1}$ and $\| J^{-1} \widehat z \|_{s + 1} \lesssim_s \| \widehat z \|_{s}$.
Hence it suffices to derive corresponding estimates for the operator
$d\Psi^{ \tau, 0}_X(\Psi^{0, \tau}_X(z)) \, \mathcal L_\tau (\Psi^{0, \tau}_X(z))^{-1}$.
By Theorem \ref{espansione flusso per correttore}, for any $\widehat w \in h^0_0$, $d\Psi^{\tau, 0}_X(w)[\widehat w]$ admits an expansion of the form
$$
  \widehat w +  \big( 0, \, \, {\cal F}_\bot \circ\sum_{k = 1}^N   a_k(w; \Psi_X^{\tau, 0}) \,\, \partial_x^{- k}{\cal F}_\bot^{- 1}[\widehat w_\bot]  
+ {\cal F}_\bot \circ\sum_{k = 1}^N   d a_k(w; \Psi_X^{\tau, 0})[\widehat w] \,\, \partial_x^{- k}{\cal F}_\bot^{- 1}[w_\bot] \big) 
+d {\cal R}_N( w; \Psi_X^{\tau, 0})[\widehat w]
$$
where in the situation at hand
$$
w =  \Psi^{0, \tau}_X(z) = z +  \big( 0, \, \, {\cal F}_\bot \circ\sum_{k = 1}^N   a_k(z; \Psi_X^{0, \tau}) \,\, \partial_x^{- k}{\cal F}_\bot^{- 1}[z_\bot] \big)  + {\cal R}_N( z; \Psi_X^{0, \tau}).
$$ 
By writing $\mathcal L_1 (w)^{-1} = \big( \text{Id} +  J \mathcal L (w) \big)^{-1}$ as  a Neumann series, its asymptotic expansion is then obtained 
from Lemma \ref{lemma potenze JM L(q,z)} (cf also proof of Theorem \ref{espansione flusso per correttore}). 
Combining these results, one obtains the corresponding asymptotic expansion of $d\Psi^{ \tau, 0}_X(\Psi^{0, \tau}_X(z)) \, \mathcal L_\tau (\Psi^{0, \tau}_X(z))^{-1}$.
The claimed estimate then follow from Lemma \ref{lemma potenze JM L(q,z)} and Theorem \ref{espansione flusso per correttore}.
\end{proof}
\medskip

As a second application of Theorem \ref{espansione flusso per correttore}, we compute the Taylor expansion of the symplectic corrector $\Psi_C(z_S, z_\bot)$ in $z_\bot$ at $0$.
This expansion will be needed in the subsequent section to show that the KdV Hamiltonian, when expressed in the new coordinates
provided by the map $\Psi_L \circ \Psi_C$, is in Birkhoff normal
form up to order three. Note that by Theorem \ref{espansione flusso per correttore}, for any $z_S \in \mathcal V'_S,$ $\widehat z_\bot \in h^0_\bot$, $1 \le k \le N$,
$$
a_k((z_S, 0); \Psi_C) = 0\,, \,\,\, d_\bot a_k((z_S, 0); \Psi_C) [\widehat z_\bot] = 0\,, \quad
{\cal R}_{N}((z_S, 0); \Psi_C) = 0 \,, \,\,\, d_\bot \big( {\cal R}_{N}((z_S, 0); \Psi_C)\big)[\widehat z_\bot] = 0 \,.
$$
Hence the Taylor expansion of ${\cal R}_{ N}(z; \Psi_C)$ in $z_\bot$ of order three at $0$ reads
\begin{equation}\label{Taylor expansion R_N}
{\cal R}_{ N}(z; \Psi_C) = {\cal R}_{ N, 2}(z; \Psi_C) + {\cal R}_{ N, 3}(z; \Psi_C)\,, \qquad \mathcal R_{N, 2} (z; \Psi_C) : =  \frac{1}{2} d^2_\bot {\cal R}_{N}((z_S, 0) ; \Psi_C) [z_\bot, z_\bot]
\end{equation}
with  the Taylor remainder term ${\cal R}_{N, 3}(z; \Psi_C)$ given by 
\begin{equation}\label{Taylor remainder term}
{\cal R}_{N, 3}(z; \Psi_C) = \int_0^1 d^3_\bot {\cal R}_{N}((z_S, t z_\bot); \Psi_C)[z_\bot, z_\bot, z_\bot] \, \frac{1}{2}(1 - t)^2 \, dt
\end{equation}
whereas for any $1 \le k \le N$, ${\cal F}_\bot \big( a_{k} ( z; \Psi_C) \partial_x^{- k} {\cal F}_\bot^{- 1}[z_\bot]  \big)$ vanishes in $z_\bot$ at $0$ up to order two.
Furthermore, according to Corollary \ref{expansion of  of differential of Psi}, for any $(z_S, 0) \in \mathcal V'$, $ {\cal R}_{N}((z_S, 0); d \Psi_C^t)=0$ and hence for any
$\widehat z \in h^1_0$,  the Taylor expansion of ${\cal R}_{N}(z; d \Psi_C^t)[\widehat z]$ of order $2$ in $z_\bot$ around $0$ reads
\begin{equation}\label{Taylor expansion of of d Psi C t}
 {\cal R}_{N}(z; d \Psi_C^t)[\widehat z] = {\cal R}_{N, 1}(z; d \Psi_C^t)[\widehat z] + {\cal R}_{N, 2}(z; d \Psi_C^t)[\widehat z] 
\end{equation}
where ${\cal R}_{N, 2}(z; d \Psi_C^t)[\widehat z]$ denotes the Taylor remainder term of order $2$.
\begin{corollary}\label{proposizione espansione taylor correttore simplettico}
 $(i)$ For any integer $N \ge 1$, the Taylor expansion of the symplectic corrector $\Psi_C(z_S, z_\bot)$ in $z_\bot$ around $0$ reads   
$$
 \Psi_C(z) = (z_S, 0)  + ( 0, z_\bot) + {\cal R}_{N, 2}(z; \Psi_C) +  \Psi_{C, 3}(z)
 $$
 where $ \Psi_{C, 3}(z) \equiv  \Psi_{C, N, 3}(z)$ is given by
 \begin{equation}\label{espansione Psi C ordini taglia}
 \Psi_{C, N, 3}(z) := \big( 0,  \, {\cal F}_\bot \circ \sum_{k = 1}^N  a_{k} ( z; \Psi_C) \partial_x^{- k} {\cal F}_\bot^{- 1}[z_\bot]  \big)  + {\cal R}_{ N, 3} (z; \Psi_{C})\,.
  \end{equation}
For any $s \geq 0$, the map $\mathcal V' \cap h^s_0 \to  h^{s + N +1}_0$, $z \mapsto {\cal R}_{N, 2}(z; \Psi_C)$  is real analytic
and  the following estimates hold: for any $z \in \mathcal V' \cap h^s_0 $, $\widehat z \in  h^s_0$,
$$
 \| {\cal R}_{N, 2}(z; \Psi_C)\|_{s + N +1} \lesssim_{s, N} \| z_\bot \|_s \| z_\bot \|_0\,, \qquad  
 \| d {\cal R}_{N, 2}(z; \Psi_C)[\widehat z]\|_{s + N +1} \lesssim_{s, N} \| z_\bot \|_0 \| \widehat z\|_s + \| z_\bot \|_s \| \widehat z\|_0\,, 
 $$
 and, if in addition $ \widehat z_1, \ldots, \widehat z_l \in  h^s_0$, $l \ge 2,$
 $$
 \| d^l {\cal R}_{N, 2}(z; \Psi_C)[\widehat z_1, \ldots, \widehat z_l] \|_{s + N +1} \lesssim_{s, N, l} \sum_{j = 1}^l \| \widehat z_j \|_s \prod_{i \neq j} \| \widehat z_i \|_0 + 
\| z_\bot \|_s \prod_{j = 1}^l \| \widehat z_j\|_0\,.
$$
Similarly, for any $s \geq 0$, the map $\mathcal V' \cap h^s_0 \to  h^{s + N +1}_0$, $z \mapsto {\cal R}_{N, 3}(z; \Psi_C)$  is real analytic
and  the following estimates hold: for any $z \in \mathcal V' \cap h^s_0 $, $ \widehat z_1, \widehat z_2 \in  h^s_0$,
$$
 \| {\cal R}_{N, 3}(z; \Psi_C)\|_{s + N + 1} \lesssim_{s, N} \| z_\bot \|_s \| z_\bot \|_0^2 \,,  
 \qquad \| d {\cal R}_{N, 3}(z; \Psi_C )[\widehat z_1]\|_{s + N +1} \lesssim_{s, N} \| z_\bot \|_s \| z_\bot \|_0 \| \widehat z_1\|_0 + \| z_\bot \|_0^2 \| \widehat z_1\|_s \,, 
 $$
$$ 
\| d^2  {\cal R}_{N, 3}(z; \Psi_C )[\widehat z_1, \widehat z_2] \|_{s + N +1} \lesssim_{s, N} 
 \| z_\bot \|_0 \big( \| \widehat z_1\|_s \| \widehat z_2\|_0 + \| \widehat z_1\|_0 \| \widehat z_2\|_s \big) + \| z_\bot \|_s \| \widehat z_1\|_0 \| \widehat z_2\|_0\,, 
$$
and if in addition $\widehat z_1, \ldots, \widehat z_l \in  h^s_0$, $l \ge 3$,
$$
 \| d^l {\cal R}_{N, 3}(z; \Psi_C )[\widehat z_1, \ldots, \widehat z_l ]\|_{s + N +1} 
\lesssim_{s, N, l} \sum_{j = 1}^l \| \widehat z_j\|_s \prod_{i \neq j} \| \widehat z_i\|_0 + \| z_\bot \|_s \prod_{j = 1}^l \| \widehat z_j\|_0\,. 
$$
\noindent
$(ii)$ For any integer $N \ge 1$ and $\widehat z \in h^1_0$, the Taylor expansion of $d \Psi_C(z_S, z_\bot)^t[\widehat z]$ in $z_\bot$ around $0$ is given by 
$$
d \Psi_C(z)^t [\widehat z]= \widehat z + \Psi_{C, 1}^t(z) [\widehat z] + \Psi_{C, 2}^t(z)[\widehat z]
 $$
 where $\Psi_{C, 1}^t(z) =  {\cal R}_{N, 1}(z; d \Psi_C^t)$ (cf \eqref{Taylor expansion of of d Psi C t}) and $\Psi_{C, 2}^t(z) [\widehat z]$ has an expansion of the form 
 $$
\Psi_{C, 2}^t (z)[\widehat z] =   \big( \, 0, \,  {\cal F}_\bot \circ \sum_{k = 1}^N  a_k(z; d\Psi_{C}^t)  \partial_x^{- k} \mathcal F^{-1}_\bot [\widehat z_\bot] +  
{\cal F}_\bot \circ \sum_{k = 1}^N  \mathcal A_k(z; d\Psi_{C}^t)[\widehat z]  \partial_x^{- k} \mathcal F^{-1}_\bot [z_\bot]  \big)
+ {\cal R}_{N, 2}(z; d \Psi_C^t)[\widehat z]
$$
with $a_k(z; d\Psi_{C}^t) $, $\mathcal A_k(z; d\Psi_{C}^t)$ given as in Corollary \ref{expansion of  of differential of Psi},
and ${\cal R}_{N, 2}(z; d \Psi_C^t)[\widehat z]$ given by \eqref{Taylor expansion of of d Psi C t}.
For any $i =1, 2$, $s \ge 0$, 
$$
 {\cal R}_{N, i}( \, \cdot \, ; \, d\Psi_{C}^t ) : {\cal V'} \cap h^s_0 \to {\cal B}(h^{s+1}_0, h_0^{s + 1 + N +1}), \,\,
z  \mapsto {\cal R}_{N, i}( z; d\Psi_{C}^t)
$$
is a real analytic maps. Furthermore, for any $z \in {\cal V'} \cap h^s_0$, $\widehat z \in h^{s+1}_0,$
  $$
   \| {\cal R}_{N, 1}( z; d\Psi_{C}^t) [\widehat z]\|_{s + 1 + N + 1}  \lesssim_{s, N}  \|  z_\bot\|_0 \| \widehat z\|_{s+1} + \| z_\bot \|_s \| \widehat z\|_1\,, 
   $$
   and if in addition $\widehat z_1, \ldots, \widehat z_l \in h^s_0$, $l \in \N$,
   $$
   \begin{aligned}
   \| d^l \big(  {\cal R}_{N, 1} & ( z; d\Psi_{C}^t)  [\widehat z] \big)[\widehat z_1,  \ldots,  \widehat z_l]  \|_{s + 1 + N + 1} \\
  & \lesssim_{s, N, l} \| \widehat z\|_{s+1} \prod_{j = 1}^l \| \widehat z_j\|_0 + \| \widehat z\|_1 \sum_{j = 1}^l \| \widehat z_j\|_s \prod_{i \neq j} \| \widehat z_i\|_0 + 
   \| z_\bot\|_s  \| \widehat z\|_1 \prod_{j = 1}^l \| \widehat z_j \|_0\,.
    \end{aligned}
  $$
  and for any $z \in {\cal V'} \cap h^s_0$, $\widehat z \in h^{s+1}_0,$ $\widehat z_1 \in h^s_0$,
  $$
   \| {\cal R}_{N, 2}( z; d\Psi_{C}^t) [\widehat z]\|_{s + 1 + N + 1}  \lesssim_{s, N}  \|  z_\bot\|_0^2 \| \widehat z\|_{s+1} + \| z_\bot \|_s \| z_\bot\|_0 \| \widehat z\|_1\,, 
   $$
   $$
    \|d \big( {\cal R}_{N, 2}  ( z; d\Psi_{C}^t)  [\widehat z] \big)[\widehat z_1] \|_{s + 1 + N + 1} 
    \lesssim_{s, N} \|  z_\bot\|_0 \| \widehat z_1 \|_0 \| \widehat z\|_{s+1} + \|  z_\bot\|_0 \| \widehat z_1 \|_s \| \widehat z\|_1 + \| z_\bot \|_s \|  \widehat z_1 \|_0 \| \widehat z\|_1 \,,
    $$
   and if in addition $\widehat z_2, \ldots, \widehat z_l \in h^s_0$, $l \ge 2$,
   $$
    \| d^l \big(  {\cal R}_{N, 2}  ( z; d \Psi_{C}^t)  [\widehat z] \big)[\widehat z_1,  \ldots,  \widehat z_l]  \|_{s + 1 + N + 1}  \qquad \qquad \qquad \qquad \qquad \qquad \qquad
    $$
    $$
   \lesssim_{s, N, l} \| \widehat z\|_{s+1} \prod_{j = 1}^l \| \widehat z_j\|_0 + \| \widehat z\|_1 \sum_{j = 1}^l \| \widehat z_j\|_s \prod_{i \neq j} \| \widehat z_i\|_0 + 
   \| \widehat z\|_1  \| z_\bot\|_s  \prod_{j = 1}^l \| \widehat z_j \|_0\,.
  $$
\end{corollary}

\begin{proof} 
$(i)$ The claimed properties of ${\cal R}_{N, 2}(z; \Psi_C)$ follow directly from Theorem \ref{espansione flusso per correttore}. In view of the formula \eqref{Taylor remainder term}
the same is true for the ones of ${\cal R}_{N, 3}(z; \Psi_C)$. 
Item $(ii)$ is a direct consequence of Corollary \ref{expansion of  of differential of Psi}. 
\end{proof}


 \section{The KdV Hamiltonian in new coordinates}\label{Hamiltoniana trasformata}
In this section we provide an expansion of the transformed KdV Hamiltonian $\mathcal H = H^{kdv} \circ \Psi$
where the map $\Psi = \Psi_L \circ \Psi_C$ is the composition of $\Psi_L$ (cf Section \ref{sezione mappa Psi L}) with the symplectic corrector $\Psi_C$ (cf Section \ref{sezione Psi C})
and $H^{kdv}$ is the KdV Hamiltonian given by 
 \begin{equation}\label{cal H kdv espansione}
  H^{kdv}(u) =\frac12 \int_0^1 u_x^2\, d x + \int_0^1 u^3 \, d x\,. 
  \end{equation}
  First we need to make some preliminary considerations. Recall that
  for any finite subset $S_+ \subset \mathbb N$, the Birkhoff map $\Psi^{kdv}$ establishes a one to one correspondance between $\mathcal M_S$ 
 and the set $M_S$ of $S-$gap potentials
 where $S = S_+ \cup (- S_+)$.
  For any $S-$gap potential $q$, the corresponding KdV actions $I = (I_S, I_\bot)$, defined in terms of the Birkhoff coordinates $\Phi^{kdv}(q)$, satisfy $I_\bot = 0.$
 Denote by $\Omega_\bot(I_S) \equiv \Omega_\bot^{kdv}(I_S)$ and $\Omega_S(I_S) \equiv \Omega_S^{kdv}(I_S)$ the  diagonal linear operators defined by 
  \begin{equation}\label{Omega_S}
 \qquad \Omega_S(I_S) :=   {\rm diag}(( \Omega_n(I_S))_{n \in S}) :h^0_S \to  h^0_S\,, (z_n)_{n \in S} \to ( \Omega_n(I_S) z_n)_{n \in S}
  \end{equation}
   \begin{equation}\label{splitting Omega}
   \qquad \Omega_\bot(I_S) := 
 {\rm diag}(( \Omega_n(I_S))_{n \in S^\bot}) :h^2_\bot \to  h^{0}_\bot\,, (z_n)_{n \in S^\bot} \to ( \Omega_n(I_S) z_n)_{n \in S^\bot}
  \end{equation}
   where for any $n \ge 1$,
   \begin{equation}\label{definition Omega n}
    \Omega_n(I_S)  \equiv \Omega_n^{kdv}(I_S) := \frac{1}{2 \pi n } \omega_n((I_S, 0))\,, \quad   
    \Omega_{-n}(I_S) \equiv \Omega_{-n}^{kdv}(I_S) :=  \Omega_{n}^{kdv}(I_S)
   \end{equation} 
  and $\omega_n(I) \equiv \omega_n^{kdv}(I)$ is the $n$th KdV frequency, viewed as a function of the actions. 
  By Lemma \ref{Lemma appendice espansione frequenze}, one has: 
  \begin{lemma}\label{espansione asintotica frequenza thomas}
  For any finite gap potential $q \in M_S$, $n \ge 1$, and $N \ge 1$ one has 
  \begin{equation}\label{asintotica frequenze kdv}
  \Omega_n(I_S) = (2 \pi n)^2 +  \sum_{k = 1}^N \frac{\Omega_{2k}^{ae}(I_S)}{(2 \pi n)^{2k}} + \frac{{\cal R}_{2N}^{\Omega_n}(I_S)}{(2 \pi n)^{2 N + 1}} 
  \end{equation}
  where $\Omega_{2k}^{ae}(I_S) = \omega_{2k -1}^{ae}(I_S, 0)$, ${\cal R}_{2N}^{\Omega_n}(I_S) = {\cal R}_{2N}^{\omega_n}(I_S, 0)$
  and $\omega_{2k - 1}^{ae}(I_S, 0)$, ${\cal R}_{2N}^{\omega_n}(I_S, 0)$ are given by Lemma \ref{Lemma appendice espansione frequenze}.
  \end{lemma}
 Assume that  $q(t)$ is a solution of the KdV equation \eqref{1.1} in $M_S$ with 
  $z(t) := \Phi^{kdv}(q(t)) \in \mathcal V$ for any $t$. Note that $z(t)$ is of the form $(z_S(t), 0)$,
 the actions $I = (I_n)_{n \ge 1}$ of $q(t)$ are independent of $t$, and $I= (I_S, 0)$
 where $I_S = ( \frac{1}{2\pi n} z_n(0) z_{-n}(0))_{n \in S_+}$. Furthermore,
 $ \partial_t  z_S (t) = J_S \Omega_S(I_S)[z_S(t)]$, or in more detail, for any $n \in S$,
$$
\partial_t  z_n (t)  =  2\pi \ii n \Omega_n (I_S)  z_n(t)\,.
$$
 Denote by $\widehat q(t)$ the solution of the equation, obtained by linearizing the KdV equation along $q(t)$, 
\begin{equation}\label{lin KdV}
  \partial_t \widehat q(t) = \partial_x d  \nabla H^{kdv}(q(t)) [\widehat q(t)]\,.
  \end{equation}
  We need to investigate $\partial_x d  \nabla H^{kdv}(q(t)) [\widehat q(t)]$ further.
 If $\widehat q(0)$ is of the form $d \Psi_L(z_S(0), 0)[ 0, \widehat z_\bot(0)]$ ($= \Psi_1(z_S(0))[\widehat z_\bot(0)]$\,) with $\widehat z_\bot(0) \in h^3_\bot$, 
 then by \eqref{definition Psi_L} (definition of $\Psi_L$)
 and \eqref{formula d Psi L} (formula of the differential $d\Psi_L$), 
  $\widehat z_\bot (t)$, defined by $\widehat q(t) = \Psi_1(z_S(t)) [\widehat z_\bot(t)]$, solves the equation 
  $$\partial_t \widehat z_\bot (t) =   J_\bot \Omega_\bot(I_S) [ \widehat z_\bot (t) ]
  $$
  or more explicitly, for any  $n \in S^\bot_+$,
 $$
 \partial_t \widehat z_n (t)  = 2\pi \ii n \Omega_n (I_S) \widehat z_n(t) ,  \qquad 
  \partial_t \widehat z_{-n} (t)  = 2\pi \ii (-n) \Omega_{-n} (I_S) \widehat z_{-n}(t)\,.
 $$
  By differentiating $\widehat q(t) = \Psi_1(z_S(t))[\widehat z_\bot(t)]$ with respect to $t$, one gets  
  \begin{align}
  \partial_t \widehat q (t)& = \Psi_1(z_S(t))[\partial_t \widehat z_\bot (t)] + d_S \big( \Psi_1(z_S(t))[\widehat z_\bot (t)] \big) [\partial_t z_S(t) ]\nonumber\\
  & = \Psi_1(z_S(t)) \big[ J_\bot \Omega_\bot(I_S) [ \widehat z_\bot (t)]  \big] +
  d_S \big( \Psi_1(z_S(t))[\widehat z_\bot(t)] \big)[ \partial_t z_S(t)]\,. 
  \label{maradona 1} 
  \end{align}
  Comparing \eqref{lin KdV} and \eqref{maradona 1} 
  and using that $ \partial_t z_S(t) = J_S \Omega_S(I_S) [z_S(t)]$, one gets 
  \begin{align}
\partial_x d \nabla H^{kdv}(q(t)) & \big[ \Psi_1(z_S(t)) [\widehat z_\bot(t) ] \big] 
=  \Psi_1(z_S(t)) \big[ J_\bot \Omega_\bot(I_S) [\widehat z_\bot(t) ] \big] \,  \nonumber\\
  & \quad + \, d_S \big( \Psi_1(z_S(t))[\widehat z_\bot(t) ] \big) [J_S \Omega_S(I_S) [z_S(t)]]\,.  \label{maradona 3}
  \end{align}
  Now apply $\Psi_1(z_S(t))^{- 1}$ to both sides of the latter equality yielding
  \begin{align}
\Psi_1(z_S(t))^{- 1}\partial_x & d \nabla H^{kdv}(q(t))  \big[ \Psi_1(z_S(t)) [\widehat z_\bot(t)]  \big] 
=  J_\bot \Omega_\bot(I_S)  [\widehat z_\bot(t)]\,  \nonumber\\
  & \quad + \, \Psi_1(z_S(t))^{- 1} d_S \big( \Psi_1(z_S(t))[\widehat z(t) ] \big)[J_S \Omega_S(I_S) [z_S(t)]]\,.  \label{maradona 3-1}
  \end{align}
  Since $\Psi_1(z_S)$ is symplectic one has $ \Psi_1(z_S)^t \partial_x^{- 1} \Psi_1(z_S) = J_\bot^{- 1}$ or
   $ \Psi_1(z_S)^{- 1} \partial_x  = J_\bot \Psi_1(z_S)^{ t},$
  implying that 
   \begin{align}
J_\bot \Psi_1(z_S(t))^t d \nabla H^{kdv}(q(t)) & \big[ \Psi_1(z_S(t))[\widehat z_\bot(t)]  \big] 
=  J_\bot \Omega_\bot(I_S) [\widehat z_\bot(t)]\,  \nonumber\\
  & \quad + \, \Psi_1(z_S(t))^{- 1} d_S \big( \Psi_1(z_S(t))[\widehat z_\bot(t) ] \big)[J_S \Omega_S(I_S) [z_S(t)]]\,.  \label{maradona 3-2}
  \end{align}
The latter identity implies that  for any $z_S \in \mathcal V_S$, $I_S= (\frac{1}{2\pi n} z_nz_{-n})_{n \in S_+}$, $q = \Psi^{kdv}(z_S, 0)$, $\widehat z_\bot \in h^3_\bot$,
   \begin{align}
 \Psi_1(z_S)^t d \nabla H^{kdv}(q) & \big[ \Psi_1(z_S) [\widehat z_\bot]  \big] 
=   \Omega_\bot(I_S) [\widehat z_\bot] + \,{\cal G}(z_S)[\widehat z_\bot]   \label{maradona 4}
  \end{align}
  where $ {\cal G}(z_S): h^0_\bot \to h^0_\bot$ is given by
  \begin{equation}\label{definizione cal M (wS)}
  \begin{aligned}
  {\cal G}(z_S)[\widehat z_\bot] & := J_\bot^{- 1} \Psi_1(z_S)^{- 1} d_S \big( \Psi_1(z_S)[\widehat z_\bot] \big)[ J_S \Omega_S(I_S) [z_S]]\,. 
  \end{aligned}
  \end{equation}
  In the next lemma we record an expansion for the operator ${\cal G}(z_S)$. 
      \begin{lemma}\label{stime cal M(wS)}
  For any integer $N \ge 1$, the operator ${\cal G}(z_S): h^0_\bot \to h^0_\bot$ admits an expansion of the form 
  $$
   {\cal G}(z_S) =  {\cal F}_\bot \circ \sum_{k =  1}^{N}  a_k( z_S;  {\cal G}) \partial_x^{- k} \circ {\cal F}_\bot^{- 1} + {\cal R}_N(z_S;  {\cal G})
  $$  
  where for any $1 \le k \le N,$ $s \ge 0,$
  the maps 
  $$
  \mathcal V_S \to H^s, \, z_S  \mapsto a_k(z_S; {\cal G})\,, \qquad 
  \mathcal V_S \mapsto {\cal B}(h^s_\bot, h^{s + N + 1}_\bot), \, z_S  \mapsto {\cal R}_N(z_S; \mathcal G)
  $$
   are real analytic. 
  \end{lemma}
  \begin{proof}
  In view of the definition \eqref{definizione cal M (wS)} of ${\cal G}$,
  the lemma follows from item (ii) of  Theorem \ref{lemma asintotica Floquet solutions} (expansion of $ \Psi_1(z_S)$)
  and  Lemma \ref{lemma composizione pseudo}.
    \end{proof}

  \smallskip
  
    After this preliminary discussion, we can now study the transformed Hamiltonian ${H}^{kdv} \circ \Psi$ where $\Psi = \Psi_L \circ \Psi_C$. 
    We split the analysis into two parts. First we expand ${\cal H}^{(1)} := { H}^{kdv} \circ \Psi_L$ and then we analyze ${\cal H}^{(2)} = {\cal H}^{(1)} \circ \Psi_C$. 
   
   \medskip
   
  \noindent 
 {\bf Expansion of ${\cal H}^{(1)} := H^{kdv} \circ \Psi_L$}
 
 \noindent
To expand $H^{kdv} \circ \Psi_L$, it is useful to write $H^{kdv}(u) $ as $H^{kdv}(u) =H_2^{kdv} (u)+ H_3^{kdv}(u)$ where
  \begin{equation}\label{forma compatta hamiltoniana d-NLS}
  H_2^{kdv}(u) := \frac12 \langle (- \partial_x^2) u\,,\,u \rangle\,, 
  \qquad H_3^{kdv}(u) := \int_0^1 u^3  \,dx\,.
  \end{equation}
   The $L^2-$gradient $\nabla H^{kdv}$ of $H^{kdv}$ and its derivative are then given by 
 \begin{equation}\label{espressione gradiente Hamiltoniana originaria nls}
  \nabla H^{kdv}(u) = - \partial_x^2  u + 3u^2\,, \quad
  d \nabla H^{kdv}(u) = - \partial_x^2 + 6u\,.
  \end{equation} 
Let $z_S \in \mathcal V_S$, $q = \Psi^{kdv}(z_S, 0)$. 
The Taylor expansion of $H^{kdv}(q+v)$ around $q$ in direction 
$v = \Psi_1(z_S)[z_\bot]$ 
with $z_\bot \in \mathcal V_\bot \cap h_\bot^1$ reads
\begin{align}
H^{kdv}(q+v) =  H^{kdv}(q )  + \langle \nabla H^{kdv}(q ), v \rangle +
\frac12 \langle d \nabla H^{kdv}(q )[v]\,,\, v \rangle +   \int_0^1 v^3 \, dx\,. \label{espansione taylor H4 nls}
\end{align}
Since $v = d\Psi^{kdv}(z_S,0)[0, z_\bot]$ one has
$\langle \nabla H^{kdv}(q ), v \rangle  = \partial_y |_{y=0} H^{kdv} (\Psi^{kdv}(z_S, yz_\bot))$.
Recall that $\mathcal H^{kdv} = H^{kdv} \circ \Psi^{kdv}$ is a function of the actions alone and $I_n = \frac{1}{2\pi n}z_nz_{-n}$, $n \ge 1$.
It implies that  
$$
\partial_y |_{y=0} H^{kdv} (\Psi^{kdv}(z_S, yz_\bot)) = \sum_{n \in S^\bot_+} \omega_n(I_S, 0) \partial_y |_{y=0} \, y^2 I_n =0,
$$
and hence $\langle \nabla H^{kdv}(q ), v \rangle  =0$.
Since $\Psi_L(z) = q  + \Psi_1(z_S)[z_\bot]$ one then gets
  \begin{align}
 {\cal H}^{(1)}(z)  = H^{kdv}(\Psi_L(z)) & = H^{kdv}(q)  + \frac12 \big\langle d \nabla H^{kdv}(q) \big[ \Psi_1(z_S)[z_\bot] \big]\,,\, \Psi_1(z_S)[z_\bot] \big\rangle + 
 \int_0^1 (\Psi_1(z_S)[z_\bot])^3 \, d x \,.  \nonumber
  \end{align}
  By formula \eqref{maradona 4}, 
  $$
   \big\langle d \nabla {H}^{kdv}(q )  \big[ \Psi_1(z_S)[z_\bot] \big]\,,\, \Psi_1(z_S)[z_\bot ] \big\rangle
   = \big\langle \Omega_\bot^{kdv}(I_S) [z_\bot]\,,\, z_\bot \big\rangle  +
  \big\langle{\cal G}(z_S) [z_\bot]\,,\, z_\bot \big\rangle\,.
  $$
   Since $\Psi_1(z_S)^t d \nabla H^{kdv}(q) \Psi_1(z_S)$ and  $\Omega_\bot^{kdv}(I_S)$ are symmetric,
  so is the operator ${\cal G}(z_S)$. 
In summary,  
     \begin{align}
{\cal H}^{(1)}(z) = \mathcal H^{kdv}_S(z) + 
\frac12 \big\langle { \Omega}^{kdv}_\bot(I_S)[ z_\bot], \,  z_\bot \big\rangle \,
+ {\cal P}_2^{(1)}(z) + {\cal P}_3^{(1)}(z)    \label{forma semifinale H nls circ Psi}
  \end{align}
where for any $z = (z_S, z_\bot) \in \mathcal V \cap h^1_0,$
  \begin{align}
  & \mathcal H^{kdv}_S(z) :=  H^{kdv}(\Psi^{kdv}(z_S, 0))\,, \quad
  & {\cal P}_2^{(1)}(z) :=  \frac12 \langle {\cal G}(z_S)[z_\bot], \, z_\bot \rangle\,, \quad 
  {\cal P}_3^{(1)}(z) := \int_0^1 (\Psi_1(z_S)[z_\bot] )^3\, d x\,. 
  \label{perturbazione composizione con Psi L}
  \end{align}
  Note that $ \mathcal H^{kdv}_S(z) = \mathcal H^{kdv}_S(\Pi_S z)$ where 
$\Pi_S : h^0_S \times h^0_{ \bot } \to h^0_S \times h^0_{ \bot }$ denotes the projection, given by  
$(\widehat z_S, \widehat z_\bot) \mapsto (\widehat z_S, 0)$ (cf \eqref{Pi S Pi bot}).
  We record that ${\cal P}_2^{(1)}(z)$ is quadratic and ${\cal P}_3^{(1)}(z)$ cubic  in $z_\bot$ where the superscript $(1)$ refers to the Hamiltonian $\mathcal H^{(1)}$.
  Recall from \eqref{definition prarproduct} that for any $a \in H^1$, the paraproduct $T_a u$ of the function $a$  with $u \in L^2$ 
  with respect to the cut-off function $\chi$ is defined as $(T_a u) (x) = \sum_{k, n \in \Z} \chi(k, n) a_k u_n e^{\ii 2 \pi (k + n) x}$
  with $u_n$, $n \in \Z$, denoting the Fourier coefficients of $u$.
  \begin{lemma}\label{lemma stima cal P2 P3}
 For any integer $N \ge 1$, there exists an integer $\sigma_N \ge N$ (loss of regularity) so that on $\mathcal V \cap h^{\sigma_N}_0$,
  the $L^2-$gradient $\nabla{\cal P}_3^{(1)}$ of  ${\cal P}_3^{(1)}$ admits the asymptotic expansion of the form
  \begin{equation}\label{epansione d nabla P3}
  \nabla {\cal P}_3^{(1)}(z) = \big( 0, \, {\cal F}_\bot \circ  \sum_{k = 0}^N   T_{a_k(z; \nabla {\cal P}_3^{(1)})} \partial_x^{- k}{\cal F}_\bot^{- 1} [z_\bot] \, \big) + {\cal R}_N(z; \nabla {\cal P}_3^{(1)})
  \end{equation}
  where for any $s \geq 0$, $1 \le k \le N$, the maps
   $$
   \mathcal V \cap h^{s + \sigma_N}_0 \to H^s, \, z  \mapsto a_k(z; \nabla {\cal P}_3^{(1)})\,, \qquad
   \mathcal V \cap h^{s \lor \sigma_N}_0 \to h^{s + N + 1}_0, \, z \mapsto {\cal R}_N(z; \nabla {\cal P}_3^{(1)})
   $$
 are real analytic. Furthermore, for any $z \in \mathcal V \cap h^{s + \sigma_N}_0$ with  $\| z \|_{\sigma_N} \leq 1$, 
 $
 \| a_k(z; \nabla {\cal P}_3^{(1)}) \|_s \lesssim_{s, N} \| z_\bot\|_{s + \sigma_N}
 $
 and if in addition $\widehat z_1, \ldots , \widehat z_l \in h^{s+ \sigma_N}_0$, $l \ge 1$, 
  \begin{equation}\label{stima P3 dopo Psi L}
   \| d^l a_k(z; \nabla {\cal P}_3^{(1)}) [\widehat z_1 , \ldots, \widehat z_l] \|_s \, \lesssim_{s, k, l} \,
   \sum_{j = 1}^l \| \widehat z_j\|_{s + \sigma_N} \prod_{i \neq j} \| \widehat z_i \|_{\sigma_N} + \| z_\bot \|_{s + \sigma_N} \prod_{j = 1}^l \| \widehat z_j\|_{\sigma_N}\,.
   \end{equation}
  Similarly, for any $ z \in \mathcal V \cap h^{s \lor \sigma_N}_0$ with  $\| z \|_{\sigma_N} \leq 1$, $\widehat z \in h^{s \lor \sigma_N}_0$,
  $ \| {\cal R}_N(z; \nabla {\cal P}_3^{(1)}) \|_{s + N + 1} \lesssim_{s, N} \| z_\bot \|_{s \lor \sigma_N} \| z_\bot \|_{\sigma_N}$ and
   \begin{equation}\label{stima P3 dopo Psi L (part 2)}
  \| d {\cal R}_N(z; \nabla {\cal P}_3^{(1)}) [\widehat z] \|_{s + N + 1} \lesssim_{s, N} 
  \| z_\bot\|_{s \lor \sigma_N } \| \widehat z \|_{\sigma_N} + \| z_\bot\|_{\sigma_N} \| \widehat z \|_{s \lor \sigma_N }
  \end{equation}
  and if in addition $\widehat z_1, \ldots , \widehat z_l \in h^{s \lor \sigma_N}_0$, $l \ge 2$, then
  \begin{equation}\label{stima P3 dopo Psi L (part 3)}
   \| d^l {\cal R}_N(z; \nabla {\cal P}_3^{(1)})[\widehat z_1, \ldots, \widehat z_l] \|_{s + N + 1} \lesssim_{s, N, l} 
  \sum_{j = 1}^l \| \widehat z_j\|_{s \lor \sigma_N} \prod_{i \neq j} \| \widehat z_i\|_{\sigma_N} + \| z_\bot \|_{s \lor\sigma_N} \prod_{j = 1}^l \| \widehat z_j\|_{\sigma_N}\,.
  \end{equation}
   \end{lemma}
  \begin{proof}
  By a straightforward calculation, one has
  $\nabla_\bot {\cal P}_3^{(1)}(z) = 3 \Psi_1(z_S)^t ( \Psi_1(z_S)[z_\bot] )^2$.
 By the Bony decomposition given in Lemma \ref{primo lemma paraprodotti}$(ii)$,
  $$
( \Psi_1(z_S)[z_\bot])^2 = 2  \, T_{\Psi_1(z_S)[z_\bot]} \Psi_1(z_S)[z_\bot] + \, {\cal R}^{(B)} (\Psi_1(z_S)[z_\bot]\,,\, \Psi_1(z_S)[z_\bot] ) \,. 
  $$
 The expansion \eqref{epansione d nabla P3} and the stated estimates follow from Theorem \ref{lemma asintotica Floquet solutions}$(ii)$,
 Corollary \ref{lemma asintotiche Phi 1 q}, and Lemmata  \ref{primo lemma paraprodotti}, \ref{lemma composizione pseudo}, \ref{lemma interpolation}.    
  \end{proof}
  
 \smallskip

\noindent  
{\bf Expansion of ${\cal H}^{(2)} := {\cal H}^{(1)} \circ \Psi_C$}

\noindent
To compute the expansion of ${\cal H}^{(2)}(z)= {\cal H}^{(1)} ( \Psi_C(z))$ on $\mathcal V' \cap h^1_0$, 
we study the composition of each of the terms in \eqref{forma semifinale H nls circ Psi} with the symplectic corrector $\Psi_C$ separately. 
Recall that $\Psi_C$ is defined on $\mathcal V' $ and takes values in $\mathcal V$.

\smallskip

\noindent
{\em Term $\mathcal H_S^{kdv} $:} 
By Corollary \ref{proposizione espansione taylor correttore simplettico}, $\Psi_C(z)$ has a Taylor expansion in $z_\bot$ around $0$ of the form
$$
\Psi_C(z) = (z_S, 0) + (0, z_\bot) + \tilde \Psi_C(z), \qquad
\tilde \Psi_C(z):= {\cal R}_{N, 2}(z; \Psi_C) +  \Psi_{C, 3}(z)\,, \qquad  \Psi_{C, 3}(z) \equiv  \Psi_{C, N, 3}(z)
$$ 
where $R_{N, 2} (z; \Psi_C)$ is the term of order two, given by $R_{N, 2} (z; \Psi_C) =  \frac{1}{2} d^2_\bot {\cal R}_{N}((z_S, 0) ; \Psi_C) [z_\bot, z_\bot]$ (cf \eqref{Taylor expansion R_N}),
and $\Psi_{C, 3}(z)$ is given by \eqref{espansione Psi C ordini taglia}
$$
\Psi_{C, 3}(z) = \big( 0,  \, {\cal F}_\bot \circ \sum_{k = 1}^N  a_{k} ( z; \Psi_C) \partial_x^{- k} {\cal F}_\bot^{- 1}[z_\bot]  \big)  + {\cal R}_{ N, 3} (z; \Psi_{C})
$$ 
with ${\cal R}_{ N, 3} (z; \Psi_{C})$ denoting the Taylor remainder term \eqref{Taylor remainder term}. 
Since $ \mathcal H^{kdv}_S(z) =  \mathcal H^{kdv}_S(\Pi_S z)$ (cf \eqref{perturbazione composizione con Psi L}),
the Taylor expansion of $\mathcal H^{kdv}_S( \Psi_C(z)) = \mathcal H^{kdv}_S(z  + \tilde \Psi_C(z))$ reads
\begin{align}\label{h nls circ PsiC}
\mathcal H_S^{kdv}(\Psi_C(z)) & = \mathcal H_S^{kdv}(z) + \langle \nabla_S \mathcal H_S^{kdv}(z), \,  \pi_S {\cal R}_{N, 2}(z; \Psi_C) \rangle + 
{\cal P}^{(2a)}_3(z)\,, 
\end{align}
where ${\cal P}^{(2a)}_3(z)$ is the Taylor remainder term of order three, given by
$$
 \langle \nabla_S \mathcal H_S^{kdv}(z),  \pi_S  \Psi_{C, 3}(z) \rangle + 
\int_0^1 (1 - y) \big\langle \, d_S \big( \nabla_S \mathcal H_S^{kdv}(z + y \tilde \Psi_C(z) ) \big) 
[ \pi_S  \tilde \Psi_C(z)], \, \pi_S  \tilde \Psi_C(z) \, \big\rangle \, d y 
$$
and $\pi_S : h^0_S \times h^0_{ \bot } \to h^0_S$ denotes the map given by $z = (z_S, z_\bot) \mapsto z_S$ (cf \eqref{pi S}).
Since $\pi_S  \Psi_{C, 3}(z) = \pi_S {\cal R}_{ N, 3} (z; \Psi_{C})$ and $\pi_S \tilde \Psi_C(z) = \pi_S {\cal R}_{N}(z; \Psi_C)$,
one has
\begin{align}\label{definizione h3 nls}
{\cal P}^{(2a)}_3(z)  & =   \langle \nabla_S \mathcal H_S^{kdv}(z),  \pi_S  R_{N, 3}(z; \Psi_C) \rangle \nonumber\\
 & +  \int_0^1 (1 - y) \big\langle \, d_S \big( \nabla_S \mathcal H_S^{kdv}(z + y \tilde \Psi_C(z) ) \big) 
[\pi_S {\cal R}_{N}(z; \Psi_C)], \, \pi_S {\cal R}_{N}(z; \Psi_C)  \, \big\rangle \, d y .
\end{align}

In the next lemma we show that $ \nabla {\cal P}^{(2a)}_3(z)$ is in $h^{s+N + 1}_0$ for any $z \in \mathcal V' \cap h^s_0.$ 
\begin{lemma}\label{stime tame h3 nls}
The Hamiltonian ${\cal P}^{(2a)}_3 : \mathcal V'  \to \R$ is real analytic and for any integers $s \geq 0$, $N \ge 1$,
the map $\mathcal V' \cap h^s_0 \to h^{s + N + 1}_0$, $z \mapsto \nabla {\cal P}_3^{(2a)}(z)$ is real analytic. 
Furthermore, for any $z \in \mathcal V' \cap h^s_0 $, and
$ \widehat z \in h^s_0$,
  $$
  \| \nabla {\cal P}^{(2a)}_3(z)\|_{s + N + 1} \lesssim_{s, N}  \, \| z_\bot\|_s \| z_\bot \|_0\,, \quad 
  \| d  \nabla {\cal P}^{(2a)}_3(z) [\widehat z_1]\|_{s + N + 1} \lesssim_{s, N} \, \| z_\bot \|_s \| \widehat z_1\|_0 +  \| z_\bot \|_0 \| \widehat z_1\|_s  
  $$
  and if in addition $ \widehat z_1, \ldots, \widehat z_l \in h^s_0$, $l \geq 2$, 
  $$
  \| d^l \nabla {\cal P}^{(2a)}_3(z) [\widehat z_1, \ldots, \widehat z_l]\|_{s + N + 1} \lesssim_{s, N, l} \,
  \sum_{j = 1}^l \| \widehat z_j\|_s\prod_{i \neq j} \|\widehat z_i \|_0 + 
  \| z_\bot \|_s \prod_{j = 1}^l \|\widehat z_j \|_0\,.
  $$
\end{lemma}
\begin{proof}
We begin by analyzing the first term
$\langle \nabla_S \mathcal H_S^{kdv}(z),  \pi_S  \Psi_{C, 3}(z) \rangle$ on the right hand side of \eqref{definizione h3 nls}.
It is given by the finite sum $\sum_{n \in S} h_n(z)$  where 
$$
h_{n}(z) :=  (\nabla \mathcal H_S^{kdv}(z))_{n} \, (\Psi_{C, 3}(z))_{-n}  
=(\partial_{z_{-n}} \mathcal H_S^{kdv}(z)) \, \langle {\cal R}_{ N, 3} (z; \Psi_{C}), e_n \rangle\,, 
\quad \forall n \in S \,,
$$
and $(e_n)_{n \in S}$ denotes the standard basis of $h^0_S$.
The derivative of $h_n$ in direction $\widehat z \in h^0_0$ then reads
$$
\begin{aligned}
\langle  \nabla h_{n}(z), \widehat z \rangle 
& = \langle  \nabla \partial_{z_{-n}} \mathcal H_S^{kdv}(z), \widehat z  \rangle \, \langle {\cal R}_{ N, 3} (z; \Psi_{C}), e_n \rangle
+ \partial_{z_{-n}} \mathcal H_S^{kdv}(z) \langle d ( {\cal R}_{ N, 3} (z; \Psi_{C})) [\widehat z]\,,\,  e_{n} \rangle\\
 & =  \langle  \nabla \partial_{z_{-n}} \mathcal H_S^{kdv}(z) , \widehat z \rangle \, \langle {\cal R}_{ N, 3} (z; \Psi_{C}), e_n \rangle
+  \partial_{z_{-n}} \mathcal H_S^{kdv}(z) \, \langle   (d {\cal R}_{ N, 3} (z; \Psi_{C}))^t[e_{n}], \widehat z \rangle
  \end{aligned}
$$
 implying that 
 $$
 \nabla h_{n}(z) = \langle {\cal R}_{ N, 3} (z; \Psi_{C}), e_n \rangle  \, \nabla \partial_{z_{-n}} \mathcal H_S^{kdv}(z) 
 + \partial_{z_{-n}} \mathcal H_S^{kdv}(z) \,\,  (d {\cal R}_{ N, 3} (z; \Psi_{C}))^t[e_{n}]\,.  
 $$
By 
Corollary \ref{proposizione espansione taylor correttore simplettico},  for any $s \ge 0$,
$\mathcal V' \cap h^s_0 \to h^{s + N + 1}_0$, $z \mapsto \nabla  h_{n}(z)$ is real analytic and satisfies the estimates 
$\| \nabla h_{n}(z)\|_{s + N + 1} \lesssim_{s, N}  \| z_\bot \|_s \| z_\bot \|_0\,.$
The estimates for the higher order derivatives of $h_n$, $n \in S,$ are obtained by differentiating the expression for $ \nabla h_{n}(z)$ 
and using the estimates of 
Corollary \ref{proposizione espansione taylor correttore simplettico}.  \\
In order to analyze the second term on the right hand side of \eqref{definizione h3 nls}
 it suffices to study the functions $h_{n, k}(z)$, $n, k \in S$, given by
$$
h_{n, k}(z; y) := \partial_{z_{-n}} \partial_{z_{-k}} \mathcal H_S^{kdv}(z + y\tilde \Psi_C(z) ) \,
\langle {\cal R}_{ N} (z; \Psi_{C}), e_{n} \rangle \, \langle {\cal R}_{ N} (z; \Psi_{C}), e_{k} \rangle 
$$
where $0 \le y \le 1$. Clearly, $h_{n, k}(z; y)$ depends continuously on $y$ as do all the derivatives with respect to the variable $z$.
Since $\mathcal H_S^{kdv}( z + y\tilde \Psi_C(z) )$ only depends on $\pi_S(z + y\tilde \Psi_C(z))$ 
one sees that
$$
\begin{aligned}
\langle \nabla h_{n, k}(z; y), \,\widehat z \rangle & = 
\langle \nabla_S \big( \, \partial_{z_{-n}} \partial_ {z_{-k}}  \mathcal H_S^{kdv}(z + y  \tilde \Psi_C(z)) \big), 
\, \pi_S \big(\rm{ Id} +  y \, d \tilde \Psi_C(z) \big) [\widehat z]  \rangle
	\, \langle \mathcal R_N (z; \Psi_C), e_n \rangle \, \langle {\cal R}_{ N} (z; \Psi_{C}) , e_{k} \rangle \nonumber\\
& \quad + \partial_{z_{-n}} \partial_ {z_{-k}} \mathcal H_S^{kdv}(z + y  \tilde \Psi_C(z)) 
	\, \langle  (d {\cal R}_{ N} (z; \Psi_{C}))^t[e_{n}], \widehat z \rangle \, \langle \mathcal R_N (z; \Psi_C), e_k \rangle \nonumber\\
& \quad + \partial_{z_{-n}} \partial_ {z_{-k}}  \mathcal H_S^{kdv}(z + y  \tilde \Psi_C(z)) 
	\langle {\cal R}_{ N} (z; \Psi_{C}), e_{n } \rangle \, \langle  (d {\cal R}_{ N} (z; \Psi_{C}))^t[e_{k}], \widehat z \rangle \nonumber\\
\end{aligned}
$$
implying that 
$$
\begin{aligned}
\nabla h_{n, k}(z; y) & = \big(\rm{ Id} +  s \ d \tilde \Psi_C(z) \big)^t 
\big[\Pi_S \nabla_S \big( \, \partial_{z_{-n}} \partial_ {z_{-k}}  \mathcal H_S^{kdv}(z + y  \tilde \Psi_C(z)) \, \big) \big]   
	\, \langle {\cal R}_{ N} (z; \Psi_{C}), e_n \rangle \, \langle {\cal R}_{ N} (z; \Psi_{C}), e_k \rangle \nonumber\\
& \quad + \partial_{z_{-n}} \partial_ {z_{-k}} \mathcal H_S^{kdv}(z + y  \tilde \Psi_C(z)) 
	\, \langle {\cal R}_{ N} (z; \Psi_{C}), e_{k} \rangle \,  (d {\cal R}_{ N} (z; \Psi_{C}))^t [e_{n}]   \nonumber\\
& \quad + \partial_{z_{-n}} \partial_ {z_{-k}}  \mathcal H_S^{kdv}(z + y  \tilde \Psi_C(z)) \,
	\langle {\cal R}_{ N} (z; \Psi_{C}), e_{n} \rangle \, (d {\cal R}_{ N} (z; \Psi_{C}))^t [e_{k}] \nonumber\\
\end{aligned}
$$
By Corollary \ref{proposizione espansione taylor correttore simplettico}, for any $s \geq 0$, 
the map $\mathcal V'  \cap h^s_0 \to h^{s + N + 1}_0$, $z \mapsto \nabla h_{n, k}(z; y)$ is real analytic and satisfies the estimate 
$\| \nabla h_{n, k}(z; y)\|_{s + N +1} \lesssim_{s, N} \| z_\bot \|_s \| z _\bot\|_0^2$. 
The estimates for the higher order derivatives are obtained by differentiating $\nabla h_{n, k}$ 
and  applying again Corollary \ref{proposizione espansione taylor correttore simplettico}.
\end{proof}

\smallskip
  
  \noindent
  {\em Term ${\cal H}_\Omega(z) :=  \frac{1}{2} \langle \Omega_\bot(I_S) [z_\bot], z_\bot \rangle $:} 
According to \eqref{splitting Omega}, the operator  $\Omega_\bot(I_S)$ reads
  \begin{equation}\label{splitting Omega bot (IS)}
  \Omega_\bot(I_S) = D^2_\bot+ \Omega_\bot^{(0)}(I_S)\,,  
  \end{equation}
  where  
  \begin{equation}\label{definizione Omega bot (0) (IS)}
  D_\bot :=   {\rm diag}_{n \in S^\bot} (2 \pi n) \,, \qquad  
  \Omega_\bot^{(0)}(I_S) := {\rm diag}_{n \in S^\bot} (\Omega_n(I_S) - 4 \pi^2 n^2)\,.
    \end{equation}
 By Lemma \ref{espansione asintotica frequenza thomas}, the following holds.
  \begin{lemma}\label{lemma Omega bot q}
  For any integer $N \ge 2$, the operator $\Omega_\bot^{(0)}(I_S)$ admits the expansion
  $$
 \Omega_\bot^{(0)}(I_S) = {\cal F}_\bot \circ \sum_{k = 2}^N  a_k(I_S; \Omega_\bot^{(0)})  \partial_x^{- k} {\cal F}_\bot^{- 1} + {\cal R}_N(I_S; \Omega_\bot^{(0)})
 $$
 where for any $2 \le k \le N$ and $s \geq 0$, the maps
 $$
 \mathbb R_{\ge 0}^{S_+} \to \R, \, I_S \mapsto a_k(I_S; \Omega_\bot^{(0)})\,, \qquad  
 \mathbb R^{S_+}_{\ge 0} \to {\cal B}(h^s_\bot, h^{s + N +1}_\bot), \, I_S \mapsto  {\cal R}_N( I_S; \Omega_\bot^{(0)})
 $$
 are real analytic. 
 
  \end{lemma}
  To analyze ${\cal H}_\Omega(\Psi_C(z))$, we write the quadratic form $ \big\langle \Omega_\bot(I_S) [z_\bot], \, z_\bot \big\rangle$, $z_\bot \in h^1_0$, as a sum
  $$
   \langle \Omega_\bot(I_S) [z_\bot], \, z_\bot \rangle  =  
   \ \langle D_\bot^2 [z_\bot], \, z_\bot  \rangle 
   +  \ \langle  \Omega_\bot^{(0)}(I_S) [z_\bot] \,, \, z_\bot  \rangle
  $$
  and consider $\langle D_\bot^2 [z_\bot], \, z_\bot  \rangle$ and $ \langle  \Omega_\bot^{(0)}(I_S) [z_\bot] \,, \, z_\bot  \rangle$ separately.
When substituting for $z_\bot$ in $ \big\langle D_\bot^2 [z_\bot], z_\bot \big\rangle$ 
the expression $\pi_\bot \Psi_C(z) = z_\bot + \pi_\bot \tilde \Psi_C(z)$, one gets
\begin{align}
  \big\langle D_\bot^2 [z_\bot + \pi_\bot \tilde \Psi_C(z)], & \, z_\bot  + \pi_\bot \tilde \Psi_C(z) \big\rangle =  
 \big\langle D_\bot^2 [z_\bot] , \, z_\bot  \big\rangle +  
 \big\langle D_\bot^2 [z_\bot] , \,  \pi_\bot \tilde \Psi_C(z) \big\rangle  \nonumber\\
& +   \big\langle D_\bot^2 [ \pi_\bot \tilde \Psi_C(z)], \, z_\bot  \big\rangle +  
 \big\langle D_\bot^2 [ \pi_\bot \tilde \Psi_C(z)],  \, \pi_\bot \tilde \Psi_C(z) \big\rangle
\end{align}
where the map $\pi_\bot$ is defined in \eqref{pi S}. With a view towards the expansion of $H^{kdv} \circ \Psi$, stated in Theorem \ref{modified Birkhoff map},
we treat the difference
$$
 \frac12  \big\langle D_\bot^2 [z_\bot + \pi_\bot \tilde \Psi_C(z)], \, z_\bot + \pi_\bot \tilde \Psi_C(z) \big\rangle  - \,
\frac12   \big\langle D_\bot^2 [z_\bot] , \, z_\bot   \big\rangle
$$
as part of the error term ${\cal P}_3(z)$.  It needs special care since the two terms
$$
\langle D_\bot^2 [z_\bot] ,  \, \pi_\bot \tilde \Psi_C(z) \rangle \,,  \qquad 
 \langle D_\bot^2 [ \pi_\bot \tilde \Psi_C(z)], \, z_\bot  \rangle
$$ 
could have the property that the associated Hamiltonian vector field is unbounded.
We write
  \begin{equation}\label{parte pericolosa Psi C}
  {\cal H}_\Omega(\Psi_C(z)) = {\cal H}_\Omega(z) + {\cal P}_3^{(2b)}(z)\,, 
  \qquad {\cal P}_3^{(2b)}(z) := {\cal H}_\Omega(\Psi_C(z)) - {\cal H}_\Omega(z)\,. 
  \end{equation}
 and note that by the mean value theorem, 
  \begin{align}
  & {\cal P}_3^{(2b)}(z) = \int_0^1 {\cal P}_{\Omega}(\tau, \Psi_{X}^{0, \tau}(z)) \, d \tau   \label{caniggia 0}
  \end{align}
where for $\tau \in [0, 1]$, $z \in \mathcal V$, ${\cal P}_{\Omega}(\tau, z)$ is defined by
  \begin{equation}\label{definition cal H Omega tau}
{\cal P}_{\Omega} (\tau, z) := \langle \nabla {\cal H}_\Omega(z), \, X(\tau, z) \rangle \,.
  \end{equation}
In a first step we analyze ${\cal P}_{\Omega} (\tau, z)$.   One has 
  \begin{align}
   \langle \nabla {\cal H}_\Omega(z), \, X(\tau, z) \rangle &  = 
     \langle \nabla_S {\cal H}_\Omega(z)\,,\, \pi_S X(\tau, z) \rangle + 
   \frac{1}{2} \langle \Omega_\bot(I_S) z_\bot, \,  \pi_\bot {X}(\tau, z)  \rangle\,. \label{termine delicato sergey}
  \end{align}
Since ${\cal H}_\Omega = \frac{1}{2} \langle \Omega_\bot(I_S) [z_\bot], \, z_\bot \rangle$
and $\Omega_\bot(I_S) = D_\bot^2 +  \Omega_\bot^{(0)}(I_S)$ one has
\begin{equation}\label{formula for nabla_S H_ Omega}
 \langle \nabla_S {\cal H}_\Omega(z)\,,\, \pi_S X(\tau, z) \rangle  
 = \frac{1}{2} \sum_{j \in S} \, (\partial_{z_{-j}} \mathcal H_\Omega(z)) \, \langle X(\tau, z), e_j \rangle
 =  \frac{1}{2} \sum_{j \in S} \, \langle \partial_{z_{-j}} \Omega^{(0)}_\bot(I_S) [z_\bot], \, z_\bot \rangle \, \langle X(\tau, z), e_j \rangle \,.
 \end{equation}
 Concerning the term $\frac{1}{2} \langle \Omega_\bot(I_S) z_\bot, \,  \pi_\bot {X}(\tau, z)  \rangle$ in \eqref{termine delicato sergey}, recall that 
  $X(\tau, z) = - {\cal L}_\tau(z)^{- 1} [ J \mathcal E(z)]$ (cf  \eqref{definizione campo vettoriale ausiliario}), 
  ${\cal L}_\tau(z) = \rm{Id} + \tau J \mathcal L(z)$ (cf \eqref{definition L tau})
  and hence 
  \begin{equation}\label{proprieta Neumann F tau w}
  X(\tau, z)  = - J \mathcal E(z) -  \tau J \mathcal L(z) [X(\tau, z)]\,.
  \end{equation}
 Since $\mathcal E(z) = (\mathcal E_S(z), 0)$ and $J^t = - J$, the term  $\big\langle \Omega_\bot(I_S) z_\bot, \,  \pi_\bot {X}(\tau, z)  \big\rangle$ 
 becomes 
  \begin{align}
 -  \big\langle \Omega_\bot(I_S) z_\bot , \, \pi_\bot \tau J \mathcal L(z) X(\tau, z) \big\rangle
   = \tau  \big\langle J_\bot \Omega_\bot(I_S) z_\bot , \, \pi_\bot  \mathcal L(z) X(\tau, z) \big\rangle \,. \label{torres 0}
  \end{align}
  By \eqref{definition L z} the component $\mathcal L_\bot^\bot(z)$ of $\mathcal L(z)$ vanishes, implying that
  $ \pi_\bot \mathcal L(z) X(\tau, z) = \mathcal L_\bot^S(z) \pi_S X(\tau, z)$.  Substituting the latter expression into \eqref{torres 0}  
  and using that $ \mathcal L_\bot^S(z)^t  = - \mathcal L_S^\bot(z)$ since $\mathcal L(z)$ is skew adjoint (cf \eqref{Pull back of LambdaG})
  then leads to
  \begin{equation}\label{cappello 10}
  \big\langle \Omega_\bot(I_S) z_\bot, \,  \pi_\bot {X}(\tau, z)  \big\rangle
  = \tau \big\langle J_\bot \Omega_\bot(I_S) z_\bot, \, \mathcal L_\bot^S(z) \pi_S X(\tau, z) \big\rangle 
   =  - \tau \big\langle \mathcal L_S^\bot(z) J_\bot \Omega_\bot(I_S) z_\bot  \,,\, \pi_S X(\tau, z) \big\rangle.
   \end{equation}
 Furthermore, by \eqref{L SS botS Sbot (1)},
  \begin{align}
  &  \mathcal L_S^\bot(z) [J_\bot \Omega_\bot(I_S) z_\bot ]= 
  \big( \big\langle \partial_x^{- 1} \Psi_1(z_S)[ J_\bot \Omega_\bot(I_S ) z_\bot ] \,,\, \partial_{z_n}  \Psi_1(z_S)[ z_\bot]   \big\rangle \big)_{n \in S} 
\,. \label{cappello 0}
  \end{align}
  Since $ \Omega_\bot(I_S ) = D_\bot^2  +  \Omega_\bot^{(0)}(I_S)$ and $J_\bot D_\bot^2 = \ii D_\bot^3$ we need
  to analyze $\Psi_1(z_S) \ii D_\bot^3$. By Remark \ref{notation coefficients expansion},
 $\Psi_1(z_S) \ii D_\bot^3$ is a bounded linear operator $h^s_\bot \to H^{s-3}_0$ for any $s \ge 0$.
  \begin{lemma}\label{lemma cal T(q)}
  For any integer $N \ge 0$, the operator ${\cal T}(z_S) := \Psi_1(z_S) \ii D_\bot^3 -  {\cal F}_\bot^{- 1} \ii D_\bot^3 {\cal F}_\bot \Psi_1(z_S)$ admits the expansion
  $$
  \begin{aligned}
   {\cal T}(z_S) & = \sum_{k = - 1}^N a_k( z_S; {\cal T}) \partial_x^{- k} {\cal F}_\bot^{- 1} + {\cal R}_N(z_S;  {\cal T})
  \end{aligned}
  $$
  where for any $s \geq 0$, $-1 \le k \le N$, the maps 
  $$
  \mathcal V_S \to H^s, \, z_S \mapsto a_k( z_S; {\cal T})\,, \qquad 
  \mathcal V_S \to {\cal B}(h^s_\bot, H^{s + N + 1}), \, z_S \mapsto {\cal R}_N( z_S;  {\cal T})
  $$ 
  are real analytic. A similar statement holds for the transpose ${\cal T}(z_S)^t$ of ${\cal T}(z_S)$. 
  \end{lemma}
  \begin{proof} 
  First note that the expression obtained from $\Psi_1(z_S) \ii D_\bot^3 -  {\cal F}_\bot^{- 1} \ii D_\bot^3 {\cal F}_\bot \Psi_1(z_S)$
  by replacing  $\Psi_1(z_S)$ by its highest order part ${\cal F}_\bot^{- 1}$ (cf Theorem \ref{lemma asintotica Floquet solutions}), 
  vanishes. Since the order of the commutator of two scalar
  pseudodifferential operators of order one is again of order one, it follows that
  the operator $ {\cal T}(z_S)$ is of order $1$, meaning that 
  the expansion of ${\cal T}(z_S)$ is of the form as stated. 
  Taking into account that $\ii D_\bot^3 =  - {\cal F}_\bot  \partial_x^3 {\cal F}_\bot^{- 1}$, 
  the claimed statements follow from Theorem \ref{lemma asintotica Floquet solutions} (expansion of $\Psi_1(z_S)$)
  and Corollary \ref{lemma asintotiche Phi 1 q} (expansion of $\Psi_1(z_S)^t$).
  \end{proof}
  Taking into account that  $J_\bot \Omega_\bot(I_S) = \ii D_\bot^3 + J_\bot \Omega_\bot^{(0)}(I_S)$,
  the operator $\partial_x^{- 1}  \Psi_1(z_S) J_\bot \Omega_\bot(I_S)$, 
  appearing in formula \eqref{cappello 0} reads
  $$
  \partial_x^{- 1}  \Psi_1(z_S) J_\bot \Omega_\bot(I_S) =
    \partial_x^{-1}\Psi_1(z_S)  \ii D_\bot^3 +  \partial_x^{- 1}  \Psi_1(z_S) J_\bot \Omega^{(0)}_\bot(I_S).
  $$
  By the definition of ${\cal T}(z_S)$ and using that $\partial_x^{- 1} {\cal F}_\bot^{-1} \ii D_\bot^3 = - \partial_x^2 {\cal F}_\bot^{-1}$ , one then gets
  $$
 \partial_x^{- 1}  \Psi_1(z_S) \ii D_\bot^3 = \partial_x^{- 1} {\cal F}_\bot^{-1} \ii D_\bot^3 {\cal F}_\bot \Psi_1(z_S) + \partial_x^{- 1} \mathcal T(z_S) 
 =  - \partial_x^2 \Psi_1(z_S) +  \partial_x^{- 1} \mathcal T(z_S)\,.
  $$
  Altogether we thus have shown that the nth component $ \big\langle \partial_x^{- 1} \Psi_1(z_S)[ J_\bot \Omega_\bot(I_S ) z_\bot ] \,,\, \partial_{z_n}  \Psi_1(z_S)[ z_\bot]   \big\rangle$, $n \in S$,
  of $\mathcal L_S^\bot(z) [J_\bot \Omega_\bot(I_S) z_\bot ]$ 
  is given by
  \begin{align}
  &  \big( \mathcal L_S^\bot(z) [J_\bot \Omega_\bot(I_S) z_\bot ] \big)_n
  = -  \big\langle  \partial_x^2  \Psi_1(z_S)[ z_\bot ] \,,\, \partial_{z_n}  \Psi_1(z_S)[ z_\bot]   \big\rangle + \big\langle {\cal T}_{1, n}(z_S)[ z_\bot ] \,,\,  z_\bot   \big\rangle \label{cappello 20}
  \end{align}
  where for any $z_S \in \mathcal V_S$, the operator ${\cal T}_{1, n}(z_S)$ is given by
  \begin{equation}\label{definizione cal T1 (q)}
  {\cal T}_{1, n}(z_S) := (\partial_{z_n} \Psi_1(z_S))^t \partial_x^{- 1} {\cal T}(z_S) +  (\partial_{z_n} \Psi_1(z_S))^t \partial_x^{- 1} \Psi_1(z_S) J_\bot \Omega^{(0)}_\bot(I_S )\,. 
  \end{equation}
  Since $(\partial_{z_n} \Psi_1(z_S))^t$ is one smoothing (cf Corollary  \ref{lemma asintotiche Phi 1 q}) and 
  $\partial_x^{- 1} {\cal T}(z_S)$ is of order zero (cf Lemma \ref{lemma cal T(q)}), one sees that
  $  {\cal T}_{1, n}(z_S)$ maps $h^0_\bot$ into $h^{1}_\bot$. More precisely, the following result holds. 
  \begin{lemma}\label{proprieta cal T 1 j (q)}
  For any $n \in S$ and $N \in \N$, the operator ${\cal T}_{1, n}(z_S)$, defined by \eqref{definizione cal T1 (q)} for $z_S \in \mathcal V_S$, admits the expansion 
  $$
  {\cal T}_{1, n}(z_S) =  {\cal F}_\bot \circ \sum_{k = 1}^N  a_k(z_S;  {\cal T}_{1, n}) \partial_x^{- k} {\cal F}_\bot^{- 1} + {\cal R}_N( z_S;   {\cal T}_{1, n})
  $$
  where for any $s \geq 0$, $1 \le k \le N$, the maps 
  $$
  \mathcal V_S \to H^s, \, z_S \mapsto a_k( z_S;  {\cal T}_{1, n})\,, \qquad 
  \mathcal V_S \to {\cal B}(h^s_\bot , h^{s + N +1}_\bot), \, z_S \mapsto {\cal R}_N( z_S;  {\cal T}_{1, n})
  $$ 
  are real analytic. 
  \end{lemma}
  \begin{proof}
  The claimed statements follow from Corollary  \ref{lemma asintotiche Phi 1 q},  Lemmata \ref{lemma Omega bot q}, \ref{lemma cal T(q)},
  and Lemma \ref{lemma composizione pseudo}. 
  \end{proof}
We now turn our attention to the term $- \big\langle  \partial_x^2  \Psi_1(z_S)[ z_\bot ] \,,\, \partial_{z_n}  \Psi_1(z_S)[ z_\bot]   \big\rangle$
 in \eqref{cappello 20}. 
By \eqref{forma compatta hamiltoniana d-NLS}
$$
d \nabla H^{kdv}(q) = - \partial_x^{2} + d \nabla H_3^{kdv}(q) = - \partial_x^{2} + 6 q , \quad H_3^{kdv}(q) := \int_0^1 q^3 dx \,.
$$ 
Hence using that $- \partial_x^2 = d \nabla H^{kdv}(q) - 6q$ one obtains for any $n \in S$
\begin{align}
- \big\langle & \partial_x^2  \Psi_1(z_S)[ z_\bot ] \,,\, \partial_{z_n}  \Psi_1(z_S)[ z_\bot]   \big\rangle 
 = \frac{1}{2} \partial_{z_n} \big\langle - \partial_x^2  \Psi_1(z_S)[ z_\bot ] \,,\,  \Psi_1(z_S)[ z_\bot]   \big\rangle \nonumber\\
& =  \frac{1}{2} \partial_{z_n} \big\langle d \nabla H^{kdv}(q) \big[ \Psi_1(z_S)[ z_\bot ]\big] \,,\,  \Psi_1(z_S)[ z_\bot]   \big\rangle 
-  \frac{1}{2} \partial_{z_n} \big\langle 6 q  \Psi_1(z_S)[ z_\bot ] \,,\,  \Psi_1(z_S)[ z_\bot]   \big\rangle\,. \nonumber\\
\end{align}
Since by \eqref{maradona 4}, 
$$  
\Psi_1(z_S)^t d \nabla H^{kdv}(q)  \Psi_1(z_S) =   \Omega_\bot(I_S) + \,{\cal G}(z_S)
$$ 
we conclude that $- \big\langle  \partial_x^2  \Psi_1(z_S)[ z_\bot ] \,,\, \partial_{z_n}  \Psi_1(z_S)[ z_\bot]   \big\rangle$ equals
\begin{align}
\frac12 \big\langle \partial_{z_n}\Omega_\bot(I_S)  [z_\bot]  \,,\,  z_\bot    \big\rangle 
+  \frac12  \big\langle  \partial_{z_n}{\cal G}(z_S)  [z_\bot]  \,,\,  z_\bot    \big\rangle  
-  \frac12 \big\langle  \partial_{z_n} \big( 6 \Psi_1(z_S)^t \, q \, \Psi_1(z_S) \big) [ z_\bot ] \,,\,   z_\bot   \big\rangle\,.
\end{align}
Using again that for any $n \in S$,
$\partial_{z_n} \Omega_\bot(I_S) =  \partial_{z_n} \Omega_\bot^{(0)}(I_S),$
one thus obtains 
$$
- \langle  \partial_x^2  \Psi_1(z_S)[ z_\bot ] \,,\, \partial_{z_n}  \Psi_1(z_S)[ z_\bot]   \rangle =   \langle {\cal T}_{2, n}(z_S)[z_\bot]\,, z_\bot \rangle
$$
where
\begin{equation}\label{definizione cal T2}
 {\cal T}_{2, n}(z_S) := \frac 12 \partial_{z_n} \big( \,  \Omega_\bot^{(0)}(I_S) + {\cal G}(z_S) - 6 \Psi_1(z_S)^t q  \Psi_1(z_S) \,  \big)\,. 
\end{equation}
\begin{lemma}\label{lemma cal T2 j}
For any $n \in S$ and any integer $N \ge 0$, the operator ${\cal T}_{2, n}(z_S) : h^0_\bot \to h^0_\bot$, defined by \eqref{definizione cal T2} for $z_S \in \mathcal V_S$, admits the expansion 
$$
{\cal T}_{2, n}(z_S) = {\cal F}_\bot \circ  \sum_{k = 0}^{N} a_k( z_S; {\cal T}_{2, n}) \partial_x^{- k} {\cal F}_\bot^{- 1} + {\cal R}_N( z_S; {\cal T}_{2, n})
$$
where for any $s \geq 0$, $0 \le k \le N,$  the maps 
$$
\mathcal V_S \to H^s, \, z_S \mapsto a_k( z_S;  {\cal T}_{2, n})\,,  \qquad 
\mathcal V_S \to {\cal B}(h^s_\bot , h^{s + N + 1}_\bot), \, z_S \mapsto  {\cal R}_N ( z_S; {\cal T}_{2, n})
$$ 
are real analytic. 
A similar statement holds for the transpose ${\cal T}_{2, n}(z_S)^t$ of the operator ${\cal T}_{2, n}(z_S)$. 
\end{lemma}
\begin{proof}
The lemma follows by Lemmata \ref{lemma asintotiche Phi 1 q}, \ref{stime cal M(wS)}, \ref{lemma Omega bot q}, and Lemma \ref{lemma composizione pseudo}. 
\end{proof}
By \eqref{definition cal H Omega tau} -- 
\eqref{formula for nabla_S H_ Omega},
 \eqref{cappello 10} -- 
 \eqref{cappello 20}, 
 and \eqref{definizione cal T2} 
 the Hamiltonian ${\cal P}_\Omega(\tau, z)$, defined in \eqref{definition cal H Omega tau}, is given by
  \begin{equation}\label{forma finale cal P Omega tau}
  \begin{aligned}
{\cal P}_\Omega(\tau, z) & = \frac12  \langle \nabla_S {\cal H}_\Omega(z)\,,\, \pi_S X(\tau, z) \rangle 
+ \frac{\tau}{2} \sum_{j \in S}  \big\langle \big(  {\cal T}_{1 , j}(z_S) +   {\cal T}_{2, j}(z_S) \big) [z_\bot]\,,\, z_\bot \big\rangle \cdot
\langle X(\tau, z)\,,\, e_j \rangle \\
& = \frac12 \sum_{j \in S}  \langle {\cal T}_{3, j}(\tau, z_S)[z_\bot]\,,\, z_\bot \rangle  \cdot  \langle X(\tau, z)\,,\, e_j \rangle
\end{aligned}
  \end{equation}
where for any $j \in S$, $z_S \in \mathcal V_S$, and $0 \le \tau \le 1,$ the operator ${\cal T}_{3, j}(\tau, z_S): h^0_\bot \to h^0_\bot$ is defined by
\begin{equation}\label{definition cal T 3 j}
{\cal T}_{3, j}(\tau, z_S) := \partial_{z_{-j}} \Omega_\bot^{(0)}(I_S) +  \tau {\cal T}_{1 , j}(z_S) + \tau {\cal T}_{2, j}(z_S)\,. 
\end{equation}
The Hamiltonian ${\cal P}_\Omega(\tau, z)$ has the following properties. 
\begin{lemma}\label{proprieta hamiltoniana cal P Omega tau psi}
For any $0 \le \tau \le 1$ and any integer $N \ge 0$, the Hamiltonian ${\cal P}_\Omega(\tau, \cdot ) : \mathcal V \to \R$ is real analytic and $\nabla {\cal P}_\Omega (\tau, z)$ admits the expansion
$$
\nabla {\cal P}_\Omega(\tau, z) = \big( 0, \, {\cal F}_\bot \circ \sum_{k = 0}^N a_k( \tau, z; \nabla {\cal P}_\Omega) \partial_x^{- k} {\cal F}_\bot^{- 1} [z_\bot] \,\big) 
+ {\cal R}_N(\tau, z;  \nabla {\cal P}_\Omega)
$$
where for any $s \geq 0$, $0 \le k \le N,$ the maps 
$$
\mathcal V \to H^s, \, z \mapsto a_k(\tau, z; \nabla {\cal P}_\Omega)\,, \qquad 
\mathcal V \cap h^s_0  \to  h^{s + N + 1}_0, \, z \mapsto {\cal R}_N(\tau, z; \nabla {\cal P}_\Omega)
$$ 
are real analytic. 
Furthermore, for any $0 \le \tau \le 1$, $z \in \mathcal V$, $\widehat z \in h^0_0$
$$
 \|  a_k(\tau, z; \nabla {\cal P}_\Omega) \|_s \lesssim_{s} \| z_\bot \|_0^2 \,, \qquad  \| d a_k (\tau, z;  \nabla {\cal P}_\Omega)[\widehat z] \|_s \lesssim_s  \| z_\bot \|_0 \| \widehat z\|_0\,, 
 $$
 and if in addition, $\widehat z_1, \ldots, \widehat z_l \in h^0_0$, $l \ge 2$,
$$
 \| d^l  a_k(\tau, z;  \nabla {\cal P}_\Omega)[\widehat z_1, \ldots, \widehat z_l] \|_s  \lesssim_{s, l} \prod_{j = 1}^l \| \widehat z_j\|_0 \,.
 $$
 Similarly, for any $0 \le \tau \le 1$, $z \in \mathcal V \cap h^s_0$, $\widehat z_1, \widehat z_2 \in h^s_0,$
$$
 \| {\cal R}_N (\tau, z; \nabla {\cal P}_\Omega) \|_{s + N + 1} \lesssim_{s, N} \| z_\bot \|_s \| z_\bot \|_0^2 \,, \quad 
 \| d {\cal R}_N(\tau, z;  \nabla {\cal P}_\Omega)[\widehat z_1] \|_{s + N + 1} \lesssim_{s, N}  \| z_\bot \|_0^2 \| \widehat z_1\|_s
 +  \| z_\bot \|_s \| z_\bot \|_0 \| \widehat z_1\|_0\,, 
 $$
 $$
  \| d^2 {\cal R}_N(\tau, z;  \nabla {\cal P}_\Omega)[\widehat z_1, \widehat z_2] \|_{s + N + 1} \lesssim_{s, N}  
 \| z_\bot \|_0 \big( \| \widehat z_1\|_s \| \widehat z_2\|_0 +  \| \widehat z_1\|_0 \| \widehat z_2\|_s\big) + \|z_\bot  \|_s \| \widehat z_1\|_0 \| \widehat z_2 \|_0\,,  
 $$
 and if in addition $\widehat z_1, \ldots, \widehat z_l \in h^s_0$, $l \ge 3$,
 $$
  \| d^l {\cal R}_N(\tau, z;  \nabla {\cal P}_\Omega) [\widehat z_1, \ldots, \widehat z_l]\|_{s+N+1}  \lesssim_{s, N, l} 
 \sum_{j = 1}^l \| \widehat z_j\|_s \prod_{i \neq j} \| \widehat z_i\|_0 + \| z_\bot \|_s \prod_{j = 1}^l \| \widehat z_j\|_0\,. 
$$
\end{lemma}
\begin{proof}
One has $\nabla_S {\cal P}_\Omega(\tau, z) = (\partial_{z_{-n}} {\cal P}_\Omega(\tau, z))_{n \in S}$ with
$$
 \partial_{z_{-n}} {\cal P}_\Omega (\tau, z) = \frac12 \sum_{j \in S} \,  \langle \partial_{z_{-n}} {\cal T}_{3, j}(\tau, z_S)[z_\bot]\,,\,  z_\bot \rangle  \cdot  \langle X(\tau, z)\,,\, e_j \rangle  
+ \frac12 \sum_{j \in S} \,  \langle {\cal T}_{3, j}(\tau, z_S)[z_\bot]\,,\, z_\bot \rangle  \cdot \langle \partial_{z_{-n}}X(\tau, z)\,,\, e_j \rangle,
$$
whereas $\nabla_\bot {\cal P}_\Omega(\tau, z)$ can be computed to be
$$
\nabla_\bot {\cal P}_\Omega(\tau, z) =  \sum_{j \in S} \,  \langle X(\tau, z)\,,\, e_j \rangle \,   {\cal T}_{3, j}(\tau, z_S)[z_\bot]    
+ \frac12 \sum_{j \in S}  \, \langle {\cal T}_{3, j}(\tau, z_S)[z_\bot]\,,\, z_\bot \rangle  \, (d_\bot X(\tau, z))^t [e_j]\,. 
$$
The claimed statements then follow by Lemmata \ref{lemma campo vettoriale}, \ref{lemma Omega bot q}, \ref{proprieta cal T 1 j (q)}, \ref{lemma cal T2 j}.
\end{proof}
We are now ready to analyze the gradient of the Hamiltonian ${\cal P}_3^{(2b)}(z) := \int_0^1 {\cal P}_{\Omega}(\tau, \Psi_{X}^{0, \tau}(z)) \, d \tau$ (cf  \eqref{caniggia 0}). 
\begin{lemma}\label{proprieta hamiltoniana cal P 3 (2b)}
The Hamiltonian ${\cal P}_3^{(2b)} : \mathcal V' \to \R$ is real analytic and for any integer $N \ge 0$, its gradient $\nabla {\cal P}_3^{(2b)} (z)$ admits the expansion
$$
\nabla {\cal P}_3^{(2b)}(z) = \big( 0, \, {\cal F}_\bot \circ \sum_{k = 0}^N  a_k(z; \nabla {\cal P}_3^{(2b)}) \, \partial_x^{- k} {\cal F}_\bot^{- 1} [z_\bot] \big) \,  + \, {\cal R}_N( z; \nabla {\cal P}_3^{(2b)})
$$
where for any $s \geq 0$, $0 \le k \le N$, the maps 
$$
\mathcal V' \to H^s, \, z \mapsto a_k( z; \nabla {\cal P}_3^{(2b)})\,, \qquad  
\mathcal V' \cap h^s_0 \to  h^{s + N + 1}_0, \, z \mapsto {\cal R}_N( z;  \nabla {\cal P}_3^{(2b)})
$$ 
are real analytic. Furthermore, the following estimates hold: for any $z \in \mathcal V'$, $\widehat z \in h^0_0$,
$$
 \|  a_k( z;  \nabla {\cal P}_3^{(2b)}) \|_s \lesssim_{s} \| z_\bot \|_0^2 \,, \quad \| d a_k( z; \nabla {\cal P}_3^{(2b)})[\widehat z] \|_s \lesssim_s  \| z_\bot \|_0 \| \widehat z\|_0\,, 
 $$
 and if in addition $\widehat z_1, \ldots, \widehat z_l \in h^0_0,$ $l \ge 2$, then
 $\| d^l  a_k(z; \nabla {\cal P}_3^{(2b)})[\widehat z_1, \ldots, \widehat z_l] \|_s  \lesssim_{s, l} \prod_{j = 1}^l \| \widehat z_j\|_0$.
 
 \noindent
 Similarly, for any $z \in \mathcal V' \cap h^s_0$, $\widehat z_1, \widehat z_2 \in h^s_0$,
 $$
 \| {\cal R}_N( z;  \nabla {\cal P}_3^{(2b)}) \|_{s + N + 1} \lesssim_{s, N} \| z_\bot \|_s \| z_\bot \|_0^2 \,, \quad
 \| d {\cal R}_N( z;  \nabla {\cal P}_3^{(2b)})[\widehat z_1] \|_{s + N + 1} \lesssim_{s, N}  \| z_\bot \|_0^2 \| \widehat z_1\|_s  +  \| z_\bot \|_s \| z_\bot \|_0 \| \widehat z_1\|_0\,, 
 $$
 $$
  \| d^2 {\cal R}_N( z; \nabla {\cal P}_3^{(2b)})[\widehat z_1, \widehat z_2] \|_{s + N + 1} \lesssim_{s, N}  
 \| z_\bot \|_0 \big( \| \widehat z_1\|_s \| \widehat z_2\|_0 +  \| \widehat z_1\|_0 \| \widehat z_2\|_s\big) + \|z_\bot  \|_s \| \widehat z_1\|_0 \| \widehat z_2 \|_0\,, 
 $$
 and if in addition $\widehat z_1, \ldots, \widehat z_l \in h^s_0,$ $l \ge 2$, then
 $$
  \| d^l   {\cal R}_N( z; \nabla {\cal P}_3^{(2b)}) [\widehat z_1, \ldots, \widehat z_l]\|_{s + N + 1}  \lesssim_{s, N, l} 
 \sum_{j = 1}^l \| \widehat z_j\|_s \prod_{i \neq j} \| \widehat z_i\|_0 + \| z_\bot \|_s \prod_{j = 1}^l \| \widehat z_j\|_0 \,. 
$$
\end{lemma}
\begin{proof}
By a straightforward computation, one has for any $z \in \mathcal V'$,
$$
\nabla {\cal P}_3^{(2b)}(z) = \int_0^1 (d \Psi_X^{0, \tau}(z))^t \nabla {\cal P}_{\Omega}(\tau, \Psi_{X}^{0, \tau}(z)) \, d \tau\,.
$$
The claimed statements then follow by applying Corollary \ref{expansion of  of differential of Psi} (expansion of $d \Psi_X^{0, \tau}(z)^t$), 
Lemma \ref{proprieta hamiltoniana cal P Omega tau psi} (expansion of $\nabla {\cal P}_\Omega(\tau, z)$), 
Theorem \ref{espansione flusso per correttore} (expansion of $\Psi_{X}^{0, \tau}(z)$), and Lemma \ref{lemma composizione pseudo}.  
\end{proof}
 
 \smallskip
  
  \noindent
  {\em Terms ${\cal P}_2^{(1)}$, ${\cal P}_3^{(1)}$:} Recall that the Hamiltonians ${\cal P}_2^{(1)}$ and ${\cal P}_3^{(1)}$ were introduced in \eqref{perturbazione composizione con Psi L}. We write
  \begin{equation}\label{cal P 12 Psi C}
 {\cal P}_2^{(1)}(\Psi_C(z)) + {\cal P}_3^{(1)}(\Psi_C(z)) = {\cal P}_2^{(1)}(z) + {\cal P}_3^{(2c)}(z)\,, \quad 
  {\cal P}_3^{(2c)}(z):=   {\cal P}_2^{(1)}(\Psi_C(z)) - {\cal P}_2^{(1)}(z) + {\cal P}_3^{(1)}(\Psi_C(z))
  \end{equation}
  where by the mean value theorem
  $$
 {\cal P}_2^{(1)}(\Psi_C(z)) - {\cal P}_2^{(1)}(z) = \int_0^1\big\langle \nabla {\cal P}_2^{(1)}\big( z + y (\Psi_C(z) - z)   \big)\,,\, \Psi_C(z) - z \big\rangle d y \,. 
  $$
  The Hamiltonian ${\cal P}_3^{(2c)}(\Psi_C(z))$ has the following properties.
  \begin{lemma}\label{composizione cal P 12 Psi C}
  The Hamiltonian ${\cal P}_3^{(2c)} : \mathcal V' \cap h^1_0 \to \R$ is real analytic and for any integer $N \ge 0$ its gradient $\nabla {\cal P}_3^{(2c)} (z)$ admits the expansion
   $$
  \nabla {\cal P}_3^{(2c)}(z) = \big( 0, \,   {\cal F}_\bot \circ  \sum_{k = 0}^N T_{a_k(z; \nabla {\cal P}_3^{(2c)})}  \partial_x^{- k} {\cal F}_\bot^{- 1}[z_\bot] \,\big) + {\cal R}_N (z;  \nabla {\cal P}_3^{(2c)})
  $$ 
  with the property that there exists an integer $\sigma_N \ge N$ (loss of regularity) such that for any $s \geq 0$, $0 \le k \le N$, the maps 
  $$
  \mathcal V' \cap h^{s + \sigma_N}  \to H^s, \, z \mapsto a_k(z; \nabla {\cal P}_3^{(2c)})\,,
  \qquad 
  \mathcal V'  \cap h_0^{s \lor \sigma_N} \to  h_0^{s + N + 1}, \, z \mapsto {\cal R}_N(z; \nabla {\cal P}_3^{(2c)})
  $$
   are real analytic. Furthermore, for any $s \ge 0,$ $z \in \mathcal V' \cap h^{s + \sigma_N}_0$ with  $\| z \|_{\sigma_N} \leq 1$, $\widehat z_1, \ldots, \widehat z_l \in h^{s + \sigma_N}_0$, $l \ge 1$,
  $$
  \begin{aligned}
   & \|   a_k( z; \nabla {\cal P}_3^{(2c)})  \|_s \lesssim_{s, N} \| z_\bot \|_{s + \sigma_N}\,, \\
   & \| d^l a_k( z; \nabla {\cal P}_3^{(2c)}) [\widehat z_1, \ldots, \widehat z_l] \|_s \lesssim_{s, N, l} 
  \sum_{j = 1}^l \| \widehat z_j\|_{s + \sigma_N} \prod_{i \neq j} \| \widehat z_j\|_{\sigma_N} + \| z_\bot \|_{s + \sigma_N} \prod_{j = 1}^l \| \widehat z_j\|_{\sigma_N}\,.
    \end{aligned}
  $$
  Similarly, for any $s \ge 0,$ $z \in \mathcal V' \cap h^{s \lor \sigma_N}_0$ with  $\| z \|_{\sigma_N} \leq 1$, $\widehat z\in h^{s \lor \sigma_N}_0$, 
  $$
  \begin{aligned}
   &\| {\cal R}_N(z ; \nabla {\cal P}_3^{(2c)})\|_{s + N + 1} \lesssim_{s, N} \| z_\bot \|_{s \lor \sigma_N} \| z_\bot \|_{\sigma_N} \,, \\
   &\| d {\cal R}_N( z; \nabla {\cal P}_3^{(2c)}) [\widehat z]\|_{s + N + 1} \lesssim_{s, N} \| z_\bot \|_{\sigma_N} \| \widehat z\|_{s \lor \sigma_N} + 
   \| z_\bot \|_{s \lor \sigma_N} \| \widehat z\|_{\sigma_N}\,, 
   \end{aligned}
   $$
   and if in addition $\widehat z_1, \ldots, \widehat z_l \in h^{s \lor \sigma_N}_0$, $ l \geq 2$,
   $$
   \| d^l {\cal R}_N( z; \nabla {\cal P}_3^{(2c)}) [\widehat z_1, \ldots, \widehat z_l] \|_{s + N + 1} \lesssim_{s, N, l} 
  \sum_{j = 1}^l \|\widehat z_j \|_{s \lor \sigma_N} \prod_{i \neq j} \| \widehat z_i\|_{\sigma_N} 
  + \| z_\bot \|_{s \lor \sigma_N} \prod_{j = 1}^l \| \widehat z_j\|_{\sigma_N}\,. 
  $$
   \end{lemma}
  \begin{proof}
   The lemma follows by differentiating  the Hamiltonian ${\cal P}_3^{(2c)}$, defined  in \eqref{cal P 12 Psi C} and then  applying Corollary \ref{proposizione espansione taylor correttore simplettico}, 
   Lemmata \ref{stime cal M(wS)}, \ref{lemma stima cal P2 P3} and using Lemmata \ref{primo lemma paraprodotti}, \ref{lemma composizione pseudo}.   
  \end{proof}
  
  \bigskip
  


\smallskip

By \eqref{forma semifinale H nls circ Psi}, \eqref{h nls circ PsiC}, \eqref{parte pericolosa Psi C}, \eqref{cal P 12 Psi C} it follows that 
for $z= (z_S, z_\bot) \in \mathcal V'$, the Hamiltonian 
${\cal H}^{(2)}(z)$ is given by 
\begin{equation}\label{cal H2 nls}
{\cal H}^{(2)} (z) = {\cal H}^{kdv}(I_S ) + \frac12 \big\langle \Omega_\bot(I_S)[z_\bot], z_\bot \big\rangle
+ {\cal P}_2^{(2)}(z) + {\cal P}_3^{(2)}(z)
\end{equation}
where 
\begin{equation} \label{definizione cap P2 (3)}
{\cal P}_2^{(2)} (z)  :=  \langle \nabla_S \mathcal H_S^{kdv}(z), \,  \pi_S {\cal R}_{N, 2}(z; \Psi_C) \rangle + {\cal P}_2^{(1)}(z)\,, \qquad 
{\cal P}_3^{(2)} (z)  := {\cal P}^{(2a)}_3(z) +{\cal P}_3^{(2b)}(z) + {\cal P}_3^{(2c)}(z).
\end{equation}
and where we recall that by \eqref{Taylor expansion R_N} and \eqref{perturbazione composizione con Psi L}
\begin{align}\label{definizione cap P2 (2)} 
\mathcal R_{N, 2} (z; \Psi_C) =  \frac{1}{2} d^2_\bot {\cal R}_{N}((z_S, 0) ; \Psi_C) [z_\bot, z_\bot]\,, \qquad
\quad {\cal P}_2^{(1)}(z) =  \frac12 \langle {\cal G}(z_S)[z_\bot], \, z_\bot \rangle\,. \quad
\end{align}
Note that ${\cal P}_2^{(2)} $ is quadratic with respect to $z_\bot$, whereas ${\cal P}^{(2)}_3$
is a remainder term of order three in  $z_\bot$. 
Being quadratic with respect to $z_\bot$, ${\cal P}^{(2)}_2$ can be written as 
\begin{equation}\label{bla bla cal P2}
{\cal P}_2^{(2)} (z) = \frac12 \langle d_\bot \nabla_\bot {\cal P}_2^{(2)}(z_S, 0)   [z_\bot], \, z_\bot \rangle\,.
\end{equation}
 The following vanishing lemma is due to Kuksin \cite{K}. Since our setup is different from the one in \cite{K},
 we include its proof for the convenience of the reader.
      \begin{lemma}\label{cancellazione finale termini quadratici}
  The Hamiltonian ${\cal P}^{(2)}_2$ vanishes on $ \mathcal V'$. 
    \end{lemma}
  \begin{proof}
  In view of \eqref{bla bla cal P2}, it suffices to prove that for any $z_S \in \mathcal V_S'$, the operator $d_\bot \nabla_\bot {\cal P}_2^{(2)}(z_S, 0)$ vanishes.
  We establish that $d_\bot \nabla_\bot {\cal P}_2^{(2)}(z_S, 0) = 0$ by studying the linearization of $\partial_t w = J \nabla \mathcal H^{(2)}(w)$ along an arbitrary
  solution $w(t)$ of the form $w(t) = (w_S(t), 0)$. 
  First we need to make some preliminary considerations.
Let $t \mapsto q(t) \in M_S$ be a solution of the KdV equation $\partial_t q = \partial_x \nabla H^{kdv}(q)$
and denote by $t \mapsto z(t) := (z_S(t), 0)$ the corresponding solution in Birkhoff coordinates, defined by $q(t) = \Psi^{kdv}(z(t))$.
It satisfies $\partial_t z(t) = J \Omega(I_S) [z(t)]$.
Furthermore, let $\widehat q(t)$ be the solution of the equation, obtained by  linearizing the KdV equation along $q(t)$, 
$$
\partial_t \widehat q(t) = \partial_x d \nabla H^{kdv}(q(t)) [\widehat q(t)] \,,
$$ 
with initial data $\widehat q^0 := d\Psi^{kdv}(z_S(0), 0)[ 0, \widehat z_{\bot}^0]$ and  $\widehat z_{\bot}^0 \in h^3_\bot$. 
Similarly, denote by $\widehat z(t)$ the solution of the equation, obtained by linearizing $\partial_t z = J \Omega(I_S) [z]$
along the solution $z(t)$ with initial data $\widehat z^0 = (0, \widehat z_{\bot}^0)$,
$\partial_t \widehat z (t) =  J d \nabla  \mathcal H^{kdv}(z(t)) [\widehat z (t)]$.
Since  $\partial_t z(t) = J \Omega(I_S) [z(t)]$ one concludes that
\begin{equation}\label{equation widehat z bot}
 \widehat z(t) = (0, \widehat z_\bot(t))\,, \qquad \partial_t \widehat z_\bot (t) =   J_\bot \Omega_\bot(I_S) [\widehat z_\bot (t)] \,.
\end{equation}
and since $\Psi^{kdv}$ is symplectic and $\mathcal H^{kdv} = H^{kdv} \circ \Psi^{kdv}$, one has  
$
\widehat q(t) = d\Psi^{kdv}(z_S(t), 0)[\widehat z(t)]\,.
$
Recall that for any $z_S \in \mathcal V'_S$,  $\Psi_L(z_S, 0) = \Psi^{kdv} (z_S, 0)$ (cf definition \eqref{definition Psi_L} of $\Psi_L$) 
and $\Psi_C(z_S, 0) = (z_S, 0)$ 
(cf Corollary \ref{proposizione espansione taylor correttore simplettico}),
implying that $ \Psi(z_S, 0) = \Psi^{kdv}(z_S, 0)$ and hence $q(t) = \Psi (z_S(t), 0)$ for any $t$.
Since  $\Psi : \mathcal V' \to H^0_0$ is symplectic and $\mathcal H^{(2)} = H^{kdv}\circ \Psi$, 
one sees that $z(t) = (z_S(t), 0)$ is also a solution of the equation 
$\partial_t w = J \nabla {\cal H}^{(2)}(w)$.
 With these preliminary considerations made, we are now ready to prove that $d_\bot \nabla_\bot {\cal P}_2^{(2)}(z_S, 0)$ vanishes.
 To this end consider the solution $\widehat w(t)$ of the equation obtained by linearizing  $\partial_t w = J \nabla {\cal H}^{(2)}(w)$ along the solution $z(t) = (z_S(t), 0)$
 with initial data $\widehat w^0 = (0, \widehat z_{\bot}^0)$. Again using that the map $\Psi$ is symplectic and $\mathcal H^{(2)} = H^{kdv}\circ \Psi$,
 it follows that $d\Psi(z(t))[\widehat w(t)]$ solves the linearized KdV equation. Since $d\Psi(z(0)) = d\Psi^{kdv}(z(0))$ and $\widehat w^0 = \widehat z^0$,
 one then concludes from the uniqueness of the initial value problem that  $d\Psi(z(t))[\widehat w(t)] = d\Psi^{kdv}(z(t))[\widehat z(t)]$ and hence 
 $\widehat w(t) = \widehat z(t)$ for any $t$. It means that $ \widehat z(t)$ satisfies also the linear equation
 $$
 \partial_t \widehat z (t) =  J d \nabla  \mathcal H^{(2)}(z(t)) [\widehat z (t)]\,.
 $$
 In view of the expansion \eqref{cal H2 nls} of $\mathcal H^{(2)}$ one then obtains
$$
\partial_t \widehat z_\bot(t) = J_\bot \Omega_\bot (I_S)[\widehat z_\bot(t)] + J_\bot d_\bot \nabla_\bot {\cal P}_2(z_S(t), 0)[\widehat z_\bot(t)]\,. 
$$
Comparing the latter identity with \eqref{equation widehat z bot} one concludes that in particular, 
$d_\bot \nabla_\bot {\cal P}_2 (z_S(0), 0) = 0$. Since the initial data  $z_S(0) \in \mathcal V_S'$ can be chosen arbitrarily,
we thus have $d_\bot \nabla_\bot {\cal P}_2 (z_S, 0) = 0$ for any $z_S \in \mathcal V_S'$ as claimed. 
 \end{proof}
In summary, we have proved the following results on the Hamiltonian $\mathcal {\cal H}^{(2)}= H^{kdv} \circ \Psi$.  
\begin{theorem} \label{stime finali grado 3 perturbazione}
The Hamiltonian $\mathcal {\cal H}^{(2)}: \mathcal V' \cap h^1_0 \to \R $ has an expansion of the form
\begin{equation}\label{forma finalissima hamiltoniana trasformata}
{\cal H}^{(2)}(z) = H^{kdv}(q) + \frac12 \langle {\Omega}_\bot(I_S)[z_\bot], \, z_\bot \rangle + {\cal P}_3^{(2)}(z)
\end{equation}
where $\Omega_\bot(I_S)$ is given by \eqref{splitting Omega} and the remainder term  ${\cal P}_3^{(2)}$, defined by  \eqref{definizione cap P2 (3)},
 satisfies the following:
 ${\cal P}_3^{(2)} : \mathcal V' \cap h^1_0  \to \R$ is real analytic and for any integer $N \ge 1$,  its gradient $\nabla {\cal P}_3^{(2)}(z)$ admits the asymptotic expansion
  $$
  \nabla {\cal P}_3^{(2)}(z) = 
  \big( 0, \,  {\cal F}_\bot \circ \sum_{k = 0}^N  T_{a_k(z; \nabla {\cal P}_3^{(2)})}  \partial_x^{- k} {\cal F}_\bot^{- 1}[z_\bot] \,\big)
   + {\cal R}_N( z; \nabla {\cal P}_3^{(2)})
  $$ 
  with the property that there there exists an integer  $\sigma_N \ge N$ (loss of regularity) so that for any $s \ge 0$, $0 \le k \le N$, the maps 
  $$
  \mathcal V'  \cap h^{s + \sigma_N}_0 \to H^s, \, z \mapsto a_k( z; \nabla{\cal P}_3^{(2)})\,, \qquad
  \mathcal V' \cap h^{s \lor \sigma_N}_0 \to h^{s + N + 1}_0, \, z \mapsto {\cal R}_N( z; \nabla {\cal P}_3^{(2)})
  $$
   are real analytic and satisfy the following estimates: for any $s \ge 0,$ $z \in \mathcal V' \cap h^{s + \sigma_N}_0$ with  $\| z \|_{\sigma_N} \leq 1$,
   $\widehat z_1, \ldots, \widehat z_l \in h^{s + \sigma_N}_0$, $l \ge 1$,
  $$
  \begin{aligned}
  & \|   a_k(z; \nabla {\cal P}_3^{(2)})  \|_s \lesssim_{s, N} \| z_\bot \|_{s + \sigma_N}\,, \\
  & \| d^l a_k( z; \nabla {\cal P}_3^{(2)}) [\widehat z_1, \ldots, \widehat z_l] \|_s \lesssim_{s, N, l} 
  \sum_{j = 1}^l \| \widehat z_j\|_{s + \sigma_N} \prod_{i \neq j} \| \widehat z_j\|_{\sigma_N} + \| z_\bot \|_{s + \sigma_N} \prod_{j = 1}^l \| \widehat z_j\|_{\sigma_N}\,.
  \end{aligned}
  $$
  Similarly, for any $s \ge 1$, $z \in \mathcal V' \cap h^{s \lor \sigma_N}_0$ with $\| z_\bot \|_{\sigma_N} \leq 1$,  $\widehat z\in h^{s \lor \sigma_N}_0$,   
  $$
    \begin{aligned}
  & \| {\cal R}_N( z; \nabla {\cal P}_3^{(2)})\|_{s + N + 1} \lesssim_{s, N} \| z_\bot \|_{s \lor \sigma_N} \| z_\bot \|_{\sigma_N} \,, \\
  &\| d {\cal R}_N( z; \nabla {\cal P}_3^{(2)}) [\widehat z]\|_{s + N + 1} \lesssim_{s, N} \| z_\bot \|_{\sigma_N} \| \widehat z\|_{s \lor \sigma_N} + \| z_\bot \|_{s \lor \sigma_N} \| \widehat z\|_{\sigma_N}\,, 
  \end{aligned}
  $$
  and if in addition  $\widehat z_1, \ldots, \widehat z_l \in h^{s \lor \sigma_N}_0$, $ l \geq 2$,
 $$
 \| d^l {\cal R}_N( z; \nabla {\cal P}_3^{(2)}) [\widehat z_1, \ldots, \widehat z_l] \|_{s + N + 1} \lesssim_{s, N, l} 
  \sum_{j = 1}^l \|\widehat z_j \|_{s \lor \sigma_N} \prod_{i \neq j} \| \widehat z_i\|_{\sigma_N} + \| z_\bot \|_{s \lor \sigma_N} \prod_{j = 1}^l \| \widehat z_j\|_{\sigma_N}\,. 
  $$
\end{theorem} 
\begin{proof}
The identity \eqref{forma finalissima hamiltoniana trasformata} folllows from formula \eqref{cal H2 nls} and Lemma \ref{cancellazione finale termini quadratici}.
The claimed asymptotic expansion of the gradient of  ${\cal P}_3^{(2)}$ and its properties follow from
Lemmata \ref{stime tame h3 nls}, \ref{proprieta hamiltoniana cal P 3 (2b)}, \ref{composizione cal P 12 Psi C} and Lemma \ref{primo lemma paraprodotti}. 
\end{proof}


\section{Summary of the proofs of Theorem \ref{modified Birkhoff map} and Theorem \ref{modified Birkhoff map 2}}\label{synopsis of proof}

In this section we summarize the proofs of Theorem \ref{modified Birkhoff map}, of its addendum, and of Theorem \ref{modified Birkhoff map 2}. 
First recall that in view of the envisioned applications, these theorems are formulated in terms of action angle coordinates on the submanifold $M_S^o$ of proper $S-$gap potentials.
 Denote by $\Xi $ the map relating action angle variables and complex Birkhoff coordinates,
$$
\Xi: \T^{S_+} \times \R_{>0}^{S_+} \times h^0_\bot \to h_S^0 \times h^0_\bot, \,  (\theta_S, I_S, z_\bot) \mapsto (z_S(\theta_S, I_S), z_\bot)\,, \quad z_{\pm n} = \sqrt{2\pi n I_n} e^{\mp\ii \theta_n}\,, \, n\in S_+\,.
$$
Clearly, $\Xi$ is symplectic and, for any $s \geq 0$, the map  $\Xi : \T^{S_+} \times \R_{>0}^{S_+} \times h^s_\bot \to h_S^0 \times h^s_\bot, $ is real analytic.
Furthermore, in view of the definition  \eqref{definition reversible structure for actions angles}, the map $\Xi$ preserves the reversible structure.
Hence the claimed results for the map $\Psi_L \circ \Psi_C \circ \Xi$ follow from the corresponding ones for the map $\Psi_L \circ \Psi_C$.
In what follows we summarize the proofs of the results for $\Psi_L \circ \Psi_C$ corresponding to the ones claimed for  $\Psi_L \circ \Psi_C \circ \Xi$.

\smallskip

\noindent 
{\em Proof of Theorem \ref{modified Birkhoff map}.}  By a slight abuse of notation, the map $\Psi$ of Theorem \ref{modified Birkhoff map} is 
defined to be the composition $\Psi_L \circ \Psi_C$.
By \eqref{definition Psi_X {tau_0, tau}}, it is defined on the neighborhood $\mathcal V' = \mathcal V'_S \times \mathcal V'_\bot$
where $\mathcal V'_S$ is a bounded neighborhood of any given compact subset $\mathcal K \subset h_S^0$ and $\mathcal V'_\bot$ is a ball in $h^0_\bot$ 
of radius smaller than $1$, centered at $0$. The expansion of $\Psi$, corresponding to the one of {\bf (AE1)}, 
follows from the expansion for the map $\Psi_L$, provided by Theorem \ref{lemma asintotica Floquet solutions}
and the one for the map $\Psi_C$, provided by Theorem \ref{espansione flusso per correttore}.

\noindent
The expansion of the transpose $d\Psi(z)^t$ of the derivative $d\Psi(z)$, corresponding to the one of {\bf (AE2)}, 
follows from the fact that $\Psi: \mathcal V'  \to L^2_0$ is symplectic, meaning that 
for any $z \in \mathcal V' $, the operator $d\Psi(z)^t: H^1_0 \to  h^1_0$ satisfies $d\Psi(z)^t = J^{-1} (d\Psi(z))^{-1} \partial_x$. 
The expansion of $\Psi(z)$ in {\bf (AE1)} then leads to an expansion 
of $d\Psi(z)$ and in turn of $(d\Psi(z))^{-1}$ and hence of $d\Psi(z)^t$. In addition, the identity $d\Psi(z)^t = J^{-1} (d\Psi(z))^{-1} \partial_x$
implies that the coefficient $a_1(z; d \Psi^t)$ in the expansion of $d \Psi^t$ satisfies $a_1(z; d \Psi^t) = - a_1(z; \Psi)$.


\noindent
The expansion of the Hamiltonian $\mathcal H^{(2)}$ and of the remainder term ${\cal P}_3^{(2)}$, 
corresponding to the one in {\bf (AE3)}, are provided in Theorem \ref{stime finali grado 3 perturbazione}.
\hfill $\square$

\smallskip

\noindent
{\em Proof of Addendum to Theorem \ref{modified Birkhoff map}.} Clearly, the Fourier transform $\mathcal F$ and its inverse preserve 
the reversible structure and by Proposition \ref{proposizione 1 reversibilita}, so do the Birkhoff map $\Phi^{kdv}$  and its inverse $\Psi^{kdv}$.
Furthermore,  by the Addendum to Theorem \ref{lemma asintotica Floquet solutions}, and the Addendum to Theorem \ref{espansione flusso per correttore}
also the maps $\Psi_L$  and $\Psi_C$ and hence $\Psi_L \circ \Psi_C$ preserve the reversible structure, as do the coefficients and the remainder terms
in their expansions as well as the transpose of their derivatives. 

\noindent
Clearly, the KdV Hamiltonian $H^{kdv}$ is reversible and therefore so is $\mathcal H^{(2)} = H^{kdv} \circ \Psi$.
By \eqref{forma finalissima hamiltoniana trasformata} one then concludes that also the remainder  $ {\cal P}_3^{(2)}$ is reversible.
\hfill $\square$

\smallskip

\noindent
{\em Proof of Theorem \ref{modified Birkhoff map 2}.} The estimates of the coefficients and the remainder in the expansion of $\Psi = \Psi_L \circ \Psi_C$, corresponding to the ones of {\bf (Est1)}, 
follow from the estimates of the coefficients and the remainder in the expansion of the map $\Psi_L$, provided by Theorem \ref{lemma asintotica Floquet solutions},
and the ones of the coefficients and the remainder in the expansion of the map $\Psi_C$, provided by Theorem \ref{espansione flusso per correttore}.

\noindent
The estimates of the coefficients and the remainder in the expansion of $d\Psi(z)^t$, corresponding to the one of {\bf (Est2)},
follow from the fact that $\Psi: \mathcal V'  \to L^2_0$ is symplectic, meaning that for any $z \in \mathcal V' $, 
$d\Psi(z)^t: H^1_0 \to h^1_0$ satisfies  $d\Psi(z)^t = J^{-1} (d\Psi(z))^{-1} \partial_x$ 
and the estimates  {\bf (Est1)} of the coefficients and the remainder in the expansion of $\Psi(z)$ which lead to corresponding estimates
of the coefficients and the remainder in the expansion of $d\Psi(z)$ and in turn of $(d \Psi(z))^{-1}.$

\noindent
The estimates of the remainder term $ {\cal P}_3^{(2)}$ in the expansion of the Hamiltonian $\mathcal H^{(2)} = H^{kdv} \circ \Psi$, corresponding to {\bf (Est3)}, 
are provided by Theorem \ref{stime finali grado 3 perturbazione}.
\hfill $\square$

\appendix

\section{Birkhoff map}\label{Birkhoff map}
In this appendix we review the Birkhoff map and properties of it, relevant for our purposes. We refer to \cite{KP} and \cite{KMT}, \cite{KP1}, \cite{KST2} for more details
in these matters.

\begin{theorem}[\cite{KP}, \cite{KST2}] \label{Theorem Birkhoff coordinates}
 There exists an open neighborhood $W$ of $L^2_0$ in $L^2_{0, \C}$ and a real analytic map
$$
\Phi^{kdv} : W \to h^0_{0,\C}\,, \qquad q \mapsto z(q) = (z_n(q))_{n \ne 0}
$$ 
with $\Phi^{kdv}(0) = 0$ so that the following holds: 
\begin{description}
\item[{\rm (B0)}] For any $n \in \N$, the complex Birkhoff coordinates $z_n(q), z_{-n}(q)$ are related to the Birkhoff coordinates $x_n(q), y_n(q)$ as introduced in \cite{KP} 
by the formulas \eqref{definition z_n}\,.
\item[{\rm (B1)}] For any $s \in \Z_{\ge 0}$, the restriction of $\Phi^{kdv}$ to  $H^s_0$ gives rise to a map $\Phi^{kdv}: H^s_0 \to h^s_0$ which is a bi-analytic diffeomorphism. 
\item[{\rm (B2)}] The map $\Phi ^{kdv}$ is canonical, meaning that on $W$, $\{ z_n, z_{- n} \} = \ii 2 \pi n$ for any $n \in \N$ and the brackets between all other coordinate functions vanish.  
\item[{\rm (B3)}] The Hamiltonian $H^{kdv} \circ (\Phi^{kdv})^{- 1}$, defined on $h_0^1$, only depends on the actions $(I_n)_{ n \in \N}$ $($cf \eqref{definition actions}$)$. 
More precisely, it can be viewed as a real analytic map ${\cal H}^{kdv}$ on a complex neighborhood of the positive quadrant $\ell^{1,3}_+$ in $\ell^{1,3}(\N, \C)$ $($cf \eqref{weighted ell_1}$)$.
\item[{\rm (B4)}] The differential $d_0\Phi^{kdv}$ of $\Phi^{kdv}$ at $0$ is the
Fourier transform $\mathcal F$ $($cf \eqref{Fourier transform}$)$.
\item[{\rm (B5)}] The nonlinear part of the Birkhoff map, $A^{kdv} := \Phi^{kdv} - \mathcal F$,  and the one of its inverse, $B^{kdv} := (\Phi^{kdv})^{-1} - \mathcal F^{- 1},$  
are one smoothing. More precisely, for any $s \in \N$,
$A^{kdv} : H_0^s \to h^{s + 1}_{0}$ and  $B^{kdv} : h^s_0 \to H^{s + 1}_{0}$ are real analytic.
\end{description}
\noindent
The inverse of $\Phi^{kdv}$ is denoted by $\Psi^{kdv}$.
\end{theorem}

To continue we first need to introduce some more notations and review properties of the Schr\"odinger operator $- \partial_x^2 + q.$
For any $s \in \Z_{\ge 0}$, denote by $H^s_\C[0, 1] \equiv H^s([0, 1], \C)$ the Sobolev space of functions
$f: [0, 1] \to \C$ with the property that for any $0 \le j \le s,$ the distributional derivative $\partial_x^j f$ is in $L^2_\C [0, 1] \equiv L^2([0, 1], \C)$.
Similarly, $H^s_{0, \C} \equiv H^s_0(\T, \C)$ denotes the Sobolev space of functions $q: \T \to \C$ in $H^s(\T, \C)$ with $\int_0^1 q(x) dx = 0.$
For any $q \in L^2_{0, \C} \equiv H^0_{0, \C}$ and $\lambda \in \C$, we denote by $y_j(x, \lambda) \equiv y_j(x, \lambda, q)$, $j = 1,2$, the fundamental solutions 
of $-y'' + qy = \lambda y$. These are the solutions satisfying the initial conditions $y_1(0, \lambda) = 1$, $y_1'(0, \lambda) = 0$ and 
$y_2(0, \lambda) = 0$, $y_2'(0, \lambda) = 1$. It is well known that for any $s \in \Z_{\ge 0}$ and $1 \le j \le 2 $, the map
$$
\C \times H^s_{0, \C} \to H^{s+2}_\C[0,1]\,, \, (\lambda, q) \mapsto y_j (\cdot, \lambda, q)
$$
is analytic (cf \cite{PT}).
 For $q$ in $L^2_{0,\mathbb{C}}$, the Schr\"odinger operator $- \partial_x^2+q$, considered on the interval $[0,2]$ with periodic boundary conditions, has a discrete spectrum.
It consists of a sequence of complex numbers 
bounded from below. We list them lexicographically and with algebraic multiplicities, i.e.,
$\lambda_0^+\preceq \lambda_1^-\preceq \lambda_1^+ \preceq\lambda_2^-\preceq \dots\; $ where $\lambda_n^\pm \equiv \lambda_n^\pm (q)$  (cf \cite{KP}).
They  are referred to as periodic eigenvalues of $q$ and satisfy the asymptotic estimates $\lambda_n^+,\lambda_n^-= n^2\pi^2 +\ell^2_n,$
valid uniformly on bounded subsets of $L^2_{0,\mathbb{C}}$ (cf \cite{KP}). 
For real valued $q$, the periodic eigenvalues are real and come in isolated pairs, meaning that
$\lambda_0^+< \lambda_1^-\leq \lambda_1^+ <\lambda_2^-\leq \lambda_2^+ <\dots\; .$
We remark that for any given finite subset $S_+ \subseteq \N$, the manifold $M_S$ of $S$-gap potentials defined in the introduction, coincides with the set 
$\big\{ q \in L^2_0 : \lambda_n^-(q)  =  \lambda_n^+(q)  \,\,  \forall \,  n \in S_+^\bot \big\}$.

\noindent
By shrinking the neigbourhood $ W$ of 
Theorem \ref{Theorem Birkhoff coordinates}, if needed, one can ensure that for any $q \in W$, the closed intervals 
$$
G_n= \{(1-t)\lambda_n^-+ t \lambda_n^+|\;0\leq t\leq 1\},\quad n\geq 1\,, \qquad    G_0= \{ t+ \lambda_{0}^+|\;-\infty < t\leq 0\}
$$
are disjoint from each other. By a slight abuse of terminology, we refer to the closed interval $G_n$, $n\geq 1$, as the $n$'th gap and to $\gamma_n \equiv \gamma_n(q)$
as the $n$'th gap length, $\gamma_n:= \lambda_{n}^+ -\lambda_{n}^-$ and denote by $\tau_n \equiv \tau_n(q)$ the middle point of $G_n$, $\tau_n=(\lambda_{n}^+ +\lambda_{n}^-)/2$. 
Due to the asymptotic behaviour of the periodic eigenvalues, $(G_n)_{n \ge 1}$ admit mutually disjoint neighbourhoods 
$U_n\subseteq \mathbb{C}$, $n \ge 1$, with $G_n \subseteq U_n$, referred to as \textit{isolating neighbourhoods}.
They can be chosen locally independently of $q$ (cf \cite{KP}).
Denote by $F(\lambda) \equiv F(\lambda, q)$ the Floquet matrix 
\begin{equation}\label{Floquet matrix}
F(\lambda) = \begin{pmatrix}
m_1(\lambda) & m_2(\lambda) \\
m_1'(\lambda) & m_2'(\lambda)
\end{pmatrix},
\qquad m_j(\lambda) = y_j(1, \lambda)\,, \quad m_j'(\lambda) = y_j'(1, \lambda)\,, \quad j = 1, 2\,,
\end{equation}
and introduce the discriminant $\Delta(\lambda) \equiv \Delta(\lambda, q) := {\rm Tr}(F(\lambda, q))$, its derivative $\dot \Delta(\lambda) \equiv \dot \Delta(\lambda, q) := \partial_\lambda \Delta(\lambda, q)$,
 and the anti-discriminant $\delta(\lambda)\equiv \delta(\lambda, q) = m_1(\lambda, q) - m_2'(\lambda, q)$. 
The functions $m_j(\lambda) \equiv m_j(\lambda, q)$ and 
$m_j'(\lambda) \equiv m_j'(\lambda, q)$ 
are analytic on $\C \times L^2_{0, \C}$. The entire functions $\Delta^2(\lambda)-4$ and $\dot \Delta(\lambda)$ have product representations 
(see \cite[Proposition B.10, Proposition B.13]{KP}) 
\begin{equation}\label{product represenations}
\Delta^2(\lambda)-4 = 4 (\lambda_0^+ - \lambda) \prod_{n \geq 1} \frac{(\lambda_{n}^+-\lambda) (\lambda_{n}^- -\lambda)}{\pi_n^4}\,, 
\qquad \dot \Delta (\lambda) = - \prod_{n\geq 1} \frac{\dot \lambda_{n} -\lambda}{\pi_n^2}
\end{equation} 
where $\pi_n=n\pi$ for any $n \geq 1$ and where the zeroes $\dot \lambda_n \equiv \dot \lambda_n(q)$ satisfy the asymptotic estimate $\dot \lambda_n = n^2\pi^2 +\ell^2_n.$
We also need to consider the operator $ - \partial_x^2 +q$ on $[0,1]$ with Dirichlet and Neumann boundary conditions. For any $q$ in $L^2_{0,\mathbb{C}}$,
the corresponding spectra are again discrete, consisting of sequences of complex numbers, bounded from below. 
They are referred to as Dirichlet and respectively, Neumann eigenvalues of $q$. We list them lexicographically and with their algebraic multiplicities 
$ \mu_1\preceq \mu_2\preceq \mu_3 \preceq \dots$ and $\nu_0\preceq \nu_1 \preceq \nu_2 \preceq \dots$\,.
The $\mu_n \equiv \mu_n(q)$ and $\nu_n \equiv \nu_n(q)$ satisfy the asymptotics $\mu_n,\nu_n=n^2\pi^2+\ell^2_n,$ valid uniformly on bounded subsets of $L^2_{0,\mathbb{C}}$. 
For real valued $q$, the Dirichlet and the Neumann eigenvalues are real and satisfy
\[
\lambda_1^- \leq\mu_1\leq \lambda_1^+ <\lambda_2^- \leq \mu_2 \leq \lambda_2^+  < \dots\,, \qquad \quad \nu_0 \leq \lambda_0^+ < \lambda_1^- \leq \nu_1  \leq \lambda_1^+ < \dots \;  .
\]
By shrinking the neighbourhood $ W$ of Theorem \ref{Theorem Birkhoff coordinates}, if needed, one can assure that for any $q \in  W$ there exist isolating neighbourhoods
 $(U_n)_{n \ge 1}$ so that for any $n\geq 1$, $\mu_n,\nu_n, \tau_n, \dot \lambda_n \in U_n$, whereas $\lambda_0^+$ and $\nu_0$ are not contained in any of the $U_n$'s (cf \cite{KP}). Isolating neighbourhoods with this additional property 
can also be chosen locally independently of $q$. Note that for $q \in W$, the Dirichlet und Neumann eigenvalues are all simple and are analytic functions on $W$. 
Similarly, $\tau_n$ and $\dot \lambda_n$, $n \in \N$, are analytic on $W$. In addition,
$m_2(\lambda)$ and $m_1'(\lambda)$ admit the product representations (cf \cite[Proposition B.6]{KP})
$$
m_2(\lambda) =  \prod_{n\geq 1} \frac{\mu_{n} -\lambda}{\pi_n^2}\,, \quad m_1'(\lambda) =  - (\nu_0 - \lambda) \prod_{n\geq 1} \frac{\nu_{n} -\lambda}{\pi_n^2}\,.
$$

Let $S_+ \subset \N$ be finite and set $S = S_+ \cup (-S_+)$. For any $s \in \Z_{\ge 0}$, we identify $h^s_0$ with $h^0_S \times h^s_\bot$ and $h^s_{0, \C}$ with $\C^S \times h^s_{\bot, \C}$.
The manifold $M_S$ of $S-$gap potentials is given by $\Psi^{kdv}(h^0_S \times \{0\})$. By item $(B1)$ of  Theorem \ref{Theorem Birkhoff coordinates} it then follows that $M_S \subset \large \cap_{s \ge 0} H^s_0$.
Actually, potentials in $M_S$ are real analytic functions. 
For our purposes, it is useful to consider the Hilbert spaces  $H^w_{0, \C} := \{ q \in L^2_{0, \C} \, : \, \| q \|_{w} \equiv \| (q_n)_{n \ne 0}\|_w < \infty \}$ and 
$h^w_{0, \C} := \{ (z_n)_{n \ne 0} \in \ell^2_{0, \C} \, : \, \|(z_n)_{n \ne 0}\|_w < \infty \}$
where
$$
\|(z_n)_{n \ne 0}\|_w := (\sum_{n \ne 0} w_n^2 |z_n|^2)^{1/2}\,, \qquad
w_n:= \langle n \rangle ^r e^{a |n|^\sigma}\, , \,\, n \in \Z\,, \qquad r \ge 0\,, \, a > 0\,, \, 0 < \sigma < 1
$$
The weight $w= (w_n)_{n \in \Z}$ is referred to as Gevrey weight and the Hilbert space $H^w_{0, \C}$ as weighted Sobolev space. Functions in $H^w_{0, \C}$ are Gevrey smooth.
Correspondingly, we define the real Hilbert spaces $H^w_{0} := \{ q \in H^w_{0, \C}  \, : \, q \mbox{ real valued}\}$ and $h^w_{0} := \{ (z_n)_{n \ne 0} \in h^w_{0, \C} \, : \, z_{-n} = {\overline{z_n}}  \,\, \forall n \ge 1 \}$.
To fix ideas we will only consider the Gevrey weight with parameters $r=0,$ $a =1$, and $\sigma = 1/2$ und denote it by $w_*$, but any other choice of a Gevrey weight would also be possible.
Note that $M_S \subset H^{w_*}_{0}$ and that $H^{w_*}_{0, \C}$ naturally embeds into $H^s_{0, \C}$ for any $s \in \Z_{\ge 0}$.
According to \cite[Addendum 1 to Theorem 5]{KP1}, we have the following

\medskip

\noindent
{\bf Addendum to Theorem \ref{Theorem Birkhoff coordinates}} {\em The restriction of $\Phi^{kdv}$ to  $H^{w_*}_0$ gives rise to a map $\Phi^{kdv}: H^{w_*}_0 \to h^{w_*}_0$ 
which is a bi-analytic diffeomorphism.  In particular, for any $q \in M_S$, there exists a neighborhood $V^*_{q}$ of $q$ in $H^{w_*}_{0, \C} \cap W$ and a neighborhood 
$\mathcal V^*_{z(q)}$ of $z(q) =\Phi^{kdv}(q)$ in $h^{w_*}_{0, \C}$ so that the restriction of $\Phi^{kdv}$ to $V^*_{q}$ gives rise to a real analytic diffeomorphism
$\Phi^{kdv}: V^*_q \to \mathcal V^*_{z(q)}$. The neighborhood $\mathcal V^*_{z(q)}$ can be chosen to be of the form $\mathcal V^*_{z_S(q)} \times \mathcal V^*_{0, \bot}$
where $\mathcal V^*_{z_S(q)}$ is a neighborhood of $z_S(q) = (z_n(q))_{n \in S}$ in $\C^S$ and $\mathcal V^*_{0, \bot}$ is a ball in $h^w_{\bot, \C}$, centered at $0$
with radius depending on $q$. 
 We denote the set $\Psi^{kdv} (\mathcal V^*_{z_S(q)} \times \{0 \})$ by  $V^*_{q, S}$. It consists of complex valued $S-$gap potentials near $q$. }




\section{Floquet solutions}\label{sezione floquet solution}
In this appendix we obtain formulas for Floquet solutions $f_{\pm n}(x, q)$, $n \in S^\bot_+$, for potentials $q$ in $M_S$
which will be used in Proposition \ref {lemma zn nabla q} of Section \ref{sezione mappa Psi L} to relate these solutions to 
the differentials of the complex Birkhoff coordinates $z_{\pm n}(q)$.
The resulting formulas are a key ingredient for proving the asymptotic expansion of the map $\Psi_L$ (cf Section~\ref{sezione mappa Psi L}). 
Without further reference, we will use the notations estabished in Appendix \ref{Birkhoff map}.

\noindent
We begin by recalling the notion of Floquet solutions of $- y'' + qy = \lambda y$ for any given $q \in L^2_0$ and $\lambda \in \R$.
The eigenvalues $\kappa_\pm(\lambda) \equiv \kappa_\pm(\lambda, q)$ of the Floquet matrix $F(\lambda)$ (cf \eqref{Floquet matrix}) are given by the roots of ${\rm det}\big( F(\lambda) - \kappa {\rm Id}_{2 \times 2}\big) = \kappa^2 - \Delta(\lambda) \kappa + 1$.
For $\lambda$ in $(\lambda_0^+, \infty) \setminus \big( \bigcup_{n \geq 1} [\lambda_n^-, \lambda_n^+] \big)$ one has
$\kappa_\pm(\lambda) = \frac{\Delta(\lambda)}{2} \mp \frac12 \sqrt[c]{\Delta^2(\lambda) - 4} \in \C$
where $\sqrt[c]{\Delta^2(\lambda) - 4}$ denotes the canonical root determined by 
${\rm sign}\big( \sqrt[c]{\Delta^2(\lambda) - 4} \big) = - \ii$ for $\lambda_0^+ < \lambda < \lambda_1^-$ (cf \cite[definition (6.10)]{KP}). 
For $\lambda \in (\lambda_0^+, \infty) \setminus \big( \bigcup_{n \geq 1} [\lambda_n^-, \lambda_n^+] \big)$, 
$m_2(\lambda) \neq 0$ ($\lambda$ is not a Dirichlet eigenvalue), $m_1'(\lambda) \neq 0$ ($\lambda$ is not a Neumann eigenvalue) 
and $(1, a_+(\lambda)) \in \C^2$ is an eigenvector of $F(\lambda)$ corresponding to the eigenvalue $\kappa_+(\lambda)$
$$
F(\lambda) \begin{pmatrix}
1 \\
a_+(\lambda)
\end{pmatrix} = \begin{pmatrix}
m_1(\lambda) & m_2(\lambda) \\
m_1'(\lambda) & m_2'(\lambda)
\end{pmatrix} 
\begin{pmatrix}
1 \\
a_+(\lambda)
\end{pmatrix} = \kappa_+(\lambda) \begin{pmatrix}
1 \\
a_+(\lambda)
\end{pmatrix}
$$
where $a_+(\lambda) \equiv a_+(\lambda, q)$ is given by
\begin{equation}\label{armadillo 1}
a_+(\lambda) = \frac{\kappa_+(\lambda) - m_1(\lambda)}{m_2(\lambda)} \quad \mbox{ or, equivalently, } \quad
 a_+(\lambda) = \frac{m_1'(\lambda)}{\kappa_+(\lambda) - m_2'(\lambda)}\,.
\end{equation}
Similarly, $(1, a_-(\lambda)) \in \C^2$, $a_-(\lambda) \equiv a_-(\lambda, q)$,
is an eigenvector of $F(\lambda)$ corresponding to the eigenvalue $\kappa_-(\lambda)$,
\begin{equation}\label{armadillo 1-}
F(\lambda) \begin{pmatrix}
1 \\
a_-(\lambda)
\end{pmatrix} =  \kappa_-(\lambda) \begin{pmatrix}
1 \\
a_-(\lambda)
\end{pmatrix}\, , \qquad
a_-(\lambda) = \frac{\kappa_-(\lambda) - m_1(\lambda)}{m_2(\lambda)} \quad \mbox{or} \quad  a_-(\lambda) = \frac{m_1'(\lambda)}{\kappa_-(\lambda) - m_2'(\lambda)}\,.
\end{equation}
If $\lambda_n^+$ is a double periodic eigenvalue, one has $\lambda_n^- = \lambda_n^+ = \tau_n$ and $F(\tau_n) = (- 1)^n {\rm Id}_{2 \times 2}$. 
By de l'Hospital's rule, the formulas in \eqref{armadillo 1} - \eqref{armadillo 1-} admit limits at such eigenvalues. 
Recall that $S_+ \subseteq \N$ is finite, $S_+^\bot = \N \setminus S_+$ and $S = S_+ \cup (- S_+)$. 
We denote by $\, \dot{} \,$ the derivative with respect to $\lambda$. 
\begin{lemma}\label{lemma armadillo 1}
For any $q \in M_S$ and $n \in S_+^\bot$, the following holds: 

\noindent
$(i)$ $(- 1)^n \dot m_2(\tau_n) > 0$, $(- 1)^{n+1} \ddot{\Delta}(\tau_n) > 0$. 

\noindent
$(ii)$ The limit $a_{\pm n} \equiv a_{\pm n}(q) := \lim_{\lambda \to \tau_n} a_{\pm}(\lambda, q)$ exists and 
\begin{equation}\label{costa rica 2}
a_{\pm n} = - \frac{\dot m_1(\tau_n)}{\dot m_2(\tau_n)} \pm \ii \frac{\sqrt[+]{(- 1)^{n + 1} \ddot{\Delta}(\tau_n)/2}}{(- 1)^n \dot m_2(\tau_n)}\,.
\end{equation}

\noindent
$(iii)$ One has $\dot{m_1}'(\tau_n) \neq 0$ and 
\begin{equation}\label{alternative formula a_n}
a_{\pm n} = \frac{\dot{m_1}'(\tau_n)}{- \dot{m_2}'(\tau_n) \pm \ii (- 1)^{n} \sqrt[+]{(- 1)^{n + 1} \ddot{\Delta}(\tau_n)/2}}\,.
\end{equation}
\end{lemma}
\noindent
\begin{remark}\label{analytic extension a_n} Recall that for any given $q \in M_S$, $V^*_{q, S}$ is the set of $S-$gap potentials, introduced at the end of Appendix \ref{Birkhoff map}.
By shrinking $W$ of Theorem \ref{Theorem Birkhoff coordinates}, if needed, the expressions in the formulas for $a_{\pm n}$, $n \in S^\bot_+$, of item (ii) and (iii) of Lemma \ref{lemma armadillo 1} 
are well defined, real analytic functions on $V^*_{q, S}$ and the formulas for $a_{\pm n}$ continue to hold for any potential in $V^*_{q, S}$ (cf Appendix \ref{Birkhoff map}).
\end{remark}
\begin{proof}
$(i)$ For any $q \in M_S$ and $n \in S_+^\bot$, $\tau_n = \mu_n$, $m_1(\tau_n) = (- 1)^n$, $m_2'(\tau_n) = (- 1)^n$, and 
$\dot{m}_2(\mu_n) = (- 1)^n\int_0^1 y_2(x, \mu_n)^2\, d x$
(cf \cite[Proposition B.4]{KP}). Since $\tau_n$ is a nondegenerate critical point of $\Delta$, one has $\dot \Delta(\tau_n) = 0$ and $ \ddot{\Delta}(\tau_n) \ne 0$.
Furthermore, one has $\Delta(\tau_n) = (- 1)^n 2$ and $(- 1)^{n + 1} \ddot{\Delta}(\tau_n) > 0$. 
$(ii)$ It is well known that $F(\lambda)$ is real analytic in $\lambda$ (cf Appendix \ref{Birkhoff map}). Expanding $m_1(\lambda)$ and $m_2(\lambda)$ at $\tau_n$, $n \in S^\bot_+$ , one has 
$$
m_1(\lambda) = (- 1)^n + \dot m_1(\tau_n)(\lambda - \tau_n) + O((\lambda - \tau_n)^2)\,, \quad 
m_2(\lambda) = \dot m_2(\tau_n) (\lambda - \tau_n) + O((\lambda - \tau_n)^2)\,,
$$
and 
$\Delta(\lambda) = (- 1)^n 2 + \frac{\ddot{\Delta}(\tau_n)}{2} (\lambda - \tau_n)^2 + O \big( (\lambda - \tau_n)^3 \big)\,.$
It follows that for $\lambda_{n - 1}^+ < \lambda < \lambda_{n +1}^-$,
$$
\sqrt[c]{\Delta^2(\lambda) - 4 } = \ii (- 1)^{n + 1} \sqrt[+]{(- 1)^{n + 1} 2 \ddot{\Delta}(\tau_n)} \, (\lambda - \tau_n) + O((\lambda - \tau_n)^2)
$$
(where we set $\sqrt[c]{\Delta^2(\tau_n) - 4 } = 0$). 
Combining these asymptotic estimates, the claimed formula \eqref{costa rica 2} then follows from \eqref{armadillo 1} - \eqref{armadillo 1-}.
$(iii)$ Since $\tau_n$ is a Neumann eigenvalue and Neumann eigenvalues are simple, it follows that $\dot{m}_1'(\tau_n) \neq 0$. 
Expanding $m_1'(\lambda)$, $m_2'(\lambda)$ at $\tau_n$ one gets 
$$
\begin{aligned}
& m_1'(\lambda) = \dot{m}_1'(\tau_n ) (\lambda - \tau_n) + O \big( (\lambda - \tau_n)^2 \big)\,,  \qquad
& m_2'(\lambda) = (- 1)^n + \dot{m}_2'(\tau_n)(\lambda - \tau_n) + O\big( (\lambda - \tau_n)^2 \big)\,.
\end{aligned}
$$
Furthermore, one has
$$
\kappa_\pm(\lambda) = (- 1)^n  \pm  \ii \frac{(- 1)^n}{2}  \sqrt[+]{(- 1)^{n +1} 2 \ddot{\Delta}(\tau_n)/2} \, (\lambda - \tau_n) + 
O\big( (\lambda - \tau_n)^2\big)\,.
$$
Formula \eqref{alternative formula a_n}
 then follows from \eqref{armadillo 1} - \eqref{armadillo 1-}.
\end{proof}

\medskip

For $q \in M_S$ and $n \in S_+^\bot$, we define the Floquet solutions $f_{\pm n}(x) \equiv f_{\pm n}(x, q)$  at $\tau_n$ by
$$
 f_{\pm n} (x, q) := y_1(x, \tau_n, q) + a_{\pm n}(q) y_2(x, \tau_n, q)\,.
$$
By Lemma \ref{lemma armadillo 1} and Remark \ref{analytic extension a_n} they are welldefined for any potential in $V^*_{q, S}$.
Furthermore, consider the normalized solutions $H_n(x) \equiv H_n(x, q)$ and $G_n(x) \equiv G_n(x, q)$ of $- y'' + q y = \tau_n y,$
\begin{equation}\label{armadillo 2}
H_n(x) := \big( - \frac{2 \dot m_2(\tau_n)}{\ddot{\Delta}(\tau_n)}\big)^{\frac12} \,
\Big( y_1(x, \tau_n) - \frac{\dot m_1(\tau_n)}{\dot m_2(\tau_n)} y_2(x, \tau_n) \Big)
\end{equation}
\begin{equation}\label{aramdillo 2}
G_n(x) := \big( - \frac{2 \dot m_2(\tau_n)}{\ddot{\Delta}(\tau_n)}\big)^{\frac12} 
\frac{\sqrt[+]{(- 1)^{n + 1} \ddot{\Delta}(\tau_n)/2}}{(- 1)^{n} \dot m_2(\tau_n)} y_2(x, \tau_n)\,.
\end{equation}
One then has $G_n(0) = 0$ and 
\begin{equation}\label{aramdillo 3}
H_n(x) + \ii G_n(x) = \big( - \frac{2 \dot m_2(\tau_n)}{\ddot{\Delta}(\tau_n)}\big)^{\frac12} f_n(x)\,, \quad
H_n(x) - \ii G_n(x) = \big( - \frac{2 \dot m_2(\tau_n)}{\ddot{\Delta}(\tau_n)}\big)^{\frac12} f_{-n}(x)\,.
\end{equation}
Note that for any $q \in M_S$, $ f_{- n}(x, q) = \overline{f_n(x, q)},$ and 
$ H_n(x, q)$ and $ G_n(x, q)$ are the normalized real and respectively imaginary parts of $f_n(x, q).$
In addition, they satisfy $H_n(0, q) > 0$ and $G'_n(0, q) > 0$ by the formulas above. 
Hence given $q \in M_S$, by shrinking $V^*_{q, S}$, if needed, we can assume that $Re \, H_n(0) > 0$ and $Re \, G'_n(0) > 0$ on $V^*_{q, S}$
for any $n \in S^\bot_+$. 
\begin{proposition}\label{Re/ImFloquet solutions} For any $q \in M_S$ and $n \in S^\bot_+$ the following holds:
(i) For any $s \in \Z_{\ge 0}$, $H_n(\cdot, p)$, $G_n(\cdot, p)$, and $f_{\pm n}(\cdot, p)$ are
analytic maps in $p \in V^*_{q, S}$ with values in $H^s_{\C}[0, 1]$. \\
(ii) For any $p \in V^*_{q, S}$, the Floquet solutions 
$$
 \big( - \frac{2 \dot m_2(\tau_n)}{\ddot{\Delta}(\tau_n)}\big)^{\frac12} f_{\pm n}(x) = H_n(x) \pm \ii G_n(x)
 $$
have the property that $H_n,$ $G_n$ are the unique solutions of $- y'' + p y = \tau_n y$,
satisfying the following normalization conditions:  \,\, $(ii1)$ $\int_0^1 H_n(x) G_n(x)\, d x = 0$ \,\, and
$$
(ii2) \,\, \int_0^1 G_n^2(x)\, d x = 1\,, \quad G_n(0) = 0\,,\quad Re \, G_n'(0) > 0\, ; \qquad
(ii3) \,\, \int_0^1 H^2_n(x) \, d x = 1\,, \quad  Re \, H_n(0) > 0\,. \qquad 
$$
\end{proposition}
\begin{proof} 
(i) Since for any $s \in \Z_{\ge 0}$, $H^{w_*}_{0, \C}$ embeds into the Sobolev space $H^s_{0, \C}$ the claimed analyticity statements follow from the results
recorded in Appendix \ref{Birkhoff map}.
(ii) Clearly, the statement on uniqueness follows from the uniqueness of the initial value problem for $- y'' + q y = \tau_n y$.
Hence it remains to prove that $G_n$ and $H_n$ satisfy items (ii1) - (ii3).  In view of item (i), it suffices to verify these normalisation conditions on $M_S$.
By \cite{PT}, Theorem 6, page 21, one has
\begin{equation}\label{armadillo 5}
\dot{m}_1'(\tau_n) = (- 1)^{n + 1} \int_0^1 y_1(x, \tau_n)^2 \,d x\,, \qquad 
\dot{m}_2(\tau_n) = (- 1)^n \int_0^1 y_2(x, \tau_n)^2 \, d x\,,
\end{equation}
\begin{equation}\label{armadillo 4}
\dot{m}_2'(\tau_n) = (- 1)^{n + 1} \int_0^1 y_1(x, \tau_n) y_2(x, \tau_n)\, d x = - \dot m_1(\tau_n)\,.
\end{equation}
To prove $(ii1)$ it is  to show that $J := \int_0^1 (Re f_n(x) )\,  y_2(x, \tau_n)\, d x= 0$. By \eqref{armadillo 2}-\eqref{armadillo 4}, 
\begin{align*}
J & = \int_0^1 \big( y_1(x, \tau_n) - \frac{\dot{m}_1(\tau_n)}{\dot{m}_2(\tau_n)} y_2(x, \tau_n) \big) y_2(x, \tau_n)\, d x 
 = (- 1)^{n + 1} \dot{m}_2'(\tau_n) - \frac{\dot m_1(\tau_n)}{\dot m_2(\tau_n)} (- 1)^n \dot m_2(\tau_n)\,.
\end{align*}
Since $\dot \Delta(\tau_n) = 0$ and at the same time $\dot \Delta(\tau_n) = \dot m_1(\tau_n) + \dot{m}_2'(\tau_n)$ 
one concludes that $J = 0$. 
$(ii2)$ By the definition of $G_n$, one has $G_n(0) = 0$ and $G_n'(0) > 0.$
To see that $\int_0^1 G_n(x)^2 \, d x = 1$, note that 
$$
\int_0^1 \big( Im f(x)\big)^2\, d x = \frac{(- 1)^{n + 1} \ddot{\Delta}(\tau_n)/2}{(\dot m_2(\tau_n))^2} (- 1)^n \dot m_2(\tau_n) = - \frac{\ddot{\Delta}(\tau_n)}{2 \dot{m}_2(\tau_n)}
$$
implying that $\int_0^1 G_n(x)^2 \, d x = 1$. 
$(ii3)$ Since $Re f_n(0) = y_1(0, \tau_n) = 1$, one has $H_n(0) > 0$ whereas 
$$
J  : = \int_0^1 Re f_n(x)^2\, d x = \int_0^1 y_1(x, \tau_n)^2 - 2 \frac{\dot m_1(\tau_n)}{\dot m_2(\tau_n)} \int_0^1 y_1(x, \tau_n) y_2(x, \tau_n)\, d x+ \Big( \frac{\dot m_1(\tau_n)}{\dot m_2(\tau_n)}\Big)^2 \int_0^1 y_2(x, \tau_n)^2\, d x 
$$
is given by (cf  \eqref{armadillo 5}-\eqref{armadillo 4})

$$
J = (- 1)^{n + 1} \dot{m}_1'(\tau_n) - 2 \frac{\dot m_1(\tau_n)}{\dot m_2(\tau_n)} (- 1)^n \dot m_1(\tau_n) + \Big( \frac{\dot m_1(\tau_n)}{\dot m_2(\tau_n)} \Big)^2 (- 1)^n \dot m_2(\tau_n) 
= (- 1)^{n + 1} \dot{m}_1'(\tau_n) + (- 1)^{n + 1} \frac{\dot m_1(\tau_n)^2}{\dot m_2(\tau_n)}\,.
$$
By the definition \eqref{armadillo 2}, it is therefore to show that 
$$
(- 1)^n \frac{\ddot{\Delta}(\tau_n)}{2 \dot m_2(\tau_n)} = \dot{m}_1'(\tau_n) + \frac{\dot m_1(\tau_n)^2}{\dot m_2(\tau_n)}
\qquad \mbox{or} \qquad 
\dot{m}_1'(\tau_n) \dot m_2(\tau_n) = (- 1)^n \ddot{\Delta}(\tau_n)/2 - \dot m_1(\tau_n)^2\,.
$$
This latter identity follows by combining the two formulas for $a_n$, given in Lemma \ref{lemma armadillo 1}. 
\end{proof}


\section{Asymptotic expansions}\label{appendix asymptotics}
The main purpose of this appendix is to provide for any  $S-$gap potential $q$ an asymptotic expansion of the Floquet solutions $f_{\pm n}(x) \equiv f_{\pm n}(x, q)$ 
as $n \to \infty$. These expansions are a key ingredient for proving the asymptotic expansion of the map $\Psi_L$,  stated in  
Theorem \ref{lemma asintotica Floquet solutions} in Section~\ref{sezione mappa Psi L}.
At the end of this appendix we provide an asymptotic expansion for the KdV frequencies for $S-$gap potentials, needed for the expansion of the KdV Hamiltonian
in the new coordinates.

\noindent 
Throughout this appendix, if not mentioned otherwise, 
we assume that $q \in M_S$ where $S = S_+ \cup (- S_+)$ and $S_+ \subset \N$ is a finite subset.  
Furthermore, $V^*_{q, S}$ is the neighborhood of $q$, introduced at the end of Appendix \ref{Birkhoff map}.
If not mentioned otherwise, 
 $\sqrt{\mu}$ denotes the principal branch $\sqrt[+]{\mu}$ of the square root, defined  for $\mu \in \C \setminus (- \infty, 0]$.
 Recall that for any $p \in V^*_{q, S}$, one has $\int_0^1 p(x)\, d x= 0$ and 
 \begin{equation}\label{costa rica 1}
f_{\pm n}(x) = y_1(x, \tau_n) + a_{\pm n} y_2(x, \tau_n)\,, \quad \forall \, n \in S_+^\bot\,,
\end{equation}
where $y_j(x, \lambda)$, $j = 1,2$, denote the fundamental solutions of 
$- y'' + p y = \lambda y$, $a_{\pm n}$ are the complex numbers given by \eqref{costa rica 2}, and
$\tau_n = (\lambda_n^+ + \lambda_n^-)/2$. Note that if $p$ is real valued, then 
$$
\tau_n(p) \in \R\,, \quad a_{-n}(p) = \overline{a_n(p)}\,, \quad \mbox{and} \quad 
f_{-n}(x, p) = \overline{ f_n(x, p)}\,, \qquad \forall n \in  S^\bot_+\,.
$$

\begin{theorem}\label{teorema espansione fn appendice}
Let  $q \in M_S$ and $N \in \Z_{ \geq 0}$. Then for any $p \in V^*_{q, S}$, the Floquet solutions $f_n$, $n \in S^\bot$, 
have an expansion as $|n| \to \infty$ of the form 
\begin{equation}\label{costa rica 3}
f_n(x, p) = e^{\ii \pi n x} \big(1 + \sum_{k = 1}^N \frac{f_k^{ae}(x, p)}{(2 \pi \ii n)^k} + \frac{{\cal R}^{f_n}_{ N}(x, p)}{(2 \pi \ii n )^{N + 1}} \big)
\end{equation}
where for any $s \geq 0$, the coefficients $V^*_{q, S} \to H^s_\C$, $p \mapsto f^{ae}_k(\cdot, p)$, $ k \ge 1$,  are real analytic and  
the remainder $V^*_{q, S} \to H^s_\C$, $p \mapsto {\cal R}_N^{f_n}(\cdot, p)$, is analytic. 
In addition, for any given $j \ge 0,$
\begin{equation}\label{costa rica 4}
\sup_{\begin{subarray}{c}
0 \leq x \leq 1 \\
n \in S^\bot
\end{subarray}} |\partial_x^j {\cal R}_{ N}^{f_n}(x, p)| \leq C_{N, j}
\end{equation}
where the constant $C_{N, j} \geq 1$ can be chosen locally uniformly for $p$ in $V^*_{q, S}$. 
\end{theorem}

\noindent
To prove Theorem \ref{teorema espansione fn appendice}, we first need to establish some auxiliary results. 

\begin{lemma}\label{lemma A2 appendice}
For any $q \in M_S$ and any integers $N, M \ge 0$, the following holds: for any $p \in V^*_{q, S}$ and $\nu \in {\C \setminus\{0\}}$, there exist solutions $y_{N, M}(x, \nu) \equiv y_{N, M}(x, \nu, p)$
of $- y'' + p y = \nu^2 y$ of the form 
\begin{equation}\label{costa rica 7}
y_{N, M}(x, \nu, p) = e^{\ii \nu x} \Big(1 + \sum_{k = 1}^N \frac{y_k^{ae}(x, p)}{(2 \ii \nu)^k} + \frac{\tilde y_{N, M}(x, \nu, p)}{(2 \ii \nu)^{N + 1}} \Big)\,.
\end{equation}
The functions $y_k^{ae}(x) \equiv  y_k^{ae}(x, p)$, $k \geq 1$, are defined inductively by
\begin{equation}\label{costa rica 10}
y_k^{ae}(x) = \int_0^x (- \partial_t^2 + p) y^{ae}_{k - 1}(t)\, d t\,, \quad y_0^{ae}(x) \equiv 1
\end{equation}
and for any $s \in  \Z_{\ge 0}$, the maps $V^*_{q, S} \to H^s_\C[0, 1]$, $p \mapsto y^{ae}_k(\cdot, p)$, are real analytic. 
The remainder $\tilde y_{N, M}(x, \nu, p)$ satisfies $\tilde y_{N, M}(0, \nu, p) = 0$, $\partial_x \tilde y_{N, M}(0, \nu, p) = \sum_{k = 1}^M \frac{\partial_x y_{N + k}^{ae}(0, p)}{(2 \ii \nu)^{k}}$
and  has the property that for any $s \in \Z_{\geq 0}$, $(\C \setminus\{0\})\times V^*_{q, S} \to H^s_\C[0, 1]$, $( \nu, p) \mapsto \tilde y_{N, M}( \cdot, \nu, p)$ 
is analytic. Furthermore, if $\nu \in \R\setminus \{0\}$ and $p$ is real valued then $y_{N, M}(x, - \nu, p) = \overline{y_{N, M}(x, \nu, p)}$.
In addition, for any given $c > 0$, the remainder $ \tilde y_{N, M}( \cdot, \nu, p)$ satisfies
\begin{equation}\label{costa rica 8}
\sup_{\begin{subarray}{c}
|\rm{Im} \, \nu| \le c, |\nu| \ge 1\\
0 \leq x \leq 1,  0 \leq j \leq M
\end{subarray}} |\partial_x^j \tilde y_{N, M}(x, \nu, p)| \leq C_{N, M}
\end{equation}
where the constant $C_{N, M}$ can be chosen locally uniformly in $p \in V^*_{q, S}$. 
\begin{remark} By \eqref{costa rica 10}, for any $k \ge 1$ and $s \ge k-1$,
the map $y_k^{ae} :  H^s_{0, \C}   \to H^{s-k+2}_\C[0, 1] ,   q \mapsto y^{ae}_k(\cdot, q) $ is real analytic.
Writing $Q(x) := \int_0^x q(t) dt$, the formulas for $y_1^{ae},$ $y_2^{ae},$ and $y_3^{ae}$ read as follows:
$$
 y_1^{ae}(x, q)=  Q(x)\,,  \qquad y_2^{ae}(x, q) = -(q(x) - q(0)) + \frac{1}{2} Q(x)^2\,,
$$
$$
y_3^{ae}(x, q) =  (q'(x) - q'(0)) - \int_0^x q(t)^2 dt + q(0)Q(x)  - q(x)Q(x) + \frac{1}{6} Q(x)^3.
$$
\end{remark}
\end{lemma}
\begin{proof}
For any $p \in V^*_{q, S}$ and $\nu \ne 0$, the solutions $y_{N, M}(x, \nu) \equiv y_{N, M}(x, \nu, p)$ of $- y'' + p y = \nu^2 y$ are obtained from solutions of the form 
\begin{equation}\label{costa rica 9}
y_{N_1}(x, \nu) = e^{\ii \nu x} \Big(1 + \sum_{k = 1}^{N_1} \frac{y_k^{ae}(x)}{(2 \ii \nu)^k} + \frac{\widetilde y_{N_1}(x, \nu)}{(2 \ii \nu)^{N_1 + 1}} \Big)
\end{equation}
with an appropriate choice of $N_1$. Solutions of the form \eqref{costa rica 9}  were studied in \cite[Chapter 1, Section 4]{Ma}. Here for any $k \geq 1$,
$y_k^{ae}(x) \equiv y_k^{ae}(x, p)$ is defined  by \eqref{costa rica 10}
and $\widetilde y_{N_1} (x, \nu) \equiv \widetilde y_{N_1} (x, \nu, p)$ is the unique solution of the inhomogeneous ODE 
\begin{equation}\label{costa rica 11}
\big( - \partial_x^2 + p - \nu^2 \big) [e^{\ii \nu x} \widetilde y_{N_1}(x, \nu)] = - 2 \ii \nu e^{\ii \nu x} \partial_x  y_{N_1+1}^{ae}(x) 
\end{equation}
satisfying the initial conditions 
\begin{equation}\label{costa rica 12}
\widetilde y_{N_1}(0, \nu) = 0\,, \qquad \partial_x \widetilde y_{N_1}(0, \nu) = 0\,.
\end{equation}
Clearly, for any $s \geq 0$ and $ k \ge 1$, the maps $V^*_{q, S} \to H^s_\C[0, 1]$, $p \mapsto y^{ae}_k(\cdot, p)$, are real analytic.
Arguing as in \cite[Chapter 1]{KP}, one sees that for any $s \geq 0$,  the map $(\C \setminus\{0\})\times V^*_{q, S} \to H^s_\C[0, 1]$, $( \nu, p) \mapsto \tilde y_{N_1}( \cdot, \nu, p)$ 
is  analytic. Furthermore, note that $ y_{N_1}(0, \nu) = 0$ and $\partial_x  y_{N_1}(0, \nu) = \ii \nu + \sum_{k = 1}^{N_1} \frac{\partial_x y_{ k}^{ae}(0)}{(2 \ii \nu)^{k }}$.
It then follows that for any $\nu \ne 0,$ $y_{N_1}(x, \nu)$ is a solution of $- y'' + p y = \nu^2 y$. Indeed 
\begin{align*}
\big( - \partial_x^2 + p - \nu^2 \big) y_{N_1} & = 
e^{\ii \nu x} \sum_{k = 1}^{N_1} \frac{- \partial_x y_k^{ae} - \partial_x^2 y_{k - 1}^{ae} + p y_{k - 1}^{ae}}{(2 \ii \nu)^{k - 1}} \\
& \qquad + e^{\ii \nu x} \frac{- \partial_x^2 y_{N_1}^{ae} + p y_{N_1}^{ae} }{(2 \ii \nu)^{N_1}} + 
(- \partial_x^2 + p - \nu^2) \Big[ e^{\ii \nu x} \frac{\widetilde y_{{N_1}}(x, \nu)}{(2 \ii \nu)^{N_1 + 1}} \Big]\,.
\end{align*}
Hence by \eqref{costa rica 10} and \eqref{costa rica 11}, $\big( - \partial_x^2 + p - \nu^2 \big) y_{N_1} = 0$.
The solution $y_{N, M}(x, \nu)$ is then defined as follows
\begin{align*}
y_{N, M}(x, \nu) & = 
e^{\ii \nu x}\Big( 1 + \sum_{k = 1}^N \frac{y_k^{ae}(x)}{(2 \ii \nu)^k} + \frac{\widetilde y_{N, M}(x, \nu)}{(2 \ii \nu)^{N + 1}}\Big) \,, \qquad
\widetilde y_{N, M}(x, \nu) =  
\sum_{k = 1}^M \frac{y_{N + k}^{ae}(x)}{(2 \ii \nu)^{k - 1}} + \frac{\widetilde y_{N + M}(x, \nu)}{(2 \ii \nu)^M}\,.
\end{align*}
Since by the definition \eqref{costa rica 10}, $y_k^{ae}(0) = 0$ for any $k \geq 1$ and by \eqref{costa rica 12}, $\widetilde y_{N + M}(0, \nu) = 0$ 
one concludes that $\widetilde y_{N, M}(0, \nu) = 0$. Furthermore, since by \eqref{costa rica 12} $\partial_x \widetilde y_{N + M}(0, \nu) = 0$, one has 
$$
\partial_x \widetilde y_{N, M}(0, \nu) = \sum_{k = 1}^M \frac{\partial_x y_{N + k}^{ae}(0)}{(2 \ii \nu)^{k - 1}}
$$
and by the arguments above,  for any $s \geq 0$,  the map $(\C \setminus\{0\})\times V^*_{q, S}  \to H^s_\C[0, 1]$, $( \nu, p) \mapsto \tilde y_{N, M}( \cdot, \nu, p)$ is  analytic. 
Furthermore, if $\nu \in \R\setminus \{0\}$ and $p$ is real valued then 
$$
y'_{N, M}(0, - \nu) = - \ii \nu + \sum_{k=1}^{N+M}  \frac{\partial_x y_{k}^{ae}(0)}{(-2 \ii \nu)^{k}} = \overline{y'_{N, M}(0, - \nu)}\,.
$$
Since $y_{N, M}(0, \pm \nu) = 1$ and $y_{N, M}(0, \nu)$ and $y_{N, M}(0, - \nu)$ both solve $-y'' + py = \nu^2 y$
it follows by the uniqueness of the initial value problem that $y_{N, M}(\cdot, - \nu) = \overline{y_{N, M}(\cdot, \nu})$.
It remains to show the estimate \eqref{costa rica 8}. Note that for any $p \in V^*_{q, S}$, the terms $\frac{y_{N + k}^{ae}(x)}{(2 \ii \nu)^{k - 1}}$, $1 \leq k \leq M$, 
satisfy an estimate of the type  \eqref{costa rica 8}. Hence it suffices to show that 
\begin{equation}\label{costa rica 13}
\sup_{\begin{subarray}{c}
|\rm{Im} \, \nu | \le c,  |\nu| \geq 1\\
0 \leq x \leq 1, 0 \leq j \leq M
\end{subarray}}  \frac{|\partial_x^j \widetilde y_{N + M}(x, \nu)|}{ |\nu|^j} \leq \widetilde C_{N, M}
\end{equation}
for some constant $\widetilde C_{N, M} > 0$.
By \eqref{costa rica 11}, \eqref{costa rica 12}, $\widetilde y_{N + M}$ solves the initial value problem 
\begin{equation}\label{costa rica 14}
\begin{aligned}
& \big( - \partial_x^2 + p - \nu^2 \big) \big[ e^{\ii \nu x} \widetilde y_{N + M}(x, \nu) \big] = - 2 \ii \nu e^{\ii \nu x} \partial_x y_{N + M +1}^{ae}(x)\,, \\
& \widetilde y_{N + M}(0, \nu) = 0\,, \quad \partial_x \widetilde y_{N + M}(0, \nu) = 0\,.
\end{aligned}
\end{equation}
By the method of the variation of the constants it is given by
\begin{equation}\label{costa rica 15}
\widetilde y_{N + M}(x, \nu) = - 2 \ii \nu \int_0^x K(x, t, \nu^2) e^{\ii \nu (t - x)} \partial_t y^{ae}_{N + M + 1}(t)\, d t
\end{equation}
where 
$K(x, t, \nu^2)  = y_1(x, \nu^2) y_2(t, \nu^2) - y_1(t, \nu^2) y_2(x, \nu^2)$, satisfying the estimates (cf \cite[Chapter 1]{KP}) 
$$
\sup_{\begin{subarray}{c}
0 \leq x \leq 1 \\
|{\rm Im} \, \nu | \le c, |\nu| \geq 1
\end{subarray}} |y_1(x, \nu^2)| \leq C\,, \quad \sup_{\begin{subarray}{c}
0 \leq x \leq 1 \\
|{\rm Im} \, \nu | \le c, |\nu| \geq 1
\end{subarray}} |\nu| |y_2(x, \nu^2)| \leq C\,,
$$
$$
\sup_{\begin{subarray}{c}
0 \leq x \leq 1 \\
|{\rm Im } \, \nu | \le c, |\nu| \geq 1
\end{subarray}} \frac{|\partial_x y_1(x, \nu^2)|}{|\nu|} \leq C\,, \quad \sup_{\begin{subarray}{c}
0 \leq x \leq 1 \\
|{\rm  Im} \, \nu | \le c, |\nu| \geq 1
\end{subarray}} | \partial_x y_2(x, \nu^2)| \leq C
$$
for some constant $C > 0$. It then follows from \eqref{costa rica 15} and \eqref{costa rica 10} that 
$
\sup_{\begin{subarray}{c}
0 \leq x \leq 1 \\
|{\rm Im} \, \nu | \le c, |\nu| \geq 1
\end{subarray}} |\widetilde y_{N + M}(x, \nu)| \leq C\,.
$
Since $K(x, x, \nu^2) = 0$ one has that 
$$
\partial_x \widetilde y_{N + M}(x, \nu) = 
- \ii \nu \widetilde y_{N + M}(x, \nu) - 2 \ii \nu \int_0^x \partial_x K(x, t, \nu^2) e^{\ii \nu(t - x)} \partial_t y^{ae}_{N + M + 1}(t)\, d t
$$
implying that 
$
\sup_{\begin{subarray}{c}
0 \leq x \leq 1 \\
|{\rm Im} \, \nu | \le c, |\nu| \geq 1
\end{subarray}} \frac{|\partial_x \widetilde y_{N + M }(x, \nu)|}{|\nu|} \leq C\,.
$
Using equation \eqref{costa rica 14} one gets 
\begin{equation}\label{costa rica 16}
- \partial_x^2 \widetilde y_{N + M}(x, \nu) = 
2 \ii \nu \partial_x \widetilde y_{N + M}(x, \nu) - p \widetilde y_{N + M}(x, \nu) - 2 \ii \nu \partial_x y^{ae}_{N + M + 1}(x)
\end{equation}
yielding 
$
\sup_{\begin{subarray}{c}
0 \leq x \leq 1 \\
|{\rm Im} \, \nu | \le c, |\nu| \geq 1
\end{subarray}} \frac{|\partial_x^2 \widetilde y_{N + M}(x, \nu)|}{|\nu|^2} \leq C\,.
$
By taking derivatives of \eqref{costa rica 16}, one then concludes that there exists a constant $C_{N,M} > 0$ so that for any $0 \leq j \leq M$,
$
\sup_{\begin{subarray}{c}
0 \leq x \leq 1 \\
|{\rm Im} \, \nu | \le c, |\nu| \geq 1
\end{subarray}} \frac{|\partial_x^j \widetilde y_{N + M}(x, \nu)|}{|\nu|^j} \leq \widetilde C_{N,M}.$
Going through the arguments of the proof one concludes that the constant $C_{N, M}$ can be chosen locally uniformly for $p \in V^*_{q, S}$. 
\end{proof}

\smallskip

Let $q \in M_S$ and $N, M \geq 0$. Then for any $p \in V^*_{q, S}$ and $\nu \ne 0$ with $|\nu|$ sufficiently large, the solutions $y_{N, M}(x, \nu)$ and $y_{N, M}(x, - \nu)$ of 
$- y'' + p y = \nu^2 y$, considered in Lemma \ref{lemma A2 appendice}, 
are linearly independent. Indeed by Lemma \ref{lemma A2 appendice}, for any $\nu \ne 0$,
$$
y_{N, M}(0, \nu) = 1\,, \quad \partial_x y_{N, M}(0, \nu) =  \ii \nu + \sum_{k = 1}^{N + M} \frac{\partial_x y_k^{ae}(0)}{( 2 \ii \nu)^k}
$$
implying that the Wronskian of $y_{N, M}(\cdot, - \nu)$ and $y_{N, M}(\cdot,  \nu)$ equals
$$
{\rm det}\begin{pmatrix}
y_{N, M}(0, - \nu) & y_{N, M}(0, \nu) \\
\partial_x y_{N, M}(0, - \nu) & \partial_x y_{N, M}(0, \nu)
\end{pmatrix} = \partial_x y_{N, M}(0, \nu) - \partial_x y_{N, M}(0, - \nu)    = 2 \ii \nu + \sum_{k = 1}^{N + M} \frac{\partial_x y_k^{ae}(0)}{(2 \ii \nu)^k} \big(1 - (- 1)^k \big)\,.
$$
Hence there exists $\nu_{b} \geq 1$ so that 
\begin{equation}\label{costa rica 16 bis}
\Big| 2 \ii \nu + \sum_{k = 1}^{N + M} \frac{\partial_x y_k^{ae}(0)}{(2 \ii \nu)^k} \big(1 - (- 1)^k \big)\Big| \geq 1, \quad \forall \nu  \mbox{ with } |\nu| \geq \nu_b\,.
\end{equation}
The bound $\nu_b$ can be chosen locally uniformly in $p \in V^*_{q, S}$. It then follows that for $|\nu| \geq \nu_b$, $y_1(x, \nu^2)$, $y_2(x, \nu^2)$ are linear combinations 
of $y_{N, M}(x, \nu)$ and $y_{N, M}(x, - \nu)$,
\begin{equation}\label{costa rica 17}
y_1(x, \nu^2) = \alpha_{N, M}(- \nu) y_{N, M}(x, \nu) +  \alpha_{N, M}(\nu) y_{N, M}(x, - \nu)\,,
\end{equation}
\begin{equation}\label{costa rica 18}
\alpha_{N, M}(\nu) =  \frac{\partial_x y_{N, M}(0,  \nu)}{\partial_x y_{N, M}(0, \nu) - \partial_x y_{N, M}(0,  -\nu)}
\end{equation}
and 
\begin{equation}\label{costa rica 19}
y_2(x, \nu^2) = \beta_{N, M}(\nu) y_{N, M}(x, \nu) + \beta_{N, M}(-\nu) y_{N, M}(x, - \nu)\,, 
\end{equation}
\begin{equation}\label{costa rica 20}
\beta_{N, M}(\nu) = \frac{1}{\partial_x y_{N, M}(0, \nu) - \partial_x y_{N, M}(0, - \nu)}\,.
\end{equation}
We note that $\beta_{N, M}(- \nu) = - \beta_{N, M}(\nu)$ and for $|\nu |$ sufficiently large, $\alpha_{N, M}(\nu) \simeq \frac 12$ and $\beta_{N, M}(\nu) \simeq \frac{1}{2 \ii \nu}$.
Furthermore, in case $p$ is real valued and $\nu \in \R$ with $|\nu | \ge \nu_b$, one has
$$
\alpha_{N, M}(- \nu) = \overline{\alpha_{N, M}(\nu)}\, \quad \mbox{and} \quad \beta_{N, M}(- \nu) = \overline{\beta_{N, M}(\nu)} = - \beta_{N, M}( \nu)\,,
$$
implying that $\beta_{N, M}(\nu)$ is purely imaginary.
Finally, to prove Theorem~\ref{teorema espansione fn appendice}, the following two additional results are needed.
It is well know that $\tau_n = (\lambda_n^+ +  \lambda_n^-)/2$ admits an asymptotic expansion as $n \to \infty$ (cf e.g. \cite[Theorem 1.3]{KST1}).
More precisely, we have the following

\begin{lemma}\label{asymptotics tau} Let $q \in M_S$ and $N \in \Z_{ \ge 0}.$ Then for any $p \in V^*_{q, S}$, $\tau_n(p),$ $n \in S^\bot_+$, has an expansion of the form
\begin{equation}\label{asintotica tau n KST}
\tau_n(p) = n^2 \pi^2 + \sum_{k = 1}^N \frac{\tau^{ae}_{2k}(p)}{(2 \pi \ii n)^{2 k}} + \frac{{\cal R}_{2N}^{\tau_n}(p)}{(2 \pi \ii n)^{2 N + 2}}
\end{equation}
where $V^*_{q, S} \to \C$, $p \mapsto \tau^{ae}_{2k}(p)$, $k \ge 1$, and $ V^*_{q, S} \to \C$, $p \mapsto {\cal R}_{2N}^{\tau_n}(p)$
are real analytic. 
As a consequence, choosing $n_0 \ge m_S :=  1+ \mbox{max}\{ n \in S\}$  so that $Re(\tau_n(p)) > 0$ for any $n \ge n_0,$ it follows that 
 $2 \ii \sqrt{\tau_n(p)}$, $n \ge n_0$, admits an expansion of the form 
\begin{equation}\label{costa rica 22}
2\ii\sqrt{\tau_n(p)}  = 2 \pi \ii n \Big(1 + \sum_{k = 2}^N \frac{\sqrt{\tau} _{2k}^{ae}(p)}{(2 \pi \ii n)^{2 k}} +  \frac{{\cal R}_{2N}^{\sqrt{\tau_n}}(p)}{(2\pi \ii n)^{2 N + 2}} \Big)
\end{equation}
where  $V^*_{q, S} \to \C$, $p \mapsto \sqrt{\tau} ^{ae}_{2k} (p)$, $k \ge 2$, and $ V^*_{q, S }\to \C$, $p \mapsto {\cal R}_{2N}^{\sqrt{\tau_n}}(p)$
are real analytic. 
In addition, the remainders ${\cal R}_{2N}^{\tau_n}(p)$ and ${\cal R}_{2N}^{\sqrt{\tau_n}}(p)$ satisfy
$$
\sup_{n \in S_+^\bot} |{\cal R}_{2N}^{\tau_n}(p) | \leq C_{N}\,, \qquad  \sup_{n \ge n_0} |{\cal R}_{2N}^{\sqrt{\tau_n}}(p) |  \leq C_{N}
$$
where the constants $C_{N} > 0$ and $n_0 > m_S$ can be chosen locally uniformly for $p \in V^*_{q, S}$.
\end{lemma}
\begin{proof}
The functions $\tau _{2k}^{ae}(q)$, $k \ge 2,$ are given by polynomial expressions of integrals of densities, involving $q$ and its derivatives up to order $2k$ (cf \cite{KST1})
and are real analytic maps, $H^k_{0, \C} \to \C, q \mapsto \tau _{2k}^{ae}(q)$. Since $\tau_n$ is real analytic on $V^*_{q, S}$ so is
$$
{\cal R}_{2N}^{\tau_n}(p) = (2 \pi \ii  n)^{2N+2} \Big(  \tau_n(p) - n^2 \pi^2 - \sum_{k = 1}^N \frac{\tau^{ae}_{2k}(p)}{(2 \pi \ii n)^{2 k}} \Big)\,.
$$ 
The claimed bounds for ${\cal R}_{2N}^{\tau_n}(p) $ were established in \cite{KST1}. 
In view of the asymptotics of $\tau_n(p)$, one finds $n_0 \ge m_S$ with the claimed properties and then obtains the coefficients $\sqrt{\tau} _{2k}^{ae}(p)$ 
from the expansion of $\tau_n(p)$ in a recursive way and concludes that they are real analytic. Since $ \sqrt{\tau_n(p)}$ is real analytic
one again concludes that the remainder term ${\cal R}_{2N}^{\sqrt{\tau_n}}(p)$ is real analytic as well and deduces the claimed bounds.
\end{proof}
The second result concerns the asymptotic expansion of the coefficients $a_n$, defined in \eqref{costa rica 2}.
\begin{lemma}\label{lemma A3 appendice}
Let $q \in M_S$ and $N \in \Z_{ \ge 0}.$ Then for any $p \in V^*_{q, S}$, $a_{n} (p)$, $n \in S^\bot$,  has an expansion of the form 
\begin{equation}\label{expansion an}
a_n(p) = \ii n \pi + \sum_{k = 0}^N \frac{a_k^{ae}(p)}{(2 \pi \ii n)^k} +  \frac{{\cal R}_{N}^{a_n}(p)}{( 2 \pi \ii n)^{N + 1}}
\end{equation}
where $V^*_{q, S} \to \C$, $p \mapsto a^{ae}_{k} (p)$, $k \ge 0$, are real analytic and $V^*_{q, S} \to \C$, $p \mapsto {\cal R}_{N}^{a_n}(p)$ is analytic. 
In addition, the remainders ${\cal R}_{N}^{a_n}(p)$ satisfy
$$
\sup_{n \in S^\bot} |{\cal R}_{N}^{a_n}(p) |  \leq C_{N}
$$
where the constant $C_{N} > 0$ can be chosen locally uniformly for $p \in V^*_{q, S}$.

\end{lemma}
\begin{proof}
To start with, we compute the leading term in the expansion of $a_n$. Recall that by \eqref{costa rica 2}, 
$$
a_{\pm n} = - \frac{\dot m_1(\tau_n)}{\dot m_2(\tau_n)}  \pm  \ii \frac{\sqrt{(- 1)^{n + 1}\ddot \Delta(\tau_n)/2}}{(- 1)^n \dot m_2(\tau_n)}, \quad n \in S_+^\bot\,.
$$
By \eqref{armadillo 5} - \eqref{armadillo 4}, for any $n \in S_+^\bot$, 
$ 
\dot m_1(\tau_n) = (- 1)^n \int_0^1 y_1(t, \tau_n) y_2(t, \tau_n)\, d t
$ and
$\dot m_2(\tau_n) = (- 1)^n \int_0^1 y_2(x, \tau_n)^2\, d x$,
yielding 
\begin{align}
a_{\pm n} & = - \frac{\int_0^1 y_1(t, \tau_n) y_2(t, \tau_n)\, d t}{\int_0^1 y_2(t, \tau_n)^2\, d t} \pm 
 \ii \frac{1}{\big(2 \int_0^1 y_2(t, \tau_n)^2\, d t \big)^{\frac12}} \sqrt{- \frac{ \ddot \Delta(\tau_n)}{ \dot m_2(\tau_n)}}\,.  \label{costa rica 25}
\end{align}
Since  $\sqrt{\tau_n} = n \pi \big( 1 + O(\frac{1}{n^4}) \big)$ (Lemma~\ref{asymptotics tau}) and hence $y_1(t, \tau_n) = \cos(n \pi t) + O(\frac{1}{n})$ and 
$y_2(t, \tau_n) = \frac{ \sin(n \pi t)}{n\pi} + O(\frac{1}{n^2})$ (cf \cite[Chapter 1]{PT}) one has
\begin{equation}\label{costa rica 27}
\int_0^1 y_1(t, \tau_n) y_2(t, \tau_n)\, d t = O\Big( \frac{1}{n^2}\Big)\,, \qquad 2 \int_0^1 y_2(t, \tau_n)^2\, d t = \frac{1}{n^2 \pi^2} \big( 1 + O (\frac{1}{n})\big)\,.
\end{equation}
To analyze the quotient $- \frac{ \ddot \Delta(\tau_n)}{ \dot m_2(\tau_n)}$ we use the product representation 
of $\dot \Delta(\lambda)$ and $m_2(\lambda)$ (cf Appendix \ref{Birkhoff map}),
$$
\dot \Delta(\lambda) = - \prod_{k \geq 1} \frac{\dot \lambda_k - \lambda}{\pi^2 k^2}\,, \qquad m_2(\lambda) = 
\prod_{k \geq 1} \frac{\mu_k - \lambda}{\pi^2 k^2}\,.
$$
Since for $n \in S_+^\bot$, $\lambda_n^+ = \dot \lambda_n= \mu_n = \tau_n$, 
$$
\ddot \Delta(\tau_n) = \frac{1}{n^2 \pi^2} \prod_{k \neq n} \frac{\dot \lambda_k - \tau_n}{\pi^2 k^2}\,, 
\qquad \dot m_2(\tau_n)  = - \frac{1}{\pi^2 n^2} \prod_{k \neq n} \frac{\mu_k - \tau_n}{\pi^2 k^2}
$$
and one concludes that
$$
- \frac{\ddot \Delta(\tau_n)}{\dot m_2(\tau_n)} = \prod_{k \in S_+} \frac{\dot \lambda_k - \tau_n}{\mu_k - \tau_n} = 
\prod_{k \in S_+} \frac{1 - \frac{\dot \lambda_k}{\tau_n}}{1 - \frac{\mu_k}{\tau_n}}\,.
$$
Altogether we thus have proved that for any $n \in S^{\bot}_+$,
$$
a_{\pm n} = \pm  \ii n \pi + O(1) \quad  \mbox{and } \quad
a_{\pm n} = -  \frac{\int_0^1 y_1(t, \tau_n) y_2(t, \tau_n)\, d t}{\int_0^1 y_2(t, \tau_n)^2\, d t}  \pm
\ii \frac{1}{\big(2 \int_0^1 y_2(t, \tau_n)^2\, d t \big)^{\frac12}} \Big( \prod_{k \in S_+} \frac{1 - \frac{\dot \lambda_k}{\tau_n}}{1 - \frac{\mu_k}{\tau_n}} \Big)^{\frac12}\,.
$$
Expressing $y_1(t, \tau_n)$ and $y_2(t, \tau_n)$ in terms of $y_{N,M}(x, \pm \sqrt{\tau_n})$  (cf \eqref{costa rica 17}, \eqref{costa rica 19}),
one obtains an expansion of the form \eqref{expansion an} where the coefficients  $a^{ae}_{k}$ can be explicitly computed by using the expansion of $y_{N, M}(x, \nu)$, 
obtained in  Lemma \ref{lemma A2 appendice} and  the one of $\sqrt{\tau_n}$ of  Lemma~\ref{asymptotics tau}. 
It follows that for any $k \ge 0$, the map $V^*_{q, S} \to \C$, $p \mapsto a^{ae}_{k} (p)$ is analytic. In case $p$ is real valued one has 
$a_{-n} = \overline{a_n}$ (cf definition \eqref{costa rica 2} of $a_{\pm n}$). By an inductive argument it then follows from the expansion
of $a_{ n}  + a_{-n}$ that for any $k \ge 0,$ the coefficient $a^{ae}_{k}$ is real valued.
With regard to the remainder term, since for any $n \in S^\bot,$ $a_{ n}$ is analytic on 
$V^*_{q, S}$ (cf Remark \ref{analytic extension a_n}) one sees that ${\cal R}_{N}^{a_n}$ is analytic on $V^*_{q, S}$.
The claimed estimates are obtained from the corresponding estimates of Lemma \ref{lemma A2 appendice} and Lemma~\ref{asymptotics tau}.
\end{proof}

\medskip

\noindent
{\bf Proof of Theorem \ref{teorema espansione fn appendice}.} Let $q \in M_S$ and  $N \ge 0$. To prove that for $p \in V^*_{q, S}$,
$f_n(x, p)$ has an expansion of the form \eqref{costa rica 3} we first note that since $y_1(x, \tau_n) =  \cos(n\pi x) + O(\frac{1}{n})$, $y_2(x, \tau_n) =  \frac{1}{n\pi} \sin(n\pi x) + O(\frac{1}{n^2})$,
and $a_{\pm n} = \ii n\pi + O(1)$, one has $f_{\pm n}(x) = e^{\pm \ii \pi n x} + O(\frac{1}{n})$. To obtain the expansion as claimed,
we want to apply Lemma \ref{lemma A2 appendice}. 
Choose $M \ge 0 $ (arbitrarily large) and  $n_0 \geq m_S$ (cf Lemma \ref{asymptotics tau}) so that $Re \, \tau_n > 0$ and in addition
$|\sqrt{\tau_n}| \geq \nu_b$ for any $n \geq n_0$ where $\nu_b \ge 1$ is given by \eqref{costa rica 16 bis}. 
We then substitute the formulas \eqref{costa rica 17} and \eqref{costa rica 19} with $\nu^2 = \tau_n$ 
into the expression \eqref{costa rica 1} for $f_{\pm n}(x) \equiv f_{\pm n}(x, p)$ to get for $n \ge n_0$
\begin{equation}\label{costa rica 21}
\begin{aligned}
f_{\pm n}(x) & = \alpha_{N, M} (- \sqrt{\tau_n}) \, y_{N, M}(x, \sqrt{\tau_n}) + \alpha_{N, M}( \sqrt{\tau_n} \big) \, y_{N, M}(x, - \sqrt{\tau_n}) \\
& \qquad + a_{\pm n} \, \Big( \beta_{N, M}(\sqrt{\tau_n}) \, y_{N, M}(x, \sqrt{\tau_n}) + \beta_{N, M}(- \sqrt{\tau_n}) \, y_{N, M}(x, - \sqrt{\tau_n}) \Big)\,.
\end{aligned}
\end{equation}
Using the expansions of  $\sqrt{\tau_n}$ (Lemma~\ref{asymptotics tau}), $a_{\pm n}$ (Lemma~\ref{lemma A3 appendice}),  
and $y_{N, M}(x, \nu)$ (Lemma~\ref{lemma A2 appendice}) one gets an expansion of $f_{ n}$, $|n| \ge n_0$, of the form \eqref{costa rica 3}
where the coefficients $f^{ae}_k$, $k \ge 1$,
and the remainder ${\cal R}^{f_n}_{ N}$ can be explicitly computed.  One verifies that for any $s \ge 0$ and $k \ge 1$,
$f^{ae}_k:V^*_{q, S} \to H^s_\C[0, 1]$ is analytic. 
Furthermore, by choosing $M$ sufficiently large and using
the estimates of the lemmas referred to above, one obtains the claimed estimate \eqref{costa rica 4} of ${\cal R}^{f_n}_{ N}$ for any $|n| \ge n_0$.
Note that at this point, we only know that $f^{ae}_k(\cdot, p)$ is an element in $H^s_\C[0, 1]$ for any $s \ge 0.$
But since $e^{-\ii \pi n x}  f_n(x)$ is one periodic in $x$, it follows by induction that for any $k \ge 1$, $f^{ae}_k(x)$ is one periodic in $x$ as well. 
Since in case $p$ is real valued, $e^{-\ii \pi n x}  f_n(x) = \overline{ e^{\ii \pi n x}  f_{-n}(x) }$
one reads off from the expansions of  $e^{-\ii \pi n x}  f_n(x)$ and $\overline{e^{\ii \pi n x}  f_{-n}(x)}$ that $f^{ae}_k$, $k \ge 1$, are real valued.
Altogether this shows that for any $s \ge 0$ and $k \ge 1$, $f^{ae}_k:V^*_{q, S} \to H^s_\C$ is real analytic.
For $n \in S^\bot$ with $|n| < n_0,$ we define ${\cal R}^{f_n}_{ N}(x)$ by
$$
{\cal R}^{f_n}_{ N}(x) =  (2 \pi \ii n )^{N + 1} \Big(e^{-\ii \pi n x}  f_n(x) - 1 - \sum_{k = 1}^N \frac{f_k^{ae}(x)}{(2 \pi \ii n)^k} \Big)\,.
$$
We then conclude that for any $n \in S^\bot$, ${\cal R}^{f_n}_{ N}(x)$ is one periodic in $x$. Furthermore,
since for any $n \in S^\bot$ and $s \ge 0$, $e^{-\ii \pi n x}  f_n: V^*_{q, S} \to H^s_\C$ is analytic (cf \eqref{costa rica 1}) it follows that
 ${\cal R}^{f_n}_{ N} : V^*_{q, S} \to H^s_\C$ is analytic as well.
Going through the arguments of the proof one sees that the estimate \eqref{costa rica 4} holds for any $n \in S^\bot$ and that the constant $C_{N,j}$ in
\eqref{costa rica 4} can be chosen  locally uniformly for  $p \in V^*_{q, S}$. \hfill $\square$

\medskip
The next result states how the map $S_{rev}$, defined in  Section \ref{introduzione paper},
acts on the functions $f_n(x, q)$ and how on the coefficients and the remainder of its expansion. 

\noindent
{\bf Addendum to Theorem \ref{teorema espansione fn appendice}.}
{\em For any $q \in M_S$ and $n \in S^\bot$
$$
f_n (x, S_{rev}q) = f_{-n}( - x, q) \,\,\,\big( = (S_{rev}f_{-n})(x, q) \, \big)\,, \quad \forall x \in \R
$$ and as a consequence,
\begin{equation}\label{formula symmetries}
f^{ae}_k (x, S_{rev}q) = (-1)^k f^{ae}_k ( - x, q)\,, \, k \ge 1, \qquad {\cal R}^{f_n}_{ N}(x, S_{rev}q) = (-1)^{N+1} {\cal R}^{f_{-n}}_{ N}( - x, q)\,.
\end{equation}
}
\noindent
{\bf Proof of Addendum to Theorem \ref{teorema espansione fn appendice}.} Let $q \in M_S$ and $n \in S^\bot_+$. By Lemma \ref{lemma 2 reversibilita}, one knows that
$\tau_n(S_{rev}(q)) = \tau_n(q)$ and
\begin{equation}\label{symmetries Floquet solution}
f_{\pm n} (x, S_{rev}q) =  \big( y_1(x) + a_{\pm n} y_2(x)\big) |_{\tau_n, S_{rev}q} = y_1( -x, \tau_n, q) - a_{\pm n}(S_{rev}q) y_2( - x, \tau_n, q)\, \qquad \forall x \in \R
\end{equation}
where   $\tau_n \equiv \tau_n(q)$.
Recall that  $a_{\pm n}$ is given by $a_{\pm n} = - \frac{\dot m_1(\tau_n)}{\dot m_2(\tau_n)}  \pm  \ii \frac{\sqrt{(- 1)^{n + 1}\ddot \Delta(\tau_n)/2}}{(- 1)^n \dot m_2(\tau_n)}$.
Note that again by Lemma \ref{lemma 2 reversibilita}, one has 
$$
m_2(\lambda, S_{rev}q ) = m_2(\lambda, q )\,, \quad  \Delta(\lambda, S_{rev}q) = \Delta(\lambda, q)\,, \quad m_1(\lambda, S_{rev}q ) = m_2'(\lambda, q)\,, \quad \forall \lambda \in \R\,.
$$
Since $\dot \Delta(\tau_n, q) = 0$ and hence  $\dot m_2'(\tau_n, q) = - \dot m_1(\tau_n, q)$ it then follows that $\dot m_1(\tau_n, S_{rev}q) =  - \dot m_1(\tau_n, q)$.
Combining all these identities one concludes that $a_{\pm n}(S_{rev}q) = - a_{\mp n}(q)$ and hence by \eqref{symmetries Floquet solution}
$$
f_{\pm n} (x, S_{rev}q) = y_1( -x, \tau_n, q) + a_{\mp n}(S_{rev}q) y_2( - x, \tau_n, q) = f_{\mp n} ( - x, q)
$$
as claimed. Considering the expansions of the latter identities one obtains \eqref{formula symmetries}. \hfill  $\square$

\medskip

To obtain the asymptotic expansion for $\Psi_L,$ presented in Section~\ref{sezione mappa Psi L},
we need to establish such an expansion for each of the factors appearing in the definition  \eqref{definition Wn} of $W_{\pm n}(q)$ for a finite gap potential $q$. 
First we consider the factor $\xi_n$ which compares the square root of the n'th action
with the n'th gap length.  For any $q \in W$ (cf Theorem \ref{Theorem Birkhoff coordinates}) with $\gamma_n(q) \ne 0,$ it is given by $\sqrt{8 I_n(q) / \gamma_n^2(q)}$.
In case $\gamma_n(q) = 0,$ it can be computed by a limiting argument. By a slight abuse of terminology,
we denote this limit also by $\sqrt{8 I_n(q) / \gamma_n^2(q)}$.

\begin{lemma}\label{lemma appendice espansione xi n}
Let $q \in M_S$ and $N \in \Z_{ \ge 0}.$ Then for any $p \in V^*_{q, S}$,
 $\xi_n(q) := \sqrt{8 I_n(q) / \gamma_n^2(q)}$, $n \in S^\bot_+$, has an expansion of the form
$$
\sqrt{n \pi} \xi_n (p) = 1 + \sum_{k = 1}^N \frac{\xi_{2k}^{ae} (p) }{(2 \pi \ii n)^{2k} } +  \frac{{\cal R}_{2N}^{\xi_n}(p)}{(2 \pi \ii n)^{2 N+ 2}} 
$$
where $V^*_{q, S} \to \C$, $p \mapsto \xi^{ae}_{2k} (p)$, $k \ge 1$, and $ V^*_{q, S} \to \C$, $p \mapsto {\cal R}_{2N}^{\xi_n}(p)$,
are real analytic. 
In addition, the remainders ${\cal R}_{2N}^{\xi_n}(p)$ satisfy
\begin{equation}\label{estimate remainder xi}
\sup_{n \in S_+^\bot} |{\cal R}_{2N}^{\xi_n}(p) |  \leq C_{N}
\end{equation}
where the constant $C_{N} > 0$ can be chosen locally uniformly for $p \in V^*_{q, S}$.
\end{lemma}
\begin{proof}
Let $q \in M_S$ and $N \in \Z_{ \ge 0}$.
Following the proof of \cite[Theorem 7.3]{KP}, for any $p \in V^*_{q, S}$ and $n \in S_+^\bot$, $8 I_n(p) /\gamma_n^2(p)$ can be computed 
by considering a sequence of $S_n$-gap potentials $(p_j)_{j \geq 1}$ in $W$ (cf Theorem~\ref{Theorem Birkhoff coordinates}) with $\gamma_n(p_j) > 0$ so that $p_j \to p$ as $j \to \infty$
where $S_n : = S \cup\{-n , n\}$. 
One then obtains in the limit the formula $8 I_n(p) / \gamma_n^2(p) = \chi(\tau_n(p), p)$ where $\chi(\lambda) \equiv \chi(\lambda, p)$ is given by
$ \frac{1}{\sqrt{\lambda - \lambda_0}} \prod_{k \in S_+} \frac{\lambda -  \dot \lambda_k }{\sqrt{(\lambda_k^+ - \lambda)( \lambda_k^- -\lambda)}}$, implying that
for $n \ge n_0$ with $n_0 \ge m_S$ chosen so that $Re( \tau_n(p) ) > 0$ for any $n \ge n_0$ (cf Lemma \ref{asymptotics tau})
\begin{equation}\label{chi (lambda) sec lemma}
 \chi(\tau_n) = \frac{1}{\sqrt{\tau_n}} \frac{1}{\sqrt{1 - \frac{\lambda_0}{\tau_n}} }
\prod_{k \in S_+} \frac{1 -  \frac{\dot \lambda_k}{\tau_n}} {\sqrt{ (1 - \frac{\lambda_k^+}{\tau_n})( 1 - \frac{\lambda_k^-}{\tau_n})  }}\,.
\end{equation}
Combining  \eqref{chi (lambda) sec lemma}
with the expansion of $\tau_n$ (cf Lemma~\ref{asymptotics tau}) then
yields the expansion
$$
\sqrt{n \pi} \xi_n(p) = \sqrt{n \pi} \sqrt{\chi(\tau_n(p), p)} = 1 + \sum_{k =1}^{N}\frac{\xi_{2k}^{ae}(p)}{(2 \pi \ii n)^{2k}} +  \frac{{\cal R}^{\xi_n}_{2N}(p)}{(2 \pi \ii  n)^{2N+2}}\,. 
$$
where $V^*_{q, S} \to \C$, $p \mapsto \xi^{ae}_{2k} (p)$, $k \ge 1$, are real analytic and $\sup_{n \in S^\bot_+}|{\cal R}^{\xi_n}_{2N}(p)|$ is bounded.
For $n \in S^\bot_+$ with $n < n_0$, ${\cal R}^{\xi_n}_{2N}(p)$ is defined by $(2 \pi \ii  n)^{2N+2}\big( \sqrt{n \pi} \xi_n(p)  - 1 - \sum_{k =1}^{N}\frac{\xi_{2k}^{ae}(p)}{(2 \pi \ii n)^{2k}}\big)$.
Since for any $n \in S^\bot_+$, $\sqrt{n \pi} \xi_n$ is real analytic on $V^*_{q, S}$ , so is $ {\cal R}_{N}^{\xi_n}(p)$. 
Going through the arguments of the proof one sees that the constant $C_N$ in \eqref{estimate remainder xi} can be chosen locally uniformly for $p \in V^*_{q, S}$.
\end{proof}

Next we prove an expansion for the factor $ \dot m_2(\tau_n(q), q) / \ddot \Delta(\tau_n(q), q)$ in  \eqref{definition Wn} for $ q \in M_S$.
More precisely, we show the following
\begin{lemma}\label{m2 tau n delta ddot lemma appendice}
Let $q \in M_S$ and $N \in \Z_{ \ge 0}.$ Then for any $p \in V^*_{q, S}$,
 $d_n(p):=  - \dot m_2(\tau_n(p), p) / \ddot \Delta(\tau_n(p), p)$, $n \in S^\bot_+$,
has an expansion of the form 
$$
d_n(p)= 
1 +  \sum_{k = 1}^N \frac{d_{2k}^{ae}(p)}{(2 \pi \ii n)^{2k}} + \frac{{\cal R}_{2N}^{d_n}(p)}{(2 \pi \ii n)^{2 N + 2}}\,,
$$
where $V^*_{q, S} \to \C$, $p \mapsto d^{ae}_{2k} (p)$, $k \ge 1$, and $ V^*_{q, S} \to \C$, $p \mapsto {\cal R}_{2N}^{d_n}(p)$,
are real analytic. 
In addition, the remainders ${\cal R}_{2N}^{d_n}(p)$ satisfy
$$
\sup_{n \in S_+^\bot} |{\cal R}_{2N}^{d_n}(p) |  \leq C_{N}
$$
where the constant $C_{N} > 0$ can be chosen locally uniformly for $p \in V^*_{q, S}$.
\end{lemma}
\begin{proof}
Let $q \in M_S$ and $N \in \Z_{\ge 0}$ be given. 
For any $p \in V^*_{q, S}$, $m_2(\lambda) \equiv m_2(\lambda, p)$ admits the product representation (cf Appendix \ref{Birkhoff map})
$m_2(\lambda) = \prod_{k \geq 1} \frac{\mu_k - \lambda}{\pi^2 k^2}$
where $(\mu_k)_{k \ge 1}$ denote the Dirichlet eigenvalues of the operator $- \partial_x^2 + p$, listed in lexicograpic order. By \eqref{product represenations}, $\dot \Delta(\lambda)$ also admits such a representation,
$\dot \Delta(\lambda) = - \prod_{k \geq 1} \frac{\dot \lambda_k - \lambda}{\pi^2 k^2}$ with $(\dot \lambda_k)_{k \ge 1}$ being listed in lexicographic order.
Since $(\mu_k)_{k \geq 1}, (\dot \lambda_k)_{k \geq 1}$ are simple 
$$
\dot m_2(\mu_n) = - \frac{1}{\pi^2 n^2} \prod_{k \neq n} \frac{\mu_k - \mu_n}{\pi^2 k^2} \neq 0 \quad {\mbox{ and }} \quad
\ddot \Delta(\dot \lambda_n) = \frac{1}{\pi^2 n^2} \prod_{k \neq n} \frac{\dot \lambda_k - \dot \lambda_n}{\pi^2 k^2} \neq 0\,.
$$ 
For any $n \in S_+^\bot$, one has $\mu_n = \dot \lambda_n = \tau_n$ and hence one concludes that for $n$ sufficiently large so that $Re \, \tau_n > 0,$
\begin{equation}\label{m2 Delta tau n}
- \frac{\dot m_2(\tau_n)}{\ddot \Delta(\tau_n)} = \prod_{k \in S_+} \frac{\mu_k - \tau_n}{\dot \lambda_k - \tau_n} = 
\prod_{k \in S_+} \frac{1 - \frac{\mu_k}{\tau_n}}{1 - \frac{\dot \lambda_k}{\tau_n}}\,.
\end{equation}
Combining this with the results of $\tau_n$ (Lemma~\ref{asymptotics tau}) then yields the expansion of the stated form.
Going through the arguments of the proof one concludes that $d_{2k}^{ae}$, $k \ge 1$, and ${\cal R}^{d_n}_{2N}$ have the claimed properties.
\end{proof}

It remains to prove that the factors $e^{\pm \ii \beta_n(q)}$, appearing in the definition  \eqref{definition Wn} of $W_{\pm n}(x, q)$,
admit an expansion as well. Clearly, it suffices to prove such an expansion for $\beta_n(q)$. 
Recall that for any $q \in M_S$ and $n \in S_+^\bot$, $\beta_n(q)$ is given by
$\beta_n(q) = \sum_{\ell \in S_+} \beta_\ell^n(q)$ (\cite[Theorem 8.5]{KP}) and by \cite[page 70]{KP},
$$
\beta_\ell^n(q) =  \int_{\lambda_\ell^-(q)}^{\mu_\ell^\ast(q)} \frac{\psi_n(\lambda, q)}{\sqrt{\Delta^2(\lambda, q) - 4}} d \lambda, \qquad
\mu_\ell^\ast (q) = (\mu_n(q), \delta(\mu_\ell)) \in \R^2
$$
where $\delta(\l ) \equiv \delta(\l, q)$ denotes the anti-discriminant. Here we used that for any $\ell \in S_+^\bot \setminus{\{n\}}$, one has $\mu_\ell = \lambda_\ell^-$
and hence $\beta_\ell^n(q) = 0$. Furthermore recall that by  \cite[Theorem 8.1]{KP} $\psi_n(\l, q)$ is an entire function of $\l$.
\begin{lemma}\label{asymptotics beta_n}
Let $q \in M_S$ and $N \in \Z_{ \ge 0}.$ After shrinking $V^*_{q, S}$, if needed, it follows that for any $p \in V^*_{q, S}$,
 $\b_n(p)$, $n \in S_+^\bot$, admits an expansion of the form
$$
\b_n(p)=  \frac{1}{n\pi} \sum_{k = 0}^N \frac{\b_{2k}^{ae}(p)}{(2 \pi \ii n)^{2k}} + \frac{{\cal R}_{2N}^{\b_n}(p)}{(2 \pi \ii n)^{2 N + 2}}\,,
$$
where $V^*_{q, S} \to \C$, $p \mapsto \b^{ae}_{2k} (p)$, $k \ge 1$, and $ V^*_{q, S} \to \C$, $p \mapsto {\cal R}_{2N}^{\b_n}(p)$,
are real analytic. 
In addition, the remainders ${\cal R}_{2N}^{\b_n}(p)$ satisfy
$$
\sup_{n \in S_+^\bot} |{\cal R}_{2N}^{\b_n}(p) |  \leq C_{N}
$$
where the constant $C_{N} > 0$ can be chosen locally uniformly for $p \in V^*_{q, S}$.
\end{lemma}
\begin{proof}
Let $p \in V^*_{q, S}$ and $n \in S_+^\bot$. Since $p$ is an $S-$gap potential, it follows from \cite[Theorem 8.5]{KP} that the quotient of $\psi_n(\lambda) \equiv \psi_n(\lambda, p)$
with $\sqrt{\Delta^2(\l) - 4} \equiv \sqrt{\Delta^2(\l, p) - 4}$ is of the form
$$
\frac{\psi_n(\lambda)}{\sqrt{\Delta^2(\l) - 4}} =  \frac{\l^M + s^n_{M-1}\l^{M-1} + \cdots + s^n_{0}}{\sqrt{\mathcal R}} \frac{n\pi}{\tau_n - \l}
$$
where up to a sign, the complex numbers $s^n_{j} \equiv s^n_{j}(p)$, $0 \le j \le M-1$, are the symmetric functions of the roots $\sigma^n_\ell$, $\ell \in S_+,$
of $\psi_n(\l)$, $M = |S|$, and  $\mathcal R(\l) \equiv \mathcal R(\l, p)$ is given by
$$
\mathcal R(\l) = (\l^+_0 - \l) \prod_{j \in S_+}( \l_j^+ - \l)( \l_j^- - \l).
$$
Here we used that for any $k \ne n$ with $\l_k^+ = \l_k^-$, the eigenvalue $\l_k^+$ is also a root of $\psi_n(\l)$
and we listed the roots $\sigma^n_\ell$, $\ell \in S_+,$ in lexicographic order. Without loss of generality
we thus may assume in the sequel that $\l_\ell^+ \ne \l^-_\ell$ for any $\ell \in S_+$. It then follows that for any $\ell \in S_+$,
$$
\b_\ell^n =   \int_{\lambda_\ell^-}^{\mu_\ell^\ast} \frac{\psi_n(\lambda)}{\sqrt{\Delta^2(\lambda) - 4}} d \lambda
= \frac{1}{n \pi }  \int_{\lambda_\ell^-}^{\mu_\ell^\ast}  \frac{\l^M}{\sqrt{\mathcal R}} \frac{n^2\pi^2}{\tau_n - \l} d \l 
+ \frac{1}{n \pi }  \sum_{j = 0}^{M-1} s^n_j    \int_{\lambda_\ell^-}^{\mu_\ell^\ast}  \frac{\l^j}{\sqrt{\mathcal R}} \frac{n^2\pi^2}{\tau_n - \l}    d \l.
$$
Using Lemma \ref{asymptotics tau}, one concludes that  $\frac{n^2\pi^2}{\tau_n - \l} = 1 + O(\frac{1}{n^2})$ and shows in a straightforward way that $\frac{n^2\pi^2}{\tau_n - \l}$
and hence the integrals  $\int_{\lambda_\ell^-}^{\mu_\ell^\ast}  \frac{\l^j}{\sqrt{\mathcal R}} \frac{n^2\pi^2}{\tau_n - \l}    d \l$ admit an expansion
in $\frac{1}{(2\pi \ii n)^{2k}}$, $k \ge 0$, with coefficients and remainder having properties as stated.
It thus remains to show that for any $0 \le j \le M-1$, $s^n_j$ also admits such an expansion.
By the uniqueness statement of \cite[Proposition D.7]{KP} (and after shrinking $V^*_{q, S}$, if needed,)   it follows that $(s^n_j)_{0 \le j \le M-1}$
is the unique solution of the following inhomogeneous, linear $M \times M$ system 
$$
\sum_{j = 0}^{M-1} s^n_j    \int_{\lambda_\ell^-}^{\l_\ell^+}  \frac{\l^j}{\sqrt{\mathcal R}} \frac{n^2\pi^2}{\tau_n - \l}    d \l 
= - \int_{\lambda_\ell^-}^{\l_\ell^+}  \frac{\l^M}{\sqrt{\mathcal R}} \frac{n^2\pi^2}{\tau_n - \l} d \l \,, \quad \forall \ell \in S_+.
$$
It then follows that $\det(E^n) \ne 0$ where $E^n \equiv E^n(p)$ denotes the $M\times M$ matrix with coefficients 
$$
E^n_{\ell j} = \int_{\lambda_\ell^-}^{\l_\ell^+}  \frac{\l^j}{\sqrt{\mathcal R}} \frac{n^2\pi^2}{\tau_n - \l}    d \l  \,, \qquad \ell \in S_+, \,\, 0 \le j \le M-1.
$$
Therefore, 
$$
(s^n_j)_{0 \le j \le M-1} = - (E^n)^{-1} (b^n_\ell)_{\ell \in S_+}, \qquad \quad  b^n_\ell = \int_{\lambda_\ell^-}^{\l_\ell^+}  \frac{\l^M}{\sqrt{\mathcal R}} \frac{n^2\pi^2}{\tau_n - \l} d \l, \quad \ell \in S_+.
$$
Using once again the expansion of $\tau_n$ of Lemma \ref{asymptotics tau} one shows that $s^n_j$, $0 \le j \le M-1$, admit an expansion
in $\frac{1}{(2\pi \ii n)^{2k}}$, $k \ge 0$, with coefficients and remainder having properties as stated. As an aside, we remark that by
 \cite[Proposition D.7]{KP}, $\lim_{n \to \infty} \sigma^n_\ell = \dot \lambda_\ell$, $\ell \in S_+$, and hence $\lim_{n \to \infty} s^n_j =  s_j$
for any $0 \le j \le M-1$ where up to signs, $(s_j)_{0 \le j \le M-1}$ are the symmetric polynomials of $\dot \l_\ell,$ $\ell \in S_+$. 
One then concludes that $\sigma^n_\ell = \dot \lambda_\ell + O(\frac{1}{n^2})$, $\ell \in S_+$, and in turn $s^n_j = s_j + O(\frac{1}{n^2})$, $0 \le j \le M-1$.
\end{proof}
We finish this appendix by proving an expansion of the KdV frequencies $\omega_n \equiv \omega_n^{kdv}$ (cf Section \ref{introduzione paper}) at finite gap potentials. 
Using the Birkhoff map, we view them as functions of the potential, which by a slight abuse of notation, we denote also by $\omega_n$.
\begin{lemma}\label{Lemma appendice espansione frequenze}
Let $q \in M_S$ and $N \in \Z_{ \ge 0}.$ Then for any $p \in V^*_{q, S}$, the KdV frequencies $\omega_n(p)$, $n \in S_+^\bot$,
have an expansion of the form 
\begin{equation}\label{expansion frequencies}
\omega_n(p)= (2 \pi n)^3 + \sum_{k = 1}^N \frac{\omega_{2k-1}^{ae}(p)}{(2 \pi n)^{2k-1}} + \frac{{\cal R}_{2N}^{\omega_n}(p)}{(2 \pi n)^{2 N + 1}}
\end{equation}
where  $V^*_{q, S} \to \C$, $p \mapsto \omega^{ae}_{2k-1} (p)$, $k \ge 1$, and $ V^*_{q, S} \to \C$, $p \mapsto {\cal R}_{2N}^{\omega_n}(p)$,
are real analytic. 
In addition, the remainders ${\cal R}_{2N}^{\omega_n}(p)$ satisfy
$$
\sup_{n \in S_+^\bot} |{\cal R}_{2N}^{\omega_n}(p) |  \leq C_{N}\,,
$$
where the constant $C_{N} > 0$ can be chosen locally uniformly for $p \in V^*_{q, S}$.
\end{lemma}
\begin{proof}
Let $q \in M_S$ and $N \in \Z_{ \ge 0}$ be given. 
The basic ingredient into our proof of \eqref{expansion frequencies} are formulas of the frequencies
in terms of periods of an Abelian differential of the second kind on the hyperelliptic Riemann surface $\Sigma_p$, 
associated to the periodic spectrum of $L_p = - \partial^2_{x} + p$ (see \cite{Dubrovin}, \cite{DKN}, \cite{IM}, \cite{KT}, \cite{MM}). 
We follow \cite[Section 2]{KT} and note that the arguments made there extend to complex valued potentials: for any $p \in V^*_{q, S}$, denote by $\Sigma_p$ the compact Riemann surface 
associated to the simple periodic eigenvalues  of $p$, 
$$
\Sigma_p := \big\{ (\lambda, \mu) \in \C^2 : \mu^2 = 
(\lambda - \lambda_0) \prod_{j \in J}(\lambda - \lambda_j^+)(\lambda - \lambda_j^-) \big\} \cup \{ \infty \}
$$
where $J \equiv J(p) := \{ j \in S_+ \, : \, \lambda^+_j(p) \ne \lambda^-_j(p) \}$. The variable $z \in \C$ around the point $z = 0$ gives a complex chart in a neighborhood of the 
branch point $\infty \in \Sigma_p$ via the substitution $\lambda = - \frac{1}{z^2}$. By construction, this chart is defined uniquely up to a change of sign 
of the variable $z$, $z \mapsto - z$, and is referred to as standard chart. Then $\Sigma_p$ admits an Abelian differential $\Omega_4$ 
of the second kind, uniquely determined by the following properties: (i) $\Omega_4$ is holomorphic on $\Sigma_p \setminus \{ \infty \}$;
(ii) near $\infty$, $\Omega_4$ is of the form 
$$
\Omega_4 = \frac{1}{z^4}\, d z + h(z)\, d z, \qquad h\,\,\text{holomorphic near}\, z = 0
$$
in the appropriate standard chart; (iii) $\int_{a_j} \Omega_4 = 0$ for any $ j \in J$
where $a_j$ are the smooth cycles around the gap $[\lambda_j^- , \lambda_j^+]$ defined in \cite[Section 2]{KT}. 
The differential $\Omega_4$ is of the form 
\begin{equation}\label{bubble 0}
\frac{\ii}{2} \frac{\lambda^{M + 1} + c_1 \lambda^M + \cdots + c_{M + 1}}{\sqrt{(\lambda - \lambda_0^+)
\prod_{j \in J} (\lambda - \lambda_j^+)(\lambda - \lambda_j^-)}}\, d \lambda 
\end{equation}
where $M:= |J(p)|$ and the coefficients $c_1, \ldots , c_{M+1}$ are real analytic functions on $V^*_{q, S}$. Then by \cite[formula (2.19)]{KT}, 
$$
\omega_n = 12 \ii \int_{b_n} \Omega_4 \,, \quad \forall n \geq 1\,,
$$
where $b_n$, $n \geq 1$, are the cycles as defined in \cite[Section 2]{KT}. Let 
$m_S := 1 + {\rm max}\{ k \in S\}$. Then for any $ n \ge m_S$, $\lambda_{n}^- = \lambda_{n}^+ = \tau_{n}$.
It then follows from the definition of the cycles $b_n$ that for any $n \ge m_S$
\begin{equation}\label{bubble 1}
\omega_n =  \omega_{m_S}  +  12 \ii  \big(2 \int_{\tau_{m_S}}^{\tau_n} \Omega_4 \big)
\end{equation}
with the appropriate choice of the root in the denominator of $\Omega_4$. 
The abelian integral $\int_{\tau_{m_S}}^\lambda \Omega_4$ has an expansion as $\lambda \to \infty$ of the form 
$$
\begin{aligned}
\int_{\tau_{m_S}}^\lambda \Omega_4 & = b_* \lambda^{\frac32} + b_0 \lambda^{\frac12} + b_1 \frac{1}{\lambda^{\frac12}} + \cdots + b_{**} 
= \lambda^{\frac12} \big( b_* \lambda + b_0 + b_1 \frac{1}{\lambda}+ \cdots  \big) + b_{**}
\end{aligned}
$$
and hence as $n \to \infty$ 
\begin{equation}\label{int tau mS tau n}
\int_{\tau_{m_S}}^{\tau_n} \Omega_4 = b_{**} + \sqrt{\tau_n} \big(b_* \tau_n + b_0 + b_1 \frac{1}{\tau_n} + \cdots \big)\,.
\end{equation}
In view of the formula \cite[(2.20)]{KT} of $\Omega_4$, the coefficients $b_* ,b_{**}, b_0, b_1, \ldots$ are real analytic functions on $V^*_{q, S}$. 
Furthermore, it is well known 
(cf e.g. \cite[Proposition 8.1]{KST2}) that since $\int_0^1 p(x)\, d x = 0$
\begin{equation}\label{asintotica rozza omega n}
\omega_n = (2 \pi n)^3 + O(\frac{1}{n})\,. 
\end{equation}
Combining \eqref{bubble 1} - \eqref{asintotica rozza omega n} 
with the results on $\tau_n$ and  $\sqrt{\tau_n}$  of Lemma~\ref{asymptotics tau}
one obtains an expansion of $\omega_n$, $n \ge m_S$, of the form \eqref{expansion frequencies}
where $\omega_{2k-1}^{ae}: V^*_{q, S} \to \C$, $k \ge 1$, and ${\cal R}^{\omega_n}_{2N}: V^*_{q, S} \to \C$, $n \ge m_S$, are real analytic
and ${\cal R}^{\omega_n}_{2N}: V^*_{q, S} \to \C$, $n \ge m_S$, have the claimed bounds. 
For $n \in S^\bot_+$ with $n < m_S$, one defines ${\cal R}^{\omega_n}_{2N}$ by \eqref{expansion frequencies} and
since $\omega_n$ are real analytic on $V^*_{q, S}$, one then concludes  that ${\cal R}^{\omega_n}_{2N}$, $n \in S^\bot_+$,
are real analytic on $V^*_{q, S}$ and satisfy the claimed bounds.
\end{proof}


\section{Reversibility structure}\label{AppendixReversability}

In this appendix we prove that the Birkhoff map $\Phi^{kdv}$ and hence also its inverse $\Psi^{kdv}$ preserve the reversible structure,
defined by the maps 
$$
 S_{rev}: L^2_0 \to L^2_0\,, \,\,\, ( S_{rev} q)(x) := q(-x) \quad  {\mbox{ and }} \quad
\mathcal S_{rev}: h^0_0 \to h^0_0\,, \,\,\, (\mathcal S_{rev} w)_n := w_{-n}\,, \,\, \forall n \ne 0.
$$
\begin{proposition}\label{proposizione 1 reversibilita}
One has 
$$
\Phi^{kdv} \circ { S}_{rev} = \mathcal S_{rev} \circ \Phi^{kdv}\,. 
$$
As a consequence, ${ S}_{rev} \circ \Psi^{kdv} = \Psi^{kdv} \circ \mathcal S_{rev}$ and by the chain rule, for any $q \in L^2_0(\T)$ and $w \in h^0_0$
$$
(d_{{ S}_{rev}(q)} \Phi^{kdv}) \circ { S}_{rev} = \mathcal S_{rev} \circ d_q \Phi^{kdv}\,, \qquad  (d_{{\mathcal S}_{rev}(w)} \Psi^{kdv}) \circ { \mathcal S}_{rev} =  S_{rev} \circ d_w \Psi^{kdv}\,.
$$ 
\end{proposition}

First we establish some preliminary results. Recall that $y_j(x, q) \equiv y_j(x, \lambda, q)$, $j = 1, 2$, denote the fundamental solutions of $- y'' + q y = \lambda y$,
$\Delta(\lambda) \equiv \Delta(\lambda, q)$ the discriminant, and $\delta(\lambda) \equiv \delta(\lambda, q)$ the anti-discriminant, 
$$
\Delta(\lambda) = y_1(1, \lambda) + y_2'(1, \lambda)\,, \qquad \delta(\lambda) = y_1(1, \lambda) - y_2'(1, \lambda)\,.
$$
In a straightforward way, one verifies the following 
\begin{lemma}\label{lemma 2 reversibilita}
For any $q \in L^2_0$, $\lambda \in \C$, $x \in \R$, 
$$
y_1(x, \lambda, { S}_{rev} q) = y_1(- x, \lambda, q)\,, \qquad     y_2(x, \lambda, { S}_{rev} q) = - y_2(- x, \lambda, q)
$$ 
or alternatively,
$$
y_1(x, \lambda, { S}_{rev} q) = \big( y_2'(1) y_1(1 - x) - y_1'(1) y_2(1 - x)\big)|_{\lambda, q}\,, \qquad
y_2(x, \lambda, { S}_{rev}q) = \big( y_2(1) y_1(1 - x) - y_1(1) y_2(1 - x) \big)|_{\lambda, q}\,. 
$$
The latter identities imply that
\begin{equation}\label{formula 1 reversibilita}
\Delta(\lambda, { S}_{rev} q) = \Delta(\lambda, q), \qquad \delta(\lambda, {S}_{rev} q) = - \delta(\lambda, q), \qquad
y_2(1, \lambda, { S}_{rev} q) = y_2(1, \lambda, q)\,. 
\end{equation}
\end{lemma}

\smallskip

An immediate consequence of  the first identity in \eqref{formula 1 reversibilita} is that 
\begin{equation}\label{formula 3 reversibilita}
\lambda^{+}_0({S}_{rev} q) = \lambda^{+}_0(q)\,, \quad \lambda^{\pm}_n({S}_{rev} q) = \lambda^{\pm}_n(q), \,\, \forall n \geq 1, \qquad \gamma_n({ S}_{rev} q) = \gamma_n(q), \,\, \forall n \geq 1\,. 
\end{equation}
Moreover by \eqref{formula 1 reversibilita},
\begin{equation}\label{formula 4 reversibilita}
\mu_n({ S}_{rev} q) = \mu_n(q)\,, \quad   \delta(\mu_n, { S}_{rev} q) = - \delta(\mu_n, q)\,, \qquad \forall n \geq 1\,. 
\end{equation}
For any $q \in L^2_0$, the action variables $I_n \equiv I_n(q)$, $n \geq 1$, are defined by contour integrals (cf. \cite[p 64]{KP}),
$$
I_n = \frac{1}{\pi} \int_{\Gamma_n} \frac{\lambda \dot \Delta(\lambda)}{\sqrt{\Delta(\lambda)^2 - 4}}\, d \lambda\,.
$$
Furthermore the normalizing factor $\xi_n \equiv \xi_n(q)$, defined for $q \in L^2_0$ with $\gamma_n(q) > 0$ by $\xi_n = \sqrt{8 I_n / \gamma_n^2}$, 
extends analytically to $L^2_0$ (cf \cite[Theorem 7.3]{KP}). By \cite[Theorem 8.5]{KP}, $\beta_n = \sum_{k \neq n} \beta_k^n$ is well defined on $L^2_0$ 
where $\beta_k^n \equiv \beta_k^n(q)$ is given by (cf \cite[p 70]{KP}) 
$$
\beta_k^n = \int_{\lambda_{k}^-}^{\mu_k^*} \frac{\psi_n(\lambda)}{\sqrt{\Delta^2(\lambda) - 4}}\, d \lambda\,, \qquad \mu_k^* = (\mu_k, \delta(\mu_k))
$$
with the sign of $\sqrt{\Delta^2(\lambda) - 4}$ determined by $\sqrt[*]{\Delta(\mu_k) - 4} = \delta(\mu_k)$. On the other hand, $\eta_n \equiv \eta_n(q)$ and $\theta_n \equiv \theta_n(q)$ 
are well defined  modul $2\pi$ on $L^2_0 \setminus Z_n$ by 
$$
\eta_n = \int_{\lambda_n^-}^{\mu_n^*} \frac{\psi_n(\lambda)}{\sqrt{\Delta^2(\lambda) - 4}}\,d \lambda \,, \qquad \theta_n = \eta_n + \beta_n\,,
$$
where $Z_n = \big\{q \in L^2_0 : \gamma_n(q) = 0 \big\}.$
One then concludes from \eqref{formula 1 reversibilita}, \eqref{formula 3 reversibilita}, \eqref{formula 4 reversibilita} that the following holds. 
\begin{corollary}\label{corollario 3 reversibilita}
For any $q \in L^2_0$ and $n \geq 1$, 
\begin{equation}\label{formula 5 reversibilita}
I_n({ S}_{rev}q) = I_n(q)\,, \qquad \xi_n({ S}_{rev} q) = \xi_n(q)\,, \qquad
\beta_n({ S}_{rev} q) = - \beta_n(q)\,.
\end{equation}
Furthermore on $L^2_0 \setminus Z_n$, $\theta_n({ S}_{rev} q) = -  \theta_n(q)$ modulo $2\pi$.
\end{corollary}

With these preparations made we now prove Proposition \ref{proposizione 1 reversibilita}. 
\smallskip

\noindent
{\it Proof of  Proposition \ref{proposizione 1 reversibilita}.} For any $n \geq 1$ and $q \in L^2_0 \setminus Z_n$, the complex Birkhoff coordinates $z_n(q), z_{- n}(q)$ are given by 
$z_n(q) = \sqrt{n \pi} \sqrt{I_n(q)} e^{- \ii \theta_n(q)}$, $z_{- n}(q) = \sqrt{n \pi} \sqrt{I_n(q)} e^{\ii \theta_n(q)}$, whereas for $q \in Z_n$, $z_n(q) = 0$ and $z_{- n}(q) = 0$. 
Hence it follows from Corollary \ref{corollario 3 reversibilita} that $z_n({ S}_{rev} q) = z_{- n}(q)$ and $z_{- n}({ S}_{rev} q) = z_n(q)$ for any $n \geq 1$. This proves that $\Phi^{kdv} \circ { S}_{rev} = \mathcal S_{rev} \circ \Phi^{kdv}$. 
\hfill $\square$


\section{Properties of pseudodifferential and paradifferential calculus}\label{appendice B}
In this appendix we collect some well known facts about pseudodifferential and paradifferential calculus on the torus. We refer to \cite{Met} for further details.
Let $\chi \in C^\infty(\R^2, \R)$ be an admissible cut-off function. It means that $\chi$ is an even function
and that there exist
$0 < \e_1 < \e_2 < 1$ so that for any $(\vartheta, \eta) \in \R^2$ and $\alpha, \beta \in \Z_{\ge 0}$,
\begin{equation}\label{cut-off paraprodotto}
\chi(\vartheta, \eta) = 1, \quad \forall |\vartheta| \leq \e_1 +  \e_1|\eta|\,, \quad \chi(\vartheta, \eta) = 0, \quad \forall |\vartheta| \geq \e_2 +  \e_2 |\eta|\,,
\end{equation}
\begin{equation}\label{order 0 cut-off}
|\partial^\alpha_\vartheta \partial^\beta_\eta \chi(\vartheta, \eta)) | \le C_{\alpha, \beta} (1 +  |\eta|)^{-\alpha -\beta}\,.
\end{equation}
For any $a \in H^1$, the paraproduct $T_a u$ of the function $a$  with $u \in L^2$ (with respect to the cut-off function $\chi$) is defined as
\begin{equation}\label{definition prarproduct}
(T_a u) (x) := \sum_{k, n \in \Z} \chi(k, n) a_k u_n e^{\ii 2 \pi (k + n) x}
\end{equation}
where we recall that $u_n$, $n \in \Z$, denote the Fourier coefficients of $u$, $u_n = \int_0^1 u(x) e^{- 2 \pi \ii n x}\, d x$. 
Note that since $a$, $u$, and $\chi$ are real valued and $\chi$ is even, $T_a u$ is real valued as well.
Given any $s, s' \in \Z$, we denote by ${\cal B}(H^s, H^{s'})$ the Banach space of all bounded linear operators
$H^s \to H^{s'}$, endowed with the operator norm $\| \cdot \|_{{\cal B}(H^s, H^{s'})}$. 
In case $s=s',$ we also write ${\cal B}(H^s)$ instead of ${\cal B}(H^s, H^{s})$. Given any linear operator $A \in {\cal B}(H^s, H^{s'})$,
we denote by $A^t$ the transpose of $A$ with respect to the $L^2-$inner product. It is an element in ${\cal B}((H^{s'})^*, (H^{s})^*)$
where $(H^{s})^*$ denotes the dual of $H^{s}$.

\begin{lemma}\label{primo lemma paraprodotti}
$(i)$ For any $s \in \Z_{\ge 0}$ and $a \in H^1,$ the linear operator $T_a: u \mapsto  T_a u$ is in ${\cal B}(H^s)$. Furthermore 
the linear map $H^1 \to  {\cal B}(H^s),$ $a \mapsto T_a$, is bounded,  $\| T_a\|_{{\cal B}(H^s)} \lesssim_s \| a \|_{1}$. 

\noindent
$(ii)$ Let $a \in H^{s_1}, b \in H^{s_2}$ and $s_1, s_2 \in \Z_{ \geq 1}$. Then 
$$
a b = T_a b + T_b a + {\cal R}^{(B)}(a, b)
$$
where the bilinear map ${\cal R}^{(B)}: H^{s_1} \times H^{s_2} \to H^{s_1 + s_2 - 1}$, $(a, b) \mapsto {\cal R}^{(B)}(a, b)$, is continuous 
and satisfies the estimate 
$$
\| {\cal R}^{(B)}(a, b)\|_{s_1 + s_2 -  1} \lesssim_{s_1, s_2} \| a \|_{s_1} \| b \|_{s_2}\,. 
$$

\noindent
$(iii)$ Let $a \in H^\rho$ with $\rho \in \Z_{ \ge 2}$. Then for any $s \geq 0$, $T_a^t - T_{ a} \in {\cal B}(H^s, H^{s + \rho - 1})$ and
$$
\| T_a^t - T_{ a}\|_{{\cal B}(H^s, H^{s + \rho - 1})}  \lesssim_{s, \rho}  \| a \|_{\rho}\,.
$$

\noindent
$(iv)$ Let $a, b \in H^\rho$ with $\rho \in \Z_{ \ge 1}$. Then for any $s \geq 0$, $T_a \circ T_b - T_{ab} \in {\cal B}\big( H^s , H^{s + \rho - 1} \big)$ and 
$$
\| T_a \circ T_b - T_{ab} \|_{{\cal B}(H^s, H^{s + \rho - 1})} \lesssim_{s, \rho} \| a \|_\rho \| b \|_\rho\,. 
$$
\end{lemma}


 
\begin{lemma}\label{lemma composizione pseudo}
$(i)$ Let $k, j \in \Z_{\geq 0}$ and $a \in C^\infty(\T)$. Then for any $s \in \Z_{\ge 0}$ and $N \in \N$ with $N \geq  k+ j$, the composition
$\partial_x^{- k} \circ a \partial_x^{- j}$ is a bounded linear operator  $H^s \to H^{s+k+j}$ which admits an expansion of the form
$$
\partial_x^{- k} \circ a  \partial_x^{- j} = a\partial_x^{-k -j} + \sum_{i = 1}^{N - k - j} C_i(k, j) \,  (\partial_x^i a ) \, \partial_x^{- k -j - i} + {\cal R}^{\psi do}_{N, k, j}(a)
$$
where $C_i(k, j)$, $1 \le i \le N-k-j$, are constants depending on $k, j$ and the remainder ${\cal R}^{\psi do}_{N, k, j}(a)$ is a bounded linear operator 
$H^s \to  H^{s + N + 1}$,
 satisfying the estimate 
\begin{equation}\label{stima resti lemma astratto OPS}
\| {\cal R}^{\psi do}_{N, k, j}(a) \|_{{\cal B}(H^s, H^{s + N + 1})} \lesssim_{s, N}  \| a \|_{s + 2N }\,.
\end{equation} 

\noindent
$(ii)$ Let $k, j \in \Z_{\geq 0}$ and $N \geq k + j$. There exists a constant $\sigma_N > N - k - j +1$ such that for any $a \in H^{\sigma_N}$
and any $s \in \Z_{\ge 0}$, the composition $\partial_x^{- k} \circ T_a  \circ \partial_x^{- j}$ is a bounded linear operator 
$H^s \to H^{s+k+j}$ which admits an expansion of the form
$$
\partial_x^{- k} \circ T_a  \circ \partial_x^{- j} =  T_{ a} \partial_x^{- k -j} +
\sum_{i= 1}^{N - k - j} C_i(k, j) T_{\partial_x^i a} \partial_x^{- k -j - i} + {\cal R}_{N, k, j}^{(B)}(a)
$$
where $C_i(k, j)$, $1 \le i \le N-k-j$, are constants depending on $k, j$ and for any $s \geq 0$, the remainder ${\cal R}_{N, k, j}^{(B)}(a)$
is a bounded linear operator $H^s \to  H^{s + N+1}$, satisfying the estimate 
\begin{equation}\label{stima resti lemma astratto OPS}
\| {\cal R}_{N, n, k}^{(B)}(a) \|_{{\cal B}(H^s, H^{s + N + 1})} \lesssim_{s, N}  \| a \|_{ \sigma_N }\,. 
\end{equation}
\end{lemma}

Finally, we record the following well known tame estimates of products of functions in Sobolev spaces. 
\begin{lemma}\label{lemma interpolation}
For any $s \in \Z_{\ge 1}$,
$$
\| u v \|_s \lesssim_s \| u \|_s \| v \|_1 + \| u \|_1 \| v \|_s\,, \quad \forall u, v \in H^s\,.
$$
\end{lemma}

\vspace{1.0cm}

\noindent
T. Kappeler, 
Institut f\"ur Mathematik, 
Universit\"at Z\"urich, Winterthurerstrasse 190, CH-8057 Z\"urich;\\
${}\qquad$  email: thomas.kappeler@math.uzh.ch \\

\noindent
R. Montalto, 
Department of Mathematics,  
University of Milan, Via Saldini 50, 20133, Milano, Italy;\\
${}\qquad$ email: riccardo.montalto@unimi.it

\noindent
and Institut f\"ur Mathematik, 
Universit\"at Z\"urich, Winterthurerstrasse 190, CH-8057 Z\"urich;\\
${}\qquad$  email: riccardo.montalto@math.uzh.ch \\

\end{document}